\documentclass[reqno,10pt]{amsart}
\usepackage{amssymb,amsmath,amsthm,amsfonts,color}

\usepackage{mathrsfs,dsfont, comment,mathscinet}

\usepackage{enumerate,esint}
\usepackage[left=2.8cm,right=2.8cm,top=2.8cm,bottom=2.8cm]{geometry}
\usepackage{mathtools}
\usepackage{graphicx}
\usepackage{bbm}
 \usepackage[abs]{overpic}
 \usepackage{tikz} %for the figures
 \usetikzlibrary{arrows, arrows.meta, patterns}%to draw arrows in pictures
\usepackage{xcolor} %for the cube
\usepackage{caption,subcaption,pgfplots} % for the zooming in

\usepackage{epstopdf}
\usepackage[colorlinks=true, pdfstartview=FitV, linkcolor=blue, 
            citecolor=blue, urlcolor=blue]{hyperref}

\allowdisplaybreaks
\usetikzlibrary{shapes,snakes,spy,calc,shapes.geometric} % to draw ellipses the first two, the third and fourth to zoom in 
\numberwithin{equation}{section}
\theoremstyle{plain}
\newtheorem{theorem}{Theorem}[section]
\newtheorem{proposition}[theorem]{Proposition}
\newtheorem{lemma}[theorem]{Lemma}
\newtheorem{example}[theorem]{Example}

\theoremstyle{definition}
\newtheorem{definition}[theorem]{Definition}
\newtheorem{remark}[theorem]{Remark}

\newcommand\R{\mathbb R}

\newcommand\M{\mathbb M}
\newcommand\N{\mathbb N}

\newcommand\ep{\varepsilon}
\newcommand\eps{\varepsilon}

\newcommand\dist{{\rm dist}}

\newcommand\wk{\rightharpoonup}

\newcommand\EEE{\color{black}}

\title[Two-well linearization for solid-solid phase transitions]
{Two-well linearization for solid-solid phase transitions}
\author[E. Davoli] {Elisa Davoli} 
\address{Institute of Analysis and Scientific Computing, Vienna University of Technology, 
Wiedner Hauptstrasse 8-10, 1040 Vienna, Austria}
\email{elisa.davoli@tuwien.ac.at}
\author[M. Friedrich] {Manuel Friedrich} 
\address{Institute for Computational and Applied Mathematics, University of M\"unster, 
Einsteinstrasse 62, D-48149 M\"unster, Germany}
\email{manuel.friedrich@uni-muenster.de}
\begin{document} 
\vskip .2truecm
\begin{abstract}

In this paper we consider nonlinearly elastic, frame-indifferent, and singularly perturbed two-well models for materials undergoing solid-solid phase transitions in any space dimensions, and we  perform a simultaneous passage to sharp-interface and small-strain limits. Sequences of deformations with equibounded energies are decomposed  via suitable Caccioppoli partitions  into the sum of piecewise constant rigid movements and suitably rescaled displacements. These  converge to limiting  partitions, deformations, and displacements, respectively. Whereas limiting deformations are simple laminates whose gradients only attain two values, the limiting displacements belong to the class of special functions with bounded variation ($SBV$). The latter  feature elastic contributions measuring the distance to simple laminates, as well as jumps associated to two consecutive phase transitions  having  vanishing distance, and thus not being detected by the limiting deformations.  By $\Gamma$-convergence we identify an  effective limiting model given by the sum of a quadratic linearized elastic energy in terms of displacements along with two surface terms. The first one is proportional to the total length of interfaces created by jumps in the gradient of the limiting deformation. The second one is proportional  to twice the total length of interfaces created by jumps in the limiting displacement,  as well as by the boundaries of limiting partitions.  A main tool of our analysis is a novel two-well rigidity estimate which has been derived in  \cite{davoli.friedrich} for  a model with anisotropic second-order perturbation.

%
%We analyze  two-well  linearization around a laminate in the setting of $d$-dimensional solid-solid phase transitions, \BBB $d\geq 2$, for a model augmented by \EEE a suitable \BBB anisotropic \EEE penalization of second variations.
%Sequences of deformations with equibounded energies are shown to be decomposable into the sum of piecewise constant rigid movements and scaled equibounded elastic displacements. The penalization is proven to be sharp in order to ensure that the limiting displacement, as the transition parameter vanishes, belongs to the class of special functions with bounded variation ($SBV$). \RRR {I dont understand this sentence.} \EEE  An effective limiting model is deduced by $\Gamma$-convergence and shown to be given by the sum of a quadratic linearized elastic energy with two surface terms, corresponding, respectively, to the total length of interfaces created by jumps in the gradient of the limiting deformation, and to twice the total length of interfaces created by jumps in the limiting displacement. \BBB These surface energies are weighted by a suitable constant $K$ representing the energy associated to a profile performing a single phase transition in an energetically optimal way. 

\end{abstract}
\maketitle

%%%%%%%%%%%%%%%%%%%%%%%%%%%%%%%%%%%%%%%%%%
\section{Introduction}

Solid-solid phase changes are the physical phenomena for which, by strong temperature or pressure variations, a  solid can modify its crystalline structure without undergoing any intermediate liquid phase. 
Well-known examples are temperature-induced phase transitions between martensite and austenite in shape-memory alloys (see, e.g., \cite{bhattacharya.kohn, cheng}), the nucleation of different ice forms at elevated pressure, or the mechanisms behind the evolution of graphite into diamond in carbon composites.

In this paper we  focus on materials exhibiting exactly two different phases by considering nonlinearly elastic, frame-indifferent, and singularly perturbed two-well models in any space dimensions. Our goal is to perform a simultaneous passage from nonlinear-to-linearized elastic energies and from diffuse-to-sharp interface descriptions of solid-solid phase transitions. We start by introducing the modeling assumptions and discussing the background. Afterwards, we describe our main results.  

%
% In order to describe our results, we need to introduce some basic notation. We focus here on materials exhibiting exactly two different admissible phases.

In the setting of nonlinear elasticity, the coexistence of two phases can be mathematically described by variational two-well problems, based on the study of energy functionals of the form

\begin{align}\label{eq: basic energy-intro}
y\in H^1(\Omega;\mathbb{R}^d)\to \int_{\Omega}W(\nabla y)\,{\rm d}x.
\end{align}
In the expression above, $\Omega\subset \mathbb{R}^d$, $d\in \mathbb{N}$,  is a bounded Lipschitz domain, representing the reference configuration of a material undergoing a solid-solid phase transition between phases $A,B\in \mathbb{M}^{d\times d}$. (Here, $\mathbb{M}^{d\times d}$ is the set of real $d\times d$ matrices.) The stored energy density $W\colon\mathbb{M}^{d\times d}\to [0,+\infty)$ in \eqref{eq: basic energy-intro} is a nonlinear, frame-indifferent function whose zero set has the two-well structure
\begin{equation*}
%\label{eq:zset}
\{F\in \mathbb{M}^{d\times d}\colon\, W(F)=0\}=SO(d)A\cup SO(d)B,
\end{equation*}
with $SO(d)$ denoting the set of proper rotations in $\mathbb{M}^{d\times d}$. The model in \eqref{eq: basic energy-intro} is disadvantaged by a quite unphysical drawback. In fact, whenever $A$ and $B$ are rank-one connected, low energy sequences for generic boundary value problems are known to possibly exhibit highly oscillatory behaviors. 
In order to prevent this effect, `phenomenological' higher order regularizations are often incorporated in the energy functional.  These may be interpreted as surface  energies  penalizing the transition between different energy wells. A concrete example is provided by the following \emph{diffuse-interface model}, where transitions between the two wells $SO(d)A$ and $SO(d)B$ are controlled by augmenting \eqref{eq: basic energy-intro} via a second-order singular perturbation:
\begin{equation}
\label{eq:sol-sol-intro}
 y\in H^2(\Omega;\R^d) \to I_{\ep}(y):=\frac{1}{\ep^2}\int_{\Omega}W(\nabla y)\,{\rm d}x+\ep^2\int_{\Omega}|\nabla^2 y|^2\,{\rm d}x.
\end{equation}
 The competition between the two energy contributions in \eqref{eq:sol-sol-intro} is tailored by the smallness parameter $\ep>0$, which introduces a length scale into the problem. (We adopt it with exponent $2$ since this will have notational advantages in the following.) As $\ep$ tends to zero, the higher-order perturbation becomes more singular, and $I_{\ep}$ behaves more similarly to a \emph{sharp-interface model}. Roughly speaking, in fact, low-energy sequences for $I_{\ep}$ exhibit transition layers between different phases of width $\eps^2$ (see, e.g., \cite{ball.james,bhattacharya, capella.otto2, kohn.muller2,muller}).  

Energy functionals as \eqref{eq:sol-sol-intro} are naturally linked to the study of classical Cahn-Hilliard-Modica-Mortola energies, cf.\ \cite{gurtin, modica, modica.mortola}, which in turn are strongly connected to the theory of minimal surfaces and to the modeling of liquid-liquid phase transitions. As the width $\ep$ of transition layers tends to zero, the behavior of Modica-Mortola energies has been shown to approach, in the sense of  $\Gamma$-convergence (see \cite{Braides:02, DalMaso:93} for an overview), that of a surface energy being proportional to the length of the interfaces between the different phases. Amidst the extensive literature, we single out the seminal contributions \cite{barroso.fonseca, bouchitte,  fonseca.tartar, owen.sternberg, sternberg,sternberg2} for a characterization of both scalar and vectorial Modica-Mortola energies, the results \cite{kohn.sternberg} for an analysis of local minimizers, \cite{ambrosio, baldo} for extensions to the multiwell scenario, and the recent contribution \cite{cristoferi.gravina} for the case of spatially dependent wells.  We finally mention \cite{stinson} for related models for Lithium-Ion batteries.

The study of analogous sharp-interface limits in the solid-solid setting has been initiated by {\sc S.~Conti}, {\sc I.~Fonseca}, and {\sc G.~Leoni} in \cite{conti.fonseca.leoni}, neglecting the effects of frame indifference. In dimension two, the frame-indifferent purview has been characterized by {\sc S.~Conti} and {\sc B.~Schweizer} for two rank-one connected wells $A$ and $B$, first in a linearized setting in \cite{conti.schweizer2}, and then in the the fully nonlinear framework of \eqref{eq:sol-sol-intro} in \cite{conti.schweizer, conti.schweizer3}. We also mention the contributions \cite{kytavsev-ruland-luckhaus, kytavsev-ruland-luckhaus2} for related microscopic models for two-dimensional martensitic transformations. 

The first analysis of sharp-interface limits for singularly perturbed frame-indifferent energies in higher dimensions $d>2$ has been obtained in our previous work \cite{davoli.friedrich}, for a slightly modified version of the model \eqref{eq:sol-sol-intro} where the energy contains a further anisotropic perturbation. More specifically, when the two wells have exactly one rank-one connection, after rotation, we can assume without loss of generality that $B-A = \kappa {\rm e}_d \otimes {\rm e}_d $ for $\kappa >0$. Then, our model reads  as follows:
 \begin{align}\label{eq: nonlinear energy-intro2}
 y\in H^2(\Omega; \R^d) \to E_{\ep,\eta}(y):=I_{\ep}(y)+ \eta^2 \int_{\Omega}\Big(  |\nabla^2 y|^2 - |\partial^2_{dd}y|^2 \Big)  \, {\rm d}x
\end{align}
for $\eta>0$. Owing to the additional anisotropic perturbation, our analysis is restricted to the case of exactly one rank-one connection. We stress that this additional energy term does not affect frame indifference, and penalizes only transitions in the direction orthogonal to the rank-one connection ${\rm e}_d \otimes {\rm e}_d$, while still allowing for phase transitions between the two different energy wells. We refer to \cite{golovaty1, golovaty2, kohn.muller, Zwicknagl} for studies of related models involving anisotropic perturbations.

 In \cite{davoli.friedrich}  we have shown that, for a suitable choice of $\eta$ (dependent on $\ep$), the functionals in \eqref{eq: nonlinear energy-intro2}  $\Gamma$-converge as $\ep\to 0$ (in the $L^1$-topology) to the sharp-interface limit  
\begin{align}
\label{eq: limiting energy-intro}
\mathcal{E}_0(y):=\begin{cases}
 K \,  \mathcal{H}^{d-1}(J_{\nabla y})&
\text{if }\nabla y\in BV(\Omega; R \lbrace A,B\rbrace) \ \text{ for some } R\in SO(d),\\
+\infty&\text{otherwise in }L^1(\Omega;\mathbb{R}^d),\end{cases}
 \end{align}
where $K$ corresponds to the energy of optimal transitions between the two phases (see \eqref{eq: our-k1} for the exact expression).   Roughly speaking, limiting deformations are necessarily piecewise affine with $J_{\nabla y}$ consisting of hyperplanes orthogonal to $e_d$  intersected with $\Omega$ (see \cite{dolzmann.muller} and Figure \ref{fig:limiting-def}).  We point out that, up to a possibly different constant $K$, the model in \eqref{eq: limiting energy-intro} is the same as the one identified in \cite{conti.schweizer}.  An essential ingredient in \cite{davoli.friedrich} is a novel \emph{two-well rigidity estimate} (see Theorem \ref{thm:rigiditythm} below). It provides stronger estimates with respect to previous results in the literature (see e.g.\  \cite{Chermisi-Conti, conti.schweizer, Jerrard-Lorent, Lorent}) by introducing a \emph{phase indicator}, which allows  to identify the predominant phase in each point of $\Omega$.

In this paper we further build upon this new rigidity estimate to combine the perspective of deriving sharp-interface limits for phase transitions with the passage from nonlinear-to-linearized elastic energies. In fact, triggered by the availability of rigidity estimates (mainly \cite{FrieseckeJamesMueller:02}) the derivation of effective linearized models has attained a great deal of attention over the last years. Their interest originates from the observation that they generally provide good approximations of the behavior of nonlinear models for deformations that are `close'  to rigid movements in a suitable sense. In fact, under the assumption that $A$ is the identity matrix ${\rm Id}$, a formal asymptotic expansion shows that, by considering deformations $y$ of the form $y={\rm id}+\ep u$ for a smooth displacement $u$, there holds 
$$\int_{\Omega} W (\nabla y)\,{\rm d}x=\int_{\Omega}W({\rm Id}+\ep\nabla u)\,{\rm d}x\sim \frac{\ep^2}{2}\int_{\Omega}D^2W({\rm Id}) \, \nabla u:\nabla u\,{\rm d}x+o(\ep^2),$$
where $D^2 W$ denotes the second-order differential of $W$ and $o(\ep^2)/\ep^2\to 0$ as $\ep$ tends to zero. In other words, the leading order behavior of the energy $W$ is completely encoded by the quadratic form of linearized elasticity $\frac{1}{2}\int_{\Omega}D^2W({\rm Id})\nabla u:\nabla u\,{\rm d}x$. Whereas $\eps^2$ is related to the width of transition layers, as explained above, the parameter $\eps$ represents the typical order of elastic strains. This heuristic argument has been made rigorous by {\sc G.~Dal Maso}, {\sc M.~Negri}, and {\sc D.~Percivale} in the seminal paper \cite{dalmaso.negri.percivale} for single-well energies under standard growth conditions.  
An extension  to the case of weaker growth conditions has been the subject of \cite{agostiniani.dalmaso.desimone}. We further refer to related studies on atomistic systems \cite{Braides-Solci-Vitali:07, Schmidt:2009}, homogenization \cite{gloria-neukamm, muller-neukamm}, viscoelasticity \cite{friedrich-kruzik}, plasticity \cite{Ulisse}, or fracture \cite{Friedrich-ARMA, Engineer, NegriToader:2013}.

Some of the aforementioned linearization results have been generalized to the multiwell setting for wells approaching the identity as $\eps \to 0$,  see e.g.\ \cite{agostiniani.blass.koumatos,  jesenko, Schmidt:08}. For fixed wells (independent of $\eps$), results are limited to  \cite{alicandro.dalmaso.lazzaroni.palombaro} (see \cite{alicandro.lazzaroni.palombaro} for an atomistic counterpart). There,  \EEE the authors consider a stronger higher-order perturbation compared to the ones in \eqref{eq:sol-sol-intro} and \eqref{eq: nonlinear energy-intro2}. In particular, they characterize, under appropriate boundary conditions, linearization around one of the two wells, i.e.,\ a crucial feature  is that \emph{only one phase} (say, the identity) is present in the limiting model. This is an effect of the stronger higher-order perturbation that, roughly speaking, \emph{prevents the occurrence of  macroscopic phase transitions} in the effective functional.  In mathematical terms, their penalization is chosen in a specific  way to ensure compactness and convergence of rescaled displacements $u= (y - {\rm id})/\eps$ in suitable Sobolev norms.

The main novelty of this work consists in providing a new perspective on solid-solid phase transitions, allowing simultaneously to have phase changes present in the limit, as well as to perform a `pointwise dependent' linearization that keeps track of the different `predominant phases' in each region of the body. We consider here sequences of energies of the form \eqref{eq: nonlinear energy-intro2} for suitable $\eps$-dependent $\eta$ (see Remark \ref{rk:seq-eta} below for details), denoted by $\mathcal{E}_{\ep}$ in the following.  We point out that $\eta$ is chosen to be `big enough' to guarantee that our quantitative rigidity estimate in Theorem \ref{thm:rigiditythm} provides enough compactness properties, but also `small enough' so that the limiting behavior of the energies is not affected by the anisotropic perturbation and no second-order derivatives of the deformations are involved in the limiting description. We refer to \cite[Remark 4.5 and paragraph before Theorem 1.1]{davoli.friedrich} for a discussion of this point.

Our first result consists in showing that to every sequence of deformations $\{y^\ep\}_{\ep}\subset H^2(\Omega;\R^d)$ with equibounded $\mathcal{E}_{\ep}$-energies we can associate a limiting deformation $y\in H^1(\Omega;\R^d)$, with $\nabla y\in BV(\Omega; R\{A,B\})$ for some $R\in SO(d)$, a limiting displacement $u\in SBV^2_{\rm loc}(\Omega;\R^d)$ (see Appendix \ref{sec:appendix}),  and a limiting Caccioppoli partition $\mathcal{P} = \lbrace P_j \rbrace_j$. The   jump set of $u$ is the (at most) countable union of hyperplanes orthogonal to $e_d$ and intersected with $\Omega$,  and the components of  $\mathcal{P}$  are given by the intersection of $\Omega$ with $d$-dimensional stripes having sides orthogonal to $e_d$.

The full statement of our result is quite technical: for this reason we present here a simplified version and refer to Theorem \ref{thm:compactness} for the precise formulation.

\begin{theorem}[Simplified compactness result]
\label{thm:compactness-meta}
Let $\Omega$ be a bounded Lipschitz domain in $\R^d$, $d\geq 2$, such that all its slices orthogonal to the $e_d$-direction are connected (see {\rm H8.}). Let $W$ satisfy {\rm H1.-H4.} Let $\{y^\ep\}_{\ep}\subset H^2(\Omega;\R^d)$ be such that $\sup_{\ep>0}\mathcal{E}_{\ep}(y^\ep)<+\infty$. Then, to every deformation $y^\ep$ we can associate a rotation $R^\ep\in SO(d)$, a Caccioppoli partition $\mathcal{P}^\ep=\{P^\ep_j\}_{j}$, phase indicators $\mathcal{M}^\ep=\{M^\ep_j\}_j\subset \{A,B\}$, and translations $\mathcal{T}^\ep=\{t^\ep_j\}_{j}\subset \R^d$,  as well as a limiting triple $(y,u,\mathcal{P})$ with $\nabla y\in BV(\Omega;R\{A,B\})$ such that
\begin{align*}
&R^{\ep}\to R,\\
&P_j^{\ep}\to P_j\quad\text{in measure for all $j$},\\
& y^{\ep}- \frac{1}{ \mathcal{L}^d(\Omega)} \int_{\Omega}y^{\ep}(x)\,{\rm d}x\to y \quad\text{strongly in }H^1(\Omega;\R^d),\\
&u^{\ep}\to u\quad\text{in measure in }\Omega,\text{ and } \ 
\nabla u^{\ep}\wk \nabla u\ \text{weakly in } L^2_{\rm loc}(\Omega;\M^{d\times d}), 
\end{align*}
where $u^\eps$ denote rescaled displacement fields associated to $\mathcal{P}^\eps, \mathcal{M}^\eps$, $\mathcal{T}^\eps$, and $R^\eps$, defined by 
\begin{align}\label{eq: rescali-def}
u^{\ep}:=\frac{y^{\ep}-\sum\nolimits_j ( R^{\ep} M_j^{\ep}\, x+t_j^{\ep})\chi_{P^{\ep}_j}}{\ep}.
\end{align} 
\end{theorem}

The assumptions on $W$ are classical regularity and coercivity conditions for two-well nonlinear elastic energies, cf.\  Subsection \ref{sec: nonlinear energy}. In particular, the statement shows that sequences of deformations with equibounded energies can be decomposed into the sum of piecewise constant rigid movements $\sum\nolimits_j ( R^{\ep} M_j^{\ep}\, x+t_j^{\ep})\chi_{P^{\ep}_j}$ and scaled   displacements $u^\eps$. The limiting quantities   $(y,u,\mathcal{P})$  play different roles in the description of the effective model: roughly speaking,  the limiting deformation $y$ encodes  the two different phases,  which are in general still present in the limit, and correspondingly indicates the surfaces where  phase transitions occur. The limiting displacement $u$ and the partition $\mathcal{P}$, instead, keep track of the situation in which in the limiting model two neighboring areas are in the same phase but at level $\ep$ they were separated by small  intermediate regions in the opposite phase having asymptotically vanishing  width as $\eps \to 0$, see Figure \ref{fig2} below for an illustration. More specifically, intermediate layers of width comparable to $\eps$ (i.e., the order of elastic strains) are encoded by the jump set of $u$ and widths asymptotically larger than $\eps$ are associated to the boundary of the partition $\partial P_j \cap \Omega$, $P_j \in \mathcal{P}$.  Finally, $u$  features also elastic displacements. %\RRR In principle, different choices for the limiting Caccioppoli partition $\mathcal{P}$ would be possible. The necessity for the limiting model to carry the highest possible amount of information about the microstructure originating at level $\ep$ is the motivation for the selection criterion for the limiting partition $\mathcal{P}$. We refer to Theorem \ref{thm:compactness}(c) for a precise statement and to the discussion in Example \ref{ex} for an in-depth analysis of this point. {\tt This is a bit technical and should be rephrased if the theorem above does not speak about the selection principle, as I suggested. (I would mention the selection principle in the proof description.) Maybe something shorter like this?} \EEE 

 In particular, Theorem \ref{thm:compactness-meta} motivates the notion of \emph{admissible triples} as the collection of triples $(y,u,\mathcal{P})$ that are attained in the sense of the convergences in Theorem \ref{thm:compactness-meta}, starting from a sequence of deformations $\{y^\ep\}_{\ep}$. In what follows, we will refer to the convergence properties in Theorem \ref{thm:compactness-meta} as \emph{tripling of the variables}. See also \cite{Friedrich-ARMA} for a related notion of convergence.

The second step of our analysis consists in providing a characterization of admissible limiting triples $(y,u,\mathcal{P})$. For ease of presentation, we collect our findings in a simplified statement and refer to Subsection \ref{sec: limiting triples} for the precise formulation of the results.

\begin{theorem}[Simplified characterization of limiting triples]
\label{thm:characterization-meta}
Let $(y,u,\mathcal{P})$ be an admissible triple for the sequence $\{y^\ep\}_\eps$. Then,
\begin{itemize}
\item $y$ and $\mathcal{P}$ are uniquely defined; 
\item $u$ is uniquely defined up to piecewise translations of the form $\sum_j t_j \chi_{P_j}$, $\lbrace t_j\rbrace_j \subset \R^d$, and  global (infinitesimal) rotations; 
\item $J_{\nabla y} \subset \bigcup\nolimits_{j=1}^\infty \partial P_j \cap \Omega$;
\item the jump  of $u$ is constant on every connected component of its jump set.
\end{itemize}
\end{theorem}
The non-uniqueness of the displacement field is simply  a consequence of the possible  non-uniqueness in the definition of $u^\eps$, see \eqref{eq: rescali-def}.  The last point of the statement represents a `laminate structure' of limiting displacement fields. This regularity of $u$ is achieved thanks to the anisotropic penalization in \eqref{eq: nonlinear energy-intro2} and neglects branching phenomena, see also Remark \ref{rem:internal jumps} for more details.

Denoting by $\mathcal{A}$ the class of all admissible limiting triples $(y,u,\mathcal{P})$, our main contribution consists in showing that the asymptotic behavior of the energies $\mathcal{E}_\ep$ is described by the functional
\begin{align}\label{eq: limiting energy-intro3}
 \mathcal{E}_0^{\mathcal{A}} (y,u,\mathcal{P}):=
\frac12\int_{\Omega}D^2W (\nabla y(x))\nabla u(x){:}\nabla u(x)\,{\rm d}x{+}K   \mathcal{H}^{d-1}(J_{\nabla y}){+}2  K \mathcal{H}^{d-1}\Big(\big(J_u\cup\big( \bigcup\nolimits_j \partial P_j\cap \Omega\big)\big)\setminus J_{\nabla y}\Big)
\end{align}
for every $(y,u,\mathcal{P})\in \mathcal{A}$. We point out that the constant $K$ in \eqref{eq: limiting energy-intro3} is the same one as in \eqref{eq: limiting energy-intro}. We observe that $\mathcal{E}_0^{\mathcal{A}}$ reduces to \eqref{eq: limiting energy-intro} for $u=0$ and $\mathcal{P}$ coinciding with the collection of connected components of the two sets $\{x\in \Omega\colon\,\nabla y(x)\in SO(d)A\}$ and $\{x\in \Omega\colon\, \nabla y(x)\in SO(d)B\}$. Analogously, $\mathcal{E}_0^{\mathcal{A}}$ coincides with the quadratic form of linearized elasticity, and hence with the limiting model in \cite{alicandro.dalmaso.lazzaroni.palombaro} for $u \in H^1(\Omega;\R^d)$, for the trivial partition $\mathcal{P}$ consisting only of $\Omega$, and for a deformation $y$ with $\nabla y = {\rm Id}$ in $\Omega$. In this sense, our limiting description combines both the effects of the sharp-interface characterizations \cite{conti.schweizer, davoli.friedrich} and those of the multiwell linearization \cite{alicandro.dalmaso.lazzaroni.palombaro}. In contrast to these results, it features an additional surface term: as described above, the jump of $u$ and the boundary of the partition encode small intermediate layers in the opposite phase at  level $\eps$  with width bigger than or comparable to $\ep$ which induce two `consecutive phase transitions',  see Figure \ref{fig2}. Our $\Gamma$-convergence result is proven under the compatibility condition that this additional term   enters the energy with double cost with respect to  single phase transitions, i.e., we suppose that  
\begin{equation}
\label{eq:comp-intro}
K^A_{\rm dp}=K^B_{\rm dp}=2K,
\end{equation}
where  $K^A_{\rm dp}$ and $K^B_{\rm dp}$ represent, roughly speaking, the energy necessary for performing these double-phase transitions at level $\ep$. (The subscript `dp' stands for `double profile'. We refer to \eqref{eq: our-k2} for their precise expression.)  Our main result reads as follows: 
\begin{theorem}
\label{thm:gamma-meta}
Let $\Omega$ be a bounded strictly star-shaped domain (see \eqref{eq: starshape})  satisfying the further connectedness assumption in {\rm H8.} Let $W$ satisfy {\rm H1.-H7.} and assume that the compatibility condition in \eqref{eq:comp-intro} holds true. Then, $\mathcal{E}_\ep$ $\Gamma$-converges to $\mathcal{E}^{\mathcal{A}}_0$ in the topology provided by the tripling of the variables in Theorem \ref{thm:compactness-meta}.
\end{theorem}

We refer to Subsection \ref{sec: nonlinear energy} and Subsection \ref{sec: refined model} for the formulation of H1.-H7. The difference between our result and the $\Gamma$-convergence analyses in \cite{conti.schweizer, davoli.friedrich} and  \cite{alicandro.dalmaso.lazzaroni.palombaro} is mostly in the adopted topology. In \cite{conti.schweizer, davoli.friedrich} an effective energy is identified in the strong $L^1$-topology for deformations $y$. The result in \cite{alicandro.dalmaso.lazzaroni.palombaro}, instead, is derived in the weak $H^1$-topology for rescaled displacements $(y-{\rm id})/\eps$. Our model combines this `global' point of view with a `local' one: the limiting Caccioppoli partition plays the role of identifying subdomains where the small-strains approximation of linearized elasticity, encoded by the limiting displacement $u$, is well posed. Finally, the surface-energy term associated to the jump set of $u$ and to $\mathcal{P}$  keeps track of  the multiple phase changes that the material had to undergo at level $\ep$ on regions with vanishing widths. 

We stress here that the focus of our study is not on minimization problems and their convergence but rather on the identification of the limiting energy functional. For completeness, we also mention that the case of incompatible wells, i.e., the setting where $A$ and $B$ have no rank-one connections, is not included in our analysis but would be much simpler to handle. Indeed, the limiting model would linearize around just one of the two phases, leading to a limiting description analogous to  \cite{alicandro.dalmaso.lazzaroni.palombaro}.

We point out that the lower bound in Theorem \ref{thm:gamma-meta} holds under no further assumptions on the two profile energies, i.e., the compatibility condition \eqref{eq:comp-intro} is only needed for the construction of recovery sequences. In Subsection \ref{subs:1d} we present a self-contained discussion showing that, under an additional assumption on  the energy density (see \eqref{eq: isotropy} below)  optimal profiles are one-dimensional and the compatibility condition in \eqref{eq:comp-intro} is indeed satisfied. This assumption is  fulfilled, e.g., when the energy only depends on the distance of the deformation gradient from the two wells, see \eqref{eq: modelli}. 

%We point out that the lower bound identified in Theorem \ref{thm:liminf} holds under no further assumptions on the two profile energies. The compatibility condition \eqref{eq:comp-intro} is only needed in the proof of its optimality in Theorem \ref{thm:limsup-new}. Eventually, in Subsection \ref{subs:1d} we present a self-contained discussion showing how under a further assumption on the energy density (see H9. below) optimal profiles are one-dimensional and the compatibility condition in \eqref{eq:comp-intro} is satisfied.
%

We close the introduction with some comments on the proof structure. The proof of Theorem \ref{thm:compactness-meta} relies on a series of intermediate results: all statements involving limiting rotations, partitions, and deformations are essentially proven in Proposition \ref{lemma: intermediate step1}. The sequence of translations and the  limiting displacements  are  first exhibited on subsets of $\Omega$ and eventually on $\Omega$ itself in Propositions \ref{lemma: intermediate step2} and \ref{lemma: intermediate step3}, respectively. Finally, a further delicate construction is needed to ensure uniqueness of  the limiting Caccioppoli partition. This  is based on a certain  \emph{selection principle}, see \eqref{eq: toinfty} below. Indeed, without such a requirement,  there might be different possible choices for the limiting partition, see the discussion in Example \ref{ex} for an in-depth analysis of this point. Key ingredients for the compactness analysis are the two-well rigidity estimate recalled in Theorem \ref{thm:rigiditythm} and a characterization of the two phase regions established in \cite[Proposition 3.7]{davoli.friedrich}, see also  Proposition \ref{lprop: phases}.

 The statements collected in Theorem \ref{thm:characterization-meta} are the subject of three different propositions. In particular, the uniqueness properties of limiting deformation,  displacement, and partition  are proven in Proposition \ref{prop:ex-coarsest-part}.  This latter one is shown to be a consequence of the selection principle described above. The characterization of the jump set of $\nabla y$ is contained in Proposition \ref{lemma: admissible-u-y-jump}, whereas that of the jump set of $u$ is the subject of Proposition \ref{lemma: admissible-u}.

As customary in $\Gamma$-convergence analysis, the proof of Theorem \ref{thm:gamma-meta} consists in first showing that $\mathcal{E}_0^{\mathcal{A}}$ provides a lower bound for the limiting behavior of the energies $\mathcal{E}_\ep$ (see Theorem \ref{thm:liminf}), and then in showing that this lower bound is indeed optimal (see Theorem \ref{thm:limsup-new}). The proof of the liminf inequality essentially relies on providing a characterization of the double-profile energies $K^M_{\rm dp}$, $M\in \{A,B\}$. An important point is to show that optimal double phase transitions are, a priori, energetically more expensive than gluing together two optimal profiles performing each a single phase transition in an energetically optimal way (in other words, $K^M_{\rm dp} \ge 2K$), see Proposition \ref{eq: KundKdo-new}. The key ingredients for proving the upper bound are explicit constructions of local recovery sequences performing energetically optimal single and double phase transitions, see Propositions \ref{lemma: local1} and \ref{lemma: local2}. Both sequences are constructed starting from a delicate slicing argument introduced in \cite{davoli.friedrich} and recalled in Proposition \ref{lemma: optimal profile} below. In addition, they are chosen  so  that they coincide with isometries far from the interfaces,  and  they can then be `glued together' in the proof of Theorem \ref{thm:limsup-new}.

The paper is organized as follows: in Section \ref{sec:state} we review the state-of-the-art and perform an overview of the main mathematical difficulties. In Section \ref{sec:main-thm} we describe our model and state the main results. Sections \ref{sec:compactness} and \ref{sec:limiting-triple} are devoted to the proofs of the compactness theorem and to the characterization of limiting triples, respectively. The proof of Theorem \ref{thm:gamma-meta} is the subject of Section \ref{sec:gamma}.

\subsection{Notation}
In what follows, we  fix $d\in \N$, $d\geq 2$, and we  consider  a bounded Lipschitz domain $\Omega \subset \R^d$. We will denote points $x\in \R^d$ as $x=(x',x_d)$, with $x'\in \R^{d-1}$ and $x_d\in \R$. In the whole paper we use standard notations for Sobolev spaces, as well as for $BV(\Omega)$ and  $SBV(\Omega)$. We refer the reader to \cite{Ambrosio-Fusco-Pallara:2000}  for definitions and main results. Some basic properties of special functions of bounded variation and Caccioppoli partitions are recalled in Appendix \ref{sec:appendix}.  We will omit the target space of our functions whenever this is clear from the context. The identity map on $\R^d$ will be denoted by  ${\rm id}$ or, with a slight abuse of notation, simply by $x$.  For $m \in \N$, the  $m$-dimensional Lebesgue and Hausdorff measures of a set will be indicated by $\mathcal{L}^m$ and $\mathcal{H}^m$, respectively. By $\fint_{\Omega}$ we denote the average integral $\frac{1}{\mathcal{L}^d(\Omega)}\int_\Omega$. 

We denote by $e_1,\dots, e_d$  and $e_{ij}$, $i,j=1,\dots,d$, the standard basis in $\R^d$ and $\mathbb{M}^{d\times d}$, respectively. We will use the notation ${\rm Id}$ for the identity matrix in $\M^{d \times d}$ and denote by $SO(d) \subset \mathbb{M}^{d\times d}$ the set of proper rotations. The sets of symmetric and skew-symmetric matrices  are  indicated by $\M^{d \times d}_{\rm sym}$ and $\M^{d \times d}_{\rm skew}$, respectively. In what follows, we will adopt the Frobenius scalar product between matrices $F:G:={\rm Tr}(F^TG)$ for every $F,G\in \mathbb{M}^{d\times d}$, and we will use the symbol $|\cdot|$ for the associated Frobenius norm.  For every set $S\subset \R^d$, we indicate by $\chi_S$ its characteristic function, defined as $\chi_S(x)=1$ if $x\in S$ and  $\chi_S(x)=0$ otherwise. Given two sets $S_1,S_2 \in \R^d$, we denote by $S_1 \triangle S_2$ their symmetric difference. Inclusions of sets  $S_1\subset S_2$ are always intended up to sets of negligible measure, i.e., $\mathcal{L}^d(S_1 \setminus S_2) = 0$.  By $B_\rho(x)$ we denote the $d$-dimensional ball of radius $\rho>0$ and center $x \in \R^d$.

%%%%%%%%%%%%%%%%%%%%%%%%%%%%%%%%%%%%%%%%%%
%%%%%%%%%%%%%%%%%%%%%%%%%%%%%%%%%%%%%%%%%%
\section{State-of-the-art, heuristics, and challenges}
\label{sec:state}

In this section we recall the state-of-the-art for sharp-interface limits in the theory of solid-solid phase transitions, and for derivations of linearized models from nonlinear elastic energies. We additionally highlight the main open questions and difficulties.

%%%%%%%%%%%%%%%%%%%%%%%%%%%%%%%%%%%%%%%%%%
\subsection{Models in nonlinear elasticity for two-well energies}\label{sec: nonlinear energy}
To  every \emph{deformation}  $y\in H^1(\Omega;\R^d)$ we associate the elastic energy 
$$\int_{\Omega}W(\nabla y)\,{\rm d}x,$$
where $W\colon\M^{d\times d}\to [0,+\infty)$ is a map representing the \emph{stored-energy density}, and satisfying the following properties:
\begin{enumerate}
\item[H1.](Regularity) $W$ is continuous; \label{H1}
\item[H2.](Frame indifference) $W(RF)=W(F)$ for every $R\in SO(d)$ and $F\in \M^{d\times d}$;
\item[H3.](Two-well structure) $W(A)=W(B)=0$, where $A={\rm Id}$, and $B=\textrm{diag }(1,1,\dots,1,1+\kappa)$, for   $\kappa>0$;
\item[H4.](Coercivity) there exists a constant $c_1>0$ such that  
\begin{align*}
 W(F) \ge c_1 \dist^2(F,SO(d)\lbrace A,B \rbrace) \ \ \ \text{for every $F\in \M^{d\times d}$;}
 \end{align*}
\item[H5.](Quadratic behavior around the two wells) there exists $\delta_W>0$ such that $W$ is of class $C^2$ in 
$$ \{F\in \M^{d\times d}:\, {\rm dist}(F, SO(d) \lbrace A,B\rbrace)<\delta_W\}.$$
\item[H6.] (Growth condition from above) there exists a constant $c_2>0$ such that 
$$  W(F) \le c_2 \dist^2(F,SO(d)\lbrace A,B \rbrace)  \ \ \ \text{for every $F\in \M^{d\times d}$.}$$ 
 \end{enumerate}

Assumptions H1.-H5.\ are standard requirements on stored-energy densities in nonlinear elasticity. We note that after an affine change of variables one can always assume that the two wells have the form given in H3., see \cite[Discussion before Proposition 5.1 and Proposition 5.2]{dolzmann.muller}. Specifically, the choice $\kappa>0$ amounts to the case of exactly one rank-one connection between $A$ and $B$, namely to the setting in which the only solution of $B-R A  = a \otimes \nu$ with $R \in SO(d)$, $a,\nu \in \R^d$, and  $|\nu|=1$  is given by $R= {\rm Id}$, $\nu= e_d$, and $a = \kappa e_d$. 
 
 We point out that assumption H6.\ is not compatible with the impenetrability condition
\begin{align}\label{eq: impenetrability}
W(F)\to +\infty\quad\text{as}\quad{\rm det}\,F\to +0, 
\end{align}
 which is usually enforced to model a blow-up of the elastic energy under strong compressions. In the derivation of sharp-interface limits for solid-solid phase transitions \cite{conti.schweizer, conti.schweizer2, davoli.friedrich}, however,  condition H6.  is instrumental for the construction of recovery sequences. (Note that, in dimension two, by means of a more elaborated construction  performed in \cite{conti.schweizer3}, assumption H6.\ may be dropped.)

In order to model solid-solid phase transitions, we analyze a nonlinear energy given by the sum of a
suitable rescaling of the elastic energy and  a singular perturbation. For every $\ep>0$,  we consider the functional   ${E}^P_{\ep}\colon H^2(\Omega;\R^d)\to [0,+\infty)$ defined by 
\begin{align}\label{eq: basic energy} 
{E}^P_{\ep}(y):=\frac{1}{\ep^2}\int_{\Omega}W(\nabla y)\,{\rm d}x+\int_{\Omega} P_\eps(\nabla^2 y)\,{\rm d}x,
\end{align}
where $P_\eps: \R^{d \times d \times d} \to [0,+\infty)$ is a function which depends on the small parameter $\eps$. In the following subsections, we will specify the choice of $P_\eps$ according to different modeling assumptions.

%In the theory of solid-solid phase transitions the
%deformation $y: \R^d \to \R^d$ of an elastic body is determined
%via a functional containing a nonconvex energy density introduced above and a
%singular perturbation.  let ${E}^P_{\ep}:H^2(\Omega;\R^d)\to [0,+\infty)$ be the functional
%

The parameter $\eps$ in the definition above represents the typical order of the strain, whereas $\eps^2$ is
related to the size of transition layers \cite{ball.james,bhattacharya, capella.otto2, kohn.muller2,muller}. The first term in the right-hand side of \eqref{eq: basic energy} favors deformations $y$ whose gradient is close to the two {wells} of $W$, whereas the second term penalizes transitions between two different values of the gradient.

In the following, we will call $A$ and $B$  the \emph{phases}. Regions of the domain where $\nabla y$ is in a neighborhood of $SO(d)A$ will be called $A$-\emph{phase regions} \EEE of $y$  and accordingly we will speak of  $B$-\emph{phase regions}. \EEE

%%%%%%%%%%%%%%%%%%%%%%%%%%%%%%%%%%%%%%%%%%
\subsection{Review of existing results}\label{sec: existing results}

We now continue by recalling some  results about sharp-interface limits  and derivation of linearized models. The exact setting of the paper and our main results can be found in Section \ref{sec:main-thm}.  There, we will also recall a more recent result on sharp-interface limits which   we proved in \cite{davoli.friedrich}, and which   represents the departure point of our analysis.

\noindent \textbf{A sharp-interface limit for a model of solid-solid phase transitions.} Classical singularly perturbed two-well problems are described by energies of the form
\begin{align}\label{eq: I functional}
I_{\ep}(y):=\frac{1}{\ep^2}\int_{\Omega}W(\nabla y)\,{\rm d}x+\ep^2\int_{\Omega}|\nabla^2 y|^2\,{\rm d}x
\end{align}
for every $y\in H^2(\Omega;\R^d)$, corresponding to the choice $P_\eps(G) = \eps^2 |G|^2$, $G \in \R^{d \times d \times d}$, in \eqref{eq: basic energy}. This subsection is devoted to a presentation of the analysis performed by {\sc S. Conti} and {\sc B. Schweizer}  \cite{conti.schweizer} which addresses the sharp-interface limit of this model in dimension two, as $\eps$ tends to zero. Although in \cite{conti.schweizer} also the case of two rank-one connections is considered, we focus here on compatible wells having exactly one rank-one connection (see assumption H3.).

 Denote by $\mathscr{Y}(\Omega)$ the class of admissible limiting deformations, defined as 
\begin{align}\label{eq: limiting deformations}
 \mathscr{Y}(\Omega):= \bigcup_{R\in SO(d)}\mathcal{Y}_R(\Omega),\ \ \ \text{where}\ \  \mathcal{Y}_R(\Omega):=\big\{& y\in H^1_{\#}(\Omega;\R^d)\colon\,\nabla y\in BV(\Omega;R\{A,B\})\big\} \ \  \text{for } R \in SO(d),
\end{align}
 where $H^1_{\#}(\Omega;\R^d) := \lbrace y \in H^1(\Omega;\R^d)\colon \, \fint_\Omega y\, {\rm d}x = 0 \rbrace$. For every open subset $\Omega'\subset \Omega$, we will adopt the notation   $\mathscr{Y}(\Omega')$   to indicate the corresponding admissible deformations.
 In \cite[Proposition 3.2]{conti.schweizer} the authors established the following compactness result.

\begin{lemma}[Compactness]
\label{lemma:comp-def}
Let $d\in \mathbb{N}$, $d\geq 2$, and let $\Omega\subset \R^d$ be a bounded Lipschitz domain. Let $W$ satisfy assumptions {\rm H1}.--{\rm H4}. Then, for all sequences $\{y^{\ep}\}_\eps\subset H^2(\Omega;\R^d)$ for which 
$$\sup_{\ep>0} I_{\ep}(y^{\ep})<+\infty,$$ 
there exists a map $y\in \mathscr{Y}(\Omega)$ such that, up to the extraction of a subsequence (not relabeled), there holds
$$y^{\ep}-\fint_{\Omega}y^{\ep}(x)\,{\rm d}x \to y\quad\text{strongly in }H^1(\Omega;\R^d).$$
\end{lemma}

The limiting deformations $y$ have the structure of a simple laminate. Indeed, {\sc G.~Dolzmann} and {\sc S.~M\"uller} \cite{dolzmann.muller} have shown that for $y \in \mathcal{Y}_R(\Omega)$  the  essential boundary of the set $T:=\{x\in \Omega\colon\,\nabla y(x)\in RA\}$ consists of subsets of hyperplanes that intersect $\partial \Omega$ and are orthogonal to $e_d$, and  that  $y$ is affine on balls whose intersection with $\partial T$ has zero $\mathcal{H}^{d-1}$-measure,  cf.\ Figure \ref{fig:limiting-def}  (see also Appendix \ref{sec:appendix} for the definition of essential boundary for a set of finite perimeter).

\begin{figure}[h]
\centering
\begin{tikzpicture}
% e_1 points right, e_2 points up, 
% four external points
\coordinate (A) at (0,3);
 \coordinate (B) at (2,0);
 \coordinate (C) at (-2,0);
 \coordinate (D) at (0,-3);
 % points on the right y=3/2x-3
 \coordinate (E) at (1,1.5);
 \coordinate (F) at (1.33,1);
  \coordinate (G) at (1.33,-1);
  \coordinate (H) at (0.66,-2);
  \coordinate (I) at (0.33,-2.5);
  \coordinate (L) at (0.2,-2.7);
  \coordinate (M) at (0.1,-2.85);
  %points on the left
   \coordinate (N) at (-1,1.5);
 \coordinate (O) at (-1.33,1);
  \coordinate (P) at (-1.33,-1);
  \coordinate (Q) at (-0.66,-2);
  \coordinate (R) at (-0.33,-2.5);
  \coordinate (S) at (-0.2,-2.7);
  \coordinate (T) at (-0.1,-2.85);
  %lines and colors
  \draw (A)--(B)--(D)--(C)--cycle;
  \draw (E)--(N);
  \draw (F)--(O);
  \draw (G)--(P);
 % \draw[dashed] (H)--(Q);
 % \draw[dashed] (I)--(R);
 % \draw[dashed] (L)--(S);
  %\draw[dashed] (M)--(T);
  \draw[fill=orange!20] (A)--(N)--(E);
  \draw[fill=blue!30] (N)--(E)--(F)--(O)--cycle;
 \draw[fill=orange!20] (O)--(F)--(B)--(G)--(P)--(C)--cycle;
 \draw[fill=blue!30] (P)--(G)--(H)--(Q)--cycle;
  \draw[dashed,fill=orange!20] (Q)--(H)--(I)--(R);
   \draw[dashed,fill=blue!20] (R)--(I)--(L)--(S)--cycle;
 \draw[dashed,fill=orange!10] (S)--(L)--(M)--(T);
 \node at (0,0) {$B$};
 \node at (0,-1.5) {$A$};
  \node at (0,1.25) {$A$};
 \node at (0,2.25) {$B$};

%% Following is for debugging purposes so you can see where the points are
%% These are last so that they show up on top
%\foreach \xy in {A, B, C, D, E, F, G, H, I, L, M, N, O, P, Q, R, S, T}{
%  \node at (\xy) {\xy};
%}
\end{tikzpicture}
\caption{The gradient of a limiting deformation $y\in \mathcal{Y}_{\rm Id}(\Omega)$, in the case in which $B-A=\kappa e_d\otimes e_d$.}
\label{fig:limiting-def}
\end{figure}
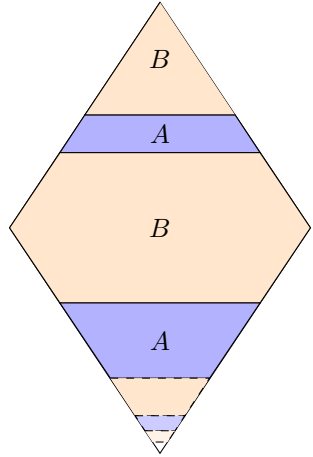

We now introduce the limiting sharp-interface energy. Denoting by $Q:=(-\tfrac12,\tfrac12)^d$ the $d$-dimensional unit cube centered in the origin and with sides parallel to the coordinate axes, we consider the \emph{optimal-profile energy} 
\begin{align}\label{eq: conti-schweizer-k}
K_{0} :=\inf\Big\{&\liminf_{\ep\to 0} I_{\ep}(y^{\ep},Q)\colon \,   \lim_{\eps \to 0}  \Vert  y^\eps -  y_0^+ \Vert_{L^1(Q)}= 0\Big\},
\end{align}
where $y_0^+\in H^1_{\rm loc}(\R^d;\R^d)$ is the continuous function with $\nabla y_0^+=A\chi_{\{x_d>0\}}+B\chi_{\{x_d<0\}}$ and $y_0^+(0)=0$. (Here, $\chi_{\{x_d>0\}}$ and $\chi_{\{x_d<0\}}$ denote the characteristic functions of the two halfplanes $\{x_d>0\}$ and $\{x_d<0\}$, respectively.) Note that $K_0$ corresponds to the energy of an optimal phase transition from $A$ to $B$, and that it is invariant under changing the roles of the two phases, i.e., invariant by replacing $y_0^+$ with the function $y^-_0\in H^1_{\rm loc}(\R^d;\R^d)$ satisfying $y^-_0(0)=0$ and $\nabla y^-_0=B\chi_{\{x_d>0\}}+A\chi_{\{x_d<0\}}$.

The sharp-interface limiting functional $I_0\colon L^1(\Omega;\R^d)\to [0,+\infty]$ is defined as
\begin{align}\label{eq: first sharp-interface}
I_0(y):=\begin{cases} K_0 \mathcal{H}^{d-1}(J_{\nabla y})&\text{if }y\in \mathscr{Y}(\Omega)\\
+\infty&\text{otherwise}.\end{cases}
\end{align}
In \cite[Theorem 3.1]{conti.schweizer} it was proved that, in the two-dimensional setting, $I_0$ is the variational limit of the sequence $\{I_{\ep}\}_\ep$ in the sense of $\Gamma$-convergence. (For an exhaustive treatment of $\Gamma$-convergence we refer the reader to \cite{Braides:02, DalMaso:93}.) 
\begin{theorem}[$\Gamma$-convergence in dimension $d=2$]\label{theorem: conti.schweizer}
Let $d=2$, let $\Omega\subset \R^2$ be a bounded, strictly star-shaped Lipschitz domain, and let $W$ satisfy {\rm H1}.--{\rm H4}.\ and {\rm H6}. Then
$$\Gamma-\lim_{\ep \to 0} I_{\ep}=I_0$$ 
with respect to the strong $L^1$-topology.
 \end{theorem}
  We recall that an open set $\Omega$ is strictly star-shaped if there exists a point $x_0\in \Omega$ such that 
  \begin{align}\label{eq: starshape}
  \{tx+(1-t)x_0:\,t\in (0,1)\}\subset \Omega \ \ \ \ \text{for every $x\in \partial \Omega$}.
  \end{align}
 Here and in the sequel, we follow the usual convention that convergence of the continuous parameter $\eps \to 0$ stands for convergence of arbitrary sequences $\lbrace \eps_i \rbrace_i$ with $\eps_i \to 0$ as $i \to \infty$, see \cite[Definition 1.45]{Braides:02}.   In \cite{conti.schweizer3}, the same $\Gamma$-convergence result as in Theorem \ref{theorem: conti.schweizer} has been obtained by dropping H6.\ via a more elaborate construction allowing to incorporate an impenetrability condition of the form \eqref{eq: impenetrability}. 

The result  in Theorem \ref{theorem: conti.schweizer}  is limited to the two-dimensional setting due to the limsup inequality: the definition
of sequences with optimal energy approximating a limit that has multiple flat interfaces relies on a deep technical construction. This so-called \emph{$H^{1/2}$-rigidity on lines} (see \cite[Section 3.3]{conti.schweizer}) is only available in dimension $d=2$. We also refer to a recent related study for microscopic models of two-dimensional martensitic transformations  \cite{kytavsev-ruland-luckhaus2}. The issue of dimensionality has been overcome in \cite{davoli.friedrich} by considering a slightly modified model, see Subsection \ref{sec: refined model} for details.

%%%%%%%%%%%%%%%%%%%%%%%%%%%%%%%%%%%%%%%%%%%%%%%%%%%%%%%%%%%%%%%%%%%

\noindent \textbf{Linearization around the identity for multiwell energies.} In the context of multiwell linearization, {\sc R. Alicandro, G. Dal Maso, G. Lazzaroni}, and  {\sc M. Palombaro} \cite{alicandro.dalmaso.lazzaroni.palombaro} investigated a multiwell energy $F_\eps \colon H^2(\Omega;\R^d) \to [0,+\infty) $  of the form 
\begin{align}\label{eq: F functional}
F_{\ep}(y):=\frac{1}{\ep^2}\int_{\Omega}W(\nabla y)\,{\rm d}x+\ep^{2- \gamma_d(r) }\int_{\Omega}|\nabla^2 y|^2\,{\rm d}x
\end{align}
for $r \in [1,2]$ and  a suitable function $\gamma_d\colon[1,2] \to (0,+\infty)$, where  for $d=2$ it holds $\gamma_2(r) = r$, cf.~\cite[Equation (1.9)]{alicandro.dalmaso.lazzaroni.palombaro}. Here, the singular higher order term penalizes transitions between different wells in a stronger way with respect to \eqref{eq: I functional}. This corresponds to the choice $P_\eps(G) = \eps^{2- \gamma_d(r)} |G|^2$, $G \in \R^{d \times d \times d}$, in \eqref{eq: basic energy}. In \cite{alicandro.dalmaso.lazzaroni.palombaro}, the problem is studied in arbitrary dimension for a finite number of different wells  and under very general growth conditions for the elastic energy and the second-order penalization. There, also the influence of external forces, under different scalings of the singular perturbation, is thoroughly discussed.  For a simple exposition, however, we present only the basic case here   and we specify the result to our two phases $A$ and $B$.  

First, \cite[Theorem 2.3]{alicandro.dalmaso.lazzaroni.palombaro} along with the well-known rigidity estimate in \cite{FrieseckeJamesMueller:02} yields the following compactness result.

\begin{lemma}[Compactness]\label{lemma:comp-giuliano}
 Let $d\in \N$, $d\geq 2$, and $r \in (1,2]$. Let $\Omega\subset \R^d$ be a bounded Lipschitz domain. Let $W$ satisfy assumptions {\rm H1}.--{\rm H4}. Then, for all sequences   $\{y^{\ep}\}_\eps\subset H^2(\Omega;\R^d)$  satisfying $\sup_{\ep>0} F_{\ep}(y^{\ep})<+\infty$ we find rotations $R^\eps \in SO(d)$, translations $t^\eps \in \R^d$, and  phases $M^\eps \in \lbrace A, B \rbrace$ such that
$$\sup_{\eps>0} \     \Big\| \frac{ y^\eps - (R^\eps M^\eps\, x + t^\eps)}{\eps} \Big\|_{W^{1,r}(\Omega)} < + \infty.$$
\end{lemma}
Additionally imposing Dirichlet boundary conditions of the form $y^\eps = {\rm id} + \eps g $ on a part of the boundary with $g \in W^{1,\infty}(\Omega;\R^d) \cap H^2(\Omega;\R^d)$, one can choose  $R^\eps = {\rm Id}$, $t^\eps = 0$, and $M^\eps = A = {\rm Id}$ in the above result, see \cite[Theorem 1.8]{alicandro.dalmaso.lazzaroni.palombaro}. This implies the uniform bound  $\sup_{\eps>0} \Vert  u^\eps \Vert_{W^{1,r}(\Omega)} < + \infty$ for the \emph{rescaled displacement fields} 
\begin{align}\label{eq: simple u-def}
u^\eps := \frac{  y^\eps  - {\rm id}}{\eps}.
\end{align}
In other words,  for sequences with bounded $F_\eps$-energy, Lemma \ref{lemma:comp-giuliano} together with prescribed boundary conditions ensures compactness in $W^{1,r}$ for rescaled displacement fields. We write the nonlinear energy in terms of the displacement fields by setting $\hat{F}_\eps (u) = F_\eps({\rm id} + \eps u)$ for $u \in H^2(\Omega;\R^d)$.

Formally, the effective linearized energy $F_0\colon W^{1,r}(\Omega;  \R^d)\to [0,+\infty]$  can be calculated by a Taylor expansion, and has the structure 
\begin{align}\label{eq: first-linearization}
F_0(u):= \begin{cases}\int_\Omega  \mathcal{Q}_{\rm lin}  \big({\rm Id},e(u)\big) \, {\rm d}x  &\text{if } u\in H^1(\Omega;\R^d),\\
+\infty & \text{otherwise}.\end{cases}
\end{align}
where $\mathcal{Q}_{\rm lin} \colon SO(d)\{A,B\}\times \M^{d\times d}\in [0,+\infty)$ is the quadratic form
\begin{equation}
\label{eq:def-Q}
\mathcal{Q}_{\rm lin} (RM,F):=\frac12 D^2W(RM)F:F
\end{equation}
for every $R \in SO(d)$, $M\in \{A,B\}$, and $F\in \M^{d\times d}$. Note that frame indifference (see H2.) implies that the energy only depends on the symmetric part $e(u) := \frac{1}{2}((\nabla u)^T + \nabla u)$ of the strain, see \eqref{eq: first-linearization}. More generally, in view of H4., one can check that (cf.\ \eqref{eq: linearization formula} below)
\begin{align}\label{eq: only symmetric}
 \mathcal{Q}_{\rm lin}  (RM,SRM) = 0  \ \ \text{ if and only if  \ \  $R \in SO(d)$, $M\in \{A,B\}$, and $S \in \mathbb{M}^{d\times d}_{\rm skew}.$}
% \mathcal{Q}_{\rm lin}  \EEE(RM,F) = \mathcal{Q}_{\rm lin}  \EEE\big(RM, \tfrac{1}{2}\big((FM^{-1}R^T)^T + FM^{-1}R^T \big) RM \EEE\big) \text{ for all  $R \in SO(d)$, $M\in \{A,B\}$, and \EEE $F \in \mathbb{M}^{d\times d}.$}
\end{align}
 The  relation  of $\hat{F}_\eps$  and  $F_0$ has been made rigorous by $\Gamma$-convergence (see \cite[Theorem 1.9]{alicandro.dalmaso.lazzaroni.palombaro}).

\begin{theorem}[Passage from nonlinear to linearized energies by $\Gamma$-convergence]\label{theorem: gamma-giuliano}
 Let $d\in \N$, $d\geq 2$, and $r \in (1,2]$. Let $\Omega\subset \R^d$ be a bounded Lipschitz domain. Let $W$ satisfy assumptions {\rm H1}.--{\rm H5}. Then
$$\Gamma-\lim_{\ep \to 0} \hat{F}_{\ep}= F_0$$ 
with respect to the weak $W^{1,r}$-topology.
 \end{theorem}
 
%%%%%%%%%%%%%%%%%%%%%%%%%%%%%%%%%%%%%%%%%%%%
 \subsection{Phase transitions and linearization: Heuristics and challenges}\label{sec: heuristics/challenges}

 Our goal is to  combine the above two approaches and to identify a model which allows both for phase transitions and for the passage to linearized energies in terms of rescaled displacement fields. As a first observation, we note that the setting in \eqref{eq: F functional} is more specific than the one considered in \eqref{eq: I functional} in the sense that deformations with finite  energy are essentially in \emph{one} phase, $A$ or $B$, see Lemma \ref{lemma:comp-giuliano}. Imposing certain boundary conditions, one can always infer that the same phase, e.g.\ $A = {\rm Id}$, is  predominant.  Then it is indeed meaningful to perform a linearization around the identity. This differs significantly from the laminate structure of the limiting configurations obtained in Lemma \ref{lemma:comp-def}, where  different phases may be active and phase transitions between the different phase regions occur, see  Figure \ref{fig:limiting-def}. In \eqref{eq: F functional}, the second-order penalization is so strong that basically phase transitions  in the limit $\ep\to 0$  are forbidden. In the following, we discuss some of the challenges in more detail (we concentrate on the planar case $d=2$ for simplicity), and then describe the approach adopted in this work. 

\begin{figure}[h]
\centering
\begin{tikzpicture}
% e_1 points right, e_2 points up, 
% four external points
\coordinate (A) at (-2,3);
 \coordinate (B) at (0,0);
 \coordinate (C) at (-4,0);
 \coordinate (D) at (-2,-3);
 % points on the right y=3/2x-3
 \coordinate (E) at (-1,1.5);
 \coordinate (F) at (-0.67,1);
  \coordinate (G) at (-0.67,-1);
  \coordinate (H) at (-1.34,-2);
  \coordinate (I) at (-1.67,-2.5);
  \coordinate (L) at (-1.8,-2.7);
  \coordinate (M) at (-1.9,-2.85);
  %points on the left
   \coordinate (N) at (-3,1.5);
 \coordinate (O) at (-3.33,1);
  \coordinate (P) at (-3.33,-1);
  \coordinate (Q) at (-2.66,-2);
  \coordinate (R) at (-2.33,-2.5);
  \coordinate (S) at (-2.2,-2.7);
  \coordinate (T) at (-2.1,-2.85);
  %lines and colors
  \draw (A)--(B)--(D)--(C)--cycle;
 % \draw (E)--(N);
 % \draw (F)--(O);
%  \draw (G)--(P);
  \draw (H)--(Q);
 % \draw[dashed] (I)--(R);
 % \draw[dashed] (L)--(S);
  %\draw[dashed] (M)--(T);
  \draw[fill=orange!20] (A)--(B)--(C)--cycle;
  \draw[pattern=north west lines, pattern color=blue] (-3.84,0.2)--(-0.16,0.2)--(B)--(C)--cycle;
  
  \draw[fill=blue!30] (B)--(H)--(Q)--(C)--cycle;
% \draw[fill=orange!20] (O)--(F)--(B)--(G)--(P)--(C)--cycle;
% \draw[fill=blue!30] (P)--(G)--(H)--(Q)--cycle;
  \draw[fill=orange!20] (Q)--(H)--(D);
   \draw[pattern=north west lines, pattern color=blue] (-2.55,-2.2)--(-1.44,-2.2)--(H)--(Q)--cycle;
  % \draw[dashed,fill=blue!20] (R)--(I)--(L)--(S)--cycle;
% \draw[dashed,fill=orange!10] (S)--(L)--(M)--(T);
 \node at (-2,-1) {$A$};
 %\node at (-2,-1.5) {$A$};
 % \node at (-2,1.25) {$A$};
 \node at (-2,1.25) {$B$};
  \node at (-2,-2.5) {$B$};
  \draw[>=triangle 45, <->] (-2.68,-1.8) -- (-1.32,-1.8);
   \node at (-1.9,-1.5) {$\eps^{r/2}$};
   \draw[>=triangle 45, <->] (-2.9,-2) -- (-2.9,-3);
    \node at (-3.4,-2.4) {$\eps^{r/2}$};
    \node at (-4,3){$(a)$};
 
 %% Following is for debugging purposes so you can see where the points are
%% These are last so that they show up on top
%\foreach \xy in {A, B, C, D, E, F, G, H, I, L, M, N, O, P, Q, R, S, T}{
% \node at (\xy) {\xy};}

 \coordinate (A2) at (4,3);
 \coordinate (B2) at (6,0);
 \coordinate (C2) at (2,0);
 \coordinate (D2) at (4,-3);
 % points on the right y=3/2x-3
 \coordinate (E2) at (5,1.5);
 \coordinate (F2) at (5.33,1);
  \coordinate (G2) at (5.33,-1);
  \coordinate (H2) at (4.66,-2);
  \coordinate (I2) at (4.33,-2.5);
  \coordinate (L2) at (4.2,-2.7);
  \coordinate (M2) at (4.1,-2.85);
  %points on the left
   \coordinate (N2) at (3,1.5);
 \coordinate (O2) at (2.67,1);
  \coordinate (P2) at (2.67,-1);
  \coordinate (Q2) at (3.34,-2);
  \coordinate (R2) at (3.67,-2.5);
  \coordinate (S2) at (3.8,-2.7);
  \coordinate (T2) at (3.9,-2.85);
  %lines and colors
 % \draw (A2)--(B2)--(G2)--(P2)--(C2)--cycle;
 % \draw (E2)--(N2);
 % \draw (F2)--(O2);
 % \draw (G2)--(P2);
 % \draw[dashed] (H)--(Q);
 % \draw[dashed] (I)--(R);
 % \draw[dashed] (L)--(S);
  %\draw[dashed] (M)--(T);
  \draw[fill=blue!30] (A2)--(B2)--(G2)--(P2)--(C2)--cycle;
  \node[shape=ellipse, minimum height=0.8cm, minimum width=1.7cm, draw=black, fill=orange!20] at (4,1.5) {};
  \node[shape=ellipse, minimum height=0.8cm, minimum width=1.7 cm, draw=black, pattern=north west lines, pattern color=blue] at (4,1.5){};
  \node[shape=ellipse, minimum height=0.1cm, minimum width=1.4cm, draw=black, fill=orange!20] at (4,1.5) {$B$};
  \draw[dashed] (3.3,1.5)--(4.7,1.5);
  \node at (4.42,1.65) {$\Gamma$};
 % \draw[fill=blue!30] (N2)--(E2)--(F2)--(O2)--cycle;
% \draw[fill=orange!20] (O2)--(F2)--(B2)--(G2)--(P2)--(C2)--cycle;
 \draw[fill=orange!20] (P2)--(G2)--(H2)--(Q2)--cycle;
  \draw[pattern=north west lines, pattern color=blue] (P2)--(G2)--(5.2,-1.2)--(2.82,-1.2)--cycle;
   \draw[pattern=north west lines, pattern color=blue] (Q2)--(H2)--(4.8,-1.8)--(3.2,-1.8)--cycle;
  \draw[fill=blue!30] (Q2)--(H2)--(D2)--cycle;
  % \draw[dashed,fill=blue!20] (R2)--(I2)--(M2)--(T2)--cycle;
% \draw[dashed,fill=orange!10] (S2)--(L2)--(M2)--(T2);
\draw[>=triangle 45, <->] (2.4,-0.9) -- (2.4,-2.1);
\draw [->] (7,-1) to [out=150,in=30] (4,-1);
\node at (7.3,-1) {$u^+$};
\draw [->] (7,-2) to [out=-150,in=-30] (4,-2);
\node at (7.3,-2) {$u^-$};
 \node at (2,-1.5) {$w_{\ep}$};
% \node at (4,0) {$B2$};
 \node at (4,-1.5) {$B$};
\node at (4,0.5) {$A$};
\node at (4,-2.5) {$A$};
\node at (2,3) {$(b)$};
 %\node at (4,2.25) {$B2$};

%% Following is for debugging purposes so you can see where the points are
%% These are last so that they show up on top
%\foreach \xy in {A2, B2, C2, D2, E2, F2, G2, H2, I2, L2, M2, N2, O2, P2, Q2, R2, S2, T2}{
% \node at (\xy) {\xy};
%}
\end{tikzpicture}
\caption{(a) Illustration of the $A$ and $B$ phase regions  of a deformation $y^\eps$ with finite energy \eqref{eq: I functional} in dimension $d=2$. The shadowed regions, where a transition of the gradient between $SO(2)A$ and $SO(2)B$ occurs, are horizontal reflecting the laminate structure of configurations with bounded energy. For the energy \eqref{eq: F functional}, the phase transition at the lower boundary is possible, whereas the transition in the upper part would lead to unbounded energies as $\eps \to 0$,  cf.\ \eqref{eq: minority phase scaling}.  (b)   In the upper part of the figure, we depict a minority island  centered around a segment $\Gamma$, which may have length $\sim 1$ in the $e_1$-direction, but width at most $\sim \eps$, cf.\ \eqref{eq: conti estimate}. Such a set necessarily has curved boundaries and is also penalized by the elastic energy in a neighborhood of the island. In the lower part, the phenomenon described in \eqref{eq: jump height} is illustrated.}
\label{fig1}
\end{figure}
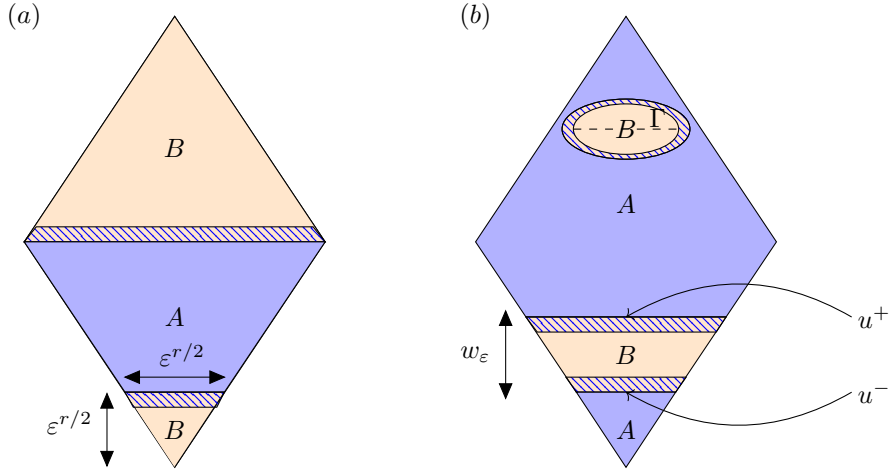

\textbf{(a) Volume of the minority phase.} In the model \eqref{eq: F functional}, the $B$-phase region, i.e., the set where the deformation gradient $\nabla y^\eps$ takes values in a neighborhood of $SO(d)B$, denoted by $T^\eps_B$ in the following,   has small $\mathcal{L}^2$-measure. Heuristically, this property can be seen as follows.  From the boundedness of the energy and H4.\ one can deduce, for a suitable definition of $T_B^{\ep}$, that      
\begin{align}\label{eq: minority phase scaling}
\mathcal{H}^{1} (\partial T^\eps_B \cap \Omega ) \le  C \Vert \dist(\nabla y^\eps,SO(2)) \Vert_{L^2(\Omega)} \Vert \nabla^2 y^\eps \Vert_{L^2(\Omega)} \le C\eps \eps^{\frac{\gamma_2(r) }{2}-1} = \eps^{\frac{r}{2}},
\end{align}
 where in the last step we used $\gamma_2(r) = r$, see below \eqref{eq: F functional}. (We refer to \cite[Proof of Proposition 3.7, Step 1]{davoli.friedrich} for details on the first inequality.) By the (relative) isoperimetric inequality  we obtain 
$$\min \lbrace  \mathcal{L}^2(T_B^{\ep}), \mathcal{L}^2(\Omega \setminus T_B^{\ep}) \rbrace \le C\eps^{r}. $$ 
Assuming that $T^\eps_B$ is the \emph{minority phase}, i.e., the minimum is attained for $T^\eps_B$, we get
\begin{align}\label{eq: sharp bound}
\mathcal{L}^2(T^\eps_B) \le C\eps^r.
\end{align}
This scaling of the area of the minority phase excludes phase transitions of the form given in Figure \ref{fig1}(a) where both $\mathcal{L}^2(T^\eps_B)$ and $\mathcal{L}^2(\Omega \setminus T^\eps_B)$ are bounded uniformly from below. It is worth mentioning that the calculation \eqref{eq: minority phase scaling} for the model \eqref{eq: I functional} (corresponding to $r=0$) would give $\mathcal{H}^{1}(\partial T_B^{\ep}) \le C$. This reflects the fact that  phase transitions in the limit $\ep\to 0$  are possible in that framework, see Lemma~\ref{lemma:comp-def}, Figure~\ref{fig:limiting-def}, and Figure~\ref{fig1}(a).

\textbf{(b) Criticality of the scaling.} For compactness of rescaled displacement fields $u^\eps = (y^\eps -{\rm id})/\eps$, see \eqref{eq: simple u-def}, we necessarily need  $\mathcal{L}^2(T^\eps_B) \to 0$ as otherwise $|\nabla u^\eps| \to +\infty$ on a set of positive measure. More precisely, since   $|\nabla u^\eps | \sim 1/\eps$ on $T^\eps_B$,  it turns out that the bound in \eqref{eq: sharp bound} is sharp in order to derive the uniform estimate $\Vert \nabla u^\eps \Vert_{L^r(\Omega)} \le C$ of Lemma \ref{lemma:comp-giuliano}. 

Recall that \eqref{eq: sharp bound} was derived from \eqref{eq: minority phase scaling} via the isoperimetric inequality. One may ask if this estimate is sharp, i.e.,  if the scaling $ \eps^{2-\gamma_2(r)} = \eps^{2-r}$ of the  penalization in   \eqref{eq: F functional} is really necessary to obtain \eqref{eq: sharp bound}. For a small region near the boundary of $\Omega$ whose boundary in $\Omega$ is a short straight line of length $\sim \eps^{\frac{r}{2}}$ (see  Figure \ref{fig1}(a)) the scaling is indeed critical. (We also refer to Example 3.2 in \cite{alicandro.dalmaso.lazzaroni.palombaro}.)  As  the interface between the two phases is horizontal, such a transition is only realizable close to the boundary. For  small inclusions of the $B$-phase in the interior, so-called \emph{minority islands}, this is impossible, see Figure \ref{fig1}(b).

\textbf{(c) Minority islands.} The situation for such  minority islands is indeed quite different. In dimension two and without a strong second-order penalization, merely under the assumption that in a neighborhood $N$ of the island the quantity $\int_N |\nabla^2 y^\eps|\, {\rm d}x$ is smaller than a universal constant independent of $\eps$,  {\sc S. Conti} and {\sc B. Schweizer} \cite[Proposition 2.1]{conti.schweizer} derived the remarkable bound
\begin{align}\label{eq: conti estimate}
\mathcal{L}^2(T_B^{\ep}) \le C \int_\Omega \dist(\nabla y^\eps, SO(2)\lbrace A,B\rbrace) \, {\rm d}x  \le C\eps, 
\end{align}
where the last step follows from the boundedness of the elastic energy. Roughly speaking, they showed that minority islands, although possibly being long in the $e_1$-direction (the direction orthogonal to the rank-one connection), have width at most $\sim \eps$ in the $e_2$-direction, cf.\  Figure \ref{fig1}(b).   Their result is indeed sharp in the sense that they provide a configuration with a minority island of length $\sim 1$ and width $\sim \eps$  such that the energy \eqref{eq: I functional} is bounded uniformly in $\eps$, see \cite[Remark 6.1] {conti.schweizer2}. A $d$-dimensional analogue has been provided in \cite[Remark 3.9]{davoli.friedrich}.

\textbf{(d) Internal jumps.} This phenomenon excludes compactness in $W^{1,r}$ for every $r>1$, even if for a sequence $\lbrace y^\eps \rbrace_\eps$ there is only a single minority island of width $\eps$ in the $e_2$-direction around a $1$-dimensional horizontal set $\Gamma$. Indeed, in that scenario  the strain  $|\nabla u^\eps |$ of the rescaled displacement fields $u^\eps$ (see \eqref{eq: simple u-def}) would scale  like $1/\eps$ on a set of $\mathcal{L}^2$-measure $\sim \eps$, and one could expect no Sobolev compactness. On the contrary, it would be natural for $u^\eps$ to converge to an $SBV$ function which jumps on $\Gamma$. In the following, we will refer to the setting   described  above   as that of \emph{internal jumps}. We again recall that this issue is excluded in the model \eqref{eq: F functional} by the bound \eqref{eq: sharp bound}. 

 \textbf{(e) Double phase transitions.} A similar phenomenon may occur in the presence of a $B$-phase layer with width $w_\eps \sim \eps$ as indicated in the lower part of Figure \ref{fig1}(b) which corresponds to  two  `consecutive phase transitions'. Heuristically, denoting by $y^\eps_+(x'),\, y^\eps_-(x')$,  $u^\eps_+(x')$, and $u^\eps_-(x')$ the traces of $y^\ep$ and $u^\eps$ on the upper and lower boundary (with respect to the $e_2$-direction) of such a layer, one expects that $y^\ep_+(x')\approx y^\ep_-(x')+w_\ep Be_2 $, and thus
\begin{align}\label{eq: jump height}
\lim_{\eps \to 0} \big( u^\eps_+(x')  - u^\eps_-(x')\big) = \lim_{\eps \to 0} \frac{y^\eps_+(x')  - y^\eps_-(x') - w_\eps e_2   } {\eps}       = const. = \kappa  \lim_{\eps \to 0} \frac{w_\eps}{\eps} e_2,
\end{align}
where we recall  \eqref{eq: simple u-def} and the fact that $(B-A)e_2 =  \kappa  e_2 $, see H3.  Consequently,  the limiting  function would jump with constant jump height $\kappa  \lim_{\eps \to 0} \frac{w_\eps}{\eps} e_2$. Interestingly, the jump height is essentially determined by $w_\eps$, i.e., by the width of the $B$-phase layer. Let us also mention an additional problem occurring if $w_\eps \gg \eps$: in this latter setting the sequence of rescaled displacement fields would not even converge to an $SBV$ function, cf. \eqref{eq: jump height}.

\textbf{The perspective of the present work.} The goal of the present contribution is to overcome the above mentioned issues. In particular, building upon a novel two-well rigidity estimate proved in \cite{davoli.friedrich} for a model augmented by a suitable anisotropic second-order penalization (see Subsection \ref{sec: refined model}), we will introduce a \emph{generalized definition of the rescaled displacement fields} which takes the presence of the two phases $A$ and $B$ in different parts of the domain into account. Roughly speaking, these displacement fields will measure the distance of the deformations $y^\eps$ from suitable rigid movements which may be different on the components of a partition of $\Omega$ induced by the $A$ and $B$ phase regions. This more flexible definition will allow us  to carry out the following tasks in any dimension $d\ge 2$:
\begin{itemize}
\item derive a linearization result for configurations where both phases are present, in particular where   phase transitions occur;
\item  obtain compactness results in a piecewise Sobolev setting for generalized rescaled displacements, despite the presence of  minority islands with macroscopic length;
\item  identify an effective limiting model comprising linearized elastic energies and contributions for single and double phase transitions. 
\end{itemize}
In our investigation, however, we do not take the presence of internal jumps into account  for this would lead to a considerably more involved limiting energy, see Remark \ref{rem:internal jumps} for a discussion in that direction. From a modeling point of view, this amounts to excluding the presence of minority islands of width $\sim \eps$ (see Figure \ref{fig1}(b)), whereas minority islands of width $\ll \eps$ are allowed. In our model, this will be achieved by considering  a suitable anisotropic second-order penalization.

%%%%%%%%%%%%%%%%%%%%%%%%%%%%%%%%%%%%%%%%%%%%%%%%%%%%%%%%%%%%%%%%%%%%%
\section{The model and main results} 
\label{sec:main-thm}

In this section we introduce our model with a refined singular perturbation,  state  the rigidity estimate proved in \cite{davoli.friedrich}, and present our main results.

%%%%%%%%%%%%%%%%%%%%%%%%%%%%%%%%%%%%%%%%%%%%%%%%%%%%%%%%%%%%%%%%%%%%%
\subsection{A model with a refined singular perturbation and its sharp-interface limit}\label{sec: refined model}

In this subsection we present the exact  mathematical setting of this paper  and recall our previous work \cite{davoli.friedrich}. We analyze a nonlinear energy given by the sum of  the non-convex elastic energy, a singular perturbation, and a higher-order penalization in the direction orthogonal to the rank-one connection. To be precise, for each $\eps,\eta>0$, we consider the functional
 \begin{align}\label{eq: nonlinear energy}
E_{\ep,\eta}(y):=\frac{1}{\ep^2}\int_{\Omega}W(\nabla y)\,{\rm d}x+\ep^2\int_{\Omega}|\nabla^2 y|^2\,{\rm d}x+\eta^2 \int_{\Omega}(|\nabla^2 y|^2-|\partial^2_{dd} y|^2)\,{\rm d}x
\end{align}
for every $y\in H^2(\Omega;\R^d)$. This corresponds to the choice 
$$P_\eps(G) = \eps^2 |G|^2 + \eta^2 \sum_{i=1}^d\sum_{\substack{(j,k)\in \{1,\dots,d\}^2,\\ (j,k)\neq (d,d)}}|G_{ijk}|^2,\quad G \in \R^{d \times d \times d},$$ in \eqref{eq: basic energy}. Note that \eqref{eq: nonlinear energy} coincides  with the energy functional in \eqref{eq: I functional} when $\eta = 0$.   In what follows, we will study the asymptotic behavior of the energies 
\begin{align}\label{eq: specific energy choice}
\mathcal{E}_{\ep}:=E_{\ep,\bar{\eta}_{\ep,d}},
\end{align}
 where  $\lbrace \bar{\eta}_{\ep,d} \rbrace_\eps$ is defined by  
\begin{equation}\label{eq:alphad}
\bar{\eta}_{\ep,d}:=\eps^{-1+\alpha(d)},\quad\text{ with }\quad \alpha(d): =1/(2d).
\end{equation}
 We refer to Remark \ref{rk:seq-eta} below for details on the choice of the parameter. We denote the restriction of $\mathcal{E}_{\ep}$ to a subset $\Omega' \subset \Omega$ by $\mathcal{E}_{\eps}(y,\Omega')$. 
In \cite[Proposition 4.3, Theorem 4.4, and Remark 4.5]{davoli.friedrich} we have shown that the asymptotic behavior of the energies $\mathcal{E}_{\ep}$ is described (via $\Gamma$-convergence in the strong $L^1$-topology) by the sharp-interface model $\mathcal{E}_0\colon L^1(\Omega;\R^d)\to   [0,+\infty]$, given by 
\begin{align*}
%\label{eq: limiting energy}
\mathcal{E}_0(y):=\begin{cases}
 K \,  \mathcal{H}^{d-1}(J_{\nabla y})&
\text{if }y\in  \mathscr{Y}(\Omega),\\
+\infty&\text{otherwise},\end{cases}
 \end{align*}
where the \emph{optimal-profile energy} is defined  by
\begin{align}\label{eq: our-k1}
K :=\inf\Big\{&\liminf_{\ep\to 0} \mathcal{E}_{\ep}(y^{\ep},Q): \   \lim_{\eps \to 0}  \Vert  y^\eps -  y_0^+ \Vert_{L^1(Q)}  = 0\Big\}.
\end{align}
Here,  $Q=(-\tfrac12,\tfrac12)^d$ again denotes  the $d$-dimensional unit cube centered in the origin, $y_0^+$ was defined below \eqref{eq: conti-schweizer-k},  and for the definition of $\mathscr{Y}(\Omega)$ we refer to \eqref {eq: limiting deformations}. Note that \eqref{eq: our-k1} is the counterpart to \eqref{eq: conti-schweizer-k} for the model in \eqref{eq: nonlinear energy}. From the definition of the optimal-profile energy  and the fact that the penalization in  \eqref{eq: nonlinear energy} (with $\eta=  \bar{\eta}_{\ep,d}$) is stronger than the one in \eqref{eq: I functional}, we deduce the  inequality $K \ge K_0$.   
As pointed out in \cite[Remark 4.5]{davoli.friedrich},  the additional penalization term in \eqref{eq: nonlinear energy} with respect to \eqref{eq: I functional} does not affect the qualitative behavior of the sharp-interface limit, only the constant in \eqref{eq: our-k1} may change. Moreover, the fact that $ \bar{\eta}_{\ep,d}  \ll  \ep^{-1}$ guarantees that, asymptotically when passing to a linearized strain regime, the resulting model does not feature second-order derivatives, see \cite[Introduction]{davoli.friedrich} and Remark~\ref{rk:qualitative} below.

We mention that anisotropic singular perturbations have already been used in related problems, see  e.g.\ \cite{kohn.muller, Zwicknagl}. In the present context, the role of the perturbation is twofold: (1) It allows us to use the two-well rigidity estimate proved in  \cite{davoli.friedrich}, see Theorem \ref{thm:rigiditythm} below.  (2) As discussed at the end of Subsection \ref{sec: heuristics/challenges},  the penalization simplifies the analysis  by excluding the formation of \emph{internal jumps}  for limiting displacement fields, see Remark \ref{rem:internal jumps} below for more details. We remark that this anisotropy is the reason why we study the case of exactly one rank-one connection.

\begin{remark}[Choice of the penalization constant]
\label{rk:seq-eta}
We briefly mention that the result in \cite{davoli.friedrich} is slightly more general in the sense that it holds also for penalization constants $\{\eta_{\ep,d}\}_\ep$ with $\eta_{\ep,d} \ll \bar{\eta}_{\ep,d}$, see \cite[(4.5)]{davoli.friedrich}, i.e.,  our choice of the penalization constant here is `less sharp'. For the sake of simplicity rather than generality, we prefer to work with \eqref{eq:alphad} since it simplifies many estimates in the following. (In particular, the statement of the rigidity estimate in Theorem \ref{thm:rigiditythm} below becomes simpler.)   
\end{remark}

 Let us now recall the two-well rigidity result which is the fundamental ingredient for the proof of the aforementioned $\Gamma$-convergence result and, at the same time, is instrumental for our work. More precisely, in the present paper, besides yielding properties on optimal sequences in \eqref{eq: our-k1} necessary for deriving   the sharp-interface limit, this estimate plays additionally a pivotal role for showing compactness of sequences with equibounded energies and for providing an optimal lower bound for the asymptotic behavior of the sequence $\{\mathcal{E}_\ep\}_\eps$.  We present here a slightly simplified version of  \cite[Theorem 3.1]{davoli.friedrich} with  $p=2$ and  $\bar{\eta}_{\ep,d}$ in place of $\eta$.

\begin{theorem}[Two-well rigidity estimate]\label{thm:rigiditythm}
Let $\Omega$ be a bounded  Lipschitz  domain in $\R^d$ with $d\geq 2$, and let   $\lbrace \bar{\eta}_{\ep,d} \rbrace_\eps$  be  as in \eqref{eq:alphad}. Suppose that $W$ satisfies  {\rm H1}.--{\rm H4}.  Let $E>0$. Then for each $\Omega' \subset \subset \Omega$ there exists a constant  $C=C(\Omega,\Omega', \kappa, c_1, E)>0$  such that for every $y\in H^2(\Omega;\R^d)$ with $\mathcal{E}_\eps(y)\leq E$ there exist  a rotation $R\in SO(d)$ and a phase indicator $\Phi \in BV(\Omega;\lbrace A,B \rbrace)$ satisfying
\begin{align}\label{eq: rigidity-new}
\Vert\nabla y-R\Phi \Vert_{L^2(\Omega')}\leq C\ep  \ \ \ \text{and} \ \ \  |D \Phi |(\Omega) \le C. 
\end{align}  
Additionally, the choice of the rotation $R$ and the phase indicator $\Phi$ is independent of the set $\Omega'\subset \subset\Omega$. If $\Omega$  is a paraxial cuboid, \eqref{eq: rigidity-new} holds on the entire domain $\Omega$ for a constant $C=C(\Omega,\kappa, c_1, E)> 0$. 
\end{theorem}

 We point out that  the result in \cite[Theorem 3.1]{davoli.friedrich} is more general. Indeed, it is stated for any $\eta\ge\eps$ and for a range of integrability exponents. The present version for the choice $\eta = \bar{\eta}_{\ep,d}$  is the counterpart  to the simplified version \cite[Theorem 1.1]{davoli.friedrich} on general bounded Lipschitz domains,  and for a non-sharp choice of $\alpha(d)$. We refer to \cite[Section 3]{davoli.friedrich} for additional motivation on this estimate, in particular for a comparison with other   quantitative rigidity estimates for multiwell energies.

 The focus of this contribution is on a $\Gamma$-convergence analysis of the energies $\mathcal{E}_{\ep}$ in a  topology different from the one specified above. It will lead to a limiting model simultaneously keeping track both of sharp interfaces between the two phases and of linearization effects. The precise topology for our $\Gamma$-convergence result  is detailed in Subsection \ref{sec: compactness} below, and the $\Gamma$-limit is presented in Subsection \ref{sec: effective model}. Due to the necessity of linearizing nonlinear elastic energies, we additionally need a      local    Lipschitz condition  for the construction of recovery sequences: besides the assumptions H1.-H6.\ stated in Subsection \ref{sec: nonlinear energy}, we also require
 \begin{itemize}
\item[H7.] (Local Lipschitz condition) there exists a constant $c_3>0$ such that
$$|W(F_1) - W(F_2)| \le c_3(1+|F_1| + |F_2|)|F_1-F_2|  \ \ \ \ \text{for all} \ \ F_1,F_2 \in \M^{d \times d}.$$  
 \end{itemize}
Moreover, for simplicity we will assume that 
 \begin{itemize}
\item[H8.] (Geometric condition) for all $t \in \R$ the set $\Omega \cap \lbrace x_d = t \rbrace$ is connected (whenever nonempty). 
 \end{itemize} 
 The latter condition is only needed for the compactness result in Theorem \ref{thm:compactness} and could be dropped at the expense of more elaborated arguments, see Remark \ref{rem: Geometric condition} for details.

%%%%%%%%%%%%%%%%%%%%%%%%%%%%%%%%%%%%%%%%%%
\subsection{Compactness}\label{sec: compactness}

This subsection is devoted to our main compactness result. Our approach consists in decomposing sequences of deformations $\lbrace y^\eps \rbrace_\eps$ with equibounded $\mathcal{E}_\ep$-energies into  the sum of two parts:
  \begin{itemize}
  \item[(a)] Piecewise rigid movements, where `piecewise' refers to associated Caccioppoli partitions induced by the $A$ and $B$ phase regions. These converge to the limit $y$ of the deformations $\lbrace y^\eps \rbrace_\eps$. 
  \item[(b)] Displacements, rescaled by $\ep$, whose strain is equibounded in $L^2$. These converge to a limiting displacement field, which is piecewise Sobolev, with possible jumps with normal in $e_d$-direction.   
\end{itemize}
In order to formulate the main result of this subsection, we need to introduce some notation. Denote by $\mathscr{P}(\Omega)$ the following collection of \emph{Caccioppoli partitions} of $\Omega$: 
\begin{align}
\label{eq:def-P}
\mathscr{P}(\Omega):=\Big\{\mathcal{P}=\{ {P}_j\}_j \text{ partition of $\Omega$}\colon &\   \bigcup\nolimits_j  \partial P_j  \cap \Omega \subset \bigcup\nolimits_{i\in \N} (\R^{d-1} \times \lbrace s_i \rbrace) \cap \Omega \text{ for  } \lbrace s_i \rbrace_i \subset \R\Big\}.
\end{align}
We point out that the partitions can be finite or may consist of countably many sets. (For simplicity, we do not specify the index set corresponding to the indices $j$.) The definition above implies that $\bigcup_j \partial P_j \cap \Omega$ consists of subsets of hyperplanes orthogonal to $e_{d}$, which extend up to the boundary of $\Omega$.  Note that   every Caccioppoli partition on the bounded domain $\Omega$ induces an ordered one just by a permutation of the indices. For this reason, throughout the paper we always tacitly assume that partitions are ordered. We will  say that $P^\eps \to P$ in measure as $\eps \to 0$ if $\chi_{P^\eps} \to \chi_P$ in $L^1$.  Let $\mathscr{U}(\Omega)$ be the set of \emph{displacements} whose jump sets are the union of countably many subsets of hyperplanes orthogonal to $e_{d}$, i.e.,
\begin{align}\label{eq: def of U}
 \mathscr{U}(\Omega)  :=\Big\{u\in SBV^2_{\rm loc}(\Omega;\R^d)\colon \  \,  J_u \subset \bigcup\nolimits_{i\in \N} (\R^{d-1} \times \lbrace s_i \rbrace) \cap \Omega \text{ for  } \lbrace s_i \rbrace_i \subset \R \Big\}.
\end{align}
For basic properties of Caccioppoli partitions and $SBV$ functions we refer to Appendix \ref{sec:appendix}. In particular, the essential boundary of a set is indicated by $\partial^*$. For sets $\Omega' \subset \Omega$ and $S \subset \Omega$, we denote by $\pi_d(S)$ the orthogonal projection of $S$ onto the $e_d$-axis, and define the  \emph{layer set} 
\begin{align}\label{eq: layer set}
L_{\Omega'}(S) = \Omega' \cap \big(\R^{d-1} \times \pi_d(S) \big).
\end{align}

We now state our main compactness result. Recall the definition of $\mathcal{Y}_R(\Omega)$ in \eqref{eq: limiting deformations}. 
\begin{theorem}[Compactness]
\label{thm:compactness}
 Let $\Omega\subset \R^d$ be a bounded Lipschitz domain satisfying {\rm H8}. Assume that $W$ satisfies assumptions {\rm H1}.--{\rm H4}.,  and  let $\{y^{\ep}\}_{\ep}\subset H^{2}(\Omega;\R^d)$ be a sequence of deformations satisfying the uniform energy estimate
\begin{equation}
\label{eq:unif-en-estimate}
\sup_{\ep>0} \mathcal{E}_{\ep}(y^{\ep}) \le C_0  < + \infty. 
\end{equation}
Then, up to the extraction of a  subsequence (not relabeled), the following holds:

(a)  (Piecewise rigidity)  There exist Caccioppoli partitions $\mathcal{P}^{\ep}:=\{{P}^{\ep}_j\}_{j}$ of $\Omega$ such that 
\begin{align}
&\label{eq: partition property}  \mathcal{H}^{d-1}\big(\bigcup\nolimits_j \partial^* P^\eps_j \big) \le C, \\
& \label{eq: partition property-new}  \sum\nolimits_j \min\big\{ \mathcal{L}^d(\Omega' \cap P^\eps_j), \mathcal{L}^d( L_{\Omega'}(P^\eps_j) \setminus P^\eps_j  )  \big\} \le C_{\Omega'} \, \eps^p   \ \ \ \ \text{for every } \Omega' \subset \subset \Omega, 
\end{align}
for some $p=p(d) \in (1,2)$, where $C$ depends only on $C_0$ and $\Omega$, and $C_{\Omega'}$ additionally on $\Omega'$. There exist associated rotations $R^{\ep}\in SO(d)$, as well as collections of phase indicators $\mathcal{M}^\ep:=\{M_j^{\ep}\}_{j}$, with $M^{\ep}_{j}\in \{A,B\}$ for every $j$ and $\ep$, such that 
\begin{equation}
\label{eq:rigidity-compactness}
\sup_{\eps >0} \  \|\nabla y^{\ep}-\sum\nolimits_j R^{\ep} M^{\ep}_j\chi_{P^{\ep}_j}\|_{L^2(\Omega')} \le   C_{\Omega'} \, \eps  \ \ \ \  \text{for every } \Omega' \subset \subset \Omega. 
\end{equation}

(b) (Limiting deformation and partition) There exist a limiting rotation $R\in SO(d)$, a limiting deformation $y \in \mathcal{Y}_R(\Omega)$, and a limiting partition $\mathcal{P}= \{{P}_j\}_{j} \in \mathscr{P}(\Omega)$ such that 
\begin{align}
&\label{eq:comp-R} R^{\ep}\to R,\\
&\label{eq:comp-sets} P_j^{\ep}\to P_j\quad\text{in measure for all $j$},\\
&\label{eq:comp-y} y^{\ep}-\fint_{\Omega}y^{\ep}(x)\,{\rm d}x\to y \quad\text{strongly in }H^1(\Omega;\R^d),\\
&\label{eq:comp-M}  \sum\nolimits_j R^\eps M^\eps_j\chi_{P^\eps_j} \rightharpoonup^* \nabla y   \quad\text{weakly* in } BV(\Omega;\M^{d \times d}). 
\end{align}

(c) (Displacements) We find collections of constants $\mathcal{T}^{\ep}:=\{t_j^{\ep}\}_{j} \subset \R^d $, associated to $\mathcal{P}^{\ep}$, satisfying
\begin{align}\label{eq: toinfty}
  \frac{|t_i^{\ep}-t_j^{\ep}|}{\ep}\to +\infty \ \ \ \ \text{for all $i \neq j$ with $\mathcal{L}^d(P_i), \mathcal{L}^d(P_j)>0$, and $\lim_{\eps \to 0} M^\eps_i = \lim_{\eps \to 0} M^\eps_j$,}
\end{align} 
 and defining the rescaled displacement fields associated to $\mathcal{P}^\eps, \mathcal{M}^\eps$, $\mathcal{T}^\eps$, and $R^\eps$ by 
\begin{align}\label{eq: rescaled disp}
u^{\ep}:=\frac{y^{\ep}-\sum\nolimits_j ( R^{\ep} M_j^{\ep}\, x+t_j^{\ep})\chi_{P^{\ep}_j}}{\ep},
\end{align} 
there exists   $u\in \mathscr{U}(\Omega)$ such that 
\begin{align}
&\label{eq:comp-u} u^{\ep}\to u\quad\text{in measure in }\Omega,\\
&\label{eq:comp-grad-u}\nabla u^{\ep}\wk \nabla u\quad\text{weakly in } L^2_{\rm loc}(\Omega;\M^{d\times d}). 
\end{align}
\end{theorem}

In view of our compactness result, sequences of deformations having equibounded energies decompose into the sum of piecewise  rigid movements  with gradients $\sum\nolimits_j R^{\ep} M^{\ep}_j\chi_{P^{\ep}_j}$,  reflecting also the different phases $A$ and $B$, and scaled $SBV$-displacements $u^\eps$ whose gradients are uniformly bounded in $L^2_{\rm loc}(\Omega; \M^{d\times d})$. Let us comment on the compactness result and on some of the proof ideas. 

The definition of the piecewise rigid movements, as well as \eqref{eq: partition property}-\eqref{eq:rigidity-compactness}, follow from the geometric two-well rigidity result recalled in Theorem \ref{thm:rigiditythm}. In particular, \eqref{eq: partition property-new} shows that each component has either small volume or coincides  (up to a small set)  with a `layer' of $\Omega'$. (We also refer to Figure \ref{fig:small-trans} below for a 2d illustration.) At this point, the passage to subdomains is necessary and in \eqref{eq:rigidity-compactness} we control the quantities only in $L^2_{\rm loc}$,  cf.\ \eqref{eq: rigidity-new}.  If $\Omega$ is a paraxial cuboid, this passage can be avoided, see Remark \ref{rem: Geometric condition} for details  in that direction.  Let us also emphasize that the rotation $R^\eps$ is defined \emph{globally}, i.e., it is independent of the components of the partition $\mathcal{P}^\eps$. 

Standard compactness results (see Theorem \ref{th: comp cacciop})  imply \eqref{eq:comp-R}-\eqref{eq:comp-sets}, whereas \eqref{eq:comp-y} follows from  Lemma \ref{lemma:comp-def}, and for \eqref{eq:comp-M}  we also take \eqref{eq:rigidity-compactness} into account. The global point of view for phase transitions given in Lemma \ref{lemma:comp-def} is combined  with a local one in \eqref{eq: toinfty}-\eqref{eq:comp-grad-u}:   the Caccioppoli partitions play the role of identifying subdomains where the small-strain displacement fields defined in \eqref{eq: rescaled disp} satisfy good compactness properties \eqref{eq:comp-u}-\eqref{eq:comp-grad-u}.  

In this context, condition \eqref{eq: toinfty} represents a selection principle for the Caccioppoli partitions.  (Note that $\lim_{\eps \to 0} M^\eps_k$ for $k=i,j$ is well defined by \eqref{eq:comp-R}, \eqref{eq:comp-sets}, and \eqref{eq:comp-M}.)   Loosely speaking, it implies that two regions of the domain in the same phase, say phase $A$, are represented in the limit by two different sets $P_i$ and $P_j$ if and only if along the sequence $\{\mathcal{P}^{\ep}\}_\eps$ there is a layer contained in the $B$-phase region lying between $P_i^{\ep}$ and $P_j^{\ep}$ whose width  is asymptotically (as $\eps \to 0$) much larger than $\eps$,   cf.\ the discussion below \eqref{eq: jump height}.    We emphasize that, without the selection principle \eqref{eq: toinfty}, there might be different possible choices for the limiting partition, as the following example shows.

\begin{example}[Non-uniqueness of limiting partition]\label{ex}
{\normalfont The choice of different partitions at level $\ep$  is  not equivalent. In particular, different  $\ep$-decompositions determine different limiting displacements and Caccioppoli partitions, which may contain a different `amount of information'. To clarify this, consider the following two-dimensional example.  (For related examples, we refer to \cite[Example 2.5]{Engineer} or \cite[Example 2.4]{Friedrich-ARMA}). Let
$$\Omega= (0,1) \times (0,2),\quad \Omega_1 = (0,1) \times (0,1), \quad \Omega_2 = (0,1) \times (1,2)$$
and for $\eps >0$ and $l \in \lbrace 1/2,1,2 \rbrace$ consider the sets
$$\Omega^{\eps,l}_3 = (0,1)  \times  (1-\eps^{l},1+\eps^{l}), \quad \Omega^{\eps,l}_1 = \Omega_1 \setminus \Omega^{\eps,l}_3, \quad \Omega^{\eps,l}_2 = \Omega_2 \setminus \Omega^{\eps,l}_3.$$ 
We define three different example sequences according to the value of $l$: first, define $\tilde{y}^{\eps,l} \in H^1(\Omega;\R^2)$ by  
\begin{align*}
\tilde{y}^{\eps,l}(x) := \begin{cases}   x & x \in \Omega^{\eps,l}_1\\  Bx - \kappa (1-\eps^l)e_2 & x \in \Omega^{\eps,l}_3 \\ x + 2\kappa \eps^{l} e_2 & x \in \Omega_2^{\eps,l} \end{cases}
\end{align*}
for every $x\in \Omega$,  where $\kappa$ is given in H3.,  and then 
$$y^{\eps,l} := \tilde{y}^{\eps,l} * \frac{1}{\eps^4}\varphi(\cdot/\eps^2),$$ where $\varphi\colon \R^2 \to \R^2$ is a standard mollifier with ${\rm supp}(\varphi) \subset B_1(0)$.  One can check that \EEE $\sup_{\ep>0} \mathcal{E}_{\ep}(y^{\eps,l})< \infty$. There are two natural alternative decompositions of the maps $y^{\ep,l}$, namely  
 \begin{align*}
(1)& \ \ y^{\ep,l}= (R^{\eps,l}M^{\eps,l}_1\,x + t_1^{\eps,l}) \chi_{P^{\eps,l}_1} + \ep u^{\eps,l},\quad\quad\text{and}\quad\quad
(2) \ \ y^{\ep,l}=\sum\nolimits_{j=1}^3 (\hat{R}^{\eps,l}\hat{M}^{\eps,l}_j\,x + \hat{t}^{\eps,l}_j)\chi_{\hat{P}^{\eps,l}_j} + \ep \hat{u}^{\eps,l},
\end{align*}
where $R^{\eps,l} = \hat{R}^{\eps,l} = {\rm Id}$ and the Caccioppoli partitions, phases, and constant translations are defined as 
\begin{align*}
(1)& \ \ P^{\eps,l}_1 = \Omega, \ \ \ \ M_1^{\eps,l} = A, \ \ \ \  t_1^{\eps,l} = 0,\\
(2)& \ \ \hat{P}^{\eps,l}_j = \Omega^{\eps,l}_j, \ \  \ \  \hat{M}^{\eps,l}_1= \hat{M}^{\eps,l}_2  = A, \  \hat{M}^{\eps,l}_3 = B,  \ \ \ \ \   \ \ \hat{t}_1^{\eps,l} =  0,  \   \hat{t}^{\eps,l}_2  = 2\kappa\eps^l  e_2 - b\eps, \  \hat{t}^{\eps,l}_3 = -\kappa (1-\eps^l)e_2,
  \end{align*}
respectively, where $b \in \R^2$ is some arbitrary translation. This leads to the different limiting displacement fields and Caccioppoli partitions
\begin{align*}
&(1) \ \  u^l = 0 \cdot \chi_{\Omega_1} + s^le_2\chi_{\Omega_2}, \ \ \ \  \ \ \ \ \ \ \ \  \ \ P^l_1 = \Omega,\\
 &(2) \ \ \hat{u}^l =  0 \cdot \chi_{\Omega_1} +  b\chi_{\Omega_2}, \ \ \ \ \ \ \ \ \ \  \  \hat{P}^l_1 = \Omega_1, \ \hat{P}^l_2 = \Omega_2, \ \hat{P}^l_3 = \emptyset, 
  \end{align*}
  where $s^l:= 2\kappa  \lim_{\eps \to 0} \eps^{l-1}$ for $l \in \lbrace 1/2,1,2 \rbrace$.

In case (2), where the sets $\Omega_1$ and $\Omega_2$ are split in the limiting partition, the limiting displacement does not provide any information  on the behavior of the deformations at the $\ep$-level.  Note that the translation $b \in \R^2$ just expresses the non-uniqueness of the limiting configuration and does not have any physically reasonable interpretation, see Proposition  \ref{prop:ex-coarsest-part} below. On the contrary, in case (1) the jump height of the limiting displacement on $\partial \Omega_1 \cap \partial \Omega_2$ provides information on the width of the intermediate layer $\Omega_3^{\eps,l}$ where the deformation is in phase $B$: The jump heights $s^2 = 0$ and $s^1 = 2\kappa$ express that the width is of order $\ll \eps$ and $\sim \eps$, respectively. As $s^{1/2} = \infty$, we observe that $u^{1/2} \notin \mathscr{U}(\Omega)$. Thus, alternative (1) is not allowed  in the case $l=1/2$ and the sets $\Omega_1$ and $\Omega_2$ have to be split in the limiting partition. The observation that coarser partitions provide more information suggests to define the partition `as coarse as possible'. This intuition is exactly reflected in the selection principle \eqref{eq: toinfty}: for $l=1,2$ we apply Case (1) and only for $l=1/2$ we apply Case (2).} \nopagebreak\hspace*{\fill}$\Box$
\end{example}

 As a consequence of Theorem \ref{thm:compactness}, we introduce the following notion of convergence.

\begin{definition}
\label{def:convergence1}
(i) We say that a sequence of deformations $\{y^{\ep}\}_\eps$ is \emph{asymptotically represented} by a limiting triple $(y,u,\mathcal{P})\in \mathscr{Y}(\Omega)\times \mathscr{U}(\Omega)\times \mathscr{P}(\Omega)$, and write
$$y^{\ep}\to (y,u,\mathcal{P}),$$
if there are sequences $\lbrace R^\eps \rbrace_\eps$, $\lbrace \mathcal{P}^\eps\rbrace_\eps$, $\lbrace \mathcal{M}^\eps\rbrace_\eps$, and $\lbrace \mathcal{T}^\eps\rbrace_\eps$ such that  \eqref{eq: partition property}--\eqref{eq:comp-grad-u} hold.\\
(ii) We call a sequence of quadruples $(R^\eps, \mathcal{P}^\eps, \mathcal{M}^\eps, \mathcal{T}^\eps)$ \emph{admissible for $\lbrace y^\eps \rbrace_\eps$} if \eqref{eq: partition property}--\eqref{eq:comp-grad-u} are satisfied.\\ 
(iii)  We call a triple $(y,u,\mathcal{P})\in \mathscr{Y}(\Omega)\times \mathscr{U}(\Omega)\times \mathscr{P}(\Omega)$ \emph{admissible for $\lbrace y^\eps \rbrace_\eps$} if $\{y^{\ep}\}_\eps$ is asymptotically represented by $(y,u,\mathcal{P})$.
\end{definition}

Although we use the notation $\to$  and call $(y,u,\mathcal{P})$ a limiting triple, it is clear
that Definition \ref{def:convergence1} cannot be understood as a convergence in the usual sense. In
particular, a specific feature of our limiting model is that in the limit $\eps \to 0$ a \emph{tripling of the variables} occurs. Another crucial aspect is given by the fact that along the sequence a characterization  in terms of quadruples is needed. Let us highlight the relation between the quadruples and the limiting triples:  the  deformation $y \in \mathscr{Y}(\Omega)$ is determined by the rotation $R^\eps$, the partitions $\mathcal{P}^\eps$, and  the phases $\mathcal{M}^\eps$, see \eqref{eq:comp-M}. For the displacement field $u$ we additionally need the translations $\mathcal{T}^\eps$, see \eqref{eq: rescaled disp}-\eqref{eq:comp-u}.  Finally, the limiting partition $\mathcal{P}$ is directly related to $\mathcal{P}^\eps$ by \eqref{eq:comp-sets}.

 We will now proceed  with a more specific characterization of the admissible limiting triples for a sequence $\lbrace y^\eps \rbrace_\eps$.

\subsection{Characterization of admissible limiting triples}\label{sec: limiting triples}

In this subsection, our aim is to give a complete characterization of all limiting  triples  $(y,u,\mathcal{P})$ which are admissible for a sequence $\{y^{\ep}\}_\eps$   considered in Theorem \ref{thm:compactness}.  This, in turn, specifies the domain of our effective energy discussed in the next subsection. Below we will see that the choice of the deformation $y$ and the partition $\mathcal{P}$ is unique. On the other hand, however, we see that $u$ is not determined uniquely: 

Consider admissible quadruples $\{(R^\eps, \mathcal{P}^\eps, \mathcal{M}^\eps, \mathcal{T}^\eps)\}_\ep$ for a sequence $\lbrace y^\eps\rbrace_\eps$ which is asymptotically represented by a triple $(y,u,\mathcal{P})$, where $\mathcal{T}^\eps = \lbrace t^\eps_j \rbrace_j$. Then, we find another sequence of admissible quadruples  $\{(\hat{R}^\eps, \hat{\mathcal{P}}^\eps, \hat{\mathcal{M}}^\eps, \hat{\mathcal{T}}^\eps)\}_\ep$ by setting $\hat{R}^\eps =  \exp(-\eps S) {R}^\eps$  for $S \in \M^{d\times d}_{\rm skew}$, $\hat{\mathcal{P}}^\eps = {\mathcal{P}}^\eps$, $\hat{\mathcal{M}}^\eps= {\mathcal{M}}^\eps$, and $\hat{\mathcal{T}}^\eps = \lbrace \hat{t}^\eps_j \rbrace_j$ with $\hat{t}^\eps_j=t^\eps_j -  \eps  t_j $ for some $t_j\in \R^d$  for all  $j$. (Here, $\exp$ denotes the matrix exponential.) In view of \eqref{eq:comp-M} and \eqref{eq: rescaled disp}--\eqref{eq:comp-u}, a short computation yields that this sequence of quadruples will give the limiting triple $(y,\hat{u},\mathcal{P})$  with 
\begin{align}\label{eq: uhat}
\hat{u}(x) = u(x) + \sum\nolimits_j t_j \chi_{P_j}(x) + S\, \nabla y(x) \,  x \quad \quad \text{for all $x \in \Omega$.}
\end{align}
To take this ambiguity of the limiting description into account, for a given deformation $y \in \mathscr{Y}(\Omega)$ and a given Caccioppoli partition $\mathcal{P} = \lbrace P_j \rbrace_j \in \mathscr{P}(\Omega)$, we introduce  the set 
\begin{align}\label{eq: infini rigid}
\mathscr{T}(y,\mathcal{P})=  \Big\{& T\colon \Omega \to \R^d\,\,\text{such that}\,\, T(x)=\sum\nolimits_j t_j\chi_{P_j}(x) + S\,  \nabla y(x) \, x, \ \  \  t_j \in \R^d, \  S \in \M^{d \times d}_{\rm skew} \Big\}
\end{align}
of corresponding piecewise translations combined with global infinitesimal rotations. We obtain the following characterization.

 \begin{proposition}[Characterization of admissible limiting triples] \label{prop:ex-coarsest-part}
 Let $\lbrace y^\eps \rbrace_\eps$ be a   sequence as in Theorem~\ref{thm:compactness}. Let $(y^1,u^1,\mathcal{P}^1)$ and  $(y^2,u^2,\mathcal{P}^2)$ be two admissible triples. Then the following assertions hold:

 \begin{itemize}
\item[(i)] $ y^1 =   y^2$ and $\mathcal{P}^1 = \mathcal{P}^2$ (up to possible reorderings  of the sets).
\item[(ii)] There exists $T\in \mathscr{T}(y^1,{\mathcal{P}}^1) = \mathscr{T}(y^2,{\mathcal{P}}^2)$ such that $u^1-u^2=T$.
\item[(iii)] For each $\tilde{T}\in \mathscr{T}(y^1,{\mathcal{P}}^1)$, the triple $(y^1,u^1 + \tilde{T},\mathcal{P}^1)$ is admissible. 
\end{itemize} 
 \end{proposition}
 
Property (i) states  that the limiting deformation is uniquely determined. It follows from \eqref{eq:comp-y}. The corresponding property for the partition is a consequence of the selection principle in \eqref{eq: toinfty}. Without such a condition other choices are possible, see Example \ref{ex}  for more details.     Property (ii) states that the admissible displacement fields for a sequence $\{y^{\ep}\}_{\ep}$   are determined uniquely up to piecewise translations and a global (infinitesimal) rotation. This non-uniqueness   has been illustrated in \eqref{eq: uhat}.

 The next result characterizes  the jump sets involved in admissible limiting triples.

 \begin{proposition}[Admissible limiting triples; jump set  and partition]\label{lemma: admissible-u-y-jump}
  Let $\lbrace y^\eps \rbrace_\eps$ be a   sequence as in Theorem \ref{thm:compactness}. Then for each admissible triple  $(y,u,\mathcal{P})$ in the sense of Definition \ref{def:convergence1} there holds 
  $$J_{\nabla y} \subset \bigcup\nolimits_{j} \partial P_j \cap \Omega. $$
 There are examples of  sequences $\lbrace y^\eps \rbrace_\eps$ such that the inclusion is strict.  
\end{proposition}
The fact that the inclusion  may be  strict can be seen in Case (2) of Example \ref{ex} (corresponding to $l = 1/2$). We also note by Proposition \ref{prop:ex-coarsest-part}(iii) that there is always an admissible displacement field $u$ with $ \bigcup\nolimits_{j} \partial P_j \cap \Omega \subset J_u$. This inclusion might be strict, see Case (1) in Example \ref{ex} with $l=1$. We proceed  with a result which specifies the jump heights of admissible limiting displacement fields.  For $u \in \mathscr{U}(\Omega)$,  the normal on $J_u $ is given by $\nu_u = e_d$. We denote by $u^+$ and $u^-$ the corresponding one-sided limits of $u$ and  we let  $[u]:= u^+-u^-$.

\begin{proposition}[Admissible limiting displacement fields; jump heights]\label{lemma: admissible-u}
Let $(y,u,\mathcal{P})$ be an admissible triple in the sense of Definition \ref{def:convergence1} and let $R \in SO(d)$ be such that $y \in \mathcal{Y}_R(\Omega)$. We have
\begin{align*}%\label{eq: properties-u}
{\rm (i)} &  \ \ [u](x) \ \text{constant for $\mathcal{H}^{d-1}$-a.e. } x \in (\R^{d-1} \times \lbrace t \rbrace) \cap \Omega   \ \  \text{for all } t \in \R \text{ with } J_u \cap (\R^{d-1} \times \lbrace t \rbrace) \neq \emptyset, \notag\\
{\rm (ii)} & \ \ [u](x) \in [0,+\infty) Re_d \text{ for $\mathcal{H}^{d-1}$-a.e. $x \in \big(J_u \setminus \bigcup\nolimits_j \partial P_j \big) \cap  \lbrace \nabla y = RA \rbrace$}, \notag\\
{\rm (iii)} & \ \ -[u](x) \in [0,+\infty) Re_d \text{ for $\mathcal{H}^{d-1}$-a.e. $x \in \big(J_u \setminus \bigcup\nolimits_j \partial P_j \big) \cap  \lbrace \nabla y = RB \rbrace$}.
\end{align*}
\end{proposition}

Roughly speaking, property (i) is a consequence of the geometry of the $A$ and $B$ phase regions induced by the rigidity estimate. We refer to \eqref{eq: partition property-new} and  to Figure \ref{fig:small-trans} for an illustration. We also refer to the discussion on the jump height in \eqref{eq: jump height}. In particular, (i) implies that the jump set  consists of subsets of hyperplanes orthogonal to $e_{d}$, which  extend up to the boundary of $\Omega$. Some intuition for point (ii) has been provided in \eqref{eq: jump height}, see also Case (1) in Example \ref{ex} with $l=1$. Point (iii) is similar by changing the roles of the phases $A$ and $B$. Note that (ii) and (iii) are well defined by Proposition \ref{lemma: admissible-u-y-jump}.

\begin{definition}
\label{def:convergence2}
In view of Theorem \ref{thm:compactness}, Proposition \ref{lemma: admissible-u-y-jump}, and Proposition \ref{lemma: admissible-u}, we introduce the set of admissible limiting triples 
$$\mathcal{A}:=\Big\{(y,u,\mathcal{P})\in \mathscr{Y}(\Omega)\times \mathscr{U}(\Omega)\times \mathscr{P}(\Omega)\colon\ \,  J_{\nabla y} \subset \bigcup\nolimits_{j=1}^\infty \partial P_j\cap\Omega,  \ \, u\text{ satisfies   (i)--(iii) in Proposition }\ref{lemma: admissible-u}\Big\}.$$
\end{definition}

\begin{remark}[Internal jumps]\label{rem:internal jumps}
{\normalfont

As discussed already heuristically in Subsection \ref{sec: heuristics/challenges}, the choice of the penalization factor \eqref{eq:alphad} simplifies the analysis  by excluding the formation of \emph{internal jumps}  for limiting displacement fields, see Proposition \ref{lemma: admissible-u}(i). This allows us to formulate our limiting model  for displacements in a piecewise Sobolev setting. Let us mention that without such a requirement the domain of the limiting model is expected to be the space of generalized functions of bounded variation $GSBD^2(\Omega)$ introduced in \cite{DM}, with an additional constraint on the jump sets of admissible functions. Note that this phenomenon is not just a technical mathematical issue, but is related to \emph{branching}, i.e., to the presence of microstructures near interfaces, see  e.g. \cite{Chan-Conti, conti-diermeier, Diermeier, kohn.muller, Zwicknagl}. Particularly, see  \cite{conti-diermeier} for a  simplified scalar model in  $SBV$ addressing the low volume-fraction of one phase, and dealing with the problem of internal jumps.  (We also refer to \cite{Diermeier} for some extensions to a vectorial model in the geometrically linear setting,  and to \cite{conti.diermeier.melching.zwicknagl} for a corresponding scaling law in the case of a martensitic  nucleus embedded in an austenitic matrix.\EEE)

} 
\end{remark} 
 
 \
%%%%%%%%%%%%%%%%%%%%%%%%%%%%%%%%%%%%%%%%%%
\subsection{The effective limiting model and $\Gamma$-convergence}\label{sec: effective model}

 This subsection is devoted to the  identification of the effective limiting model.   We start by introducing the limiting energy functional. 
We preliminarily recall that, in view of assumption H5., the stored energy density $W$ is $C^2$ in a neighborhood of the set $SO(d)\{A,B\}$. We also recall  the quadratic form  $\mathcal{Q}_{\rm lin}$  defined in \eqref{eq:def-Q}, Definition \ref{def:convergence2}, and the asymptotic optimal-profile energy in \eqref{eq: our-k1}. We define the functional
\begin{align}\label{eq: limiting energy}
 \mathcal{E}_0^{\mathcal{A}} (y,u,\mathcal{P}):=
\int_{\Omega}\mathcal{Q}_{\rm lin} (\nabla y,\nabla u) \,{\rm d}x +   K   \mathcal{H}^{d-1}(J_{\nabla y})+ 2  K \mathcal{H}^{d-1}\Big(\big(J_u\cup\big( \bigcup\nolimits_j \partial P_j\cap \Omega\big)\big)\setminus J_{\nabla y}\Big)
\end{align}
for every $(y,u,\mathcal{P})\in \mathcal{A}$. Note that the elastic term is well defined as $\nabla y(x) \in SO(d)\{A,B\}$ for a.e.\ $x \in \Omega$.

 We briefly compare this energy to the limiting models in Subsection \ref{sec: existing results} and explain the relation to  $\mathcal{E}_\eps$  introduced in \eqref{eq: specific energy choice}. First, the elastic energy is more general than the one in \eqref{eq: first-linearization} as it accounts for the two different phases indicated by $\nabla y$. Moreover, in contrast to \eqref{eq: first sharp-interface},  the   functional contains two surface  terms: the jumps of $\nabla y$ represent the energy associated to single phase transitions between $A$- and $B$-phases,  already appearing in \eqref{eq: first sharp-interface}.  The second surface term corresponds to two  `consecutive phase transitions', i.e.,  two transitions  with a small intermediate layer  whose width vanishes as $\eps \to 0$, which remain undetected by $y$. More generally speaking,  by relaxation in the limit $\eps \to 0$,  the first term (single transition) and the second term (double transition)  effectively correspond to  an odd and an even number of consecutive phase transitions, respectively, cf.\ Figure \ref{fig2}. Note that the second surface term enters the energy with double cost with respect to  single phase transitions.  This term itself has two contributions:    recalling the selection principle for the partition in \eqref{eq: toinfty}, small intermediate layers of width $\sim \eps$ are associated to $J_u$  in the limit $\eps \to 0$  and layers with asymptotically much larger width are encoded by the partition $\mathcal{P}$. Layers of width $\ll \eps$ do not  affect  the limiting energy. This is illustrated in Example \ref{ex}. 
 
 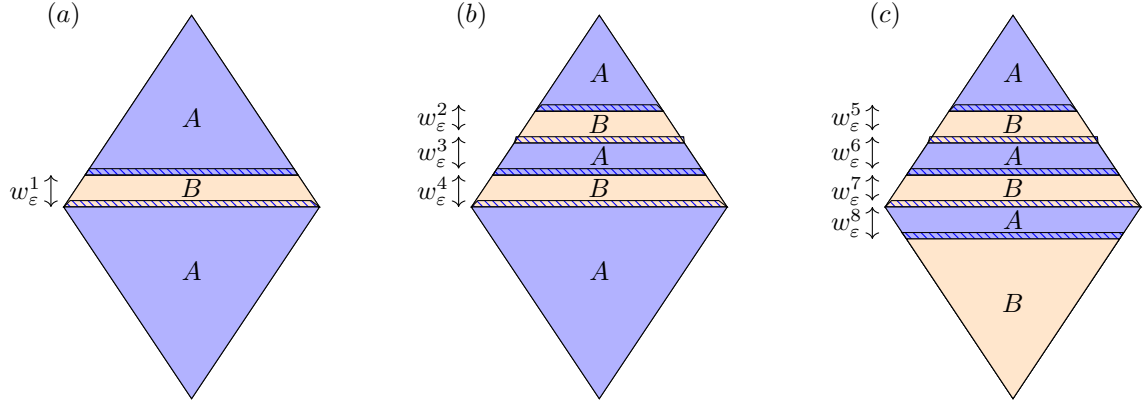
\begin{figure}[h]
\centering
\begin{tikzpicture}
% e_1 points right, e_2 points up, 
% four external points

\begin{scope}[shift={(-2.62,0)}, scale=.85]

\coordinate (A) at (-2,3);
 \coordinate (B) at (0,0);
 \coordinate (C) at (-4,0);
 \coordinate (D) at (-2,-3);
 % points on the right y=3/2x-3
 \coordinate (E) at (-1,1.5);
 \coordinate (F) at (-0.67,1);
  \coordinate (G) at (-0.67,-1);
  \coordinate (H) at (-1.34,-2);
  \coordinate (I) at (-1.67,-2.5);
  \coordinate (L) at (-1.8,-2.7);
  \coordinate (M) at (-1.9,-2.85);
  %points on the left
   \coordinate (N) at (-3,1.5);
 \coordinate (O) at (-3.33,1);
 \coordinate (U) at (-3.665, 0.5);
  \coordinate (V) at (-0.335, 0.5);
  \coordinate (P) at (-3.33,-1);
  \coordinate (Q) at (-2.66,-2);
  \coordinate (R) at (-2.33,-2.5);
  \coordinate (S) at (-2.2,-2.7);
  \coordinate (T) at (-2.1,-2.85);
  %lines and colors
  \draw (A)--(B)--(D)--(C)--cycle;
 % \draw (E)--(N);
 % \draw (F)--(O);
%  \draw (G)--(P);
  \draw (H)--(Q);
 % \draw[dashed] (I)--(R);
 % \draw[dashed] (L)--(S);
  %\draw[dashed] (M)--(T);

   \draw[fill=orange!20] (U)--(V)--(B)--(C)--cycle;
  \draw[pattern=north west lines, pattern color=blue] (-3.9,0.1)--(-0.13,0.1)--(B)--(C)--cycle;
 
  \draw[fill=blue!30] (B)--(D)--(C)--cycle;
   \draw[fill=blue!30] (A)--(V)--(U)--cycle;

     \draw[pattern=north west lines, pattern color=blue] (-3.6, 0.6)--(-0.4, 0.6)--(V)--(U)--cycle;
% \draw[fill=orange!20] (O)--(F)--(B)--(G)--(P)--(C)--cycle;
% \draw[fill=blue!30] (P)--(G)--(H)--(Q)--cycle;
 % \draw[fill=orange!20] (Q)--(H)--(D);
 %  \draw[pattern=north west lines, pattern color=blue] (-2.55,-2.2)--(-1.44,-2.2)--(H)--(Q)--cycle;
  % \draw[dashed,fill=blue!20] (R)--(I)--(L)--(S)--cycle;
% \draw[dashed,fill=orange!10] (S)--(L)--(M)--(T);

  \node at (-2,1.4) {$A$};
   \node at (-2,0.3) {$B$};
    \node at (-2,-1) {$A$};
   
     \draw[<->] (-4.2,0) -- (-4.2,0.5);
      \node at (-4.6,0.25) {$w_\ep^1$};
  %\draw[>=triangle 45, <->] (-2.68,-1.8) -- (-1.32,-1.8);
  % \node at (-1.9,-1.5) {$\eps^{r/2}$};
  % \draw[>=triangle 45, <->] (-2.9,-2) -- (-2.9,-3);
  %  \node at (-3.4,-2.4) {$\eps^{r/2}$};
 
 %% Following is for debugging purposes so you can see where the points are
%% These are last so that they show up on top
%\foreach \xy in {A, B, C, D, E, F, G, H, I, L, M, N, O, P, Q, R, S, T, U, V}{
%\node at (\xy) {\xy};}
 \node at (-4,3) {$(a)$};
 \end{scope}

\begin{scope}[shift={(2.8,0)}, scale=.85]

\coordinate (A) at (-2,3);
 \coordinate (B) at (0,0);
 \coordinate (C) at (-4,0);
 \coordinate (D) at (-2,-3);
 % points on the right y=3/2x-3
 \coordinate (E) at (-1,1.5);
 \coordinate (F) at (-0.67,1);
  \coordinate (G) at (-0.67,-1);
  \coordinate (H) at (-1.34,-2);
  \coordinate (I) at (-1.67,-2.5);
  \coordinate (L) at (-1.8,-2.7);
  \coordinate (M) at (-1.9,-2.85);
  %points on the left
   \coordinate (N) at (-3,1.5);
 \coordinate (O) at (-3.33,1);
 \coordinate (U) at (-3.665, 0.5);
  \coordinate (V) at (-0.335, 0.5);
  \coordinate (P) at (-3.33,-1);
  \coordinate (Q) at (-2.66,-2);
  \coordinate (R) at (-2.33,-2.5);
  \coordinate (S) at (-2.2,-2.7);
  \coordinate (T) at (-2.1,-2.85);
  %lines and colors
  \draw (A)--(B)--(D)--(C)--cycle;
 % \draw (E)--(N);
 % \draw (F)--(O);
%  \draw (G)--(P);
  \draw (H)--(Q);
 % \draw[dashed] (I)--(R);
 % \draw[dashed] (L)--(S);
  %\draw[dashed] (M)--(T);
  \draw[fill=orange!20] (N)--(E)--(F)--(O)--cycle;
   \draw[pattern=north west lines, pattern color=blue] (-3.3,1.1)--(-0.68,1.1)--(F)--(O)--cycle;
   \draw[fill=orange!20] (U)--(V)--(B)--(C)--cycle;
  \draw[pattern=north west lines, pattern color=blue] (-3.9,0.1)--(-0.13,0.1)--(B)--(C)--cycle;
 
  \draw[fill=blue!30] (B)--(D)--(C)--cycle;
   \draw[fill=blue!30] (A)--(E)--(N)--cycle;
    \draw[pattern=north west lines, pattern color=blue] (-2.9,1.6)--(-1.1,1.6)--(E)--(N)--cycle;
    \draw[fill=blue!30] (F)--(V)--(U)--(O)--cycle;
     \draw[pattern=north west lines, pattern color=blue] (-3.6, 0.6)--(-0.4, 0.6)--(V)--(U)--cycle;
% \draw[fill=orange!20] (O)--(F)--(B)--(G)--(P)--(C)--cycle;
% \draw[fill=blue!30] (P)--(G)--(H)--(Q)--cycle;
 % \draw[fill=orange!20] (Q)--(H)--(D);
 %  \draw[pattern=north west lines, pattern color=blue] (-2.55,-2.2)--(-1.44,-2.2)--(H)--(Q)--cycle;
  % \draw[dashed,fill=blue!20] (R)--(I)--(L)--(S)--cycle;
% \draw[dashed,fill=orange!10] (S)--(L)--(M)--(T);
 \node at (-2,2.1) {$A$};
 %\node at (-2,-1.5) {$A$};
 % \node at (-2,1.25) {$A$};
 \node at (-2,1.3) {$B$};
  \node at (-2,0.8) {$A$};
   \node at (-2,0.3) {$B$};
    \node at (-2,-1) {$A$};
    \draw[<->] (-4.2,1.6) -- (-4.2,1.2);
     \node at (-4.6,1.4) {$w_\ep^2$};
    \draw[<->] (-4.2,0.6) -- (-4.2,1.1);
    \node at (-4.6,0.85) {$w_\ep^3$};
     \draw[<->] (-4.2,0) -- (-4.2,0.5);
      \node at (-4.6,0.25) {$w_\ep^4$};
  %\draw[>=triangle 45, <->] (-2.68,-1.8) -- (-1.32,-1.8);
  % \node at (-1.9,-1.5) {$\eps^{r/2}$};
  % \draw[>=triangle 45, <->] (-2.9,-2) -- (-2.9,-3);
  %  \node at (-3.4,-2.4) {$\eps^{r/2}$};
 
 %% Following is for debugging purposes so you can see where the points are
%% These are last so that they show up on top
%\foreach \xy in {A, B, C, D, E, F, G, H, I, L, M, N, O, P, Q, R, S, T, U, V}{
%\node at (\xy) {\xy};}
 \node at (-4,3) {$(b)$};
 \end{scope}

\begin{scope}[shift={(3.2,0)}, scale=.85]

% e_1 points right, e_2 points up, 
% four external points
\coordinate (A2) at (4,3);
 \coordinate (B2) at (6,0);
 \coordinate (C2) at (2,0);
 \coordinate (D2) at (4,-3);
 % points on the right y=3/2x-3
 \coordinate (E2) at (5,1.5);
 \coordinate (F2) at (5.33,1);
  \coordinate (G2) at (5.33,-1);
  \coordinate (H2) at (4.66,-2);
  \coordinate (I2) at (4.33,-2.5);
  \coordinate (L2) at (4.2,-2.7);
  \coordinate (M2) at (4.1,-2.85);
  %points on the left
   \coordinate (N2) at (3,1.5);
 \coordinate (O2) at (2.67,1);
 \coordinate (U2) at (2.335, 0.5);
  \coordinate (V2) at (5.665, 0.5);
  \coordinate (P2) at (8.67,-1);
  \coordinate (Q2) at (3.34,-2);
  \coordinate (R2) at (3.67,-2.5);
  \coordinate (S2) at (3.8,-2.7);
  \coordinate (T2) at (3.9,-2.85);
   \coordinate (X2) at (2.335, -0.5);
  \coordinate (Y2) at (5.665, -0.5);
  %lines and colors
  \draw (A2)--(B2)--(D2)--(C2)--cycle;
 % \draw (E)--(N);
 % \draw (F)--(O);
%  \draw (G)--(P);
  \draw (H2)--(Q2);
 % \draw[dashed] (I)--(R);
 % \draw[dashed] (L)--(S);
  %\draw[dashed] (M)--(T);
  \draw[fill=orange!20] (N2)--(E2)--(F2)--(O2)--cycle;
   \draw[pattern=north west lines, pattern color=blue] (2.7,1.1)--(5.32,1.1)--(F2)--(O2)--cycle;
   \draw[fill=orange!20] (U2)--(V2)--(B2)--(C2)--cycle;
  \draw[pattern=north west lines, pattern color=blue] (2.1,0.1)--(5.87,0.1)--(B2)--(C2)--cycle;
 
  \draw[fill=blue!30] (B2)--(D2)--(C2)--cycle;
   \draw[fill=blue!30] (A2)--(E2)--(N2)--cycle;
    \draw[pattern=north west lines, pattern color=blue] (3.1,1.6)--(4.9,1.6)--(E2)--(N2)--cycle;
    \draw[fill=blue!30] (F2)--(V2)--(U2)--(O2)--cycle;
     \draw[pattern=north west lines, pattern color=blue] (2.4, 0.6)--(5.6, 0.6)--(V2)--(U2)--cycle;
     
       \draw[pattern=north west lines, pattern color=blue] (2.26, -0.4)--(5.73,-0.4)--(Y2)--(X2)--cycle;
        \draw[fill=orange!20] (X2)--(Y2)--(D2)--cycle;
% \draw[fill=orange!20] (O)--(F)--(B)--(G)--(P)--(C)--cycle;
% \draw[fill=blue!30] (P)--(G)--(H)--(Q)--cycle;
 % \draw[fill=orange!20] (Q)--(H)--(D);
 %  \draw[pattern=north west lines, pattern color=blue] (-2.55,-2.2)--(-1.44,-2.2)--(H)--(Q)--cycle;
  % \draw[dashed,fill=blue!20] (R)--(I)--(L)--(S)--cycle;
% \draw[dashed,fill=orange!10] (S)--(L)--(M)--(T);
 \node at (4,2.1) {$A$};
 %\node at (-2,-1.5) {$A$};
 % \node at (-2,1.25) {$A$};
 \node at (4,1.3) {$B$};
  \node at (4,0.8) {$A$};
   \node at (4,0.3) {$B$};
    \node at (4,-0.2) {$A$};
    \node at (4,-1.5) {$B$};
  %\draw[>=triangle 45, <->] (-2.68,-1.8) -- (-1.32,-1.8);
  % \node at (-1.9,-1.5) {$\eps^{r/2}$};
  % \draw[>=triangle 45, <->] (-2.9,-2) -- (-2.9,-3);
  %  \node at (-3.4,-2.4) {$\eps^{r/2}$};
 \draw[<->] (1.8,1.6) -- (1.8,1.2);
     \node at (1.4,1.4) {$w_\ep^5$};
    \draw[<->] (1.8,0.6) -- (1.8,1.1);
    \node at (1.4,0.85) {$w_\ep^6$};
     \draw[<->] (1.8,0.1) -- (1.8,0.5);
      \node at (1.4,0.25) {$w_\ep^7$};
       \draw[<->] (1.8,-0.5) -- (1.8,0);
 \node at (1.4,-0.25) {$w_\ep^8$};
\node at (2,3) {$(c)$};

\end{scope}

%% Following is for debugging purposes so you can see where the points are
%% These are last so that they show up on top
%\foreach \xy in {A2, B2, C2, D2, E2, F2, G2, H2, I2, L2, M2, N2, O2, P2, Q2, R2, S2, T2, U2, V2}{
% \node at (\xy) {\xy};
%}
\end{tikzpicture}
\caption{Illustration of  situations corresponding to even and odd numbers of consecutive phase transitions. We assume that $w_\ep^i\to 0$ as $\ep\to 0$ and that $\liminf_{\ep\to 0}w^i_\ep/\ep>0$ for $i=1,\dots,8$. The shaded regions describe the areas in which the phase transitions occur. (a) We depict here the case of two phase transitions: the intermediate phase has infinitesimal width $w^1_\ep$ and thus disappears in the limit. Its presence at level $\ep$, though, still  affects  $\mathcal{E}_0^\mathcal{A}$. Indeed, in the second surface term, the length of the interface between the two limiting $A$-regions will enter the energy with density $2K$.  (b) The case of three intermediate phases is depicted. Although being different from (a) on the level $\eps$, this situation leads to the same effective energy. In this sense, two  intermediate phases  `compensate each other' in the limit. Note that the jump height of the limiting function is determined by $w^2_\ep$ and $w^4_\ep$ only. (c) We illustrate here the situation of five phase transitions: the energy contribution is accounted for in $\mathcal{E}_0^\mathcal{A}$ by the first surface term, i.e., the length of the interface between the limiting $A$- and $B$-regions, reflected by $J_{\nabla y}$, will enter the energy only with density $K$.}
\label{fig2}
\end{figure}

 \begin{remark}[Second-gradient terms]
 \label{rk:qualitative}
The effective model described in \eqref{eq: limiting energy} does not contain second-gradient terms neither in $y$ nor  in $u$. Indeed, the choice $\bar{\eta}_{\ep,d} \ll \ep^{-1}$ guarantees that the effects of higher-order contributions, in particular of their anisotropic part, enter the limiting energy only in terms of  the  value of the constant $K$, but no dependence on second-order derivatives persists in the model after the limiting passage.
 \end{remark}

The main contribution of this paper consists in showing  that the sequence $\{\mathcal{E}_\eps\}_\ep$  is asymptotically described by  $\mathcal{E}_0^{\mathcal{A}}$, in the sense of $\Gamma$-convergence in the topology introduced in Definition \ref{def:convergence1}. As a preliminary observation, we note that the limiting energy is invariant under changes of  the asymptotic representative.

\begin{remark}[Energy invariance for different asymptotic representatives]\label{rem: AR}
Suppose that a sequence $\lbrace y^\eps \rbrace_\eps$ is asymptotically represented by two triples  $(y^1,u^1,\mathcal{P}^1), (y^2,u^2,\mathcal{P}^2) \in \mathcal{A}$. Then,  $\mathcal{E}_0^{\mathcal{A}} (y^1,u^1,\mathcal{P}^1) = \mathcal{E}_0^{\mathcal{A}} (y^2,u^2,\mathcal{P}^2)$. This follows from \eqref{eq: only symmetric}, \eqref{eq: infini rigid},  Proposition \ref{prop:ex-coarsest-part}, and \eqref{eq: limiting energy}. 
 \end{remark}

 Our first result  shows that  $\mathcal{E}_0^{\mathcal{A}}$  provides a lower bound for the asymptotic behavior of the energy functionals $\lbrace \mathcal{E}_{\ep}\rbrace_\eps$. 

\begin{theorem}[$\Gamma$-liminf inequality]
\label{thm:liminf}
 Let $\Omega\subset \R^d$ be a bounded Lipschitz domain satisfying {\rm H8}.  Let $W$ satisfy assumptions {\rm H1}.--{\rm H5}., let $(y,u,\mathcal{P})\in \mathcal{A}$,  and let $\{y^{\ep}\}_\eps \subset H^2(\Omega;\R^d)$ be such that $y^{\ep}\to (y,u,\mathcal{P})$. Then
\begin{equation*}
\liminf_{\ep\to 0}\mathcal{E}_{\ep}(y^{\ep})\geq \mathcal{E}_0^{\mathcal{A}}(y,u,\mathcal{P}).
\end{equation*}
\end{theorem}

 Our second result is the proof  that the lower bound identified in Theorem \ref{thm:liminf} is optimal.  For the construction of recovery sequences we need slightly stronger assumptions: we require that the set is strictly star-shaped (see \eqref{eq: starshape}), we assume H6.\ and H7., and we need a specific condition for  the asymptotic optimal-profile energy.  In order to state our result, we need some additional notation.  Define the set of sequences 
\begin{align}\label{ew: W sequence}
\mathcal{W}_d := \big\{ \lbrace w_\eps \rbrace_\eps: \  w_\eps \in (0,\infty),  \  w_\eps \to 0, \ \liminf_{\eps \to 0} \,   (w_\eps/\eps) >0  \big\},
\end{align}
and define the functions 
\begin{align}\label{eq: step functions}
y_{\rm dp}^A := e_d \chi_{\lbrace x_d>0\rbrace}, \quad \quad \quad y_{\rm dp}^B := -e_d \chi_{\lbrace x_d>0\rbrace}.
\end{align}
For  $M \in \lbrace A,B\rbrace$, we introduce   the \emph{double-profile energy}
\begin{align}\label{eq: our-k2}
 K^M_{\rm dp} {:=}\sup_{h>0 }\sup_{\lbrace w_\eps \rbrace_\eps \in \mathcal{W}_d} \inf\Big\{&\limsup_{\ep\to 0}  \mathcal{E}_{\ep}\big(y^{\ep}, Q'\times (-h,h)\big)\colon \,             \frac{y^\eps -  M x}{w_\eps} \to   y_{\rm dp}^M   \text{ in measure in \EEE $Q'\times (-h,h)$}\Big\},
\end{align} 
where here and in the following $Q' :=(-\frac{1}{2},\frac{1}{2})^{d-1} \subset \R^{d-1}$. 
We defer a discussion about the definition of $K^M_{\rm dp}$, and proceed with the $\Gamma$-limsup inequality.

\begin{theorem}[$\Gamma$-limsup inequality]
\label{thm:limsup-new}
Let $\Omega\subset \R^d$ be a bounded, strictly star-shaped, Lipschitz domain in $\R^d$ satisfying {\rm H8}. Let $W$ satisfy assumptions {\rm H1}.--{\rm H7}., and suppose that  $K^A_{\rm dp} = K^B_{\rm dp} = 2K$. Let $(y,u,\mathcal{P})\in \mathcal{A}$. Then  there exists $\{y^{\ep}\}_\eps\subset H^2(\Omega;\R^d)$ such that 
$y^{\ep}\to (y,u,\mathcal{P})$ in $\mathcal{A}$, and
\begin{equation*}
\limsup_{\ep\to 0} \mathcal{E}_{\ep}(y^{\ep})\leq \mathcal{E}_0^{\mathcal{A}}(y,u,\mathcal{P}).
\end{equation*}
\end{theorem}
 The notion of strictly star-shaped sets will allow us to reduce the constructions to the case of finitely many phase transitions, similarly to the investigation in \cite{conti.schweizer}. The additional assumptions H6.\ and H7.\ are instrumental to control the nonlinear elastic energies of the recovery sequence, whenever the gradient is away from the two wells. We now address definition \eqref{eq: our-k2} and explain the condition $K^A_{\rm dp} = K^B_{\rm dp} = 2K$.

First, in order to understand the role of the sequences $\mathcal{W}_d$ defined in \eqref{ew: W sequence}, recall the setting in Figure \ref{fig2}(a). The case in which, locally at level $\ep$, two portions of the domain in the same phase are separated by an intermediate region in the opposite phase, is reflected by an energy contribution in the limiting functional $\mathcal{E}_0^{\mathcal{A}}$ whenever the  width  of the `intermediate layer' behaves asymptotically as one of the sequences in $\mathcal{W}_d$. We recall that, if $\liminf_{\eps \to 0} \,   (w_\eps/\eps) \in (0,+\infty)$, this is encompassed by the jump set of the limiting displacement $u$, whereas the opposite scenario is captured by the limiting partition $\mathcal{P}$.

Intuitively, the value  $K_{\rm dp}^A$ in  \eqref{eq: our-k2} provides an upper bound for the energy of an optimal profile which contains two phase transitions, first from $A$ to $B$ and then from $B$ to $A$, with an  intermediate layer in the $B$-phase of width $\{w_\ep\}_\ep$,  see Figure \ref{fig2}(a). The interpretation of   $K_{\rm dp}^B$  is the same after interchanging the roles of the phases. The \emph{compatibility condition}  $K^A_{\rm dp} = K^B_{\rm dp} = 2K$  is needed in the construction of recovery sequences. On the one hand, it seems a natural condition as $K$ and $K^A_{\rm dp},  K_{\rm dp}^B$  correspond to the case of one and two phase transitions, respectively. On the other hand, for general densities $W$ we are able to prove only one inequality and the other inequality only under extra assumptions on $W$. More precisely, we have the following.

\begin{proposition}[Relation of $K$,  $K_{\rm dp}^A$, and $K^B_{\rm dp}$: inequality]
\label{prop:cell-form-chaper2}
The values $K$,  $K_{\rm dp}^A$, and $K_{\rm dp}^B$  introduced in   \eqref{eq: our-k1} and \eqref{eq: our-k2} satisfy $\min\lbrace  K^A_{\rm dp},K_{\rm dp}^B \EEE \rbrace \ge 2K$. 
\end{proposition}

We now discuss an additional assumption on $W$ which implies equality.  Assume that the energy density additionally satisfies 
\begin{align}\label{eq: isotropy}
W(F) \ge W\big({\rm Id} + (|F e_d| - 1)e_{dd}\big) \quad \quad \quad \text{for all } F \in \mathbb{M}^{d \times d}.
\end{align}
As we will show in Lemma \ref{lemma: 1d}, this condition ensures that optimal profiles are one-dimensional.  It can be understood as a generalization of condition (${\rm H}_3$) in \cite{conti.fonseca.leoni} where one-dimensionality of profiles has been discussed for a two-well problem without frame indifference. Note that this condition is compatible with frame indifference. A  model case is a situation where the energy only depends on the distance of the two wells, i.e.,
\begin{align}\label{eq: modelli}
W(F) =  \phi\big(  {\rm dist}(F, SO(d)A),   {\rm dist}(F, SO(d)B) \big) \quad \quad \quad \text{for all } F \in \mathbb{M}^{d \times d},
\end{align} 
where $\phi\colon ([0,\infty))^2 \to [0,\infty)$ is a smooth function with  $c_1 (\min \lbrace t_1,t_2 \rbrace)^2 \le \phi(t_1,t_2)  \le c_2 (\min \lbrace t_1,t_2 \rbrace)^2$ for all $t_1,t_2 \in [0,\infty)$ which is increasing in both entries. We refer to  \eqref{eq: model case for details} below for details.

Given condition \eqref{eq: isotropy}, we are able to show the following.  

\begin{proposition}[Relation of $K$,  $K_{\rm dp}^A$, and $K^B_{\rm dp}$: equality]\label{prop: compat}
Suppose that \eqref{eq: isotropy} holds.  The values $K$,  $K_{\rm dp}^A$, and $K_{\rm dp}^B$  introduced in   \eqref{eq: our-k1} and \eqref{eq: our-k2} satisfy $K^A_{\rm dp} = K_{\rm dp}^B  = 2K$. 
\end{proposition}

We do not have an explicit example, but we conjecture that for certain energy densities one might indeed have $\min\lbrace K^A_{\rm dp},  K_{\rm dp}^B  \rbrace > 2K$.  Moreover, in contrast to  \eqref{eq: first sharp-interface} and  \eqref{eq: our-k1}, we cannot apply a symmetry argument to show that   $K_{\rm dp}^B$  equals   $K_{\rm dp}^A$.   In general,   $K_{\rm dp}^A$   and  $K_{\rm dp}^B$  might be different.
  
 Intuitively, $\min\lbrace  K^A_{\rm dp}, K_{\rm dp}^B  \rbrace > 2K$ means that two optimal profiles in \eqref{eq: our-k1} cannot  be combined  to a competitor in \eqref{eq: our-k2} without essentially increasing the energy. In any case, if e.g. $K_{\rm dp}^A > 2K$, the energy would probably depend  on the width of the intermediate $B$-layer and the limiting energy \eqref{eq: limiting energy} would necessarily also depend on the jump height of $u$. We do not pursue this more complicated case here, but only provide a result under the aforementioned compatibility condition. In this case, the cost of a double phase transition always equals $2K$, independently of the width of the intermediate layer.

This concludes the presentation of our results. The remainder of the paper is devoted to the proofs. The proof of Theorem \ref{thm:compactness} is the subject of Section \ref{sec:compactness}. In particular, the limiting deformations, rotations, and partitions are identified in Proposition \ref{lemma: intermediate step1}, whereas the limiting displacement fields are exhibited in Proposition \ref{lemma: intermediate step2} and Proposition \ref{lemma: intermediate step3}. The remaining part of the proof of Theorem \ref{thm:compactness} consists in showing that partitions and translations at the $\ep$-level can be chosen so that the selection principle in \eqref{eq: toinfty} holds true. The characterization of limiting triples described in Subsection \ref{sec: limiting triples} is provided in Section \ref{sec:limiting-triple}. Theorems \ref{thm:liminf} and \ref{thm:limsup-new} are proven in Subsections \ref{subs:liminf} and \ref{subs:limsup}.

The main step of the proof of the lower bound in Theorem \ref{thm:liminf} consists in showing that in the `bulk part' of the domain and around the different limiting interfaces the asymptotic behavior of the energies can be bounded from below by the elastic energy and by the two surface terms, respectively. Key ingredients are the notions of optimal-profile and double-profile energy functions (see \eqref{eq: k-intro} and \eqref{eq: k2-intro}), as well as Propositions \ref{prop:cell-form}--\ref{eq: KundKdo-new}, providing a characterization of the local behavior of the energy around the different limiting interfaces. The former was proven in \cite[Propostion 4.6]{davoli.friedrich}.  The proof of the latter is carried out in  Subsection \ref{sec: cell-formula}. 

The proof of Theorem \ref{thm:limsup-new} relies on two main intermediate results, which are proven in Subsection \ref{subs:local}: (1) in Proposition \ref{lemma: local1} we generalize \cite[Proposition 4.7]{davoli.friedrich} to construct local recovery sequences around single phase transitions; (2) in Proposition \ref{lemma: local2} we prove the corresponding result for double phase transitions. Eventually, in Subsection \ref{subs:1d} we show that under \eqref{eq: isotropy} optimal profiles for single phase transitions are one-dimensional (see Lemma \ref{lemma: 1d}), and that $K^{A}_{\rm dp}=K^{B}_{\rm dp}=2K$ (see Proposition \ref{prop: compat}).
\EEE

 \section{Compactness  analysis}
 \label{sec:compactness}

 This section is devoted to the proof of our compactness result in Theorem \ref{thm:compactness}.   We proceed in several steps: we first identify sequences of rotations, phase indicators, and partitions, as well as a limiting deformation and partition such that  \eqref{eq: partition property}-\eqref{eq:comp-M} hold, see Proposition \ref{lemma: intermediate step1}. Then, Proposition \ref{lemma: intermediate step2} and Proposition  \ref{lemma: intermediate step3} are devoted to the construction of (sequences of) translations and the definition of displacement fields, see \eqref{eq: rescaled disp}--\eqref{eq:comp-grad-u}, first on subsets of $\Omega$ and eventually on $\Omega$ itself. Finally, a further delicate construction is needed to show that by a suitable choice of the partitions and the translations also  the  selection principle \eqref{eq: toinfty} can be guaranteed.

 In what follows, we will use the notion of sets of finite perimeter and Caccioppoli partitions. We refer to Appendix \ref{sec:appendix} for basic properties. Before we start, we recall the two-well rigidity estimate in Theorem~\ref{thm:rigiditythm} and point out that the result hinges on the following characterization of the two phase regions (see \cite[Proposition 3.7 and Remark 3.8]{davoli.friedrich}). 
We refer to Figure \ref{fig:small-trans} for a two-dimensional visualization.

\begin{proposition}[Decomposition into phases]\label{lprop: phases}
 Let $\Phi$ be the phase indicator identified in Theorem \ref{thm:rigiditythm}, and define $T :=\{ \Phi =A\}$. Then
\begin{align}\label{eq: propertiesT}
{\rm (i)}& \ \ \mathcal{H}^{d-1}(\partial^* T \cap \Omega) \le c  \mathcal{E}_\ep(y), \notag\\
{\rm (ii)} & \ \  \int_{\partial^* T \cap \Omega} |\langle \nu_T , e_i \rangle| \, {\rm d}\mathcal{H}^{d-1} \le c \ep^{2-\alpha(d)}\, \mathcal{E}_\ep (y) \ \  \text{for} \ \ i=1,\ldots, d-1, \notag \\  
{\rm (iii)} & \ \ \int_{-\infty}^{+\infty}\mathcal{H}^{d-2}\Big( \big(\R^{d-1} \times \lbrace t \rbrace\big) \cap \partial^* T \cap \Omega\Big)  \, {\rm d}t \le c \ep^{2-\alpha(d)}\,  \mathcal{E}_\ep (y),  
 \end{align}
where $\nu_T$ denotes the outer normal to $T$, $\partial^* T$ its essential boundary, $\alpha(d)$ is the quantity introduced in \eqref{eq:alphad}, and $ \mathcal{E}_\ep $ is the energy functional defined in \eqref{eq: nonlinear energy}--\eqref{eq: specific energy choice}. \EEE  
\end{proposition}

We point out that the statement in \cite[Proposition 3.7]{davoli.friedrich} is more general but reduces to the proposition above for the choice $\eta=\bar{\eta}_{\ep,d}$ (see \eqref{eq:alphad}).

In the proof of the compactness result, the set $T$ will be the starting point for constructing the partitions. Properties \eqref{eq: propertiesT}(i),(ii) are crucial to show \eqref{eq: partition property} and to pass to a limiting partition in $\mathscr{P}(\Omega)$ by compactness. Item \eqref{eq: propertiesT}(iii) is instrumental to prove \eqref{eq: partition property-new}.

\begin{figure}[h!]
\centering
\begin{tikzpicture}
% e_1 points right, e_2 points up, 
% four external points
\coordinate (A) at (-3,3);
 \coordinate (B) at (3,3);
 \coordinate (C) at (3,-3);
 \coordinate (D) at (-3,-3);
 % points on the right 
 \coordinate (E) at (3,2.1);
 \coordinate (F) at (3,1.4);
 \coordinate (G) at (3,0.6);
 \coordinate (H) at (3,-0.1);
 \coordinate (I) at (3,-0.7);
 \coordinate (L) at (3,-1.4);
 \coordinate (M) at (3,-1.9);
 \coordinate (N) at (3,-2.6);
  %points on the left
  \coordinate (O) at (-3,2.1);
 \coordinate (P) at (-3,1.4);
 \coordinate (Q) at (-3,0.6);
 \coordinate (R) at (-3,-0.1);
 \coordinate (S) at (-3,-0.7);
 \coordinate (T) at (-3,-1.4);
 \coordinate (U) at (-3,-1.9);
 \coordinate (V) at (-3,-2.6);
  %lines and colors

 \draw (A)--(B)--(C)--(D)--cycle;
 \draw[fill=cyan, opacity=0.2] (O)--(E)--(F)--(P)--cycle;
 \draw [black] (-3,2.1)  to [out=-45,in=175] (3,2.1);
  \draw[fill=cyan, opacity=0.2] (Q)--(G)--(H)--(R)--cycle;
   \node[shape=ellipse, minimum height=0.1cm, minimum width=1.4cm, draw=black, fill=red!10, opacity=0.8] at (1.7, 0.25){$A$};
   \draw[fill=cyan, opacity=0.2] (S)--(I)--(L)--(T)--cycle;
    \draw [black] (-3,-0.7)  to [out=-45,in=175] (-1,-0.8) to [out=-10, in=-175] (1.8,-1.2)  to [out=20, in=-185] (3,-1.4);
    \draw[fill=cyan, opacity=0.2] (U)--(M)--(N)--(V)--cycle;
     \node[shape=ellipse, minimum height=0.1cm, minimum width=1.4cm, draw=black, fill=red!10, opacity=0.8] at (-0.6, -2.25){$B$};
  \draw[<->] (3.3,2.1)--(3.3,1.4);
 
\node at (-1.8,0.6) {$B$};
\node at  (-0.6,2.1) {$A$};
\node at  (1.8,-1.9) {$A$};
\node at (3.9,1.75) {$ \ep^{7/4}$};
%% Following is for debugging purposes so you can see where the points are
%% These are last so that they show up on top
%\foreach \xy in {A, B, C, D, E, F, G, H, I, L, M, N, O, P, Q, R, S, T,U,V}{
%  \node at (\xy) {\xy};
%}
\end{tikzpicture}
\caption{A visualization of phase regions in dimension $d=2$. The (anisotropic) second-order penalization guarantees that phase transitions occur inside cylindrical layers of height $\ep^{7/4}$.  (Note that $\alpha(d) = 1/4$ for $d=2$.) Additionally, $\ep^{7/4}$ is an upper bound on the height of minority islands in the $e_2$-direction. In other words, connected components of the phase regions have either small volume or coincide  (up to a small set)  with a layer of $\Omega$. In higher dimensions, a similar interpretation is possible, up to higher order terms.}
\label{fig:small-trans}
\end{figure}
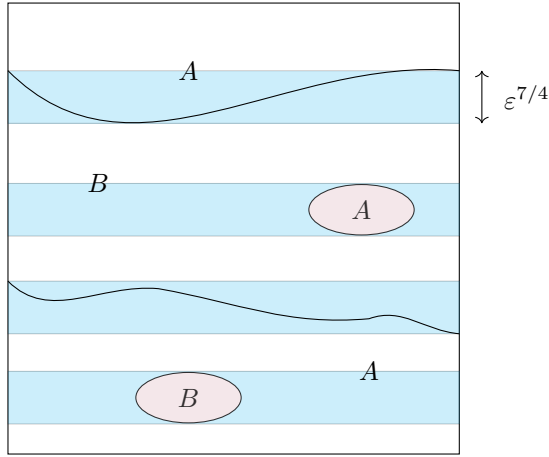

 We now start by identifying the limiting deformation and limiting partition. Recall the definition of $\mathcal{Y}_R(\Omega)$ and $\mathscr{P}(\Omega)$  in \eqref{eq: limiting deformations} and \eqref{eq:def-P}, respectively.

 \begin{proposition}[Deformations  and partitions]\label{lemma: intermediate step1}
 Let $\Omega$ be a bounded Lipschitz domain satisfying H8. Suppose that $W$ fulfills H1.--H4.  Let $\{y^{\ep}\}_{\ep}\subset H^2(\Omega;\R^d)$ be a sequence of deformations satisfying  \eqref{eq:unif-en-estimate}.  Then, we find a sequence of triples $(R^\eps, \mathcal{P}^\eps, \mathcal{M}^\eps)$, a limiting rotation $R \in SO(d)$, a limiting deformation $y \in \mathcal{Y}_R(\Omega)$, and a limiting partition $\mathcal{P} = \lbrace P_j\rbrace_j\in \mathscr{P}(\Omega) $ such that \eqref{eq: partition property}-\eqref{eq:comp-M} hold  after extracting a  subsequence. The components of $\mathcal{P}$ are  connected. 
 \end{proposition} 
 
We point out that in Theorem \ref{thm:compactness} the components are not connected in general. At this intermediate stage, however, constructing the partition with this additional property is instrumental for the definition of displacement fields in Propositions \ref{lemma: intermediate step2} and  \ref{lemma: intermediate step3} below as it allows to apply Poincar\'e inequalities on each component.

\begin{proof}[Proof of Proposition \ref{lemma: intermediate step1}]
Let $\{y^{\ep}\}_{\ep}\subset H^2(\Omega;\R^d)$ be a sequence of deformations satisfying  \eqref{eq:unif-en-estimate}.  We denote the orthogonal projection of $\Omega$ onto the $e_d$-axis  by the interval $(a,b)$.  

\emph{Step 1: Preliminary estimates.} First, we apply Theorem \ref{thm:rigiditythm} to obtain sequences of rotations $\lbrace R^\ep \rbrace_\ep \subset SO(d)$ and phase indicators $\lbrace  \Phi^\ep  \rbrace_\ep \subset BV(\Omega;\lbrace A,B\rbrace)$  such that 
for all $\Omega' \subset \subset \Omega$   
\begin{align}\label{eq: rigidity in compactness}
\Vert\nabla  y^\ep -R^{\ep} \Phi^{\ep} \Vert_{ L^2(\Omega') \EEE } = \Vert \nabla y^\ep - (R^\ep A\chi_{T^\eps} + R^\ep B \chi_{\Omega \setminus T^\eps} )\Vert_{L^2( \Omega' \EEE )}\leq C_{\Omega'}\ep\quad\text{and}\quad |D\Phi^\ep|(\Omega)\leq C,
\end{align}
where  $T^\ep = \lbrace  \Phi^\ep  = A \rbrace$ denotes the \emph{$A$-phase regions}, see Proposition \ref{lprop: phases}, $C_{\Omega'}$ depends on $\Omega'$,  and $C$ is related to $C_0$ in \eqref{eq:unif-en-estimate}.

In the following, we will need to apply the relative isoperimetric inequality on sections of the form $\Omega \cap \lbrace x_d = t \rbrace$, $t \in (a,b)$. In general, the involved constant may depend on $t$. As a remedy, we pass to suitable subsets of $\Omega$ with  properties independent of $t$: for $\eps>0$, we can choose $\Omega^\eps \subset \subset \Omega$ with Lipschitz boundary, satisfying H8., and 
\begin{align}\label{eq: Omegaeps}
\sup\nolimits_{\eps>0} \mathcal{H}^{d-1}(\partial \Omega^\eps) < + \infty, \ \ \ \ \ \ \ \ \ \lim_{\eps \to 0 } d_H(\Omega,  \Omega^\eps) = 0
\end{align}
($d_H$ denotes the Hausdorff distance) such that for each $t \in (a,b)$ and each set of finite perimeter $E \subset \Omega^\eps \cap \lbrace x_d = t \rbrace$  there holds
\begin{align}\label{eq: isoperimetric}
\min \Big\{ \mathcal{H}^{d-1}(E), \ \mathcal{H}^{d-1}\Big( (\Omega^\eps \cap \lbrace x_d = t \rbrace) \setminus E  \Big) \Big\} \le \eps^{-\alpha(d)} \big(\mathcal{H}^{d-2}( \partial^*E \cap \Omega^\eps) \big)^\frac{d-1}{d-2},
\end{align}    
where $\alpha(d)$ is defined in \eqref{eq:alphad}. (For $d=2$, the left hand side has to be interpreted as zero if $\mathcal{H}^{0}( \partial^*E \cap \Omega^\eps) = 0$.) Indeed, these sets can be constructed as follows.

For fixed $\rho>0$, let  $\Omega^\rho \subset \subset \Omega$ be a Lipschitz domain satisfying H8.\ which is a finite union of cylindrical sets of the form $\omega \times  (h^-,h^+)$ for $\omega \subset \R^{d-1}$ Lipschitz, i.e., there are only a finite number of different shapes for $\Omega^\rho \cap \lbrace x_d = t \rbrace$, denoted by $\omega_i \times \lbrace t \rbrace$ for Lipschitz domains $\omega_i$, $i=1,\ldots,N^\rho$.  (We do not include $\rho$ in the notation for simplicity.) Given $t \in (a,b)$, choose $\omega_i$ such that  $\omega_i \times \lbrace t \rbrace = \Omega^\rho \cap \lbrace x_d = t \rbrace$ and consider $E \subset \Omega^\rho \cap \lbrace x_d = t \rbrace$. Then we can apply the relative isoperimetric inequality on $\omega_i$ to obtain \eqref{eq: isoperimetric} for a constant $C^\rho_i$ depending on $\omega_i$ in place of $\eps^{-\alpha(d)}$ and $\Omega^\rho$ in place of $\Omega^\eps$. (See \cite[Theorem 2, Section 5.6.2]{EvansGariepy92}; note  that the theorem in the reference above is stated and proved in a ball, but that the argument only relies on Poincar\'e inequalities, and thus easily extends to bounded Lipschitz domains.) Choose an infinitesimal sequence $\lbrace \rho_k\rbrace_k \subset (0,+\infty)$ and a corresponding strictly decreasing infinitesimal sequence $\lbrace \ep_{k}\rbrace_k \subset (0,+\infty)$ such that the sequence $\lbrace \Omega^{\rho_k} \rbrace_k$ satisfies \eqref{eq: Omegaeps} (with $\Omega^{\rho_k}$ in place of $\Omega^\ep$) and 
\begin{equation*}
%\label{eq:cond-rho}
\max_{i=1,\ldots, N^{\rho_k}} C^{\rho_k}_i \le \eps_{k}^{-\alpha(d)}.
\end{equation*}
To conclude, we apply the following diagonalization argument: for $\ep\in [\ep_{k}, \ep_{{k-1}}]$ we set $\rho^\ep:=\rho_{k-1}$. The claim follows by  considering the sets $\Omega^{\ep}:=\Omega^{\rho^\ep}$.

\emph{Step 2: Construction of auxiliary partitions.}  We start the actual proof by constructing a finite partition of $T^\eps \cap \Omega^\eps$ as follows: we define $f^\eps\colon (a,b) \to (0,+\infty)$ by 
\begin{align}\label{eq: fdef}
f^\eps(t) = \mathcal{H}^{d-1}( \lbrace x_d = t \rbrace \cap T^\eps \cap \Omega^\eps) \ \ \ \ \text{for } t \in (a,b).
\end{align}
We observe that $f^\eps \in BV((a,b))$, and that its total variation can be estimated by
\begin{align}\label{eq: NNparti1}
|Df^\eps|(a,b)  \le  \mathcal{H}^{d-1}(\partial^* T^\eps \cup \partial\Omega^\eps).
\end{align}
In fact, for any $\psi \in C^\infty_c(\Omega)$ with $\psi =1$ on $\Omega^\eps$, we get by Fubini's theorem that    
\begin{align*}
|Df^\eps|(a,b) & = \sup_{\substack{\varphi \in C^1_c(a,b) \\ \Vert \varphi\Vert_{L^\infty(a,b)} \le 1}} \  \int_{(a,b)} f^\eps \varphi' \, {\rm d}t =
\sup_{\substack{\varphi \in C^1_c(a,b) \\ \Vert \varphi\Vert_{L^\infty(a,b)} \le 1}} \ \int_{\Omega} \chi_{T^\eps \cap \Omega^\eps}(x',x_d) \ \varphi'(x_d) \,  {\rm d}(x',x_d), \\
& =  \sup_{{\varphi \in C^1_c(a,b),  \Vert \varphi\Vert_{L^\infty(a,b)} \le 1}} \ \int_{\Omega} \chi_{T^\eps \cap \Omega^\eps}(x',x_d) \  {\rm div} \big(\psi(x) \varphi(x_d)e_d \big) \,  {\rm d}(x',x_d),
\end{align*}
 where we write $x=(x',x_d)$ with $x'\in \R^{d-1}$ and $x_d\in \R$. Therefore, we obtain 
\begin{align*}
|Df^\eps|(a,b) \le
 \sup_{\varphi \in C^1_c(\Omega;\R^d ), \, \Vert \varphi\Vert_{L^\infty(\Omega)} \le 1} \  \int_\Omega \chi_{T^\eps \cap\Omega^\eps} \ {\rm div}(\varphi) \,  {\rm d}x = |D\chi_{T^\eps\cap\Omega^\eps}|(\Omega).
 \end{align*}
Then, \eqref{eq: NNparti1} follows from \cite[(3.29), (3.62)]{Ambrosio-Fusco-Pallara:2000}. 

We set $p:= 1+ \frac{3}{2d(2d-3)} \in (1,2)$. (The choice becomes clear later.) Choose  $\sigma_\eps \in (\eps^p/2,\eps^p)$ such that
\begin{align}\label{eq: NNparti2}
\mathcal{H}^0\big(\partial^* \lbrace f^\eps \le \sigma_\eps \rbrace \cap (a,b)\big) \le 2 \eps^{-p} \int_{\eps^p/2}^{\eps^p} \mathcal{H}^0\big(\partial^* \lbrace f^\eps \le s \rbrace \cap (a,b)\big) \, {\rm d}s \le 2\eps^{-p} |Df^\eps|(a,b),   
\end{align}
where the last step follows from the coarea formula for $BV$ functions (see \cite[Theorem 3.40]{Ambrosio-Fusco-Pallara:2000}). We choose $ a < d_1 < d_2 < \ldots < d_{m-1} < b$ such that $\partial^* \lbrace f^\eps \le \sigma_\eps \rbrace \cap (a,b) = \lbrace d_j \rbrace_{j=1}^{m-1}$, where $m-1 \le 2\eps^{-p} |Df^\eps|(a,b)$ by \eqref{eq: NNparti2}. We define a finite partition of $T^\eps \cap \Omega^\eps$ consisting of the sets
\begin{align}\label{eq: NNparti3}
\tilde{P}^\eps_j = T^\eps \cap \Omega^\eps \cap \lbrace  d_{j-1} < x_d < d_j \rbrace,\ \ \ \ \ j =1,\ldots,m, 
\end{align}
where we let $d_0 = a$ and $d_m = b$. In view of the definition in \eqref{eq: fdef}, we can estimate the `upper' and `lower' boundary of $\tilde{P}^\eps_j$ by  $\mathcal{H}^{d-1}((\partial^* \tilde{P}^\eps_j \cap \Omega^\eps) \setminus \partial^* T^\eps) \le 2\sigma_\eps \le 2\eps^p$. Therefore, since $m-1 \le 2\eps^{-p} |Df^\eps|(a,b)$ by \eqref{eq: NNparti2}, \eqref{eq: NNparti1} yields
\begin{align}\label{eq: NNparti4}
\sum\nolimits_{j=1}^m \mathcal{H}^{d-1}(\partial^* \tilde{P}_j^\eps) \le 2m\EEE\eps^p+ \mathcal{H}^{d-1}(\partial^*T^\eps\cup \partial \Omega^\eps) \le 5\mathcal{H}^{d-1}(\partial^* T^\eps \cup \partial \Omega^\eps) + 2\eps^p.
\end{align}
We repeat the above procedure for $\Omega^\eps \setminus T^\eps$ in place of $T^\eps$ and obtain a finite partition of $\Omega^\eps \setminus T^\eps$ which we denote by $\lbrace\tilde{P}^\eps_j\rbrace_{j=m+1}^{n}$. Repeating the argument in \eqref{eq: NNparti4} we get  $ \sum\nolimits_{j=m+1}^n \mathcal{H}^{d-1}(\partial^* \tilde{P}_j^\eps) \le 5\mathcal{H}^{d-1}(\partial^* T^\eps \cup \partial \Omega^\eps) + 2\eps^p$. We set $\tilde P_{n+1}^\eps = (\Omega \setminus \Omega^\eps) \cap T^\eps$ and $ \tilde P_{n+2}^\eps= \Omega \setminus (\Omega^\eps \cup T^\eps)$.  Since $\mathcal{H}^{d-1}(\partial \Omega) < + \infty$, by \eqref{eq: Omegaeps},  \eqref{eq:unif-en-estimate}, and Proposition \ref{lprop: phases}(i) we conclude 
 \begin{align}\label{eq: Ptilde}
\sum\nolimits_{j=1}^{n+2} \mathcal{H}^{d-1}(\partial^* \tilde{P}^\eps_j ) \le   10\mathcal{H}^{d-1}(\partial^* T^\eps \cup \partial \Omega^\eps) + 4\eps^p + \mathcal{H}^{d-1}(\partial(\Omega \setminus \Omega^\eps)) + 2\mathcal{H}^{d-1}(\partial^* T^\eps \cap (\Omega \setminus \Omega^\eps) ) \le  C
 \end{align}
for a constant $C>0$ independent of $\eps$. For later purposes, we note that each set $\tilde{P}^\eps_j$  is  either  contained in $T^\eps$ or  in  $\Omega \setminus T^\eps$.

\emph{Step 3: Limiting rotation, deformation, and partition.} Up to the extraction of a subsequence  (not relabeled), we may assume that 
$$R^\ep \to R \in SO(d),$$
 i.e., we directly have \eqref{eq:comp-R}. Applying Lemma \ref{lemma:comp-def}, up to passing to a further subsequence, we find  $y \in \mathscr{Y}(\Omega)$, see \eqref{eq: limiting deformations}, \EEE such that   \eqref{eq:comp-y} holds.  By \eqref{eq: rigidity in compactness} we get that there exists $\Phi\in BV(\Omega;\{A,B\})$ such that
 $$\Phi^\ep\wk^*\Phi\quad\text{weakly* in }BV(\Omega;\{A,B\})$$
 and hence almost everywhere in $\Omega$. By \eqref{eq: rigidity in compactness}, \eqref{eq:comp-R},  and   \eqref{eq:comp-y}  we then get $y \in \mathcal{Y}_R(\Omega)$.

 By \eqref{eq: Ptilde} and the compactness theorem for Caccioppoli partitions (Theorem \ref{th: comp cacciop})  we obtain a limiting partition $ \tilde{\mathcal{P}} := \EEE \lbrace \tilde{P}_j \rbrace_j$ such that $\tilde{P}_j^\eps \to \tilde{P}_j$ in measure for all indices $j$ (up to a subsequence). Note that the components $\lbrace \tilde{P}_j \rbrace_j$ are possibly not indecomposable. Therefore, we let $\mathcal{P} = \lbrace P_j \rbrace_j$ be the partition consisting of the connected components of $\lbrace \tilde{P}_j \rbrace_j$. (This partition exists due to \cite[Theorem 1]{Ambrosio-Morel},  see also Appendix \ref{sec: caccio}.)  By the lower semicontinuity of the Hausdorff measure and \eqref{eq: Ptilde} we also deduce
\begin{align}\label{eq: limiting parti}
\sum\nolimits_j \mathcal{H}^{d-1}(\partial^* P_j) = \sum\nolimits_j \mathcal{H}^{d-1}(\partial^* \tilde{P}_j) \le C. 
\end{align}
We close this step of the proof by showing that $\mathcal{P} \in \mathscr{P}(\Omega)$.  Clearly, by the definition of $\mathcal{P}$, it suffices to prove $\tilde{\mathcal{P}}\in \mathscr{P}(\Omega)$. \EEE  To this end, it suffices to show that 
\begin{align}\label{eq: nu-directioni}
\nu_{\tilde{P}_j}(x) = \pm e_d\quad\text{for}\quad\mathcal{H}^{d-1}\text{-a.e. }x \in \partial^* \tilde{P}_j \cap \Omega,
\end{align}
 where $\nu_{\tilde{P}_j}$ denotes the outer unit normal to  $\tilde{P}_j$. Let $\Omega' \subset \subset \Omega$. Fix $i\in \{1,\dots,d-1\}$. Since the function $\varphi(\nu) = |\langle \nu, e_i \rangle|$ is $BV$-elliptic (see \cite[Theorem 5.20, Example 5.23]{Ambrosio-Fusco-Pallara:2000}), lower semicontinuity results for sets of finite perimeter \cite[Theorem 2.1]{AmbrosioBraides2} imply 
\begin{align}\label{eq: lsc}
\int_{\partial^* \tilde{P}_j\cap \Omega'} |\langle \nu_{\tilde{P}_j}, e_i \rangle| \, {\rm d}\mathcal{H}^{d-1} \le \liminf_{\eps \to 0} \int_{\partial^* \tilde{P}^\eps_j \cap \Omega'} |\langle \nu_{\tilde{P}^\eps_j}, e_i \rangle| \, {\rm d}\mathcal{H}^{d-1}.
\end{align}
For $\eps$ sufficiently small we have $\Omega' \subset \Omega^\eps$, see \eqref{eq: Omegaeps}. Then  the definition of $\tilde{P}^\ep_j$ (see \eqref{eq: NNparti3})  implies
\begin{align}\label{eq: lsc2}
 \liminf_{\eps \to 0} \int_{\partial^* \tilde{P}^\eps_j \cap \Omega'} |\langle \nu_{\tilde{P}^\eps_j}, e_i \rangle| \, {\rm d}\mathcal{H}^{d-1} \le  \liminf_{\eps \to 0} \int_{\partial^* T^\eps \cap \Omega} |\langle \nu_{T^\eps}, e_i \rangle| \, {\rm d}\mathcal{H}^{d-1} 
\end{align}
since $ \nu_{\tilde{P}^\eps_j}(x) = \pm e_d$ for $\mathcal{H}^{d-1}$-a.e.\ $x \in \partial^* \tilde{P}^\eps_j \setminus \partial^* T^\eps $.  In view of Proposition \ref{lprop: phases}(ii) and   \eqref{eq:unif-en-estimate},  recalling the definition of $\alpha(d) =1/(2d)$ in \eqref{eq:alphad}, we obtain by \eqref{eq: lsc}--\eqref{eq: lsc2}  
$$\int_{\partial^* \tilde{P}_j \cap \Omega'} |\langle  \nu_{\tilde{P}_j}, e_i \rangle| \, {\rm d}\mathcal{H}^{d-1} = 0\quad\text{for every}\ i=1,\dots,d-1.$$
  Thus, \eqref{eq: nu-directioni} holds since $\Omega' \subset \subset \Omega$ was arbitrary.   Therefore, $\tilde{\mathcal{P}}  \in \mathscr{P}(\Omega)$ and then also ${\mathcal{P}}  \in \mathscr{P}(\Omega)$.  \EEE

 \emph{Step 4: Definition of  the sequence of partitions and phase indicators.}   We now define the partitions $\mathcal{P}^\eps$ and the phase indicators $\mathcal{M}^\eps$, and show  \eqref{eq: partition property}, \eqref{eq:rigidity-compactness}, \eqref{eq:comp-sets}, and \eqref{eq:comp-M}.  The proof of \eqref{eq: partition property-new} is deferred to Step 5 below. Let $\mathcal{P}^\eps  = \lbrace P_j^\eps \rbrace_j$  be the partition consisting of the nonempty components of 
\begin{align}\label{eq: Peps}
\lbrace \tilde{P}^\eps_k \cap P_j\colon \ j,k \in \N\rbrace.    
\end{align}
 Since $\tilde{P}_k^\eps \to \tilde{P}_k$ for all indices $k$ and $P_j \subset \tilde{P}_k$ for some $k$, we  clearly get that \eqref{eq:comp-sets} holds. Additionally, property \eqref{eq: partition property} follows from   \eqref{eq: Ptilde}--\eqref{eq: limiting parti}.

 Recall \EEE that each component of $\mathcal{P}^\eps$ is contained in $T^\eps$ or $\Omega \setminus T^\eps$, see  the  sentence below \eqref{eq: Ptilde}. We define the sequence $\mathcal{M}^\eps = \lbrace M^\eps_j \rbrace_{j, \eps}$   by  $M_j^\eps = A$ for all $j$ such that $P^\eps_j \subset T^\eps$, and   $M_j^\eps = B$ otherwise. Then \eqref{eq:rigidity-compactness} follows from  \eqref{eq: rigidity in compactness}.   This along with \eqref{eq:comp-y}  also  implies   
\begin{align}\label{eq: strong loc}
\sum\nolimits_j R^\eps M^\eps_j\chi_{P^\eps_j} \to  \nabla y\quad\text{strongly in }L^2_{\rm loc}(\Omega;\M^{d \times d}).
\end{align}
Due to \eqref{eq: partition property}, we have  $\sum\nolimits_j \mathcal{H}^{d-1}(\partial^* P^\eps_j) \le C,$
which yields 
$$\big|D\big(\sum\nolimits_j R^\eps M^\eps_j\chi_{P^\eps_j}\big)\big|(\Omega) \le C.$$ 
 This along with \eqref{eq: strong loc} and a  $BV$ compactness argument yields \eqref{eq:comp-M}. 
 
\emph{Step 5: Proof of \eqref{eq: partition property-new}.} It remains to prove \eqref{eq: partition property-new}. Choose $\Omega' \subset \subset \Omega$ and let $\eps$ be sufficiently small such that $\Omega' \subset \Omega^\eps$, see \eqref{eq: Omegaeps}.  
We show \eqref{eq: partition property-new} only for the components of $\mathcal{P}^\eps$ which are contained in $T^\eps \cap \Omega^\eps$ since for components contained in $\Omega^\eps \setminus T^\eps$ the argument is the same.  Denote by $\pi_d(P^\eps_j)$   the orthogonal projection of $P^\eps_j$ onto the $e_d$-axis. In view of \eqref{eq: NNparti2}--\eqref{eq: NNparti3}, \eqref{eq: Peps}, and the fact that $\mathcal{P} \in \mathscr{P}(\Omega)$,  we can decompose the collection of components into the two sets    
\begin{align}\label{eq: isi2}
\mathcal{J}^\eps_1 &= \big\{ P_j^\eps \subset T^\eps \cap \Omega^\eps\colon\, \mathcal{H}^{d-1}(P^\eps_j \cap \lbrace x_d = t \rbrace)  \le \sigma_\eps  \ \text{for a.e.\ }   t \in \pi_d(P^\eps_j)\big\} , \notag\\ 
 \mathcal{J}^\eps_2 &= \big\{ P_j^\eps\subset T^\eps \cap \Omega^\eps\colon \, \mathcal{H}^{d-1}(P^\eps_j \cap \lbrace x_d = t \rbrace) > \sigma_\eps \ \text{for a.e.\ }   t \in \pi_d(P^\eps_j)\big\} . 
\end{align}
First, since $\sigma_\eps \le \eps^p$, we clearly get by Fubini's theorem that 
\begin{align}\label{eq: isi3}
\sum\nolimits_{P^\eps_j \in \mathcal{J}^\eps_1}  \mathcal{L}^d( \Omega^\eps  \cap P^\eps_j) \le (b-a) \sigma_\eps \le C\eps^p,  
\end{align}
where $C$ only depends on $\Omega$. We now consider the components in $\mathcal{J}^\eps_2 $. We let  
\begin{align}\label{eq: isi12}
I_j^\eps = \big\{t \subset \pi_d(P^\eps_j)\colon\, \mathcal{H}^{d-1}\big( ( \Omega^\eps \setminus P^\eps_j) \cap\lbrace x_d  =t\rbrace \big) > \sigma_\eps \big\} \quad\text{for every }j\in \mathcal{J}_2^\ep.
\end{align}
 Since $\sigma_\eps \le \eps^p$, we get
\begin{align}\label{eq: isi4}
\sum\nolimits_{P^\eps_j \in \mathcal{J}^\eps_2} \int_{\pi_d(P^\eps_j) \setminus I^\eps_j} \mathcal{H}^{d-1}\big(  (\Omega^\eps \setminus  P^\eps_j) \cap \lbrace x_d=t \rbrace  \big) \,{\rm d}t \le (b-a) \sigma_\eps \le C\eps^p.  
\end{align}
On the other hand, for a.e. $t \in I_j^\eps$ we get by \eqref{eq: isoperimetric}, applied for $E = P_j^\eps \cap \lbrace x_d = t\rbrace$, and  by \eqref{eq: isi2}, \eqref{eq: isi12} that 
\begin{align*}
&\sigma_\eps \le \min\big\{\mathcal{H}^{d-1}(P^\eps_j \cap \lbrace x_d = t\rbrace ), \mathcal{H}^{d-1}\big( (\Omega^\eps \setminus P^\eps_j) \cap \lbrace x_d = t\rbrace \big) \big\}\\
&\quad \le \eps^{-\alpha(d)}  \big(\mathcal{H}^{d-2}(\partial^* (P^\eps_j \cap \lbrace x_d = t\rbrace) \cap \Omega^\eps )\big)^{\frac{d-1}{d-2}}. 
\end{align*}
As $\sigma_\eps \ge \eps^p/2$, we find  $1/2 \le (1/2)^{(d-2)/(d-1)} \le \eps^{-(\alpha(d)+p)(d-2)/(d-1)} \mathcal{H}^{d-2}(\partial^* P^\eps_j \cap \lbrace x_d = t\rbrace  \cap \Omega^\eps )$.  
Integrating over $I_j^\eps$ and summing over the components $\mathcal{J}^\eps_2$, we get 
$$ 
\sum\nolimits_{P^\eps_j \in \mathcal{J}^\eps_2} \mathcal{L}^1(I^\eps_j ) \le  C \eps^{\frac{-(\alpha(d)+p)(d-2)}{d-1}} \sum\nolimits_j  \int_{a}^b\mathcal{H}^{d-2}\big(\partial^* P^\eps_j \cap \lbrace x_d = t\rbrace \cap \Omega^\eps \big)  \, {\rm d}t.
$$
We recall \eqref{eq: Peps}  and  the fact that $\mathcal{P} \in \mathscr{P}(\Omega)$. Moreover, we have $\bigcup_j\partial^* \tilde{P}^\eps_j \cap \Omega^\eps \cap \lbrace x_d = t\rbrace \subset \partial^* T^\eps \cap \Omega \cap \lbrace x_d = t\rbrace$ for a.e.\ $t \in (a,b)$, where $\mathcal{H}^{d-1}$-a.e. $x \in \partial^* T^\eps$ is contained in the boundary of at most two different components, see \eqref{eq: NNparti3}. Then, \eqref{eq:unif-en-estimate} and Proposition \ref{lprop: phases}(iii) yield 
$$  
\sum_{P^\eps_j \in \mathcal{J}^\eps_2} \mathcal{L}^1(I^\eps_j ) \le  C \eps^{\frac{-(\alpha(d)+p)(d-2)}{d-1}} \int_{-\infty}^\infty\mathcal{H}^{d-2}(\partial^* T^\eps \cap \lbrace x_d = t\rbrace \cap \Omega )  \, {\rm d}t \le C\eps^{-(\alpha(d)+p)(d-2)/(d-1)}\eps^{2-\alpha(d)},
$$
where $C>0$ depends on $C_0$. Recalling $p= 1+ \frac{3}{2d(2d-3)}$ and $\alpha(d)=1/(2d)$, this yields $\sum_{P^\eps_j \in \mathcal{J}^\eps_2} \mathcal{L}^1(I^\eps_j ) \le C\eps^p$ by an elementary computation. This along with \eqref{eq: isi4} and the fact that   $\mathcal{H}^{d-1}(\Omega^\eps \cap \lbrace x_d  =t \rbrace  ) \le ({\rm diam}(\Omega))^{d-1}$ for all $t \in (a,b)$ yields  
\begin{align}\label{eq: isi5}
\sum\nolimits_{P^\eps_j \in \mathcal{J}^\eps_2}  \mathcal{L}^d( L_{\Omega^\eps}(P^\eps_j) \setminus  P^\eps_j)  & \le ({\rm diam}(\Omega))^{d-1} \sum\nolimits_{P^\eps_j \in \mathcal{J}^\eps_2}   \mathcal{L}^1(I^\eps_j ) \notag \\ 
& \ \ \  + \sum\nolimits_{P^\eps_j \in \mathcal{J}^\eps_2}  \int_{\pi_d(P^\eps_j) \setminus I^\eps_j} \mathcal{H}^{d-1}\big(  (\Omega^\eps \setminus  P^\eps_j) \cap \lbrace x_d=t \rbrace  \big) \,{\rm d}t \le  C\eps^p,  
\end{align}
where the constant $C$ depends only on $\Omega$ and $C_0$, and $L_{\Omega^\eps}(P^\eps_j)$ is defined in \eqref{eq: layer set}.   By combining  \eqref{eq: isi3} and \eqref{eq: isi5} we get \eqref{eq: partition property-new} since $\Omega^\eps \supset \Omega'$ (for $\eps$ small enough). This concludes the proof. 
\end{proof}

 \begin{remark}[Geometry of $\Omega$]\label{rem: Geometric condition}
 {\normalfont

(i) Condition H8. could be dropped at the expense of more elaborated estimates. First, in \eqref{eq: partition property-new}, $L_{\Omega'}(P_j)$  would have to be replaced by the connected components of $L_{\Omega'}(P_j)$ which intersect $P_j$. Accordingly, the isoperimetric inequality \eqref{eq: isoperimetric}, applied in Step 5 of the proof,  would need  to be applied separately in each of the components of $\Omega^\eps \cap \lbrace x_d = t \rbrace$ to get an estimate along the lines of \eqref{eq: isi5}. 

(ii) The passage to a subdomain in \eqref{eq: partition property-new}--\eqref{eq:rigidity-compactness} is not needed if $\Omega$ is a paraxial cuboid: in this case, Theorem \ref{thm:rigiditythm} can be replaced by an equivalent statement directly on $\Omega$, see \cite[Theorem 3.1 and Remark 3.2]{davoli.friedrich}. Moreover, the isoperimetric inequality \eqref{eq: isoperimetric} in Step 5 can be performed on the (identical) cuboids $\Omega \cap \lbrace x_d  = t\rbrace$ of dimension $d-1$.}
 \end{remark}

% \BBB We point out that in dimension $d=2$ and for $\Omega$ simply connected the compactness result holds without passing to a subdomain, as Theorem \ref{thm:rigiditythm} can be replaced by an equivalent statement directly on the whole set $\Omega$. The same holds in any dimension $d$ if $\Omega$ is a paraxial cube (see \cite[Theorem 3.1 and Remark 3.2]{davoli.friedrich}).\EEE

 Recall the definition of $\mathscr{U}(\Omega)$ in \eqref{eq: def of U}.  The next step will be to identify limiting displacement fields for subsets $\Omega' \subset \subset \Omega$. Before that, we state an elementary local property of partitions that we will use several times.

 \begin{lemma}[Local property  of partitions]\label{lemma: localization}
Let $K\subset \subset  \Omega$. Then, for each $\mathcal{P} \in \mathscr{P}(\Omega)$, the set $K$ only intersects a finite number of sets contained in  $\mathcal{P}$.
\end{lemma}

\begin{proof}
The result is a direct consequence of the compactness of $K$, and of the definition of  $\mathscr{P}(\Omega)$.
\end{proof}

 \begin{proposition}[Rescaled displacement fields on  subdomains]\label{lemma: intermediate step2}
 Consider the setting of Proposition \ref{lemma: intermediate step1}. Let $\Omega' \subset \subset \Omega$, and denote by $\lbrace P_j \rbrace_{j=1}^N$ the components of $\mathcal{P}$ which intersect $\Omega'$, see Lemma \ref{lemma: localization}. Then there exist  $u \in \mathscr{U}(\Omega')$ with $J_u \subset \bigcup_j \partial P_j$ and collections of constants $\lbrace t_j^\eps\rbrace_{j=1}^N$ for $\eps>0$ such that the rescaled displacements  $u^{\ep}\colon\Omega'\to \R^d$ defined by 
 \begin{align}\label{eq: uepsdef}
 u^\eps(x) := \eps^{-1}\sum\nolimits_{j=1}^N \big(y^{\ep}(x)- ( R^{\ep} M_j^{\ep}\, x+t_j^{\ep}) \big)\chi_{P^{\ep}_j}(x) + \eps^{-1}\sum\nolimits_{j>N} \big(y^{\ep}(x)- R^{\ep} M_j^{\ep}\, x \big)\chi_{P^{\ep}_j}(x)
 \end{align}
for $x \in \Omega'$ satisfy (up to a subsequence, not relabeled)
 \begin{align}\label{eq: uepsdef2}
 u^{\ep}\to u\quad \text{ in measure in }\ \Omega', \quad\quad\quad\quad  \nabla u^{\ep}\wk \nabla u\quad\text{weakly in } L^2(\Omega';\M^{d\times d}).
 \end{align}
    \end{proposition} 

We note that the second addend in \eqref{eq: uepsdef} is intended to be zero if $\lbrace P^\eps_j\rbrace_j$ consists only of $N$ components.

\begin{proof}
First, we recall that the components $\lbrace P_j \rbrace_{j=1}^N$ are connected by definition, that  $\mathcal{H}^{d-1}(\partial P_j \setminus \partial^* P_j) =0$, and that $\nu_{P_j} = \pm e_d$ for  $\mathcal{H}^{d-1}$-a.e. $x \in \partial \EEE P_j \cap \Omega$, where the latter two properties follow from the fact that $\mathcal{P} \in \mathscr{P}(\Omega)$.  Possibly choosing another set $\Omega' \subset \subset \Omega'' \subset \subset \Omega$ we can assume that the sets $P_j \cap \Omega''$, $j= 1,\ldots, N$, are connected and have Lipschitz boundary. Clearly, it suffices to show the statement for $\Omega''$ in place of $\Omega'$. For simplicity, we still denote this set by $\Omega'$.

  Let  $(R^\eps, \mathcal{P}^\eps, \mathcal{M}^\eps)$ be the triples identified in Proposition \ref{lemma: intermediate step1}.  By \eqref{eq:rigidity-compactness} we get 
\begin{align}\label{eq: veps2}
\Big\|  \sum\nolimits_{j} \big( \nabla y^\eps - R^\eps M_j^{\ep}\big)\,\chi_{P^{\ep}_j}  \Big\|_{L^2(\Omega')} \le C\eps  
\end{align}
for a constant $C>0$ depending  on $\Omega'$.

\emph{Step 1: Poincar\'e estimate on each component.}   Since $P_j \cap \Omega'$ is connected with Lipschitz boundary, we can choose an increasing sequence of smooth connected sets $K_n\subset\subset P_j\cap\Omega' $ such that $\mathcal{L}^d( (P_j \cap \Omega') \setminus K_n) \to 0$ as $n\to\infty$. The sets can be chosen such that the functions 
\begin{equation}
\label{eq:def-fj}
f^{n,\eps}_j(x) := \eps^{-1}(y^\eps(x) - R^\eps  \,  M_j^\eps   \,   x - t^{n,\eps}_j)\quad\text{ for every }x \in K_n,
\end{equation}    
for suitable $t^{n,\eps}_j \in \R^d$, satisfy a Poincar\'e estimate
\begin{equation}
\label{eq:new-est}
\Vert  f^{n,\eps}_j \Vert_{L^p(K_n)} \le C\Vert \nabla f^{n,\eps}_j \Vert_{L^p(K_n)},
\end{equation}
where the constant $C$ depends on $P_j$, but is independent of $\eps$ and $n$. By \eqref{eq:comp-sets} and \eqref{eq: partition property-new}  we get   $P^\eps_j \cap \Omega' \to P_j \cap \Omega'$ and $L_{\Omega'}(P^\eps_j) \to P_j \cap \Omega'$ in measure as $\eps \to 0$. The latter and the fact that  $K_n\subset\subset \Omega' \cap P_j$ show that $K_n \subset L_{\Omega'}(P^\eps_j)$ for $\eps$ small enough (depending on $n$). Thus, by using again \eqref{eq: partition property-new}  and  $\mathcal{L}^d((P^\eps_j \cap \Omega') \triangle (P_j \cap \Omega'))\to 0$, we get 
  \begin{equation}
\label{eq: saturation}
\mathcal{L}^d(K_n\setminus {P}^\ep_j) \le  \mathcal{L}^d\big(K_n\setminus L_{\Omega'}(P^\eps_j)\big) + \mathcal{L}^d\big(L_{\Omega'}(P^\eps_j)\setminus {P}^\ep_j\big) = \mathcal{L}^d\big(L_{\Omega'}(P^\eps_j)\setminus {P}^\ep_j\big)  \leq C\eps^p 
\end{equation}
for $\eps$ small enough depending on $n$, where $p=p(d) \in (1,2)$ is fixed. Let $L$ be a sufficiently large constant  (independent of $\eps,n$)  such that 
\begin{align*}
%\label{eq: Ldef}
\dist(F,SO(d)\lbrace A,B\rbrace) \ge |F-R^\eps M_j^\eps|/2 \ \ \ \ \text{for all} \ \ F \in \M^{d \times d} \text{ with } |F-R^\eps M_j^\eps| \ge L.
\end{align*}
Then, $\Vert \nabla f^{n,\eps}_j \Vert_{L^p(K_n)}$ can be controlled by
\begin{align*}
 &\Vert \nabla f^{n,\eps}_j \Vert_{L^p({P}_j^\eps\cap K_n)} + \Vert \nabla f^{n,\eps}_j  \Vert_{L^p((K_n \setminus {P}_j^\eps) \cap \lbrace |\nabla y^\eps - R^\eps M^\eps_j| \le L \rbrace)}  + \Vert \nabla f^{n,\eps}_j  \Vert_{L^p((K_n \setminus {P}_j^\eps) \cap \lbrace |\nabla y^\eps-R^\eps M^\eps_j| > L \rbrace)} \notag\\
& \le \frac{1}{\eps} \Vert \nabla y^\eps - R^\eps M^\eps_j \Vert_{L^p({P}_j^\eps \cap \Omega')} + \frac{L}{\eps}  \big(\mathcal{L}^d( K_n \setminus {P}^\eps_j  )\big)^{\frac{1}{p}} \EEE + \frac{2}{\eps}\Vert \dist(\nabla y^\eps,SO(d)\lbrace A,B\rbrace) \Vert_{L^p(\lbrace |\nabla y^\eps-R^\eps M^\eps_j| > L \rbrace)}.
\end{align*}
Using  H\"older's inequality for $p<2$, \eqref{eq: veps2}, \eqref{eq: saturation}, as well as  \eqref{eq: nonlinear energy}, \eqref{eq:unif-en-estimate} together with H4.\ we obtain the uniform estimate $\Vert \nabla f^{n,\eps}_j \Vert_{L^p(K_n)} \le C$ for $C>0$ independent of $n$ and $\eps$.  Then \eqref{eq:new-est} yields    
\begin{equation}
\label{eq:t-and-C}
\Vert  f^{n,\eps}_j   \Vert_{W^{1,p}(K_n)} \le C. 
\end{equation}
We now show that the translations $\lbrace t^{n,\eps}_j\rbrace_\eps$ and thus the functions $\lbrace f^{n,\eps}_j\rbrace_\eps$ can actually be chosen \emph{independently} of $n$. Recall that $K_n \supset K_1$ for all $n \in \N$.  In view of \eqref{eq:def-fj} and \eqref{eq:t-and-C}, we have 
\begin{align}\label{eq: compare-diff-n}
 \ep^{-1}|t^{n,\ep}_j-t^{m,\ep}_j| \, \mathcal{L}^d(K_1) \le  \|f^{\ep,n}_j\|_{L^1(K_n)}+\|f^{\ep,m}_j\|_{L^1(K_m)}  \leq C
\end{align}
for every $m,n\in \N$, where the constant $C$ is independent of $n$, $m$, and $\eps$.
Thus, for every $\ep>0$ we get that $\{t^{n,\ep}_j\}_{n}$ is a bounded sequence, and up to the extraction of a  subsequence (not relabeled) there exists $t^{\ep}_j$ such that 
\begin{align}\label{eq: limiting constants}
t^{n,\ep}_j\to t^{\ep}_j\quad\text{as }n\to +\infty.
\end{align}
  The constants $t^\eps_j$ are the ones from the statement of the proposition. By \eqref{eq: compare-diff-n} we  get $\ep^{-1}|t^{n,\ep}_j-t^{\ep}_j| \le C$ for a constant $C>0$ independent of $n$ and $\eps$. This along with \eqref{eq:t-and-C} yields that the functions 
\begin{align}\label{eq: feps-defi}  
  f^{\eps}_j(x) := \eps^{-1}(y^\eps(x) - R^\eps  \,  M_j^\eps   \,   x -t^\eps_j) \quad\text{ for every }x \in P^\eps_j, 
  \end{align}
satisfy for all $n \in \N$ and all  $\eps$ small enough (depending on $n$) 
$$\|f^{\ep}_j\|_{W^{1,p}(K_n)}\leq C,$$
where the constant $C>0$ is independent of $\eps$ and $n$.  Thus, by a compactness and a diagonal argument there exists a function $f_j \in W^{1,p}(P_j\cap \Omega';\R^d)$ such that (up to  a subsequence) 
\begin{align}\label{eq: first-lim-pass}
f^\eps_j\rightharpoonup f_j \ \ \ \ \text{weakly in } W^{1,p} (P_j\cap \Omega';\R^d).
\end{align}

\emph{Step 2: Definition of the limiting displacement field.} Recall the    functions $f_j$ identified in \eqref{eq: first-lim-pass} and the constants $t_j^{\ep}$ from \eqref{eq: limiting constants}.  We set $u := \sum_{j=1}^N f_j \chi_{P_j}$ on $\Omega'$ and define $u^\eps$ as in \eqref{eq: uepsdef}. Below we will show that indeed  $u \in \mathscr{U}(\Omega')$, see  \eqref{eq: def of U}, but now we first confirm  \eqref{eq: uepsdef2}. In view of \eqref{eq: feps-defi}, we get that $u^\eps = f_j^\eps$ on $P_j^\eps\cap \Omega'$.  We claim that, up to a further subsequence, there holds
\begin{align}\label{eq: convergence-for-gradu}
{\rm (i)}&  \ \   u^\eps \to f_j =  u \ \ \ \text{ in measure on} \ P_j\cap \Omega' \quad \text{for all } j=1,\ldots,N, \notag \\ 
{\rm  (ii)}&  \ \  \nabla u^\eps   \rightharpoonup      \nabla u \quad \text{weakly in } L^2(\Omega';\mathbb{M}^{d\times d}).  
\end{align}
In fact, \eqref{eq: first-lim-pass} along with  \eqref{eq:comp-sets} and $u^\eps = f_j^\eps$ on $P_j^\eps\cap \Omega'$ implies measure convergence on $P_j\cap \Omega'$. This yields (i). To see (ii), we use  \eqref{eq: uepsdef} and \eqref{eq: veps2} to get 
$$ \nabla u^\eps  = \eps^{-1}   \sum\nolimits_{j} \big( \nabla y^\eps - R^\eps M_j^{\ep}\big)\,\chi_{P^{\ep}_j}  \rightharpoonup g$$
weakly in $L^2(\Omega';\M^{d \times d})$ for a suitable function $g$. Again by \eqref{eq: first-lim-pass} we get $g = \nabla f_j$ on each $P_j \cap \Omega'$, and therefore $g = \nabla u$ a.e.\ on $\Omega'$. This yields (ii).  Clearly, \eqref{eq: convergence-for-gradu} implies \eqref{eq: uepsdef2}.

It remains to check that  $u \in \mathscr{U}(\Omega')$. \EEE Recall that only the components $P_j$, $j=1,\ldots,N$, intersect $\Omega'$. Since $f_j \in W^{1,p}(P_j\cap \Omega';\R^d)$ for all $j=1,\ldots,N$, we get $J_u \subset \bigcup_{j=1}^N \partial P_j$. Thus, we find $\mathcal{H}^{d-1}(J_u)< + \infty$ since $\mathcal{P}$ is a Caccioppoli partition. More precisely, as $\mathcal{P} \in \mathscr{P}(\Omega)$, the jump set of $u$ is contained in $(d-1)$-dimensional hyperplanes orthogonal to $e_d$. It thus remains to show that $u \in SBV^2(\Omega';\R^d)$. First, note that $\nabla u \in L^2(\Omega';\M^{d \times d})$  by \eqref{eq: convergence-for-gradu}(ii).  Since each $P_j \cap \Omega'$ has Lipschitz boundary,  we get that $u|_{P_j\cap \Omega'} \in H^1(P_j\cap \Omega';\R^d)$, and the trace of $u$ on $\partial P_j \cap \Omega'$ exists. As the number of sets $\lbrace P_j \rbrace_j$ intersecting $\Omega'$ is finite, we obtain $u \in SBV^2(\Omega';\R^d)$ by applying \cite[Theorem 3.84]{Ambrosio-Fusco-Pallara:2000}.
\end{proof}

We next show that the translations can be defined so that there exists a limiting rescaled displacement field on the whole domain $\Omega$.

 \begin{proposition}[Rescaled displacement fields]\label{lemma: intermediate step3}
 Consider the setting of Proposition \ref{lemma: intermediate step1}. Then there exist collections of constants $\mathcal{T}^\eps = \lbrace t_j^\eps\rbrace_{j}$ for $\eps >0$ and $u \in \mathscr{U}(\Omega)$ with $J_u \subset \bigcup_j \partial P_j$  such that the rescaled displacements  $u^{\ep}$ defined in \eqref{eq: rescaled disp} satisfy \eqref{eq:comp-u}--\eqref{eq:comp-grad-u}.  
 \end{proposition}

\begin{proof}
Consider a sequence $\{\Omega_n\}_n$ of open sets, compactly contained in $\Omega$, satisfying $\Omega_n\subset \Omega_{n+1}$ for every $n\in \N$, and such that  $\mathcal{L}^d(\Omega\setminus \Omega_n)\to 0$ as $n \to \infty$.   We denote by $\lbrace {\mathcal{P}}^\eps\rbrace_\eps$ and $\mathcal{P}$ the partitions identified in Proposition \ref{lemma: intermediate step1}. In view of Lemma \ref{lemma: localization}, we can reorder the partition ${\mathcal{P}} =\lbrace P_j \rbrace_j$ in a specific way and can choose integers ${N}_1 \le {N}_2 \le \ldots$ such that  $ \lbrace {P}_j \rbrace_{j=1}^{{N}_n}$ indicate the components of ${\mathcal{P}}$   which intersect $\Omega_n$. For each $n\in\N$, the translations given by Proposition \ref{lemma: intermediate step2} (with $\Omega_n$ in place of $\Omega'$) are denoted by $\lbrace {t}_j^{\eps,n} \rbrace_{j=1}^{{N}_n}$. The displacement fields on $\Omega_n$ defined in \eqref{eq: uepsdef} are denoted by  $u^{\eps,n}$. We denote their limits by $u^n \in \mathscr{U}(\Omega_n)$ and recall that $J_{u^n} \subset \bigcup_j \partial P_j$. By a diagonal argument, we may suppose that there exists a subsequence of $\eps$ (not relabeled) such that \eqref{eq: uepsdef2} holds for all $n \in \N$, i.e.,
 \begin{align}\label{eq: uepsdef2-new}
 u^{\ep,n}\to u^n\quad \text{ in measure in \EEE } \Omega_n, \quad\quad\quad\quad  \nabla u^{\ep,n}\wk \nabla u^n\quad\text{weakly in } L^2(\Omega_n;\M^{d\times d}).
 \end{align}
Now it is elementary to check that for each $n \in \N$  there holds
\begin{align}\label{eq: controlled difference}
\lim_{\eps \to 0}\eps^{-1} (t^{\eps,n}_j - t^{\eps,n+1}_j) \ \ \ \text{exists and is finite for all $1 \le j \le N_{n}$.}
\end{align}
Indeed, this follows from $\mathcal{L}^d(P_j \cap \Omega_{n})>0$ for all $1 \le j \le N_{n}$, and the fact that
$$ \eps^{-1}  (t^{n}_j - t^{n+1}_j)\chi_{P^\eps_j\cap \Omega_n} =  ( u^{\eps,n+1} - u^{\eps,n}) \chi_{P^\eps_j\cap \Omega_n}  \to   (u^{n+1}-u^{n} )  \chi_{P_j\cap \Omega_n}$$
in measure, see  \eqref{eq:comp-sets} and \eqref{eq: uepsdef}--\eqref{eq: uepsdef2}, as well as \eqref{eq: uepsdef2-new}.

We define the collection of translations $\mathcal{T}^\eps = \lbrace t_j^\eps\rbrace_ {j}$ as follows: for each $j$, choose $n \in \N$ such that $N_{n-1} <  j \le  N_n$, and set $t_j^\eps = t_j^{\eps,n}$, where we define $N_0 =0$ for convenience. We define $u^{\ep}\colon\Omega\to \R^d$  as in \eqref{eq: rescaled disp}. By recalling the definition of  $u^{\eps,n}$   in \eqref{eq: uepsdef}, we get that the restriction of $u^{\eps}$ on $\Omega_n$, for $n \in \N$, satisfies
$$u^\eps =   u^{\eps,n} + \sum\nolimits_{j=1}^{N_n} \eps^{-1}(t_j^{\eps,n}- t_j^{\eps})  \chi_{P^\eps_j \cap \Omega_n} - \sum\nolimits_{j> N_n} \eps^{-1} t_j^{\eps} \chi_{P^\eps_j \cap \Omega_n}  \ \ \ \text{on $\Omega_n$}.$$
We introduce the function  $v^n \in \mathscr{U}(\Omega_n)$  by
\begin{align}\label{eq: unvn} 
v^n = u^n + \sum\nolimits_{j=1}^{N_n} \,\Big( \lim_{\eps \to 0} \eps^{-1}( t_j^{\eps,n}- t_j^{\eps}) \Big) \,  \chi_{P_j \cap \Omega_n},
\end{align}
which is well defined by  \eqref{eq: controlled difference} and the fact that $t_j^\eps = t_j^{\eps,m}$ for the index $1 \le m \le n$ such that $N_{m-1}< j\le N_m$. In view of  \eqref{eq:comp-sets}, \eqref{eq: uepsdef2-new}, and the fact that  $P_j \cap \Omega_n = \emptyset$ for all $j>N_n$, we then get 
\begin{align}\label{eq: controlled difference2}
 u^{\ep}\to v^n\quad \text{ in measure on } \Omega_n, \quad\quad\quad\quad  \nabla u^{\ep}\wk \nabla v^n\quad\text{weakly in } L^2(\Omega_n;\M^{d\times d}).
 \end{align}
 This also shows that $v^n=v^m$ on $\Omega_n$ for all $n \le m$. This observation allows to define the function $u\colon \Omega \to \R^d$ by $u = v^n$ on $\Omega_n$ for all $n \in \N$.  The fact that $J_{u^n} \subset \bigcup_j \partial P_j$  along with \eqref{eq: unvn} also yields $J_u \subset \bigcup_j \partial P_j$. Clearly, we get $u \in \mathscr{U}(\Omega)$ since $v^n \in \mathscr{U}(\Omega_n)$ for all $n \in \N$. Finally, by \eqref{eq: controlled difference2} and the fact that $u = v^n$ on $\Omega_n$  we get that $u^{\ep}$ satisfies \eqref{eq:comp-u}--\eqref{eq:comp-grad-u}. This concludes the proof.   
\end{proof}

We conclude this section with the  proof of  Theorem \ref{thm:compactness}. Given the above constructions, it remains to show that the partitions and translations can be chosen in a specific way  such that also the selection principle \eqref{eq: toinfty} is satisfied. Although the realization of this is very technical, the main idea is quite simple: whenever two components violate \eqref{eq: toinfty}, they are combined, and they are replaced by a single component in the partition.

\begin{proof}[Proof of Theorem \ref{thm:compactness}]
Let $\{y^{\ep}\}_{\ep}\subset H^2(\Omega;\R^2)$ be a sequence of deformations satisfying  \eqref{eq:unif-en-estimate}. Consider a sequence $\{\Omega_n\}_n$ of open sets compactly contained in $\Omega$, satisfying $\Omega_n\subset \Omega_{n+1}$ for every $n\in \N$, and such that $\mathcal{L}^d(\Omega\setminus \Omega_n)\to 0$. We will prove that, after extracting a subsequence in $\eps$ (not relabeled), for each $n \in \N$ there exists a sequence of quadruples $(R^{\eps}, \mathcal{P}^{\eps,n}, \mathcal{M}^{\eps,n}, \mathcal{T}^{\eps,n})$ with $\mathcal{P}^{\eps,n} = \lbrace P_j^{\eps,n}\rbrace_j$, $\mathcal{M}^{\eps,n} = \lbrace M_j^{\eps,n}\rbrace_j$, $\mathcal{T}^{\eps,n} = \lbrace t_j^{\eps,n}\rbrace_j$ and limiting triples $(y,u^n,\mathcal{P}^n) \in \mathscr{Y}(\Omega)\times \mathscr{U}(\Omega) \times \mathscr{P}(\Omega)$ such that \eqref{eq: partition property}--\eqref{eq:comp-M} and  \eqref{eq: rescaled disp}--\eqref{eq:comp-grad-u} hold, and additionally we have
\begin{align}\label{eq: toinfty-onK}
 \frac{|t_i^{\ep,n}-t_j^{\ep,n}|}{\ep}\to +\infty \ \ \ \ &\text{for all $i \neq j$ }\ \text{with $P_i^{n}\cap \Omega_n \neq \emptyset$, $P_j^{n}\cap \Omega_n \neq \emptyset$,}\  \text{and $\lim_{\eps \to 0} M^{\eps,n}_i = \lim_{\eps \to 0} M^{\eps,n}_j$},
\end{align} 
where $\lbrace P^n_j\rbrace_j$ denote the components of the limiting partition $\mathcal{P}^n$. Note that the deformation $y$ and the rotations  $R^\eps$ can be chosen independently of $n \in \N$. Moreover, we will see that the objects can be constructed such that for each $n \ge m$ and each  $\eps >0$ we have  
\begin{align}\label{eq: toinfty-onK2}
{\rm (i)}& \ \ \text{for all $j$ there exists $l_j$ such that \ }  P_j^{\eps,m}  \subset   P_{l_j}^{\eps,n},  \notag \\
{\rm (ii)} & \ \  \text{for all $j$ we have }  M^{\eps,n}_{l_j} = M_j^{\eps,m} \text{ with $l_j$ given in {(i)}}, \notag \\
{\rm (iii)} & \ \  \text{if } \mathcal{L}^d(P_j^{\eps,m} \cap \Omega_m)>0, \text{ then } t^{\eps,n}_{l_j} = t_j^{\eps,m}  \text{ with $l_j$ given in {(i)}},\notag\\
{\rm (iv)} & \ \  u^{\eps,n} = u^{\eps,m} \text{ on $\Omega_m$} \ \ \text{ and } \  \ \nabla u^{\eps,n} = \nabla u^{\eps,m} \text{ on $\Omega$},  
\end{align}
where $u^{\eps,n}$ denote the rescaled displacement fields given in \eqref{eq: rescaled disp} for the quadruples  $(R^{\eps}, \mathcal{P}^{\eps,n}, \mathcal{M}^{\eps,n}, \mathcal{T}^{\eps,n})$. We defer the proof to Step 2 below and first show that this implies Theorem \ref{thm:compactness} for a suitable diagonal sequence (Step 1).
%\ \text{ and } \  P_{i_j}^{\eps,n} \setminus P_j^{\eps,m}  \subset \Omega \setminus \Omega_m 

\emph{Step 1: Extracting a diagonal sequence.} First,  we find by  \eqref{eq:comp-u}  on $\Omega_n$ and $\Omega_m$, and by  \eqref{eq: toinfty-onK2}(iv) that for all $n \ge m$ there holds $u^n = u^m$ on $\Omega_m$ and $\nabla u^n = \nabla u^m$ on $\Omega$.  
This observation allows to define the function $u\colon \Omega \to \R^d$ by $u = u^n$ on $\Omega_n$ for all $n \in \N$. Clearly, we get $u \in \mathscr{U}(\Omega)$ since $u^n \in \mathscr{U}(\Omega)$ for all $n \in \N$.  In particular, there holds  for all $n \in \N$
\begin{align}\label{eq: the same!}
u =  u^n\quad  \text{ on } \ \Omega_n, \ \ \ \ \ \nabla u = \nabla u^n \quad\text{on $\Omega$}.
 \end{align}
As $\mathcal{P}^{\eps,n}$ is a coarsening of $\mathcal{P}^{\eps,m}$   for all $n \ge m$ by  \eqref{eq: toinfty-onK2}(i), we get that 
 $\mathcal{P}^{n}$ is a coarsening of $\mathcal{P}^{m}$   for all $n \ge m$ by \eqref{eq:comp-sets}. This gives $\sum_j \mathcal{H}^{d-1}(\partial P^n_j) \le \sum_j \mathcal{H}^{d-1}(\partial P^1_j) < + \infty$ for all $n \in \N$.  By Theorem \ref{th: comp cacciop}    there exists a partition $\mathcal{P} = \lbrace P_j \rbrace_j$  such that $P_j^{n}\to P_j$ in measure for all $j \in \N$. Note that this convergence also implies $\mathcal{P}  \in \mathscr{P}(\Omega)$. This and  \eqref{eq:comp-sets} for each $m \in \N$ yield
\begin{align}\label{eq: the same!8}
\lim_{n\to \infty}\sum\nolimits_j \mathcal{L}^d(P_j^n \triangle P_j) =  0, \ \ \ \ \ \ \ \lim_{\eps \to 0} \sum\nolimits_j \mathcal{L}^d(P_j^{\eps,m} \triangle P_j^m) = 0 \ \ \ \text{for all $m \in \N$},
\end{align}
where $\triangle$ denotes the symmetric difference of two sets, see below Theorem \ref{th: comp cacciop}. Thus, by Attouch's diagonalization lemma \cite[Lemma 1.15 and Corollary 1.16]{attouch}, we can choose a diagonal sequence $\lbrace n(\eps) \rbrace_\eps$ such that 
\begin{align}\label{eq: partition convergence}
P^{\eps,n(\eps)}_j \to P_j \ \ \ \text{ in measure as $\ep\to 0$  for all indices $j$.}
\end{align}
We now define the triples $\mathcal{P}^\eps = \mathcal{P}^{\eps,n(\eps)}$, $\mathcal{M}^\eps= \mathcal{M}^{\eps,n(\eps)}$, and $\mathcal{T}^\eps = \mathcal{T}^{\eps,n(\eps)}$, and check  that \eqref{eq: partition property}--\eqref{eq:comp-grad-u} hold for the limiting triple  $(y,u,\mathcal{P})$.

 First,  \eqref{eq: partition property}--\eqref{eq: partition property-new}  follow directly from the corresponding properties of the partitions  $\mathcal{P}^{\eps,n}$. We observe that  \eqref{eq: toinfty-onK2}(i),(ii)  yield  
\begin{align*}
\sum\nolimits_j R^{\ep} M_j^{\ep,  n(\ep)}\chi_{P_j^{\ep,  n(\ep) }} = \sum\nolimits_j R^{\ep} M^{\ep,1}_j\chi_{P^{\ep,1}_j}.
\end{align*}
This implies \eqref{eq:rigidity-compactness}, \eqref{eq:comp-R}, \eqref{eq:comp-y}, and \eqref{eq:comp-M}  by using the corresponding properties for the triple  $(R^{\eps}, \mathcal{P}^{\eps,1}, \mathcal{M}^{\eps,1})$. Property \eqref{eq:comp-sets} follows from \eqref{eq: partition convergence}.

Consider the rescaled displacement fields $u^{\eps,n(\eps)}$ defined in   \eqref{eq: rescaled disp}. For each $m \in \N$ we have $u^{\eps,n(\eps)} \to u^m =u$  in measure   on $\Omega_m$ by \eqref{eq: toinfty-onK2}(iv), \eqref{eq: the same!}, and \eqref{eq:comp-u}  for $m$. As $m$ was arbitrary, we get  \eqref{eq:comp-u}. In a similar fashion, \eqref{eq:comp-grad-u} follows also by taking  into account   \eqref{eq: toinfty-onK2}(iv), \eqref{eq: the same!}, and \eqref{eq:comp-grad-u} for each $m$.

It remains to check \eqref{eq: toinfty}. To this end, we fix $i \neq j$ such that $\mathcal{L}^d(P_i), \mathcal{L}^d(P_j)>0$, and  $\lim_{\eps \to 0 }M_i^{\eps,n(\eps)} = \lim_{\eps \to 0 }M_j^{\eps,n(\eps)}$. In view of \eqref{eq: the same!8}--\eqref{eq: partition convergence}, we can fix $m \in \N$ (independently of $\eps$) and $\eps_0 = \eps_0(m)>0$  such that for all $0<\eps \le \eps_0$ we have for $k=i,j$
\begin{align}\label{eq: NNNNN}
{\rm (i)} \ \ \mathcal{L}^d(P^{m}_k\cap \Omega_m)>0, \ \mathcal{L}^d(P^{\eps,m}_k\cap \Omega_m)>0  \ \ \ \ \text{ and }  \ \ \ \ {\rm (ii)} \ \  \mathcal{L}^d(P_k^{\eps, n(\eps)} \triangle P_k^{\eps,m}) \le \frac{1}{2} \mathcal{L}^d(P_k^{\eps, m}).
\end{align}
 (To see (ii), we use   $\mathcal{L}^d(P_k^{\eps,n(\eps)} \triangle P_k^{\eps,m}) \le \mathcal{L}^d(P_k^{\eps,n(\eps)} \triangle P_k) +  \mathcal{L}^d(P_k \triangle P_k^{m})+ \mathcal{L}^d(P_k^{m} \triangle P_k^{\eps,m}) \to 0$ and $ \mathcal{L}^d(P_k^{ \eps, m}) \to  \mathcal{L}^d(P_k^{m})$ as $\eps \to 0$.) Possibly by passing to a smaller $\eps_0$,  we can also suppose that $n(\eps) \ge m$  for all $\eps \le \eps_0$. By \eqref{eq: toinfty-onK2}(i) for $n=n(\eps)$ we find a component $P^{\eps,n(\eps)}_{l_k}$ which contains $P^{\eps,m}_k$ up to an $\mathcal{L}^d$-negligible set for $k=i,j$. By  \eqref{eq: NNNNN}(ii) we necessarily have  that $\mathcal{L}^d(P^{\eps,n(\eps)}_k\cap P^{\eps,m}_{k})>0$. Thus, $k = l_k$. This along with \eqref{eq: NNNNN}(i) and \eqref{eq: toinfty-onK2}(ii),(iii) shows $M^{\eps,n(\eps)}_k = M^{\eps,m}_k$ and  $t^{\eps,n(\eps)}_k = t^{\eps,m}_k$ for $k=i,j$.   Then, also    $\lim_{\eps \to 0} M^{\eps,m}_i = \lim_{\eps \to 0} M^{\eps,m}_j$  and therefore, taking also    \eqref{eq: toinfty-onK}, \eqref{eq: NNNNN}(i) into account, we finally get 
    \begin{align*}
\lim_{\eps  \to 0} \frac{ |t_i^{\eps,n(\eps)}-t_j^{\eps,n(\eps)}|}{ \eps} = \lim_{\eps  \to 0} \frac{|t_i^{\ep,m}-t_j^{\ep,m}|}{ \eps }= + \infty.
  \end{align*}

% Let $\eps_0>0$ such that $n(\eps) \ge m$  for all $\eps \le \eps_0$. Possibly by passing to a smaller $\eps_0$, in view of   \eqref{eq: the same!8}, we can also suppose that  $P^{\eps,m}_i$ and $P^{\eps,m}_j$ intersect $\Omega_m$ for all $\eps \le \eps_0$. By \eqref{eq: toinfty-onK2}  we find $t^{\eps,n(\eps)}_j = t^{\eps,m}_j$, where we used that by \eqref{eq: partition convergence} for $m$ large enough the indices $j$ and $i_j$ have to coincide. \RRR {\tt Elisa: why is this uniform in $\ep$?} \EEE  Likewise,   $t^{\eps,n(\eps)}_i = t^{\eps,m}_i$.  

  \emph{Step 2: Coarsening scheme.} We inductively construct sequences of quadruples $(R^{\eps}, \mathcal{P}^{\eps,n}, \mathcal{M}^{\eps,n}, \mathcal{T}^{\eps,n})$ and limiting triples $(y,u^n,\mathcal{P}^n)$ for $n \in \N$ such that \eqref{eq: partition property}--\eqref{eq:comp-M}, \eqref{eq: rescaled disp}--\eqref{eq:comp-grad-u}, and \eqref{eq: toinfty-onK}--\eqref{eq: toinfty-onK2}  hold.

  We start with $n=1$. We apply  Proposition \ref{lemma: intermediate step3} to obtain rotations $R^\eps$  and triples $(\hat{\mathcal{P}}^{\eps}, \hat{\mathcal{M}}^{\eps}, \hat{\mathcal{T}}^{\eps})$, as well as a limiting triple $(y,\hat{u},\hat{\mathcal{P}})$ such that \eqref{eq: partition property}--\eqref{eq:comp-M} and \eqref{eq: rescaled disp}--\eqref{eq:comp-grad-u} hold. We write $\hat{\mathcal{P}}^\eps = \lbrace \hat{P}^\eps_j\rbrace_j$, $\hat{\mathcal{M}}^\eps = \lbrace \hat{M}^\eps_j\rbrace_j$, and $\hat{\mathcal{T}}^\eps = \lbrace \hat{t}^\eps_j\rbrace_j$. We modify the triples to get sequences which also satisfy \eqref{eq: toinfty-onK}.

\emph{Coarsening scheme for $n=1$.} We  construct ${\mathcal{P}}^{\eps,1}$, ${\mathcal{T}}^{\eps,1}$, and ${\mathcal{M}}^{\eps,1}$, as well as the limiting partition  ${\mathcal{P}}^1$ and the limiting displacement $u^1$ by the following iterative scheme: suppose that two components $\hat{P}_i$ and $\hat{P}_j$ of $\hat{\mathcal{P}}$ with $i\neq j$  violate  \eqref{eq: toinfty-onK}  on $\Omega_1$, i.e., 
\begin{align}\label{eq: uniform-eps}
\liminf\nolimits_{\eps \to 0}\eps^{-1}  |\hat{t}^{\eps}_i - \hat{t}^{\eps}_j| <+\infty, \ \ \ \ \  \lim_{\eps \to 0} \hat{M}^\eps_i = \lim_{\eps \to 0} \hat{M}^\eps_j,  \ \ \ \ \ \hat{P}_i\cap \Omega_1\neq \emptyset,  \ \ \ \ \   \hat{P}_j\cap \Omega_1\neq \emptyset. 
\end{align}
First, by passing to a subsequence in $\eps$ (not relabeled), we get $\hat{M}^{\eps}_i = \hat{M}^{\eps}_j$ for all $\eps$. Now, we replace $\hat{P}_i$ and $\hat{P}_j$ in $\hat{\mathcal{P}}$ by $P^1_* := \hat{P}_i \cup \hat{P}_j$.   In a similar fashion, we replace $\hat{P}^{\eps}_i$ and $\hat{P}^{\eps}_j$ in $\hat{\mathcal{P}}^{\eps}$ by $P^{\eps,1}_* := \hat{P}^{\eps}_i \cup \hat{P}^{\eps}_j$ for each $\eps>0$. Accordingly, on the set ${{P}}^{\eps,1}_*$ we introduce the translation $t^{\eps,1}_* = \hat{t}^{\eps}_i$ and the phase $M^{\eps,1}_* := \hat{M}^{\eps}_i = \hat{M}^{\eps}_j$ for each $\eps>0$. In view of Lemma \ref{lemma: localization}, only finitely many components of $\hat{\mathcal{P}}$ intersect $\Omega_1$. Thus, we can repeat this construction at most a finite number of times until, for the resulting partition $\mathcal{P}^1$ and the triples $({\mathcal{P}}^{\eps,1}, {\mathcal{M}}^{\eps,1}, {\mathcal{T}}^{\eps,1})$, each pair of components $P^1_i$ and $P^1_j$ satisfies  \eqref{eq: toinfty-onK}. This concludes the construction in the case $n=1$. (The definition of the resulting displacement field $u^1$ will be indicated below.) 

We check that   \eqref{eq: partition property}--\eqref{eq:comp-M}, \eqref{eq: rescaled disp}--\eqref{eq:comp-grad-u}, and \eqref{eq: toinfty-onK} are satisfied. First, \eqref{eq: toinfty-onK} clearly holds true by construction. To confirm the other properties, we assume for simplicity that the above coarsening scheme was applied only once for two sets $\hat{P}_i$ and $\hat{P}_j$ intersecting $\Omega_1$ since the general case follows by induction.   First, \eqref{eq:comp-R} and \eqref{eq:comp-y}  are not affected by the modification, and therefore still hold. Since the function $\sum_jR^\eps \hat{M}_j^\eps \chi_{\hat{P}^\eps_j}$ remains unchanged by construction, also \eqref{eq:rigidity-compactness}  and \eqref{eq:comp-M} are still satisfied. To see \eqref{eq: partition property} and \eqref{eq:comp-sets}, it suffices to recall that $P^{\eps,1}_*  =  \hat{P}^{\eps}_i \cup \hat{P}^{\eps}_j$ which implies   that $P^{\eps,1}_* \to P^1_* = \hat{P}_i \cup \hat{P}_j$ in measure. We now show \eqref{eq: partition property-new} for $\Omega' \subset\subset \Omega$. As $\mathcal{L}^d(\hat{P}_k \cap \Omega_1)>0$  for $k=i,j$,   for $\eps$ small enough, \eqref{eq: partition property-new} and \eqref{eq:comp-sets} (for $\hat{\mathcal{P}^\eps}$) imply $\mathcal{L}^d(\Omega' \cap \hat{P}^\eps_k) \ge  \mathcal{L}^d( L_{\Omega'}(\hat{P}^\eps_k) \setminus \hat{P}^\eps_k) $ for $k=i,j$. This also yields $\mathcal{L}^d(\Omega' \cap P^{\eps,1}_*) \ge  \mathcal{L}^d( L_{\Omega'}(P^{\eps,1}_*) \setminus P^{\eps,1}_*) $ for $\eps$ small enough. Therefore, since $\mathcal{L}^d( L_{\Omega'}(P^{\eps,1}_*) \setminus P^{\eps,1}_*) \le \sum_{k=i,j}\mathcal{L}^d( L_{\Omega'}(\hat{P}^\eps_k) \setminus \hat{P}^\eps_k) $, \eqref{eq: partition property-new} holds, as well. We now finally introduce the limiting displacement field and check \eqref{eq:comp-u}--\eqref{eq:comp-grad-u}. We observe 
$$u^{\eps,1} - \hat{u}^{\eps} = \eps^{-1} (\hat{t}_j^{\ep}-\hat{t}_i^{\ep})\chi_{\hat{P}_j^{\eps}}$$
where $u^{\eps,1}$ and $\hat{u}^{\eps}$ are the corresponding displacement fields defined in  \eqref{eq: rescaled disp} with respect to the quadruples $(R^\eps,{\mathcal{P}}^{\eps,1}, {\mathcal{M}}^{\eps,1}, {\mathcal{T}}^{\eps,1})$ and  $(R^\eps,\hat{\mathcal{P}}^{\eps}, \hat{\mathcal{M}}^{\eps}, \hat{\mathcal{T}}^{\eps})$, respectively.  By \eqref{eq: uniform-eps}  we obtain   $\eps^{-1}(\hat{t}_j^{\ep}-\hat{t}_i^{\ep}) \to   t_0 \in \R^d$, possibly passing to a subsequence (not relabeled). This implies that $u^{\eps,1}$ converges in measure  to
 \begin{align}\label{eq: u1def}
u^1 : = \hat{u} + t_0 \chi_{\hat{P}_j} \in \mathscr{U}(\Omega)
\end{align}
 and gives  \eqref{eq:comp-u}. Finally, \eqref{eq:comp-grad-u} follows from $\nabla {u}^{\eps,1} = \nabla \hat{u}^{\eps}$ and $\nabla {u}^{1} = \nabla \hat{u}$.

Now suppose that the quadruples $(R^{\eps}, \mathcal{P}^{\eps,n-1}, \mathcal{M}^{\eps,n-1}, \mathcal{T}^{\eps,n-1})$ and the limiting triple $(y,u^{n-1},\mathcal{P}^{n-1})$  in step $n-1$ have been constructed such that \eqref{eq: partition property}--\eqref{eq:comp-M}, \eqref{eq: rescaled disp}--\eqref{eq:comp-grad-u}, and \eqref{eq: toinfty-onK} hold, and \eqref{eq: toinfty-onK2} is satisfied up to step $n-1$.  We define  the objects in step $n$ as follows: if \eqref{eq: toinfty-onK} holds with respect to the set $\Omega_n$, we simply set $(\mathcal{P}^{\eps,n}, \mathcal{M}^{\eps,n}, \mathcal{T}^{\eps,n})=(\mathcal{P}^{\eps,n-1}, \mathcal{M}^{\eps,n-1}, \mathcal{T}^{\eps,n-1})$, and observe that all properties are automatically satisfied.

If \eqref{eq: toinfty-onK} is violated, the strategy is to apply the coarsening scheme described above to modify the partitions and translations such that all properties, in particular \eqref{eq: toinfty-onK}--\eqref{eq: toinfty-onK2}, are fulfilled.

\emph{Coarsening scheme for general $n$.}  If two components ${P}^{n-1}_i$ and ${P}^{n-1}_j$ violate  \eqref{eq: toinfty-onK} (with respect to the set $\Omega_n$), we combine them to one component $P^n_* := {P}^{n-1}_i\cup {P}^{n-1}_j$ and similarly we define $P^{\eps,n}_* := {P}^{\eps,n-1}_i\cup {P}^{\eps,n-1}_j$ for all $\eps>0$. Moreover, we define the phase $M^{\eps,n}_* = M^{\eps,n-1}_i = M^{\eps,n-1}_j$ for all $\eps>0$. Concerning the translation $t_*^{\eps,n}$, we proceed as follows: we observe that at most one of the two sets ${P}^{n-1}_i$ and ${P}^{n-1}_j$ intersects $\Omega_{n-1}$. Indeed, it is not possible that both sets intersect $\Omega_{n-1}$ as \eqref{eq: toinfty-onK} holds by construction in step $n-1$, and we assumed that ${P}^{n-1}_i$ and ${P}^{n-1}_j$ violate  \eqref{eq: toinfty-onK} with respect to $\Omega_{n} \supset \Omega_{n-1}$. Suppose that (at most) ${P}^{n-1}_i$ intersects $\Omega_{n-1}$. We define $t_*^{\eps,n} := t_{i}^{\eps,n-1}$. We repeat this procedure (at most a finite number of times, cf.\  Lemma \ref{lemma: localization}) until all pairs of components satisfy \eqref{eq: toinfty-onK}.

Then, for the resulting quadruple, \eqref{eq: toinfty-onK} is satisfied by construction. Exactly as before in the step $n=1$, we can check that \eqref{eq: partition property}-\eqref{eq:comp-M} and \eqref{eq: rescaled disp}-\eqref{eq:comp-grad-u} hold. Finally, let us confirm \eqref{eq: toinfty-onK2}: (i)  follows  from the fact that in the procedure we iteratively have combined two components. Similarly, (ii) is a consequence of the fact that only sets with the same phase are combined. Finally,  (iii) and (iv) follow from the definition of the translations in the coarsening scheme and the fact that, if two components are combined, at least one did not intersect $\Omega_{n-1}$. 

We perform this coarsening scheme for each $n \in \N$. Note that in each step we pass to a further subsequence in $\eps$ (not relabeled). Then, \eqref{eq: toinfty-onK}--\eqref{eq: toinfty-onK2} follow for each $n \in \N$ for a suitable diagonal sequence. 
\end{proof}

 \begin{remark}[Local properties of jump sets]\label{remark: localization}
For later purposes, we remark that each $K \subset \subset \Omega$  intersects only a finite number of $(d-1)$-dimensional hyperplanes orthogonal to $e_d$ which intersect $J_u$. This can be seen as follows: the construction of the displacement fields in the previous proof  shows  that $J_{u^n} \subset \bigcup_j \hat{P}_j$ for all $n\in \N$. This follows from \eqref{eq: u1def} and the fact that $J_{\hat{u}} \subset \bigcup_j \hat{P}_j$, see  Proposition  \ref{lemma: intermediate step3} for $\hat{u}$ and $\hat{P}_j$ in place of $u$ and $P_j$, respectively. Therefore, also   $J_{u} \subset \bigcup_j \hat{P}_j$ by \eqref{eq: the same!}. The desired property now follows from Lemma \ref{lemma: localization}.
\end{remark}

\EEE

We close this section by mentioning that the definition and construction of the partition in the previous proof is inspired by  \cite[Section 5]{Friedrich-ARMA} where in a different context partitions with a property of type \eqref{eq: toinfty} are called  \emph{coarsest partitions}.

\section{Analysis  of admissible limiting configurations}
\label{sec:limiting-triple}

 This section is devoted to the proofs of Proposition \ref{prop:ex-coarsest-part}, Proposition \ref{lemma: admissible-u-y-jump}, and Proposition \ref{lemma: admissible-u}.    We first show that limiting deformations and partitions are uniquely identified whereas limiting displacements may differ by global infinitesimal rotations and piecewise translations.

 \begin{proof}[Proof of  Proposition \EEE \ref{prop:ex-coarsest-part}]
  Let $\lbrace y^\eps \rbrace_\eps$ be a   sequence as in Theorem \ref{thm:compactness} and let $(y^1,u^1,\mathcal{P}^1)$,  $(y^2,u^2,\mathcal{P}^2)$ be two admissible triples. We start with the proof of (i). First,  $y^1 =y^2$ \EEE  follows directly from \eqref{eq:comp-y}. In what follows, we thus simply denote the deformation by $y$. Suppose by contradiction that the two partitions $\mathcal{P}^1 = \lbrace P^1_j\rbrace_j$ and $\mathcal{P}^2 = \lbrace P_j^2 \rbrace_j$ are different. Up to reordering we may assume that $P^1_1 \cap {P}^2_1$ and $P^1_2 \cap P^2_1$ have positive $\mathcal{L}^d$-measure.

  Let $({R}^{\eps,1}, {\mathcal{P}}^{\eps,1}, {\mathcal{M}}^{\eps,1}, {\mathcal{T}}^{\eps,1})$ and 
 $( {R}^{\eps,2},  {\mathcal{P}}^{\eps,2}, {\mathcal{M}}^{\eps,2}, {\mathcal{T}}^{\eps,2})$ be sequences of quadruples converging to the limiting triples $(y,u^1,\mathcal{P}^1)$ and $(y,u^2,{\mathcal{P}}^2)$, respectively, in the sense of \eqref{eq: partition property}--\eqref{eq:comp-grad-u}. By \eqref{eq:comp-R} we have  $\lim_{\eps \to 0} {R}^{\eps,1} = \lim_{\eps \to 0} {R}^{\eps,2}  = R \in SO(d)$, where  $R$ is such that  $y \in \mathcal{Y}_R(\Omega)$. By  \eqref{eq:comp-sets}, \eqref{eq:comp-M}, and the fact that $P^1_1 \cap {P}^2_1$ and $P^1_2 \cap P^2_1$ have positive $\mathcal{L}^d$-measure,  we then obtain for all $\eps$ small enough
  \begin{align}\label{eq: same-phases}
  M^{\eps,1}_1 = M^{\eps,1}_2 = M^{\eps,2}_1. 
  \end{align}
  Since the rescaled displacement fields $u^{\eps,1}$ and ${u}^{\eps,2}$,  defined in \eqref{eq: rescaled disp} with respect to the two different quadruples, converge in measure  in $\Omega$ by \eqref{eq:comp-u}, we observe that also
 $$\frac{1}{\eps}\Big(\sum\nolimits_j (R^{\eps,1} M^{\eps,1}_j\,x + t^{\eps,1}_j) \chi_{P^{\eps,1}_j} - \sum\nolimits_j ({R}^{\eps,2}{M}^{\eps,2}_j\,x + {t}^{\eps,2}_j) \chi_{{P}^{\eps,2}_j}\Big)$$  
 converges  in measure in $\Omega$.   In view of \eqref{eq:comp-sets}, \eqref{eq: same-phases}, and the fact that $P^1_1 \cap {P}^2_1$ and $P^1_2 \cap {P}^2_1$ have positive $\mathcal{L}^d$-measure, we obtain
\begin{align}\label{eq: rot-and-t-diff}
 |R^{\eps,1}-R^{\eps,2}| +  |t_1^{\eps,1} - {t}_1^{\eps,2}| + |t_2^{\eps,1} - {t}_1^{\eps,2}| \le C\eps
\end{align}
uniformly in $\eps$ for some $C>0$.   This is an elementary property for affine mappings. (See, e.g., \cite[Lemma 3.4]{FriedrichSolombrino}; the function $\psi$ therein can be chosen as in  \cite[Remark 2.2]{FriedrichSolombrino2}.) By the triangle inequality this particularly yields $|t_1^{\eps,1} - t_2^{\eps,1}| \le C\eps$. This, however, contradicts \eqref{eq: toinfty} in view of  \eqref{eq: same-phases}.  This concludes the proof of (i).

In the following, we denote the unique partition by $\mathcal{P} = \lbrace P_j \rbrace_j$ to simplify notation. We now show (ii). To this end, fix $P_j$ with positive measure. In view of \eqref{eq:comp-sets} and \eqref{eq:comp-M}, we find  $M_j^{\eps,1} = M_j^{\eps,2}$   for $\eps$ small enough. As  $u^{\eps,1} - {u}^{\eps,2}$ converges in measure in $\Omega$ by \eqref{eq:comp-u}, we thus obtain $|{R}^{\eps,1}  - R^{\eps,2} | \le C \eps$ and $ |t^{\eps,1}_j - t^{\eps,2}_j| \le C_j\eps$  for a constant $C>0$ depending only on $\Omega$, and some $C_j>0$  depending on $j$ but not on $\eps$,  see \eqref{eq: rot-and-t-diff} for a similar argument. Using the  formula
(see \cite[(3.20)]{FrieseckeJamesMueller:02})
\begin{align}\label{eq: linearization formula}
\Big| \frac{(F R^T)^T + FR^T}{2} - {\rm Id}\Big| = \dist(F,SO(d)) + {\rm O}(|F-R|^2) \ \ \  \ \text{for $F \in \M^{d\times d}$, $R \in SO(d)$,}
\end{align}
 we obtain $S^\eps \in \M^{d \times d}_{\rm skew}$ with $|S^\eps| \le C$ such that
 $${R}^{\eps,2}  - R^{\eps,1}   =  ({R}^{\eps,2} \,  (R^{\eps,1})^T  - {\rm Id}) \,  R^{\eps,1} =   (\eps S^\eps + {\rm O}(\eps^2)) \, R^{\eps,1}.$$
  Thus, possibly passing to a  subsequence (not relabeled),  we find $S \in \M^{d \times d}_{\rm skew}$ and  for each $j \in \N$ with $\mathcal{L}^d(P_j) >0$ a constant  $t_j \in \R^d$ such that $\eps^{-1}(t^{\eps,2}_j - {t}^{\eps,1}_j) \to t_j$  and  $\eps^{-1}({R}^{\eps,2}  - R^{\eps,1}) \to S R$, where $R \in SO(d)$ is such that $y \in \mathcal{Y}_R(\Omega)$. In particular, note that $S$ is independent of the component $P_j$.  By \eqref{eq:comp-M},   \eqref{eq: rescaled disp}--\eqref{eq:comp-u}, and the fact that $M_j^{\eps,1} = M_j^{\eps,2}$  for $\eps$ small enough \EEE we get for almost every $x \in P_j$ 
 $$  u^1(x) - {u}^2(x) \EEE = \lim_{\eps \to 0} \,  \big(u^{\eps,1}(x) - {u}^{\eps,2}(x)\big) = \lim_{\eps \to 0}  \frac{1}{\eps} \Big( (R^{\eps,2} - {R}^{\eps,1}) M_j^{\eps,1} \, x +  t_j^{\ep,2}-{t}_j^{\ep,1} \Big) =   S \, \nabla y(x)\,  x + t_j.$$
Recalling the definition in \eqref{eq: infini rigid}  we obtain (ii). 

We finally show (iii). To this end, fix $\tilde{T}\in  \mathscr{T}(y,\mathcal{P})$, say $\tilde{T}(x) = \sum_j \tilde{t}_j \chi_{P_j}(x)  + \tilde{S} \, \nabla y(x) \, x $ for $x \in \Omega$. We have to show that $(y,u^1 + \tilde{T},\mathcal{P} )$  is an admissible triple. Recall that the quadruples $({R}^{\eps,1}, {\mathcal{P}}^{\eps,1}, {\mathcal{M}}^{\eps,1}, {\mathcal{T}}^{\eps,1})$ converge to $(y,u^1,\mathcal{P} )$ in the sense of \eqref{eq: partition property}-\eqref{eq:comp-grad-u}. 

   We let $\bar{\mathcal{P}}^\eps = \mathcal{P}^{\eps,1}$,  $\bar{\mathcal{M}}^\eps = \mathcal{M}^{\eps,1}$  and define
$\bar{\mathcal{T}}^\eps = \lbrace \bar{t}_j^\eps \rbrace_j$ by $\bar{t}_j^\eps = {t}_j^{\eps,1}  -  \eps  \tilde{t}_j$ for all indices $j$. Moreover, we let $\bar{R}^\eps \in SO(d)$  be  such that $|\bar{R}^\eps - ({\rm Id} -  \eps \tilde{S}){R}^{\eps,1}| = \dist(({\rm Id}  -  \eps \tilde{S}){R}^{\eps,1},SO(d))$, which by \eqref{eq: linearization formula}  (for $F=({\rm Id} - \eps \tilde{S}){R}^{\eps,1}$ and $R={R}^{\eps,1}$) implies
\begin{align}\label{eq: rigid motions}
\bar{R}^\eps   =  ({\rm Id} - \eps \tilde{S})   \, {R}^{\eps,1} + {\rm O}(\eps^2).
\end{align}
We now see that $(\bar{R}^\eps, \bar{\mathcal{P}}^\eps, \bar{\mathcal{M}}^\eps, \bar{\mathcal{T}}^\ep)$ converges to $(y, {u}^1 + \tilde{T},\mathcal{P})$ in the sense of \eqref{eq: partition property}-\eqref{eq:comp-grad-u}.   Indeed, as  $|\bar{R}^\eps - {R}^{\eps,1}| \le C\eps$,  the  properties \eqref{eq: partition property}-\eqref{eq:comp-M} are  satisfied.   Property \eqref{eq: toinfty} follows from the corresponding property for $\mathcal{T}^{\eps,1}$ and the definition of $\bar{\mathcal{T}}^\eps$.   Define $\bar{u}^\eps$ as in \eqref{eq: rescaled disp}. To confirm \eqref{eq:comp-u}, we  calculate for almost every $x \in P_j$ using \eqref{eq:comp-M}   and \eqref{eq: rigid motions}
\begin{align*}
 \lim_{\ep\to 0}  \big( \bar{u}^\eps(x)-{u}^{\eps,1}(x)\big) &=  \lim_{\ep \to 0}\frac{1}{\eps} \big( ( {R}^{\eps,1} - \bar{R}^\eps ) \,  M_j^{\eps,1} \, x +   {t}_j^{\ep,1} - \bar{t}_j^{\ep}\big)  = \lim_{\ep\to 0}  \frac{1}{\eps} ({R}^{\eps,1} - \bar{R}^\eps) \,  M_j^{\eps,1} \, x   + \tilde{t}_j \\ &  =   \tilde{S} \, \nabla y(x) \, x +  \tilde{t}_j. 
\end{align*}
Using \eqref{eq:comp-u} for $u^1$, we find  $\bar{u}^\eps \to u^1 + \tilde{T}$ in measure on the bounded set $\Omega$. This yields   \eqref{eq:comp-u}. Finally, \eqref{eq:comp-grad-u} follows from a similar computation. 
 \end{proof}

 We proceed by characterizing the jump set of the gradients of limiting deformations.

\begin{proof}[Proof of Proposition \ref{lemma: admissible-u-y-jump}]
 As $y \in \mathcal{Y}_R(\Omega)$, we recall that $\partial \{x\in \Omega\colon\,\nabla y(x)\in RA\}$ consists of subsets of hyperplanes orthogonal to $e_d$, see below Lemma \ref{lemma:comp-def}. Now, assume by contradiction that $J_{\nabla y} \not\subset \bigcup\nolimits_{j} \partial P_j \cap \Omega$. Then, by   $\mathcal{P} \in\mathscr{P}(\Omega)$ and Lemma  \ref{lemma: localization}, we find a stripe $D := \lbrace t_0-\rho < x_d < t_0+\rho \rbrace \cap \Omega'$, with $\Omega' \subset \subset \Omega$, $t_0 \in \R$, and $\rho>0$ small, such that $D \subset P_j$ for some $j \in \N$ and (up to reflection) $D \cap \lbrace x_d > t_0 \rbrace \subset \lbrace \nabla  y = RA \rbrace$,  $D \cap \lbrace x_d < t_0 \rbrace \subset \lbrace \nabla y = RB \rbrace$. In view of \eqref{eq:comp-R}--\eqref{eq:comp-sets}, however, this contradicts \eqref{eq:comp-M}.  
To see that the inclusion might be strict, we refer to Case (2) in Example \ref{ex} with $l=1/2$.
 \end{proof}

We conclude this section with a characterization of the jump heights of limiting displacements. 
 
\begin{proof}[Proof of Proposition \ref{lemma: admissible-u}]
We first observe that it suffices to show  that, if $\Omega' \subset \subset \Omega$, then  the result  holds  for  every  $x \in \Omega'$. Consider a (subset of a) hyperplane $S := \lbrace x_d = t_0\rbrace \cap \Omega'$ \EEE with $\mathcal{H}^{d-1}\big(S \cap  J_u   \big)>0$. We distinguish two situations: 
$$\text{(a)}\quad \mathcal{H}^{d-1}\Big(S \cap  \bigcup\nolimits_j\partial P_j  \Big)= 0\quad\text{and}\quad\text{(b)}\quad \mathcal{H}^{d-1}\Big(S \cap  \bigcup\nolimits_j\partial P_j  \Big)> 0.$$
 To simplify notation, we set without restriction $t_0 = 0$. We start with Case (a). Choose another set $\Omega''$ with $\Omega' \subset \subset \Omega'' \subset \subset \Omega$. As $\mathcal{P} \in \mathscr{P}(\Omega)$, by   Lemma  \ref{lemma: localization} and  Remark \ref{remark: localization} we find $\rho>0$ small enough such that the cylindrical set    $D:= \omega \times (- \rho,\rho)$, $\omega \subset \R^{d-1}$, satisfies $D \cap \lbrace x_d =0 \rbrace = S$, is  contained in a single component $P_j$,  is contained in $\Omega''$, and  satisfies 
  \begin{equation}
 \label{eq:jump-set-u-one}
 J_u\cap D\subset  S = \lbrace x_d = 0\rbrace \cap \Omega'. 
 \end{equation} 
By Proposition \ref{lemma: admissible-u-y-jump},  it is not restrictive to concentrate on the case $\nabla y = RA$ on $D  \subset P_j \EEE$, which corresponds to proving properties (i) and (ii) of the statement. Analogously, property (iii) may be derived after some modifications in the notation. 
 
   \emph{Step 1: Case (a), property (ii).} Let $({R}^{\eps}, {\mathcal{P}}^{\eps}, {\mathcal{M}}^{\eps}, {\mathcal{T}}^{\eps})$  be sequences of quadruples converging to  $(y,u,\mathcal{P})$ in the sense of \eqref{eq: partition property}--\eqref{eq:comp-grad-u}, and define $u^\eps$ as in \eqref{eq: rescaled disp}.  Assume also that $\mathcal{J}^\ep$ is the (at most countable) set of indices for the partition $\mathcal{P}^\ep$.  We denote  by $\mathcal{J}_1^\eps$  the indices with $\mathcal{L}^d(\Omega'' \cap P^\eps_j) \le   \mathcal{L}^d( L_{\Omega''}(P^\eps_j) \setminus P^\eps_j  )$, and  we  let $\mathcal{J}^\eps_2 =  \mathcal{J}^\ep \EEE \setminus \mathcal{J}^\eps_1$.  By \eqref{eq: partition property-new}, \eqref{eq:comp-sets},  \eqref{eq:comp-u},  \eqref{eq:comp-grad-u}, Fubini's theorem,  and Fatou's lemma we get that for  $\mathcal{H}^{d-1}$-a.e.\  $x'\in  \omega$ there exists a sequence $\lbrace \eps_k \rbrace_k \subset (0,+\infty)$ with $\eps_k \to 0$ such that for a.e.\ $0<\rho'<\rho$ we have   
\begin{align}\label{eq: slicing properties3}
{\rm (i)}& \ \ (x',-\rho'), (x',\rho')   \in P^{\eps_k}_{j} \  \text{for all $k$  large enough}, \ \ \ u^{\eps_k}(x',\pm \rho') \to u(x',\pm \rho') \ \text{as $k \to \infty$},   \notag \\ 
{\rm (ii)} & \ \  \sum_{j \in \mathcal{J}^\eps_1}  \mathcal{L}^1\big( P_j^{\eps_k} \cap (\lbrace x' \rbrace \times (-\rho',\rho')) \big) + \sum_{j \in \mathcal{J}^\eps_2}   \mathcal{L}^1\Big( \big(L_{\Omega''}(P^{\eps_k}_j) \setminus P^{\eps_k}_j\big) \cap \big(\lbrace x' \rbrace \times (-\rho',\rho')\big) \Big) \le \bar{C}(x') \,  \eps^p_k, \notag\\   
{\rm (iii)} & \ \ \int_{-\rho'}^{\rho'} |\nabla u^{ \eps_k}(x',t)|^2\, {\rm d}t \le  \bar{C}(x'),  
\end{align}
where $\bar{C}(x')>0$ depends on $\Omega''$ and  $x'$, but is independent of $\rho'$  and $\lbrace \eps_k\rbrace_k$. We point out that in general the sequence $\lbrace \eps_k\rbrace_k$ depends on $x'$. For later purposes, however, we note that, for a.e.\ pair of points  $x_1',x_2'\in \omega$, we can choose a single sequence  $\lbrace \eps_k\rbrace_k$ such that \eqref{eq: slicing properties3} holds.

Fix $x' \in \omega$ and $0<\rho'<\rho$ such that \eqref{eq: slicing properties3} is satisfied. For notational simplicity, we drop the subscript $k$ of the corresponding sequence $\lbrace \eps_k \rbrace_k$ and we omit the dependence on $x'$. Define    
\begin{align}\label{eq: E slice def}
  \mathcal{B}^\eps (x';\rho') := \Big\{ t \in  (-\rho',\rho'):   \, \sum\nolimits_j    M^\eps_j \chi_{P^\eps_j}(x',t)  = B \Big\}.
\end{align}
By the fundamental theorem of calculus,  in view of the definition of $u^\eps$ in \eqref{eq: rescaled disp}, we get    
\begin{align*}%\label{eq: slicing-prop}
y^\eps(x',\rho') - y^\eps(x',-\rho') &= \int_{-\rho'}^{\rho'} \partial_d y^\eps(x',t)\, {\rm d}t  \notag\\ 
& =  \eps \int_{-\rho'}^{\rho'} \partial_d u^\eps(x',t)\, {\rm d}t  +   \mathcal{L}^1( \mathcal{B}^\eps  (x';\rho')) \, R^\eps B \, e_d+ \big( 2\rho' - \mathcal{L}^1( \mathcal{B}^\eps  (x';\rho')) \big) R^\eps A \,  e_d. 
\end{align*} 
Thus, by  \eqref{eq: slicing properties3}(iii) and H\"older's inequality we find 
\begin{align}\label{eq: liminf2}
\eps^{-1} \big|y^\eps(x',\rho') - y^\eps(x',-\rho') - 2\rho' R^\eps A \, e_d -   \mathcal{L}^1( \mathcal{B}^\eps (x';\rho')) \, R^\eps(B-A) \, e_d    \big|  \le  (2\bar{C}(x')\rho')^{1/2}.
\end{align}
 Since  $\nabla y = RA$ on $D \subset P_j $, we get $M_j^\eps =A$ for $\eps$ sufficiently small by \eqref{eq:comp-M}. Thus, by \eqref{eq: rescaled disp} and    \eqref{eq: slicing properties3}(i), we also have 
\begin{align*}
&\eps^{-1}\big(y^\eps(x',\rho') - y^\eps(x',-\rho') -  2\rho' R^\eps A \,  e_d\big)   =  u^\eps(x',\rho') - u^\eps(x',-\rho')
\end{align*}
 for every $\eps$ sufficiently small. Recall the definition of $\kappa$ in H3.  By \eqref{eq:comp-R}, \eqref{eq: slicing properties3}(i), and  \eqref{eq: liminf2}, up to passing to a further subsequence (depending on $\rho'$), we get that $ \ell(x';\rho') := \lim_{\eps \to 0} \eps^{-1} \mathcal{L}^1( \mathcal{B}^\eps (x';\rho') ) \ge 0$ exists, is finite, and satisfies  
\begin{align}\label{eq: get contradiction-prop}
|u(x',\rho') - u(x',-\rho') -   \kappa  \,  \ell(x';\rho') \,  Re_d| \le (2\bar{C}(x')\rho')^{1/2}.
\end{align}
 Here, we used that $\bar{C}(x')$ is independent of $\eps$. On the other hand,  the fundamental theorem of calculus  for the limiting displacement  together with \eqref{eq:jump-set-u-one} yields
\begin{align}\label{eq: contribution elastic part-prop}
|u(x',\rho') - u(x',-\rho') - [u](x',0) | \le  \int_{-\rho'}^{\rho'} |\partial_d u(x',t)|\,{\rm d}t \le (2\bar{C}(x')\rho')^{1/2},
\end{align}
where the last inequality follows by \eqref{eq: slicing properties3}(iii),  H\"older's inequality,  and a lower semicontinuity argument. By combining \eqref{eq: get contradiction-prop} and \eqref{eq: contribution elastic part-prop} we deduce
\begin{align}\label{eq: jump height estimate}
\big|[u](x',0) -   \kappa \,  \ell(x';\rho') \, Re_d     \big| \le 2  (2\bar{C}(x')\rho')^{1/2}. 
\end{align}
Property (ii) in Case (a) now  follows by recalling  that $ \ell(x';\rho') \ge 0$, \EEE by  the fact that $\bar{C}(x')$ may depend on $x'$ but is independent of $\rho'$,  and by considering a sequence $\rho' \to 0$  such that \eqref{eq: slicing properties3} holds.  (We briefly note that property (iii) corresponds to $\nabla y = RB$ on $D  \subset P_j$. This case can be treated  along similar lines, by interchanging the roles of $A$ and $B$.)

 \emph{Step 2: Case (a), property (i).} \EEE  We now show property (i) by contradiction, where without restriction we treat the  case $\nabla y = RA$ on $D  \subset P_j$.  If the statement  were wrong, we would find $x'_1,x'_2 \in \omega$  and $0<\rho' <\rho$  such that for each $x'_i$, $i=1,2$, \eqref{eq: slicing properties3} holds (with $x'_i$ in place of $x'$, for a single sequence $\lbrace \eps_k \rbrace_k$) and such that 
\begin{align}\label{eq: jump difference}
\big|[u](x'_1,0) - [u](x'_2,0) | \ge 5 (2\bar{C}\rho')^{1/2},
\end{align}
 where we set $\bar{C} = \max_{i=1,2} \bar{C}(x_i')$. We again drop the index $k$ of the sequence $\lbrace \eps_k\rbrace_k$. Define $ \mathcal{B}^\eps (x_i';\rho')$  as in \eqref{eq: E slice def}  for $i=1,2$. Repeating the reasoning in Step 1, see particularly \eqref{eq: jump height estimate}, we find $|[u](x'_i,0) -  \kappa  \,  \ell(x'_i;\rho')\,  Re_d     | \le 2 (2\bar{C}\rho')^{1/2}$ for $i=1,2$, where  the  limits    $\ell(x'_i;\rho') := \lim_{\eps \to 0} \eps^{-1} \mathcal{L}^1( \mathcal{B}^\eps (x'_i;\rho') ) $  can again be assumed to exist after passage   to a subsequence (not relabeled). By the triangle inequality and \eqref{eq: jump difference}, we find $ \kappa |\ell(x'_1;\rho') - \ell(x'_2;\rho')| \ge  (2\bar{C}\rho')^{1/2}$. This implies
$${ \inf_{\eps>0} \eps^{-1} \big| \mathcal{L}^1( \mathcal{B}^\eps (x'_1;\rho') ) -  \mathcal{L}^1( \mathcal{B}^\eps  (x'_2;\rho') ) \big| >0}.$$
In view of the definition \eqref{eq: E slice def}, this contradicts \eqref{eq: slicing properties3}(ii)  since $p>1$. This concludes the proof of (i) and of Case (a).

 \emph{Step 3: Case (b), property (i).}  To complete the proof of the proposition, it remains to show  assertion (i) in Case (b). (Note that assertions (ii) and (iii) are trivial in this case.) In this situation, possibly passing to a smaller $\rho$, by Lemma \ref{lemma: localization} we get that the set  $D = \omega \times (-\rho,\rho)$ considered in Case (a), see before \eqref{eq:jump-set-u-one}, only intersects two components $P_{j_1}$ and $P_{j_2}$,  with $D \cap P_{j_1} =   D \cap \lbrace x_d < 0 \rbrace$ and $D \cap P_{j_2} =   D \cap \lbrace x_d > 0 \rbrace$.  In a similar fashion to \eqref{eq: slicing properties3}, in view of  \eqref{eq: partition property-new}, \eqref{eq:comp-sets},  \eqref{eq:comp-u}, and  \eqref{eq:comp-grad-u}, Fatou's lemma yields that for $\mathcal{H}^{d-1}$-a.e.\  $x'\in  \omega$ there exists an infinitesimal  sequence $\lbrace \eps_k\rbrace_k$ such that for a.e.\ $0<\rho'<\rho$ there holds 
\begin{align}\label{eq: slicing properties-prop}
(x',-\rho') \in P^{\eps_k}_{j_1}, \  \  (x',\rho') \in P^{\eps_k}_{j_2} \  \text{ for all $k$ large enough}, \ \ \ \  u^{\eps_k}(x',\pm \rho') \to u(x',\pm \rho') \ \text{as $k \to \infty$}, 
\end{align}
and properties (ii) and (iii) of \eqref{eq: slicing properties3} are satisfied. Given $x' \in \omega$ and $0<\rho'<\rho$,  arguing exactly as in the proof of \eqref{eq: liminf2} in Case (a), we find (we again drop the index $k$ and the dependence on $x'$  in the sequel)
\begin{align*}
\eps^{-1} \big|y^\eps(x',\rho') - y^\eps(x',-\rho') - 2\rho' R^\eps A \, e_d -   \mathcal{L}^1( \mathcal{B}^\eps (x';\rho')) \,   R^\eps(B-A) \, e_d    \big|  \le  (2\bar{C}(x')\rho')^{1/2},
\end{align*}
 where  $ \mathcal{B}^\eps  (x';\rho')$  is defined  in \eqref{eq: E slice def}.  By \eqref{eq:comp-M}, for $\eps$ sufficiently small, we may assume that $M^\eps_{j} = M_j$ for $j=j_1,j_2$.  Thus,  in view of \eqref{eq: rescaled disp} and \eqref{eq: slicing properties-prop}, we get
\begin{align*}
\eps^{-1} \big( y^\eps(x',\rho') - y^\eps(x',-\rho') -  \rho' \, R^\eps \,  (M_{j_1}+ M_{j_2})   \,  e_d \big) -  \ep^{-1}(t^{\ep}_{j_2}-t^{\ep}_{j_1}) = u^\eps(x',\rho') - u^\eps(x',-\rho').
\end{align*}
 This along with the previous estimate entails
\begin{align}\label{eq:entail}
\big|u^\eps(x',\rho') - u^\eps(x',-\rho')  - v_\eps(x';\rho')   \big|  \le  (2\bar{C}(x')\rho')^{1/2},
\end{align}
where for brevity we have set
\begin{align}\label{eq: veps def}
v_\eps(x';\rho') :=  \eps^{-1}  \mathcal{L}^1(\mathcal{B}^\eps  (x';\rho')) \,  R^\eps(B-A) \, e_d  +  \eps^{-1}  \rho' R^\eps \,  (2A - (M_{j_1} + M_{j_2}))  \, e_d      -  \ep^{-1}(t^{\ep}_{j_2}-t^{\ep}_{j_1}).  
\end{align}
Then  \eqref{eq: slicing properties-prop} and \eqref{eq:entail}  show that  there exists \EEE  a constant vector  $v(x';\rho') \in \R^d$  depending on $\rho'$ and $x'$ such that, up to the extraction of a subsequence (not relabeled), there holds $v_\eps(x';\rho') \to v(x';\rho')$. By using \eqref{eq:jump-set-u-one} and \eqref{eq: slicing properties3}(iii), we get that \eqref{eq: contribution elastic part-prop} also holds in the present situation. Then, similar to the proof of \eqref{eq: jump height estimate} in Case (a), we obtain by \eqref{eq: slicing properties-prop} and \eqref{eq:entail} 
\begin{align}\label{eq: jump height estimate-prop}
|[u](x',0) -   v(x';\rho')| \le 2(2\bar{C}(x')\rho')^{1/2}. 
\end{align}
The proof of property (i)  is now obtained by contradiction by following the lines of the proof in Case (a): suppose that  there were $x'_1,x'_2 \in \omega$  and $0<\rho' <\rho$  such that for each $x'_i$, $i=1,2$,  \eqref{eq: slicing properties-prop} and \eqref{eq: slicing properties3} (ii),(iii)  hold (with $x'_i$ in place of $x')$, and   the two points are   such that $\big|[u](x'_1,0) - [u](x'_2,0) | \ge 5 (2\bar{C}\rho')^{1/2}$, where as before $\bar{C} := \max_{i=1,2} \bar{C}(x_i')$. By \eqref{eq: jump height estimate-prop} this yields $\big|v(x'_1;\rho') - v(x'_2;\rho')| \ge  (2\bar{C}\rho')^{1/2}$. In  view of \eqref{eq: veps def}, this however contradicts  \eqref{eq: slicing properties3}(ii).  This concludes the proof.   
\end{proof}

\section{Derivation of the effective linearized energy}
\label{sec:gamma}
This section is devoted to the proof of our $\Gamma$-convergence result for the sequence of energies $\mathcal{E}_\eps = E_{\eps,\bar{\eta}_{\ep,d}}$ introduced in \eqref{eq: nonlinear energy} (with $\bar{\eta}_{\ep,d}$ from \eqref{eq:alphad}) and the limiting energy $\mathcal{E}_0^{\mathcal{A}}$ defined in \eqref{eq: limiting energy}. In Subsections \ref{subs:liminf} and \ref{subs:limsup} we prove Theorems \ref{thm:liminf} and \ref{thm:limsup-new}, respectively. A key ingredient for the liminf inequality is a characterization of the double-profile energy $K^{M}_{\rm dp}$ (see \eqref{eq: our-k2}),  in particular its connection to the optimal-profile counterpart $K$ (see  \eqref{eq: our-k1}).  This result is subject of Proposition \ref{eq: KundKdo-new} and is proven in Subsection \ref{sec: cell-formula}. 
The proof of the limsup inequality is performed under the additional assumption that
\begin{equation}
\label{eq:2-cell-eq}
K^{M}_{\rm dp}=2K \quad\text{ for $M\in \{A,B\}$},  
\end{equation}
and essentially relies on Propositions  \ref{lemma: local1}  and  \ref{lemma: local2}. The latter provide constructions of local recovery sequences around interfaces performing a single and a double phase transition, respectively, and coinciding with isometries far from the interfaces. Their proofs are contained in Subsection \ref{subs:local}.
 Finally, in Subsection \ref{subs:1d} we show that, under the additional assumption in \eqref{eq: isotropy}, condition \eqref{eq:2-cell-eq} can be verified.  This  hinges on the property that in this case optimal profiles for single phase transitions are one dimensional, see Lemma \ref{lemma: 1d}.

 \subsection{The liminf inequality} 
 \label{subs:liminf}

  In this subsection we  show that the functional $\mathcal{E}^{\mathcal{A}}_0$ is a lower bound for the asymptotic behavior of the energy functionals $\mathcal{E}_{\ep}$.  As a preparation, we introduce the notion of optimal-profile and double-profile energy functions, and we state their main properties. 

Consider $\omega \subset \R^{d-1}$ open and  bounded, and let  $h >0$. For brevity, we use the following notation  for cylindrical sets
\begin{equation}
\label{eq:def-dlh}
D_{\omega,h} := \omega\times (-h,h).
\end{equation} 
We define the \emph{optimal-profile energy   function}  
\begin{align}\label{eq: k-intro}
\mathcal{F}(\omega;h) = \inf\Big\{\liminf_{\ep\to 0} \mathcal{E}_{\ep}(y^{\ep},D_{\omega,h}): \   \lim_{\eps \to 0}  \Vert  y^\eps -  y_0^+ \Vert_{L^1(D_{\omega,h})}  = 0\Big\}
\end{align}
for every $\omega\subset \mathbb{R}^{d-1}$ and $h>0$, where $y_0^+$ was defined below \eqref{eq: conti-schweizer-k}. As mentioned there, due to the invariance of the energy functionals $\mathcal{E}_{\ep}$  under the operation $Ty(x) = -y(-x)$, the  optimal-profile energy   is independent of the direction in which the transition between the two phases $A$ and $B$ occurs, i.e., in \eqref{eq: k-intro} we can replace $y_0^+$ by the continuous function $y_0^-  \in H^1_{\rm loc}(\R^d;\R^d)$  with $y_0^-(0)=0$ and $\nabla y_0^-=B\chi_{\{x_d> 0\}}+A\chi_{\{x_d<0\}}$.  
We refer to \cite[Lemma 3.2]{conti.schweizer2} for details. We start with the property that the optimal-profile energy is independent of $h$ and depends on $\omega$ only in terms of $\mathcal{H}^{d-1}(\omega)$. The following characterization has been proved in \cite[Proposition 4.6]{davoli.friedrich}. 

\begin{proposition}[Optimal-profile energy function]
\label{prop:cell-form}
For all $h>0$ and all  open, bounded   sets  $\omega \subset \R^{d-1}$ with $\mathcal{H}^{d-1}(\partial\omega)=0$ there holds  $\mathcal{F}(\omega;h)  = K\, \mathcal{H}^{d-1}(\omega)$, where $K$ is the constant from  \eqref{eq: our-k1}. 
\end{proposition}

In a similar fashion, we investigate properties of the double-profile energy given in \eqref{eq: our-k2}. Recall $\mathcal{W}_d$ in  \eqref{ew: W sequence}. We define the set of functions jumping on the interface by  
\begin{align}\label{eq: UUU}
\mathcal{U}_{\rm dp}(D_{\omega,h}) := \big\{ u \in  SBV^2_{\rm loc}(D_{\omega,h};\R^d) \colon \,  \mathcal{H}^{d-1}(J_u )>0, \ J_u \subset \omega \times \lbrace 0 \rbrace \big\}.
\end{align}
Then, for $M \in \lbrace A,B \rbrace$, we  define the \emph{double-profile energy   function}  
\begin{align}\label{eq: k2-intro}
\mathcal{F}^M_{\rm dp}(\omega;h)  = \inf_{u \in \mathcal{U}_{\rm dp}(D_{\omega,h})}  \inf_{\lbrace w_\eps \rbrace_\eps \in \mathcal{W}_d} \inf\Big\{\liminf_{\ep\to 0} \mathcal{E}_{\ep}(y^{\ep},D_{\omega,h}): \  \frac{y^\eps -  Mx}{w_\eps} \to u  \text{  in measure  in $D_{\omega,h}$ as $\eps \to 0$}\Big\},
\end{align}
for every $\omega\subset \mathbb{R}^{d-1}$ and $h>0$. The double-profile energy can be characterized as follows.

\begin{proposition}[Double-profile energy   function]\label{eq: KundKdo-new}
For all $h>0$, all  open, bounded   sets  $\omega \subset \R^{d-1}$ with $\mathcal{H}^{d-1}(\partial\omega)=0$, and  for $M\in\lbrace A,B \rbrace$  there holds 
\begin{align}\label{eq: dpe-cha}
K^M_{\rm dp} \,  \mathcal{H}^{d-1}(\omega) \ge \mathcal{F}^M_{\rm dp}(\omega,h) \ge 2K \, \mathcal{H}^{d-1}(\omega),
\end{align}
 where $K$ and  $K^M_{\rm dp }$ are defined   in \eqref{eq: our-k1} and  \eqref{eq: our-k2}, respectively.
\end{proposition}

Note that the  result in particular implies Proposition \ref{prop:cell-form-chaper2}.  Moreover, in the case  $2K=K^M_{\rm dp}$, equality holds in \eqref{eq: dpe-cha}.   (We refer to Subsection \ref{subs:1d} for a setting in which this condition is fulfilled). We defer the proof of Proposition \ref{eq: KundKdo-new} to Subsection \ref{sec: cell-formula} below. At this stage, we only mention that it is achieved in two steps: we first show that  $\mathcal{F}^M_{\rm dp}(\omega,h)$ is independent of $h$ and depends on $\omega$ only in terms of $\mathcal{H}^{d-1}(\omega)$, see Proposition \ref{prop:cell-form-new} below. Then, in a second step we address the connection between  $\mathcal{F}^M_{\rm dp}(Q',1)$, $K^M_{\rm dp}$, and $2K$, see Proposition \ref{eq: KundKdo}.   We now proceed with the proof of the liminf inequality.  \EEE

\begin{proof}[Proof of Theorem \ref{thm:liminf}]
Let $(y,u,\mathcal{P}) \in \mathcal{A}$, see Definition \ref{def:convergence2}, and let $y^{\ep}\to (y,u,\mathcal{P})$ in the sense of Definition \ref{def:convergence1}, i.e., there are sequences $\lbrace R^\eps \rbrace_\eps$, $\lbrace \mathcal{P}^\eps\rbrace_\eps$, $\lbrace \mathcal{M}^\eps\rbrace_\eps$, and $\lbrace \mathcal{T}^\eps\rbrace_\eps$ such that  \eqref{eq: partition property}--\eqref{eq:comp-grad-u} hold. Suppose that $y \in \mathcal{Y}_R(\Omega)$ for $R \in SO(d)$, see \eqref{eq: limiting deformations}.   To simplify the exposition, we suppose that $\fint_{\Omega}y^{\ep}\,{\rm d}x = 0$, i.e., by \eqref{eq:comp-y} we get 
\begin{align}\label{eq: L1-convg}
y^{\ep}\to y\quad\text{strongly in }H^1(\Omega;\R^d).
\end{align}
By Proposition \ref{prop:ex-coarsest-part}(iii), Proposition \ref{lemma: admissible-u}(i), and Remark \ref{rem: AR}, possibly passing to another  displacement field being admissible for the sequence $\lbrace y^\eps \rbrace_\eps$,  we may without restriction assume that
\begin{align}\label{eq: jump-part2}
\bigcup\nolimits_j \partial P_j \cap \Omega \subset J_u.
\end{align}
As $\Omega$ has Lipschitz boundary,  by the definition of the set $\mathcal{A}$ in  Definition \ref{def:convergence2} and  by  Proposition \ref{lemma: admissible-u}(i) there exist sequences    $\{\omega_i^y\}_i$, $\{\omega_i^u\}_i$   of Lipschitz domains in $\R^{d-1}$    and sequences $\{\alpha_i^y\}_i$, $\{\alpha_i^u\}_i$  of real numbers such that
\begin{align}\label{eq: jumpo}
J_{\nabla y} =  \bigcup\nolimits_{i\in \N} \omega_i^y \times \{\alpha_i^y\}\quad \text{and} \quad  J_u\setminus J_{\nabla y} =  \bigcup\nolimits_{i\in \N} \omega_i^u \times \{\alpha_i^u\}.
\end{align}
Let $\delta >0$. We can find $I_y, I_u \in \N$ such that
\begin{align}\label{eq: lowerbound1}
\mathcal{H}^{d-1}(J_{\nabla y}) -\delta \le \sum\nolimits_{i=1}^{I_y} \mathcal{H}^{d-1}\big(\omega^y_i \times \lbrace \alpha^y_i \rbrace\big), \ \ \ \ \  \mathcal{H}^{d-1}(J_u \setminus J_{\nabla y}) -\delta \le \sum\nolimits_{i=1}^{I_u} \mathcal{H}^{d-1}\big(\omega^u_i \times \lbrace \alpha^u_i \rbrace\big).
 \end{align}
Moreover, we choose $h>0$ such that the cylindrical sets  (see \eqref{eq:def-dlh})  $\alpha^y_i e_d + D_{\omega^y_i,h}$, $i=1,\ldots,I_y$, and $\alpha^u_i e_d + D_{\omega^u_i,h}$, $i=1,\ldots,I_u$, are pairwise disjoint, and do not intersect the interfaces $ \lbrace \omega^y_i \times \lbrace \alpha^y_i \rbrace \rbrace_{i >I_y}$ and $\lbrace\omega^u_i \times \lbrace \alpha^u_i \rbrace\rbrace_{i >I_u}$. The latter is possible due to $J_{\nabla y} \subset \bigcup\nolimits_ {j} \partial P_j \cap \Omega$ (see definition of $\mathcal{A}$), Lemma \ref{lemma: localization}, and Remark \ref{remark: localization} which   imply  that the interfaces $\lbrace \omega^y_i \times \lbrace \alpha^y_i \rbrace\rbrace_{i >I_y}$ and $\lbrace \omega^u_i \times \lbrace \alpha^u_i \rbrace\rbrace_{i >I_u}$ can only accumulate at $\partial \Omega$, see \cite[Proof of Proposition 3.1]{conti.schweizer2} for details,  and the lower part of Figure \ref{fig:limiting-def} for an illustration. 

By possibly passing to a smaller $h>0$ (not relabeled),  we can choose $\tilde{\omega}^y_i \subset\subset \omega^y_i$ and $\tilde{\omega}^u_i \subset\subset \omega^u_i$ with Lipschitz boundary such that 
\begin{align}\label{eq: lowerbound2}
\mathcal{H}^{d-1}(\omega^y_i) & \le \mathcal{H}^{d-1}(\tilde{\omega}^y_i)  + \delta/I_y  \ \ \ \text{for $i=1,\ldots,I_y$},\notag\\ 
\mathcal{H}^{d-1}(\omega^u_i)  & \le \mathcal{H}^{d-1}(\tilde{\omega}^u_i)  + \delta/I_u  \ \ \ \text{for $i=1,\ldots,I_u$},
\end{align}
and such that  
$$D^y_i := \alpha^y_ie_d + D_{\tilde{\omega}_i^y,h} \subset \subset \Omega \text{ for $i=1,\ldots,I_y$}, \quad \quad \quad D^u_i:=\alpha^u_ie_d + D_{\tilde{\omega}_i^u,h}\subset \subset \Omega \text{ for $i=1,\ldots,I_u$},$$
  see Figure \ref{fig:limiting-sets1} below.  
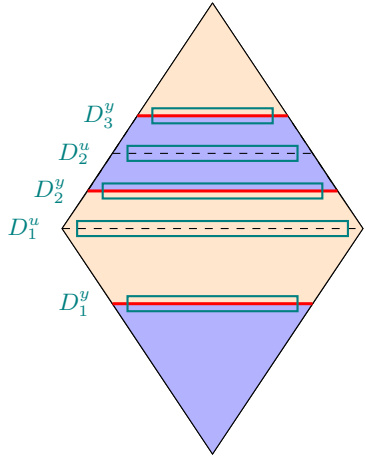
\begin{figure}[h]
\centering
\begin{tikzpicture}
% e_1 points right, e_2 points up, 
% four external points
\coordinate (A) at (0,3);
 \coordinate (B) at (2,0);
 \coordinate (C) at (-2,0);
 \coordinate (D) at (0,-3);
 % points on the right y=3/2x-3
 \coordinate (E) at (1,1.5);
 \coordinate (F) at (1.33,1);
  \coordinate (G) at (1.33,-1);
  \coordinate (H) at (0.66,-2);
  \coordinate (I) at (0.33,-2.5);
  \coordinate (L) at (0.2,-2.7);
  \coordinate (M) at (0.1,-2.85);
  %points on the left
   \coordinate (N) at (-1,1.5);
 \coordinate (O) at (-1.33,1);
 \coordinate (U) at (-1.66, 0.5);
 \coordinate (V) at (1.66, 0.5);
  \coordinate (P) at (-1.33,-1);
  \coordinate (Q) at (-0.66,-2);
  \coordinate (R) at (-0.33,-2.5);
  \coordinate (S) at (-0.2,-2.7);
  \coordinate (T) at (-0.1,-2.85);
  %lines and colors
  \draw (A)--(B)--(D)--(C)--cycle;
  \draw (E)--(N);
   \draw (G)--(P);
  
 % \draw[dashed] (H)--(Q);
 % \draw[dashed] (I)--(R);
 % \draw[dashed] (L)--(S);
  %\draw[dashed] (M)--(T);
  \draw[fill=orange!20] (A)--(N)--(E);
   \draw (A)--(E);
  \draw[dashed, fill=blue!30] (N)--(E)--(V)--(U)--cycle;
  \draw (E)--(V);
  \draw (N)--(U);
  \draw[dashed] (O)--(F);
 % \draw[dashed, fill=blue!30] (O)--(F)--(V)--(U)--cycle;
 \draw[fill=orange!20] (U)--(V)--(B)--(G)--(P)--(C)--cycle;
 \draw[dashed](C)--(B);
  %\draw[dashed, fill=orange!20] (U)--(V)--(B)--(C)--cycle;
 \draw[fill=blue!30] (P)--(G)--(D)--cycle;
 %  \draw[dashed,fill=blue!20] (R)--(I)--(L)--(S)--cycle;
% \draw[dashed,fill=orange!10] (S)--(L)--(M)--(T);
 \draw[line width=0.40 mm,red] (N)--(E);
 \draw[line width=0.40 mm,red] (U)--(V);
 \draw[line width=0.40 mm,red] (P)--(G);
% \draw (R)--(I);
 \draw[thick, teal] (-0.8,1.6)--(0.8,1.6)--(0.8, 1.4)--(-0.8,1.4)--cycle;
 \node at (-1.5,1.5) {\color{teal}\small $D_3^y$};
\draw[thick, teal] (-1.13,1.1)--(1.13,1.1)--(1.13, 0.9)--(-1.13,0.9)--cycle;
 \node at (-1.83,1) {\color{teal}\small $D_2^u$};
\draw[thick, teal] (-1.46,0.6)--(1.46,0.6)--(1.46, 0.4)--(-1.46,0.4)--cycle;
 \node at (-2.16,0.5) {\color{teal}\small $D_2^y$};
\draw[thick, teal] (-1.8,0.1)--(1.8,0.1)--(1.8, -0.1)--(-1.8,-0.1)--cycle;
 \node at (-2.5,0) {\color{teal}\small $D_1^u$};
\draw[thick, teal] (-1.13,-1.1)--(1.13,-1.1)--(1.13, -0.9)--(-1.13,-0.9)--cycle;
 \node at (-1.83,-1) {\color{teal}\small $D_1^y$};
%\draw[thick, teal] (-0.86,-2.1)--(0.86,-2.1)--(0.86, -1.9)--(-0.86,-1.9)--cycle;
% \node at (-1.16,-2) {\color{teal}\small $D2$};
%\draw[thick, teal] (-0.53,-2.4)--(0.53,-2.4)--(0.53, -2.6)--(-0.53,-2.6)--cycle;
% \node at (-0.83,-2.5) {\color{teal}\small $D1$};
%% Following is for debugging purposes so you can see where the points are
%% These are last so that they show up on top
%\foreach \xy in {A, B, C, D, E, F, G, H, I, L, M, N, O, P, Q, R, S, T, U, V}{
%  \node at (\xy) {\xy};
%}
\end{tikzpicture}
\caption{A visualization of the different interfaces and sets under \eqref{eq: jump-part2}. The phase regions associated to $A$ and $B$ are colored in blue and orange, respectively. The cylindrical sets $\{D_i^y\}_{i=1,\dots, \mathcal{I}_y}$ and $\{D_i^u\}_{i=1,\dots, \mathcal{I}_u}$ are drawn in green. The corresponding interfaces in $J_{\nabla y}$ and $J_u$ are highlighted with thick red and dashed black lines, respectively. }
\label{fig:limiting-sets1}
\end{figure}

 Moreover, it is also not restrictive to assume that
\begin{align}\label{eq: small vol}
\sum\nolimits_{i=1}^{I_y} \mathcal{L}^d(D_i^y) + \sum\nolimits_{i=1}^{I_u} \mathcal{L}^d(D_i^u) \le \delta.
\end{align}
We define the set 
\begin{align}\label{eq: omegadelta}
\Omega_\delta := \big\{ x\in \Omega\colon \, {\rm dist}(x,\partial \Omega) >\delta \big\}  \setminus \Big( \bigcup\nolimits_{i=1}^{I_y} D^y_i \cup \bigcup\nolimits_{i=1}^{I_u} D^u_i \Big).  
\end{align}
The main steps of the proof will consist in estimating the surface energies by
\begin{align}\label{eq: check1}
{\rm (i)} & \ \ \liminf_{\ep \to 0} \mathcal{E}_{\ep}\Big(y^{\ep},\bigcup\nolimits_{i=1}^{I_y} D_i^{y}\Big)  \ge    K\big(\mathcal{H}^{d-1}(J_{\nabla y}) -2\delta\big), \notag \\
{\rm (ii)} & \ \ \liminf_{\ep \to 0} \mathcal{E}_{\ep}\Big(y^{\ep},\bigcup\nolimits_{i=1}^{I_u} D_i^{u}\Big)  \ge  2K\big(\mathcal{H}^{d-1}(J_{u}\setminus J_{\nabla y}) -2\delta\big),
\end{align}
and the elastic energy by 
\begin{align}\label{eq: check2}
\liminf_{\ep\to 0} \mathcal{E}_\eps(y^\eps,\Omega_\delta ) \geq \int_{\Omega_{\delta}} \mathcal{Q}_{\rm lin}(\nabla y,\nabla u)\,{\rm d}x,
\end{align}
where the quadratic form $ \mathcal{Q}_{\rm lin}$ is defined in \eqref{eq:def-Q}. Once these estimates have been settled, in view of  \eqref{eq: limiting energy}, we then indeed obtain $\liminf_{\ep\to 0} \mathcal{E}_\eps(y^\eps,\Omega)  \ge  \mathcal{E}_0^{\mathcal{A}} (y,u,\mathcal{P})$ by letting $\delta \to 0$, by taking \eqref{eq: jump-part2} as well as \eqref{eq: small vol}--\eqref{eq: omegadelta} into account, and by using monotone convergence. Let us now prove \eqref{eq: check1} and \eqref{eq: check2}.

\emph{Step 1: Proof of \eqref{eq: check1}{\rm (i)}.}  By \eqref{eq: L1-convg}, $y \in \mathcal{Y}_R(\Omega)$,  \eqref{eq: jumpo}, and the fact that the sets $\lbrace D^y_i\rbrace_i$ are pairwise disjoint and contain only one interface, we get for each $i=1,\ldots,I_y$  that
$$R^{-1} y^\eps(\cdot + \alpha_i^y e_d)  \to  y_0^+  \quad \quad \quad  \text{or} \quad \quad \quad  R^{-1} y^\eps(\cdot + \alpha_i^y e_d)  \to  y_0^-   \quad \quad \quad     \text{ in } L^1(D_{\tilde{\omega}_i^y,h};\R^d).$$
Therefore,  by H2.,  \eqref{eq: k-intro}, and the comment thereafter we obtain 
\begin{equation*}
\liminf_{\ep \to 0} \mathcal{E}_{\ep}\Big(y^{\ep},\bigcup\nolimits_{i=1}^{I_y} D_i^{y}\Big)\ge\sum\nolimits_{i=1}^{I_y}\liminf_{\ep \to 0} \mathcal{E}_{\ep}\Big( R^{-1} y^\eps(\cdot + \alpha_i^y e_d), D_{\tilde{\omega}_i^y,h}\Big)\geq \sum\nolimits_{i=1}^{I_y} \mathcal{F}(\tilde{\omega}_i^y;h).
\end{equation*}
Then, by Proposition \ref{prop:cell-form} and \eqref{eq: lowerbound1}--\eqref{eq: lowerbound2} we get 
\begin{equation*}
%\label{eq:liminf-part1}
\liminf_{\ep \to 0} \mathcal{E}_{\ep}\Big(y^{\ep},\bigcup\nolimits_{i=1}^{I_y} D_i^{y}\Big)  \ge  K \sum\nolimits_{i=1}^{I_y} \mathcal{H}^{d-1}(\tilde{\omega}_i^y) \geq  K\big(\mathcal{H}^{d-1}(J_{\nabla y}) -2\delta\big).
\end{equation*}
This shows \eqref{eq: check1}(i).

\emph{Step 2: Proof of \eqref{eq: check1}{\rm (ii)}.} By \eqref{eq: jumpo} and the fact that the cylindrical sets are chosen to be pairwise disjoint  and to contain only one interface we know that $\nabla y$ is constant on each $D_i^u$, $i=1\ldots,I_u$. We choose $M_i \in \lbrace A, B \rbrace$ such that $\nabla y = RM_i$ on $D_i^u$. We will distinguish two cases, indicated by the index sets
\begin{align}\label{eq: I12}
\mathcal{I}_1 := \Big\{ i= 1\ldots I_u\colon\, (\omega^u_i \times \lbrace \alpha^u_i\rbrace)  \cap \bigcup\nolimits_j \partial P_j  \cap \Omega \EEE= \emptyset \Big\}, \quad \quad \quad \mathcal{I}_2 := \lbrace 1,\ldots,I_u\rbrace \setminus \mathcal{I}_1.
\end{align}

\emph{Step 2(a): $i \in \mathcal{I}_1$.} In view of  \eqref{eq: jumpo}, \eqref{eq: I12}, and the fact that the cylindrical sets are  pairwise disjoint and contain only one interface, we get $D_i^u \subset P_k$ for some index $k$.   Then by  \eqref{eq:comp-sets}, \eqref{eq:comp-M}, \eqref{eq: rescaled disp}, and  \eqref{eq:comp-u}  we get  as $\eps \to 0$
\begin{align}\label{eq: strange pointwise} 
\eps^{-1}\big(y^{\ep}-  R^{\ep} M_i\, x- t_k^{\ep}\big)  \to u \quad \quad \quad \text{  in measure in $D_i^u$}.
\end{align}
As the cylindrical sets are pairwise disjoint and contain only one interface, we find $u(\cdot + \alpha^u_i e_d ) \in \mathcal{U}_{\rm dp}(D_{\tilde{\omega}_i^u,h})$ (recall \eqref{eq: UUU}).  We define the function 
$$\bar{y}^\eps(x) \coloneqq  (R^\eps)^T y^\eps (x + \alpha_i^ue_d)  - (R^\eps)^Tt_k^\eps - M_i\alpha_i^ue_d$$
 for $x \in D_{\tilde{\omega}_i^u,h}$, and  we  note by \eqref{eq:comp-R} and \eqref{eq: strange pointwise} that $\eps^{-1}(\bar{y}^\eps   -M_ix) \to \bar{u}$ in measure in $D_{\tilde{\omega}_i^u,h}$, where $\bar{u} := R^Tu(\cdot + \alpha^u_i e_d )  \in \mathcal{U}_{\rm dp}(D_{\tilde{\omega}_i^u,h})$. Then, the sequences $\lbrace \bar{y}^\eps\rbrace_\eps$ and $\lbrace w_\eps \rbrace_\eps \in \mathcal{W}_d$ defined by   $w_\eps : = \eps $ for all $\eps$ are admissible in \eqref{eq: k2-intro}. Thus, by the translational and rotational  invariance of the energy we get 
\begin{align}\label{eq: lower-i1}
\liminf_{\eps \to 0} \mathcal{E}_\eps \big(y^\eps , D_i^u \big) = \liminf_{\eps \to 0}\mathcal{E}_\eps \big(\bar{y}^\eps, D_{\tilde{\omega}_i^u,h} \big) \ge \mathcal{F}^{M_i}_{\rm dp}(\tilde{\omega}_i^u;h). 
\end{align}
 
\emph{Step 2(b): $i \in \mathcal{I}_2$.} In this case, by \eqref{eq: jumpo} and the fact that the cylindrical sets are pairwise disjoint and contain only one interface, $D_i^u$ intersects two components $P_k$ and $P_l$, namely $\tilde{\omega}^u_i \times (\alpha^u_i - h,\alpha^u_i   )   \subset   P_k $  and $\tilde{\omega}^u_i \times (\alpha^u_i ,\alpha^u_i + h)   \subset   P_l $. As before, we have $\nabla y = RM_i$ on $D_i^u$.  Let $w_\eps : = |t^\eps_k - t^\eps_l| $, where $t^\eps_k, t^\eps_l$  are the elements  from the translations $\mathcal{T}^\eps$  corresponding to the sets $P_k^\ep$ and $P_l^\ep$.  By  \eqref{eq:comp-sets}, \eqref{eq:comp-M}, \eqref{eq: rescaled disp}, and \eqref{eq:comp-u}   we get  as $\eps \to 0$
\begin{align}\label{eq: strange pointwise2} 
\eps^{-1}\big(y^{\ep}- R^{\ep} M_i\, x  -   t_j^{\ep}\big)  \to u \quad \quad \text{  in measure in   $D_i^u \cap P_j$ for $j \in \lbrace k,l\rbrace$}.
\end{align}
By \eqref{eq: toinfty} we find $w_\eps/\eps \to \infty$. Moreover, for a.e.\ $x_k \in D_i^u \cap P_k$ and  a.e.\ $x_l \in D_i^u \cap P_l$, by multiplying  \eqref{eq: strange pointwise2}  with $\eps$ and using \eqref{eq: L1-convg} we get  $\limsup_{\eps \to 0} |t_k^\eps - t_l^\eps| \le |y(x_k)-y(x_l)|+ |M_i||x_k-x_l|$. This implies  that  $\lim_{\eps\to 0} w_\eps = \lim_{\eps \to 0}|t^\eps_k-t^\eps_l| = 0$ as $y$ is continuous. Thus,  $\lbrace w_\eps \rbrace_\eps \in \mathcal{W}_d$, see \eqref{ew: W sequence}. By possibly passing to a subsequence (not relabeled), we may suppose that $(t^\eps_l - t^\eps_k)/ w_\eps \to  t_0 \in \R^d$.  We check that 
\begin{align}\label{eq: second-check}
 \frac{y^{\ep}- (R^{\ep} M_i\, x  +   t_k^{\ep})}{w_\eps}  \to t_0 \chi_{\lbrace x_d \ge \alpha_i^u\rbrace} \quad \quad \text{  in measure in  $D_i^u$.}
\end{align}
In fact, by \eqref{eq: strange pointwise2} and  $\eps /w_\eps \to 0$, we first get  
$$w_\eps^{-1} \big(y^{\ep}- R^{\ep} M_i\, x  -   t_k^{\ep}\big) =  (\eps/w_\eps) \, \eps^{-1} \big(y^{\ep}- R^{\ep} M_i\, x  -   t_k^{\ep} \big)   \to 0 \quad \quad \text{ in measure in  $D_i^u \cap P_k$},$$
and by again using \eqref{eq: strange pointwise2}, $\eps /w_\eps \to 0$, as well as $(t^\eps_l - t^\eps_k)/ w_\eps \to  t_0$ we find \begin{align*}
w_\eps^{-1} \big( y^{\ep}- R^{\ep} M_i\, x  -   t_k^{\ep} \big)  = (\eps/w_\eps) \,  \eps^{-1} \big( y^{\ep}- R^{\ep} M_i\, x  -   t_l^{\ep}\big)   +  w_\eps^{-1} \big( t_l^\eps - t_k^\eps \big)   \to t_0
\end{align*}
in measure in  $D_i^u \cap P_l$. Now, by \eqref{eq: second-check} and  by arguing along the lines of  \eqref{eq: strange pointwise}--\eqref{eq: lower-i1} we can define a sequence $\lbrace \bar{y}^\eps\rbrace_\eps$ via rotation and shifting such that   $\lbrace \bar{y}^\eps\rbrace_\eps$ and $\lbrace w_\eps \rbrace_\eps \in \mathcal{W}_d$ are admissible in \eqref{eq: k2-intro}.  Then, we deduce
\begin{align}\label{eq: lower-i2}
\liminf_{\eps \to 0} \mathcal{E}_\eps \big(y^\eps , D_i^u \big)  \ge \mathcal{F}^{M_i}_{\rm dp}(\tilde{\omega}_i^u;h). 
\end{align}
We now conclude the proof of \eqref{eq: check1}(ii) as follows: combining \eqref{eq: lower-i1}, \eqref{eq: lower-i2}, and  Proposition \ref{eq: KundKdo-new} we get
\begin{align*}
\liminf_{\ep \to 0} \mathcal{E}_{\ep}\Big(y^{\ep},\bigcup\nolimits_{i=1}^{I_u} D_i^{u}\Big) \ge \sum\nolimits_{i=1}^{I_u} \mathcal{F}^{M_i}_{\rm dp}(\tilde{\omega}_i^u;h) \ge 2K \sum\nolimits_{i=1}^{I_u}  \mathcal{H}^{d-1}(\tilde{\omega}_i^u).
\end{align*}
Then, \eqref{eq: check1}(ii) follows from \eqref{eq: lowerbound1}--\eqref{eq: lowerbound2}.

\emph{Step 3: Proof of \eqref{eq: check2}.} We start by recalling the definition of $u^\eps$ in \eqref{eq: rescaled disp} and by noting that \eqref{eq:comp-grad-u}  implies
\begin{align}\label{eq: strain bound}
\int_{\Omega_\delta}|\nabla u^{\ep}|^2\,{\rm d}x\leq C_\delta \quad \quad \quad \text{for all $\eps>0$},
\end{align} 
where $C_\delta>0$ depends on the set $\Omega_\delta$ defined in \eqref{eq: omegadelta}, and thus on $\delta$. We now define two small exceptional sets: first, we let  $\alpha\in (0,1)$, and  we define the set of large linearized strains by 
\begin{align}\label{eq: strain-om}
\Omega^\eps_{\rm strain} := \lbrace x \in \Omega_\delta \colon  |\nabla u^{\ep}(x)| \ge \ep^{-\alpha} \rbrace.  
\end{align}
By Chebyshev's inequality and \eqref{eq: strain bound} we estimate
 \begin{equation}
\label{eq:meas-u}
\mathcal{L}^d\big(\Omega_{\rm strain}^{\ep}\big)\leq \ep^{2\alpha}\int_{\Omega_\delta}|\nabla u^{\ep}|^2\,{\rm d}x\leq C_\delta\ep^{2\alpha}.
\end{equation}
Moreover, by \eqref{eq:comp-M}  and by the continuous embedding of $BV(\Omega;\mathbb{M}^{d \times d})$ into $L^1(\Omega;\mathbb{M}^{d \times d})$  we find  a sequence $\{\delta_{\ep}\}_{ \eps}\subset (0,+\infty)$  such that
$\delta_{\ep}\to 0$ and 
\begin{align} \label{eq:delta-ep-slow}
\lim_{\ep\to 0} \frac{1}{\delta_{\ep}}\int_{\Omega} \big|\sum\nolimits_{j} (R^{\ep}M^{\ep}_j)\chi_{P^{\ep}_j}-\nabla y\big|\,{\rm d}x=0.
\end{align}
Then, we define the set
\begin{align}\label{eq: phase-om}
\Omega^\eps_{\rm phase} := \Big\{ x \in \Omega_\delta \colon \, \big|\sum\nolimits_{j} (R^{\ep}M^{\ep}_j)\chi_{P^{\ep}_j}(x)-\nabla y(x)\big|\ge \delta_{\ep}\Big\}
\end{align}
of points where the phases along the sequence differ by at least $\delta_\eps$ from the phases in the limit.  Clearly, \eqref{eq:delta-ep-slow} entails  
\begin{equation}
\label{eq:meas-y}
\lim_{\ep\to 0} \mathcal{L}^d(\Omega^\eps_{\rm phase}) \leq \lim_{\ep\to 0} \frac{1}{\delta_{\ep}}\int_{\Omega} \big|\sum\nolimits_{j} (R^{\ep}M^{\ep}_j)\chi_{P^{\ep}_j}-\nabla y\big|\,{\rm d}x=0.
\end{equation}
By combining \eqref{eq:meas-u} and \eqref{eq:meas-y} we find
\begin{align}\label{eq: omegagood}
\lim_{\eps \to 0}  \mathcal{L}^d(\Omega_\delta \setminus \Omega_{\rm good}^\eps) = 0, \quad \quad \quad \text{where } \ \Omega^\eps_{\rm good} := \Omega_\delta \setminus \big(\Omega_{\rm strain}^{\ep} \cup \Omega^\eps_{\rm phase}  \big).
\end{align}
By \eqref{eq: nonlinear energy} and the definition in \eqref{eq: rescaled disp} we get 
\begin{align}\label{eq: firsti step}
\mathcal{E}_\eps(y^\eps, \Omega_\delta)  \ge \frac{1}{\ep^2}\int_{\Omega_{\delta}} W(\nabla y^{\ep})\,{\rm d}x \geq\frac{1}{\ep^2}\sum\nolimits_{j}\int_{\Omega_{\rm good}^{\ep}\cap P_j^{\ep}} W\big( R^{\ep}M^{\ep}_j + \ep\nabla u^{\ep}(x) \big)\,{\rm d}x.
\end{align}
By assumptions H2., H3., and H5.\ we can perform a Taylor expansion and write 
$$W(RM + F) = \frac{1}{2}D^2W(RM) F : F + \omega_W(F)$$ 
for all $F \in \mathbb{M}^{d \times d}$ with $|F| < \delta_W$, where $\omega_W\colon\mathbb{M}^{d\times d} \to \R$ satisfies 
\begin{equation}
\label{eq:cont-mod}
\lim_{\rho\to 0^+} \eta_W(\rho) = 0, \quad \quad \quad \text{where } \ \eta_W(\rho) :=  \sup\Big\{\frac{ \omega_W(F)}{|F|^2}:\,|F|\leq \rho\Big\}.
\end{equation}
This expansion along with \eqref{eq: strain-om},  \eqref{eq: firsti step}, and the fact that $\Omega^\eps_{\rm good} \cap \Omega^\eps_{\rm strain} = \emptyset$  yields   for $\ep$ small enough  
\begin{align*}
\mathcal{E}_\eps(y^\eps, \Omega_\delta) &\ge \sum\nolimits_{j}\int_{{\Omega_{\rm good}^{\ep}\cap P_j^{\ep}}} \Big( \frac12  D^2W(R^{\ep}M^{\ep}_j)\nabla u^{\ep}:\nabla u^{\ep} + \frac{1}{\eps^2} \omega_W(\eps \nabla u^\eps) \Big) \, {\rm d}x\notag\\
& = \sum\nolimits_{j}\int_{{\Omega_{\rm good}^{\ep}\cap P_j^{\ep}}} \Big( \frac12  D^2W(R^{\ep}M^{\ep}_j)\nabla u^{\ep}:\nabla u^{\ep} + |\nabla u^\eps|^2 \frac{ \omega_W(\eps \nabla u^\eps)}{|\eps \nabla u^\eps|^2} \Big) \, {\rm d}x \notag\\
& \ge \sum\nolimits_{j}  \frac12  \int_{{\Omega_{\rm good}^{\ep}\cap P_j^{\ep}}}  D^2W(R^{\ep}M^{\ep}_j)\nabla u^{\ep}:\nabla u^{\ep}\,{\rm d}x - \eta_W\big(\eps^{1-\alpha}\big) \Vert \nabla u^\eps \Vert^2_{L^2(\Omega_{\rm good}^{\ep})}.
\end{align*}
Then, by \eqref{eq: strain bound} and \eqref{eq:cont-mod} we get
\begin{align}\label{eq: secondi step}
\liminf_{\ep\to 0}  \mathcal{E}_\eps(y^\eps, \Omega_\delta) \ge \liminf_{\ep\to 0} \sum\nolimits_{j}  \frac12  \int_{{\Omega_{\rm good}^{\ep}\cap P_j^{\ep}}}  D^2W(R^{\ep}M^{\ep}_j)\nabla u^{\ep}:\nabla u^{\ep}\,{\rm d}x. 
\end{align}
By H5.,  \eqref{eq: phase-om},  and the fact that $\Omega^\eps_{\rm good} \cap \Omega^\eps_{\rm phase} = \emptyset$ we find
$$
\Big|\sum\nolimits_{j} \int_{\Omega^\eps_{\rm good}\cap P_j^{\ep}} \big(D^2W(R^{\ep}M^{\ep}_j)-D^2W(\nabla y)\big)\,\nabla u^{\ep}:\nabla u^{\ep}\,{\rm d}x\Big|\leq \hat{\delta}_{\ep}\int_{\Omega_{\rm good}^\eps}|\nabla u^{\ep}|^2\,{\rm d}x,
$$
where $\lbrace \hat{\delta}_\eps \rbrace_\eps \subset (0,+\infty)$ is a sequence depending on $W$ and $\lbrace {\delta}_\eps \rbrace_\eps$, which satisfies $\hat{\delta}_\eps \to 0$. This along with \eqref{eq: strain bound} and \eqref{eq: secondi step} yields
\begin{align}\label{eq: thirdi step}
\liminf_{\ep\to 0}  \mathcal{E}_\eps(y^\eps, \Omega_\delta) \ge \liminf_{\ep\to 0}   \frac12  \int_{{\Omega_{\rm good}^{\ep}}}  D^2W(\nabla y)\nabla u^{\ep}:\nabla u^{\ep}\,{\rm d}x. 
\end{align}
In view of \eqref{eq:comp-grad-u} and \eqref{eq: omegagood}, there holds $\nabla u^{\ep}\chi_{\Omega_{\rm good}^{\ep}}\rightharpoonup \nabla u$ weakly in $L^2(\Omega_\delta;\mathbb{M}^{d \times d})$.  Note that $D^2W(RM)$ is positive semidefinite for $M \in \lbrace A,B\rbrace$ by H2.\ and H3.  Thus, by \eqref{eq: thirdi step} and the  weak  lower semicontinuity of convex integral functionals, we conclude  that
$$\liminf_{\ep\to 0}  \mathcal{E}_\eps(y^\eps, \Omega_\delta)  \ge \frac12  \int_{\Omega_\delta}  D^2W(\nabla y)\nabla u:\nabla u\,{\rm d}x.  $$
This along with the definition in \eqref{eq:def-Q} shows  \eqref{eq: check2}.  This concludes the proof.
\end{proof}

%%%%%%%%%%%%%%%%%%%%%%%%%%%%%%%%%%%%%%%%%%

 \subsection{The limsup inequality}
 \label{subs:limsup}
In this subsection we   prove the optimality of  the lower bound identified in Theorem \ref{thm:liminf},  under the additional condition that $2K=K^M_{\rm dp}$, for $M\in \{A,B\}$, cf.\ \eqref{eq: our-k1} and  \eqref{eq: our-k2}. 
We first collect some basic properties of the elastic energy density. 
\begin{lemma}[Elementary properties of the energy density]\label{lemma: limsup-W}
Let $W\colon\mathbb{M}^{d\times d}\to [0,+\infty)$ satisfy assumptions H1.--H5.\ and H7. 
Let $0 <\delta \le \delta_W/2$, where $\delta_W$ is the constant introduced in $\mathrm{ H5.}$ We define $ \mathcal{V}_\delta = \lbrace F\in \M^{d \times d}\colon \, {\rm dist}(F,SO(d)\lbrace A,B \rbrace) <\delta \rbrace$.  Then there exists a constant $C>0$  only depending on $W$, a constant $C_\delta>0$ additionally depending on $\delta$, and  $\rho_\delta>0$ with $\rho_\delta \to 0$ as $\delta\to 0$ such that
\begin{align*}
{\rm (i)}& \ \ W(F +  G)  \le W(F) + C\sqrt{W(F)}|G| + \frac{1}{2} D^2W(F) \, G:G + \rho_\delta|G|^2 \ \ \text{for all $F \in \mathcal{V}_\delta$, $G \in B_{\delta}(0) $,}\\
{\rm (ii)}& \ \ W(F + G)  \le W(F) + C_\delta\sqrt{W(F)}|G|  \ \ \text{for all $F \in \M^{d\times d} \setminus \mathcal{V}_\delta$, $G \in B_{\delta}(0)$},
\end{align*}
where  $B_{\delta}(0) \subset \M^{d\times d}$ denotes the open ball centered at $0$ with radius $\delta$.  
\end{lemma}
The proof of this lemma is postponed to the end of this subsection.

We proceed with the construction of local recovery sequences around the interfaces. To this end, recall the definition of $K$ in \eqref{eq: our-k1}. Let $y_0^+$ and $y_0^-$ be the maps defined right after \eqref{eq: conti-schweizer-k}. We recall the notion of cylindrical sets from \eqref{eq:def-dlh}  and the definition of strictly star-shaped domains in \eqref{eq: starshape}.  We start by stating the local construction of recovery sequences for a single phase transition.

\begin{proposition}[Local recovery sequence for single phase transition]\label{lemma: local1}
Let  $d\in \mathbb{N}$, $d\geq 2$. Let $\Omega \subset \R^d$ be a bounded, strictly star-shaped Lipschitz domain. Let $\omega' \subset \R^{d-1}$ be a bounded Lipschitz domain and $h>0$ such that $\partial \omega' \times (-h,h)$ does not intersect ${\Omega}$. Then,  there exist sequences $\lbrace v^+_\eps \rbrace_\eps, \lbrace v^-_\eps \rbrace_\eps \subset H^2(D_{\omega',h}\cap \Omega;\R^d)$ with 
\begin{equation}
\label{eq:local1-1}
v^\pm_\eps \to y^\pm_0 \quad\text{ in } H^1(D_{\omega',h}\cap\Omega ;\R^d),
\end{equation}
 such that
 \begin{equation}
 \label{eq:local1-3a}
 \lim_{\eps \to 0} \mathcal{E}_{\ep} (v^\pm_\eps, D_{\omega',h} \cap\Omega)  =  K \, \mathcal{H}^{d-1}\big((\omega' \times \lbrace 0 \rbrace) \cap \Omega\big), 
 \end{equation}
 and for $\eps$ sufficiently small we have
\begin{equation}
\label{eq:local1-3} v^\pm_\eps = 
\begin{cases} 
I^\pm_{1,\eps} \circ y_0^\pm & \text{if $x_d \ge   3h/4$},\\
 I^\pm_{2,\eps} \circ y_0^\pm & \text{if $x_d \le -  3h/4$},\\
  \end{cases}  \end{equation}
where $\{I^\pm_{1,\eps}\}_{ \eps}$  and $\{I^\pm_{2,\eps}\}_{ \eps}$ are sequences of isometries which converge to the identity as $\eps \to 0$.
\end{proposition}

We emphasize that the above statement  means that for \emph{any} sequence $\lbrace \eps_i\rbrace_i$ converging to zero a local
recovery sequence can be constructed. The crucial point is that the sequence $\lbrace v^\pm_\eps \rbrace_\eps$ is rigid away from
the interface. This will allow us to appropriately `glue together' local recovery sequences around different interfaces.

The next result provides a local construction of recovery sequences for the case in which  two  consecutive phase transitions create small intermediate layers at level $\ep$ between two portions of the material in the same phase, cf.\ Figure \ref{fig2}.  Owing to the compatibility condition that $2K=K^M_{\rm dp}$, for $M\in \{A,B\}$, cf.\ \eqref{eq: our-k1} and  \eqref{eq: our-k2},   this provides a double energetic contribution. Recall the mappings $y^M_{\rm dp}$  defined in \eqref{eq: step functions}.

\begin{proposition}[Local recovery sequence for double phase transitions]\label{lemma: local2}
Let  $d\in \mathbb{N}$, $d\geq 2$. Let $\Omega \subset \R^d$ be a bounded, strictly star-shaped Lipschitz domain. Let $\omega' \subset \R^{d-1}$ be a bounded Lipschitz domain and $h>0$ such that $\partial \omega' \times (-h,h)$ does not intersect ${\Omega}$. Let $M \in \lbrace A,B \rbrace$ and  suppose that the constant $K^M_{\rm dp}$ defined in \eqref{eq: our-k2} satisfies   $K^M_{\rm dp} = 2K$.   Then, for  every $\{w_\ep\}_\ep\subset\mathcal{W}_d$ there exists a sequence $\lbrace v^M_\eps \rbrace_\eps\subset H^2(D_{\omega',h}\cap \Omega;\R^d)$ with 
\begin{equation}
\label{eq:local2-1}
\frac{v^M_\eps -  Mx}{w_\eps} \to  y^M_{\rm dp}\quad\text{ in measure on } \Omega \cap D_{\omega',h}\end{equation}
 such that
\begin{equation}
\label{eq:local2-2}\lim_{\eps \to 0} \mathcal{E}_\eps (v^M_\eps, \Omega \cap D_{\omega',h})  = 2K \, \mathcal{H}^{d-1}\big((\omega' \times \lbrace 0 \rbrace) \cap \Omega\big),  \ \ \ \ \ v^M_\eps = \begin{cases} I^{M}_{1,\ep} \circ Mx  & \text{if $x_d \ge 3h/4$} \\ I^{M}_{2,\ep} \circ Mx   & \text{if $x_d \le - 3h/4$},\end{cases} \end{equation}
where $\{I^{M}_{1,\ep}\}_\ep$ and $\{I^{M}_{2,\ep}\}_\ep$ are sequences of isometries converging to the identity as $\eps \to 0$. 
\end{proposition}

We defer the proofs of Propositions   \ref{lemma: local1}  and  \ref{lemma: local2} to  Subsection \ref{subs:local}. (Let us mention that in the special case  $\Omega= D_{\omega',h}$ the statement in Proposition  \ref{lemma: local1} has already  been proven in \cite[Proposition 4.7]{davoli.friedrich}, and here we address the generalization to strictly star-shaped Lipschitz domains $\Omega$.)  We  continue with the proof of the limsup inequality. As a final preparation, we introduce the following convention: we say that a sequence of functions $\lbrace v^\eps \rbrace_\eps$  converges to $v$   \emph{up to translation} if there exist $\lbrace \alpha_\eps \rbrace_\eps \subset \R$ and $\lbrace b_\eps \rbrace_\eps \subset \R^d$ such that 
\begin{align}\label{eq: convention}
 v^\eps(\cdot  - \alpha_\eps e_d ) - b_\eps \to v
 \end{align}
 with respect to a given topology. In a similar fashion, we say that two functions $v_1,v_2$ coincide \emph{up to translation} if $v_2 =v_1(\cdot - \alpha e_d) - b$ for $\alpha \in \R$ and $b \in \R^d$.

\begin{proof}[Proof of Theorem \ref{thm:limsup-new}]
Let $(y,u,\mathcal{P})\in \mathcal{A}$. Without loss of generality,  after a rotation, \EEE we can assume that   $y\in \mathcal{Y}_{\rm Id}(\Omega)$. Moreover, similarly to the proof of Theorem   \ref{thm:liminf}, it is also not restrictive to assume that 
 \begin{align}\label{eq: jump-part}
J_{\nabla y} \subset \bigcup\nolimits_j \partial P_j \cap \Omega \subset J_u.
\end{align}
In fact, the first inclusion always holds true by Definition \ref{def:convergence2}, and by using   Proposition \ref{lemma: admissible-u}(i) we may pass to another  displacement field of the form $\tilde{u} = u + \mathcal{T}(y,\mathcal{P})$, see \eqref{eq: infini rigid}, such that the second inclusion holds for $\tilde{u}$ in place of $u$.  In view of Remark \ref{rem: AR},  this does not affect the energy and we observe that a recovery sequence   $\lbrace y^\eps\rbrace_\eps$ for  $(y,\tilde{u},\mathcal{P})$ in the sense of Definition \ref{def:convergence1}  is also admissible for the original triple $(y,u,\mathcal{P})$ by  Proposition \ref{prop:ex-coarsest-part}(iii). As a further preliminary remark, we observe that by a diagonal argument  it suffices to find for every $\delta>0$ a recovery sequence  $\lbrace y^\eps\rbrace_\eps$ for  $(y,u, \mathcal{P})$ such that
\begin{align}\label{eq: delta-en}
\limsup_{\eps \to 0} \mathcal{E}_\eps(y^\eps) \le \mathcal{E}_0^{\mathcal{A}}(y,u,\mathcal{P}) + \delta.
 \end{align}
In this context, we point out that the asymptotic representation introduced in Definition \ref{def:convergence1}  is based on the convergences  \eqref{eq: partition property}--\eqref{eq:comp-grad-u} which themselves are metrizable, i.e., diagonal arguments are applicable.

 For convenience of the reader, we start with a short outline of the proof: in Steps 1--2 we explain that it is not restrictive to treat only  problems with a finite number of interfaces and that one can assume  $\nabla u$ to be smooth. In Step 3 we construct local approximate sequences around the interfaces. These  are then `glued together'  to obtain an auxiliary recovery sequence $\{\tilde{y}^\ep\}_\ep$ converging to $y$, and capturing  correctly the surface energy of the limiting triple $(y,u,\mathcal{P})$, see Step 4. To recover the displacement field $u$ in the limit and to estimate the elastic contributions correctly, we then perturb $\{\tilde{y}^\ep\}_\ep$ by adding a term of order $\eps$. We  check that this new sequence  $\{y^\ep\}_\ep$  indeed satisfies $y^\eps \to (y,u,\mathcal{P})$ (Step 5) and $\limsup_{\eps \to 0} \mathcal{E}_\eps(y^\eps) \le \mathcal{E}_0^{\mathcal{A}}(y,u,\mathcal{P})$ (Step 6). Finally, Step 7 is devoted to some technical estimates.

\emph{Step 1: Reduction to  a  finite number of interfaces.}
  Using the star-shapedness of the domain  (say, with respect to the origin)  along with Remark \ref{remark: localization},  one can apply a scaling argument to reduce the problem to    \EEE limiting configurations where  $J_u$ consists of a finite number of disjoint interfaces orthogonal to $e_d$.   For details on this argument we refer to  \cite[Proof of Proposition 5.1]{conti.schweizer2} and also  \cite[Proof of Theorem~4.4]{davoli.friedrich}, Step I).  We just mention that, for $\rho>1$, one considers rescaled triples $(y_\rho,u_\rho,\mathcal{P}_\rho)$ of the form  $y_\rho(x) = \rho y(x/\rho)$, $u_\rho(x) = \rho u(x/\rho)$, and $P_j^\rho = \rho P_j \cap \Omega$ for each component $P^\rho_j \in \mathcal{P}_\rho$. This sequence satisfies  $\mathcal{E}_0^{\mathcal{A}}(y_\rho,u_\rho,\mathcal{P}_\rho) \to \mathcal{E}_0^{\mathcal{A}}(y,u,\mathcal{P})$ as $\rho \to 1$.   The geometrical intuition is that, since infinitely many interfaces can only occur close to the boundary (see also  the  lower part of Figure \ref{fig:limiting-def}), a rescaling allows to reduce the study to a finite  number of interfaces.    It suffices to construct recovery sequences for $(y_\rho,u_\rho,\mathcal{P}_\rho)$ since a recovery sequence for $(y,u,\mathcal{P})$  can then be obtained by a diagonal argument.
  
   Summarizing, by \eqref{eq: jump-part}   we can suppose that there exist finitely many Lipschitz domains $\omega_{i} \subset \R^{d-1} $  and $\alpha_i \in  \R$ for $i=1,\ldots,I$ such that 
\begin{align}\label{eq: jump-part-new}
J_{\nabla y} \cup \bigcup\nolimits_j (\partial P_j \cap \Omega) \cup J_u = J_u  = \bigcup\nolimits_{i=1}^I (\omega_i \times \lbrace \alpha_i \rbrace).
\end{align}
Since $\Omega$ is star-shaped, we have  that $\Omega\setminus J_u$ is the union of  $I+1$  connected components which we indicate as  $\{B_i\}_{i=1}^{I+1}$. The  sets are ordered such that $\partial B_i \cap \partial B_{i+1} = \omega_i \times \lbrace \alpha_i\rbrace$ for $i=1,\ldots,I$, and the outer normal to $B_i$ on $\partial B_i \cap \partial B_{i+1}$ is given by $e_d$  (see Figure \ref{fig:limiting-sets} below).

\emph{Step 2: Reduction to displacement fields with smooth gradient.}
In a similar fashion, we can also suppose that  $u \in \mathscr{U}(\Omega)$ has a smooth gradient:  
by Proposition \ref{lemma: admissible-u} we find $\lbrace b_i\rbrace_{i=1}^{I+1} \subset \R e_d$ such that the mapping
\begin{align}\label{eq: def u strich}
u' := u-\sum\nolimits_{i=1}^{I+1} b_i \chi_{B_i}
\end{align}  
satisfies $u'\in H^1(\Omega;\R^d)$. Choose a smooth sequence $\lbrace u_k'\rbrace_k \subset C^\infty(\overline{\Omega};\R^d)$ approximating $u'$ in $H^1(\Omega;\R^d)$ and observe that $u_k := u_k' + \sum_{i=1}^{I+1} b_i \chi_{B_i} \in \mathscr{U}(\Omega)$ satisfies $u_k \to u$ in $L^1(\Omega;\R^d)$ and $\nabla u_k \to \nabla u$ in $L^2(\Omega;\mathbb{M}^{d\times d})$.     Again by a diagonal argument  and by using that the limiting energy $\mathcal{E}_0^{\mathcal{A}}$ is continuous with respect to the strong $L^2$-convergence of displacement-gradients (see \eqref{eq: limiting energy}),  it suffices to construct recovery sequences  for displacement fields $u \in \mathscr{U}(\Omega)$ such that $\nabla u \in C^\infty(\overline{\Omega};\mathbb{M}^{d \times d})$.

\emph{Step 3:  Local  construction of  the approximate  recovery sequence.} We now start with the construction of recovery sequences around the interfaces. For brevity, we set $J_{\mathcal{P}} = \bigcup_j \partial P_j \cap \Omega$. In view of \eqref{eq: jump-part} and \eqref{eq: jump-part-new}, we can write 
\begin{align}
\label{eq: diff-i}
J_{\nabla y} = \bigcup\nolimits_{i \in \mathcal{I}_y} (\omega_i \times \lbrace \alpha_i \rbrace), \quad \quad    J_{\mathcal{P}} \setminus J_{\nabla y}= \bigcup\nolimits_{i \in \mathcal{I}_{\mathcal{P}}} (\omega_i \times \lbrace \alpha_i \rbrace),  \quad \quad  J_u \setminus (J_{\nabla y} \cup J_{\mathcal{P}})= \bigcup\nolimits_{i \in \mathcal{I}_u} (\omega_i \times \lbrace \alpha_i \rbrace),
\end{align}
where $\mathcal{I}_y$, $\mathcal{I}_{\mathcal{P}}$, and $\mathcal{I}_u$ are three pairwise disjoint index sets with $\mathcal{I}_y \cup \mathcal{I}_{\mathcal{P}} \cup \mathcal{I}_u = \lbrace 1,\ldots,I\rbrace$. 

\begin{figure}[h]
\centering
\begin{tikzpicture}
% e_1 points right, e_2 points up, 
% four external points
\coordinate (A) at (0,3);
 \coordinate (B) at (2,0);
 \coordinate (C) at (-2,0);
 \coordinate (D) at (0,-3);
 % points on the right y=3/2x-3
 \coordinate (E) at (1,1.5);
 \coordinate (F) at (1.33,1);
  \coordinate (G) at (1.33,-1);
  \coordinate (H) at (0.66,-2);
  \coordinate (I) at (0.33,-2.5);
  \coordinate (L) at (0.2,-2.7);
  \coordinate (M) at (0.1,-2.85);
  %points on the left
   \coordinate (N) at (-1,1.5);
 \coordinate (O) at (-1.33,1);
 \coordinate (U) at (-1.66, 0.5);
 \coordinate (V) at (1.66, 0.5);
  \coordinate (P) at (-1.33,-1);
  \coordinate (Q) at (-0.66,-2);
  \coordinate (R) at (-0.33,-2.5);
  \coordinate (S) at (-0.2,-2.7);
  \coordinate (T) at (-0.1,-2.85);
  %lines and colors
  \draw (A)--(B)--(D)--(C)--cycle;
  \draw (E)--(N);
   \draw (G)--(P);
  
 % \draw[dashed] (H)--(Q);
 % \draw[dashed] (I)--(R);
 % \draw[dashed] (L)--(S);
  %\draw[dashed] (M)--(T);
  \draw[fill=orange!20] (A)--(N)--(E);
   \draw (A)--(E);
  \draw[dashed, fill=blue!30] (N)--(E)--(V)--(U)--cycle;
  \draw (E)--(V);
  \draw (N)--(U);
  \draw[dashed] (O)--(F);
 % \draw[dashed, fill=blue!30] (O)--(F)--(V)--(U)--cycle;
 \draw[fill=orange!20] (U)--(V)--(B)--(G)--(P)--(C)--cycle;
 \draw[dashed](C)--(B);
  %\draw[dashed, fill=orange!20] (U)--(V)--(B)--(C)--cycle;
 \draw[fill=blue!30] (P)--(G)--(H)--(Q)--cycle;
  \draw[fill=blue!30] (Q)--(H)--(D)--cycle;
 %  \draw[dashed,fill=blue!20] (R)--(I)--(L)--(S)--cycle;
% \draw[dashed,fill=orange!10] (S)--(L)--(M)--(T);
 \node at (0,0.75) {\small $B6$};
  \node at (0,0.25) {\small $B5$};
  \node at (0,-0.5) {\small $B4$};
 \node at (0,-1.5) {\small $B3$};
 \node at (0,-2.25) {\small $B2$};
  \node at (0,-2.75) {\small $B1$};
  \node at (0,1.25) {\small $B7$};
 \node at (0,2.25) {\small $B8$};
 \draw[line width=0.40 mm,red] (N)--(E);
 \draw[line width=0.40 mm,red] (U)--(V);
 \draw[line width=0.40 mm,red] (P)--(G);
 \draw (R)--(I);
 \draw[thick, teal] (-1.2,1.6)--(1.2,1.6)--(1.2, 1.4)--(-1.2,1.4)--cycle;
 \node at (-1.5,1.5) {\color{teal}\small $D7$};
\draw[thick, teal] (-1.53,1.1)--(1.53,1.1)--(1.53, 0.9)--(-1.53,0.9)--cycle;
 \node at (-1.83,1) {\color{teal}\small $D6$};
\draw[thick, teal] (-1.86,0.6)--(1.86,0.6)--(1.86, 0.4)--(-1.86,0.4)--cycle;
 \node at (-2.16,0.5) {\color{teal}\small $D5$};
\draw[thick, teal] (-2.2,0.1)--(2.2,0.1)--(2.2, -0.1)--(-2.2,-0.1)--cycle;
 \node at (-2.5,0) {\color{teal}\small $D4$};
\draw[thick, teal] (-1.53,-1.1)--(1.53,-1.1)--(1.53, -0.9)--(-1.53,-0.9)--cycle;
 \node at (-1.83,-1) {\color{teal}\small $D3$};
\draw[thick, teal] (-0.86,-2.1)--(0.86,-2.1)--(0.86, -1.9)--(-0.86,-1.9)--cycle;
 \node at (-1.16,-2) {\color{teal}\small $D2$};
\draw[thick, teal] (-0.53,-2.4)--(0.53,-2.4)--(0.53, -2.6)--(-0.53,-2.6)--cycle;
 \node at (-0.83,-2.5) {\color{teal}\small $D1$};
%% Following is for debugging purposes so you can see where the points are
%% These are last so that they show up on top
%\foreach \xy in {A, B, C, D, E, F, G, H, I, L, M, N, O, P, Q, R, S, T, U, V}{
%  \node at (\xy) {\xy};
%}
\end{tikzpicture}
\caption{A visualization of the different interfaces and sets after the rescaling in Step~1 and under \eqref{eq: jump-part}.  The phase regions associated to $A$ and $B$ are colored in blue and orange, respectively. The interfaces associated to the sets $\mathcal{I}_y$ and $\mathcal{I}_u$ are highlighted with thick red and dashed black lines, respectively. The remaining interfaces correspond to the set $\mathcal{I}_{\mathcal{P}}$. The connected components of $\Omega\setminus J_u$ are indicated as  $\{B_i\}_{i=1}^{8}$, whereas the cylindrical sets $\{D_i\}_{i=1}^{7}$ around   the interfaces   (see \eqref{eq: shrinko}) are drawn in green.}
\label{fig:limiting-sets}
\end{figure}
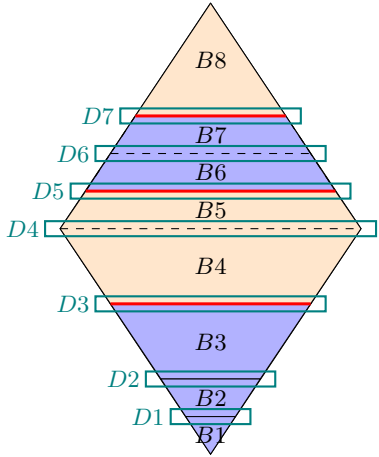

As explained in \cite[Proof of Proposition 5.1]{conti.schweizer2},  we can choose Lipschitz domains $\omega_i' \supset \supset \omega_i$  as well as $h>0$ such that  the sets $\partial \omega_i' \times (\alpha_i-h,\alpha_i + h)$ do not intersect ${\Omega}$, the different cylindrical sets $D_i:=\alpha_ie_d +D_{\omega_i',h}$ are pairwise disjoint,  and one   has  
\begin{align}\label{eq: shrinko}
(\omega_i' \times \lbrace \alpha_i \rbrace) \cap \Omega  = \omega_i \times \lbrace \alpha_i \rbrace.
\end{align}
We again refer to Figure \ref{fig:limiting-sets} for an illustration. We now distinguish the  cases of the three index sets  $\mathcal{I}_y$, $\mathcal{I}_{\mathcal{P}}$, and $\mathcal{I}_u$: first, we fix $i \in \mathcal{I}_y$. As the sets $D_i$ are pairwise disjoint, we get that on $D_i \cap \Omega$ the function  $y$  coincides with $y_0^+$ or $y_0^-$ up to translation (recall convention \eqref{eq: convention}). Thus, by Proposition \ref{lemma: local1}  we can find  a sequence $\{v_{\ep}^+\}_{\ep}$  or $\{v_{\ep}^-\}_{\ep}$  such that \eqref{eq:local1-1} holds up to translation, \eqref{eq:local1-3a}--\eqref{eq:local1-3} are satisfied,  and the sequence  $\{v_{\ep}^+\}_{\ep}$  or $\{v_{\ep}^-\}_{\ep}$ converges to $y$ in $L^1(D_i \cap \Omega;\R^d)$. 
 
For $i \in \mathcal{I}_{\mathcal{P}} \cup \mathcal{I}_u$,   we observe that $y$ coincides up to translation with $Mx$  on $D_i \cap \Omega$ for some $M \in \lbrace A,B\rbrace$.  If $i \in \mathcal{I}_{\mathcal{P}}$, we apply   \EEE    Proposition~\ref{lemma: local2} for the sequence $w_\eps = \sqrt{\eps}$.   If $i \in \mathcal{I}_{u}$, we apply       Proposition~\ref{lemma: local2} for $w_\eps = |b_{i+1} - b_i| \eps$, cf.\ \eqref{eq: def u strich}. In this context, we also note that by Proposition \ref{lemma: admissible-u}, the fact that $\nu_u = e_d$ on $J_u$, and the ordering of the sets $\lbrace B_i\rbrace_{i=1}^{I+1}$ (see Step 1), we have  $(b_{i+1}-b_i) \chi_{\lbrace x_d >0 \rbrace} = |b_{i+1} - b_i| y_{\rm dp}^M$ with $y_{\rm dp}^M$ defined in \eqref{eq: step functions}.   In both cases, we obtain  a sequence $\{v_{\ep}\}_{\ep} \subset H^2(D_i \cap\Omega;\R^d)$ such that \eqref{eq:local2-1} holds up to translation, \eqref{eq:local2-2} is fulfilled, and  $v_\eps \to y$  in measure  on $D_i \cap \Omega$. More precisely, \eqref{eq:local2-1} and the definition of $\lbrace w_\eps \rbrace_\eps$ in each case imply
\begin{align}\label{eq: Pandu-old}
{\rm (i)} &\ \ \eps^{-1}\big(v_\eps - y\big) \to (b_{i+1} - b_i) \chi_{\lbrace x_d \ge \alpha_i\rbrace} \ \  \  \ \,   \ \quad \quad  \, \text{ on $D_i \cap\Omega$ for $i \in \mathcal{I}_u$},\notag\\
{\rm (ii)} &\ \  {\eps}^{-1/2} \big(v_\eps  - y\big)  \to y_{\rm dp}^M(\cdot - \alpha_i e_d), \ \ \quad  \quad\quad \quad    \  \ \, \text{ on $D_i \cap\Omega$ for $i \in \mathcal{I}_{\mathcal{P}}$,}
\end{align}
where both properties hold in the sense of measure convergence. 

For convenience, we denote this local sequence by $\{v_\ep^i\}_\eps \subset H^2(D_i \cap \Omega;\R^d)$  for each $i=1,\ldots,I$. For later purposes,  by using Lemma \ref{lemma:comp-def}  we note that
\begin{align}\label{eq: pointwi-c}
v_\eps^i \to y \ \ \  \text{strongly in $H^1(D_i\cap \Omega;\R^d)$ for all $i=1,\ldots,I$.}
\end{align}

\emph{Step 4: Global construction of the recovery sequence.} Recall that $\Omega \setminus J_u = \bigcup_{i=1}^{I+1} B_i$, and let $B'_i := B_i \setminus \bigcup_{j=1}^I D_{j}$ for all $i=1,\ldots,I+1$.     Owing to Propositions \ref{lemma: local1} and \ref{lemma: local2}, using \eqref{eq: pointwi-c}, and arguing as in \cite[Proof of Proposition 3.5]{conti.schweizer},  we then choose iteratively isometries $\{I_i^\ep\}_{i=1}^I$ and $\{\hat{I}_i^\ep\}_{i=1}^{I+1}$ such that all isometries converge to the identity as $\ep\to 0$, and setting
$$\tilde{y}^\ep:=I_i^\ep\circ v^i_\ep\quad\text{on }D_i\cap \Omega\qquad\text{and}\qquad \tilde{y}^\ep:=\hat{I}_i^\ep\circ y\quad\text{on }B_i',$$ 
the maps $\tilde{y}^\ep\colon\Omega\to \R^d$ satisfy $\{\tilde{y}^{\ep}\}_\ep\subset H^2(\Omega;\R^d)$ and 
\begin{align}\label{eq: nochmal strong}
\tilde{y}^\ep\to y \text{ strongly in $H^1(\Omega;\R^d)$}.
\end{align}
Moreover, by \eqref{eq: Pandu-old} we get that    
\begin{align}\label{eq: Pandu}
{\rm (i)} &\ \ \eps^{-1}\big(\tilde{y}^\eps - I_i^\eps \circ y\big) \to (b_{i+1} - b_i) \chi_{\lbrace x_d \ge \alpha_i\rbrace}    \quad  \quad \quad \quad \quad \quad \quad  \text{if $i \in \mathcal{I}_u$},\notag\\
{\rm (ii)} &\ \ \eps^{-1/2}\big(\tilde{y}^\eps - I_i^\eps \circ y\big)\to  y_{\rm dp}^{M_i}(\cdot - \alpha_ie_d), \quad \quad \quad \quad \quad \ \ \,   \quad \quad   \, \  \text{if $i \in \mathcal{I}_{\mathcal{P}}$},
\end{align}
where both convergences hold in measure in $D_i \cap \Omega$,  and $M_i \in \lbrace A,B\rbrace$ is such that $\nabla y = M_i$ on $D_i \cap \Omega $ if $i\in \mathcal{I}_{\mathcal{P}}$. 
 Note that, up to translations, it is not restrictive to suppose that  $\fint_\Omega \tilde{y}^\eps\, {\rm d}x = 0$.  By construction we have
\begin{align}\label{eq: second gradient-same}
\nabla \tilde{y}^\ep\in SO(d)\{A,B\} \quad \quad \text{and} \quad \quad \nabla^2 \tilde{y}^{\eps} =0 \quad \quad \text{ on }    \bigcup\nolimits_{i=1}^{I+1} B_i' = \Omega \setminus \bigcup\nolimits_{i=1}^I D_i
\end{align}
for every $\ep$. Thus, again by the properties of the sequences $\{v_\ep^i\}_\ep$ obtained from  Propositions \ref{lemma: local1} and \ref{lemma: local2}, we get by \eqref{eq: jump-part}, \eqref{eq: diff-i}, \eqref{eq: shrinko}, and \eqref{eq: second gradient-same}  that
\begin{align*}
\limsup_{\eps \to 0}\mathcal{E}_\eps(\tilde{y}^\eps)& = \lim_{\eps \to 0}\, \sum\nolimits_{i=1}^I \mathcal{E}_\eps(v_\eps^i,D_i) \le K\sum\nolimits_{i \in \mathcal{I}_y} \mathcal{H}^{d-1}(\omega_i \times \lbrace \alpha_i\rbrace)+   2K  \sum\nolimits_{i \in \mathcal{I}_{\mathcal{P}} \cup \mathcal{I}_u} \mathcal{H}^{d-1}(\omega_i \times \lbrace \alpha_i\rbrace)\notag\\
& =   K \,  \mathcal{H}^{d-1}(J_{\nabla y})+ 2  K \, \mathcal{H}^{d-1}\Big(\big(J_u\cup\big( \bigcup\nolimits_j \partial P_j\cap \Omega\big)\big)\setminus J_{\nabla y}\Big).
\end{align*}
  By  \eqref{eq: limiting energy} we then conclude that  
\begin{align}\label{eq: energy-tilde}
\limsup_{\eps \to 0}\mathcal{E}_\eps(\tilde{y}^\eps) +  \int_{\Omega} \mathcal{Q}_{\rm lin}(\nabla y,\nabla u)\,{\rm d}x \le   \mathcal{E}_0^{\mathcal{A}} (y,u,\mathcal{P}).
\end{align}
So far, we have  constructed a sequence $\{\tilde{y}^\ep\}_\ep\subset H^2(\Omega;\R^d)$ satisfying  $\tilde{y}^\ep\to y$ strongly in $H^1(\Omega;\R^d)$  and \eqref{eq: energy-tilde}.   In view of \eqref{eq: energy-tilde}, we can apply Theorem \ref{thm:compactness} to obtain a limiting triple $(\tilde{y},\tilde{u},\tilde{\mathcal{P}})$ such that $\tilde{y}^\eps \to (\tilde{y},\tilde{u},\tilde{\mathcal{P}})$ in the sense of Definition \ref{def:convergence1}. We also  note  by \eqref{eq:comp-y}, \eqref{eq: nochmal strong}, and $\fint_\Omega \tilde{y}^\eps\,  {\rm d}x = 0$ that $\tilde{y} = y$. Then, by  \eqref{eq: jump-part}, \eqref{eq: energy-tilde}, and    Theorem \ref{thm:liminf} we find
\begin{align}\label{eq: great inequ}
\int_{\Omega}\mathcal{Q}_{\rm lin}(\nabla y,\nabla \tilde{u})\,{\rm d}x + 2K \mathcal{H}^{d-1}\Big(\big(J_{\tilde{u}}\cup\big( \bigcup\nolimits_j \partial \tilde{P}_j\cap \Omega\big)\big)\setminus J_{\nabla y}\Big) \le 2K \mathcal{H}^{d-1}(J_u \setminus J_{\nabla y}).
\end{align}
We write $\tilde{\mathcal{P}}= \lbrace \tilde{P}_j \rbrace_j$. We will prove that 
\begin{align}\label{eq: part-prop}
{\rm (i)} & \ \    \bigcup\nolimits_j \partial \tilde{P}_j \cap \Omega = \bigcup\nolimits_j \partial  {P}_j \cap \Omega,\notag \\
{\rm (ii)} & \ \      J_u = J_{\tilde{u}}\cup\big( \bigcup\nolimits_j \partial \tilde{P}_j\cap \Omega\big).
\end{align}
  In particular, (i)  yields $\mathcal{P} = \tilde{\mathcal{P}}$. We defer the proof of \eqref{eq: part-prop} to Step 7 below and now proceed with the construction of the recovery sequence. Note that in general $\tilde{u} \neq u$, and therefore we need to perturb   $\{\tilde{y}^\ep\}_\ep$ to obtain a sequence such that the rescaled displacement fields converge to $u$. To this end, for each $\eps>0$ we let
\begin{align}\label{eq: final def}
y^\eps := \tilde{y}^\eps + \eps u',
\end{align}
where $u'$ is the (smooth) function corresponding to $u$ defined in \eqref{eq: def u strich}. We now check that $y^\eps \to (y,u,\mathcal{P})$ in the sense of Definition \ref{def:convergence1} (Step 5) and then compute the energy of the sequence (Step 6).   

\emph{Step 5: Convergence to  the  limiting triple.} The goal of this step is to show that $y^\eps \to (y,u,\mathcal{P})$ in the sense of Definition \ref{def:convergence1}. Owing to \eqref{eq: jump-part} and recalling $y \in \mathcal{Y}_{\rm Id}(\Omega)$, we choose $M_j \in \lbrace A,B\rbrace$ such that $\nabla y = M_j$ on each component $P_j$.  Similarly to \eqref{eq: def u strich}, by the fact that $J_{\tilde{u}} \subset J_u$ (see \eqref{eq: part-prop}) and Proposition \ref{lemma: admissible-u}(i) applied for $\tilde{u}$ we find  $\lbrace\tilde{b}_i\rbrace_{i=1}^{I+1} \subset \R e_d$ such that  $\tilde{u}' := \tilde{u}-\sum_{i=1}^{I+1} \tilde{b}_i \chi_{B_i} \in H^1(\Omega;\R^d)$. By \eqref{eq: great inequ} and   \eqref{eq: part-prop}(ii) we get   $\int_{\Omega}\mathcal{Q}_{\rm lin}(\nabla y,\nabla \tilde{u})\,{\rm d}x = \int_{\Omega}\mathcal{Q}_{\rm lin}(\nabla y,\nabla \tilde{u}')\,{\rm d}x =0$. Note that $F \mapsto \mathcal{Q}_{\rm lin}(M,FM)$ is positive definite on $\mathbb{M}^{d \times d}_{\rm sym}$ by  \eqref{eq: only symmetric}. Therefore, by Korn's and Poincar\'e's inequalities  and  the fact that  $\tilde{u}'  \in H^1(\Omega;\R^d)$, it is elementary to check that $\tilde{u}' = \sum_j ( SM_j  x + \tilde{s}_j) \chi_{P_j}$ for some $S \in \mathbb{M}^{d \times d}_{\rm skew}$  and  suitable $\lbrace \tilde{s}_j\rbrace_j \subset \R^d$. (Note that the skew symmetric matrix $S$  here is necessarily independent of the set $P_j$ as $\tilde{u}'  \in H^1(\Omega;\R^d)$.) Consequently, we get 
\begin{align}\label{eq: def u tilde}
\tilde{u} =\sum\nolimits_j (S M_j  x + \tilde{s}_j) \chi_{P_j} +\sum\nolimits_{i=1}^{I+1}   \tilde{b}_i \chi_{B_i}.
\end{align}
Since $\lbrace B_i\rbrace_{i=1}^{I+1}$ is a refinement of the partition $\lbrace P_j\rbrace_j$ (see \eqref{eq: jump-part-new} and Figure \ref{fig:limiting-sets}), we find for each $i = 1,\ldots, I+1$ a corresponding index $j_i$ such that $B_i \subset P_{j_i}$. For $i \in \mathcal{I}_{\mathcal{P}} \cup \mathcal{I}_u$, this implies   
\begin{align}\label{eq: spefi-jump}
[\tilde{u}] = \tilde{b}_{i+1} + \tilde{s}_{j_{i+1}}  - (\tilde{b}_i +  \tilde{s}_{j_{i}} ) \   \text{ on } \omega_i \times \lbrace \alpha_i\rbrace = \partial B_{i} \cap \partial B_{i+1},
\end{align}
where $j_i = j_{i+1}$ if  $ i \in \mathcal{I}_u$, cf.\ \eqref{eq: diff-i}. Let $\{({R}^\ep,  {\mathcal{P}}^\ep, {\mathcal{M}}^\ep, {\mathcal{T}}^\ep)\}_\ep$ be the   quadruples given by Theorem \ref{thm:compactness}  for $\lbrace \tilde{y}^\eps \rbrace_\eps$ such that  \eqref{eq: partition property}--\eqref{eq:comp-grad-u}  hold.
In particular,  \eqref{eq:comp-sets}  and \eqref{eq: rescaled disp}--\eqref{eq:comp-u} yield  
 \begin{align}\label{eq: def u tilde-lim}
\eps^{-1}\big(\tilde{y}^{\ep}-   (R^{\ep} M_j^{\ep}\, x+t_j^{\ep})\big) \to \tilde{u} \quad \quad \text{in measure on $P_j$ for every $j$}.
\end{align}
Fix $i \in \mathcal{I}_u$ as defined in \eqref{eq: diff-i}, and recall that $D_i \cap \Omega \subset P_j$ for some index $j$. By \eqref{eq: Pandu}(i), the fact that $\nabla y = M_j$ on $P_j$, and  by  a compactness argument for affine mappings we find (for a subsequence, not relabeled) $\eps^{-1}(I^\eps \circ y - (R^{\ep} M_j^{\ep}\, x+t_j^{\ep}))  \to  S_i M_jx +  d_i$ pointwise  almost everywhere on $D_i \cap \Omega$ for suitable $S_i \in \mathbb{M}^{d\times d}_{\rm skew}$ and $d_i \in \R^d$. (We omit the details here and refer to the proof of Proposition \ref{prop:ex-coarsest-part} above for a very similar argument.) This along with \eqref{eq: Pandu}(i) and \eqref{eq: def u tilde-lim} yields  
\begin{align}\label{eq: strange estimate}
\tilde{u} =   (b_{i+1} - b_i) \chi_{\lbrace x_d \ge \alpha_i\rbrace} +   S_i M_jx + d_i  \quad \quad    \text{on  $D_{i}  \cap \Omega$}. 
\end{align}
  Then, in view of \eqref{eq: spefi-jump}  and the fact that $j_i = j_{i+1}$ for  $ i \in \mathcal{I}_u$, we check that $b_{i+1} - b_i = \tilde{b}_{i+1} - \tilde{b}_i$ for all $i \in \mathcal{I}_u$.   Therefore, by \eqref{eq: def u tilde} there exist  $\lbrace s_j \rbrace_j \subset \R^d$ such that 
\begin{align}\label{eq: reprisi}
\tilde{u} =  \sum\nolimits_j ( S M_j  x + {s}_j) \chi_{P_j} +\sum\nolimits_{i=1}^{I+1}   {b}_i \chi_{B_i}. 
\end{align}
We define $\bar{u} = u + \sum\nolimits_j (SM_j x + {s}_j) \chi_{P_j}$. We observe that $u-\bar{u} \in \mathcal{T}(y,\mathcal{P})$, and  by \eqref{eq: def u strich} and \eqref{eq: reprisi}  we note that
\begin{align}\label{eq: uutilde}
\bar{u} = \tilde{u} + u'. 
\end{align}
In view of \eqref{eq: final def},   \eqref{eq: def u tilde-lim}, and \eqref{eq: uutilde}, we find that 
\begin{align*}
\lim_{\eps \to 0}\eps^{-1} \big({y}^{\ep}-  ( R^{\ep} M_j^{\ep}\, x+t_j^{\ep})\big)  & =  \tilde{u}+ \lim_{\eps \to 0} \eps^{-1} (y^\eps - \tilde{y}^{\eps}) =  \tilde{u} + u' = \bar{u}
\end{align*}
in measure on $P_j$ for every $j$. In other words, by \eqref{eq:comp-sets} this means 
\begin{align}\label{eq: good conv-new}
u^\eps \to \bar{u} \quad  \text{in measure in $\Omega$},
\end{align}
where $\lbrace u^\eps\rbrace_\eps$ is defined in \eqref{eq: rescaled disp} with respect to $\lbrace y^\eps \rbrace_\eps $ and the quadruples $\{({R}^\ep, {\mathcal{P}}^\ep, {\mathcal{M}}^\ep,  {\mathcal{T}}^\ep)\}_\ep$. Now, we see that $(y,\mathcal{P},\bar{u})$ is an admissible limit for the quadruples  $\{({R}^\ep,  {\mathcal{P}}^\ep, {\mathcal{M}}^\ep, {\mathcal{T}}^\ep)\}_\ep$. Indeed, all properties apart from \eqref{eq:rigidity-compactness}, \eqref{eq:comp-y}, and \eqref{eq:comp-u}--\eqref{eq:comp-grad-u} follow from the corresponding properties of $\lbrace \tilde{y}^\eps \rbrace_\eps$. For \eqref{eq:rigidity-compactness} and \eqref{eq:comp-y} we additionally take \eqref{eq: final def} and $u' \in C^\infty(\overline{\Omega};\R^d)$ into account, and for \eqref{eq:comp-u} we use \eqref{eq: good conv-new}. Finally, to see \eqref{eq:comp-grad-u}, we use $\nabla \tilde{u}^\eps \to \nabla \tilde{u}$ in $L^2_{\rm loc}(\Omega;\mathbb{M}^{d \times d})$, where $\tilde{u}^\eps$ is defined in  \eqref{eq: rescaled disp} with respect to $\lbrace \tilde{y}^\eps \rbrace_\eps $, and $\nabla \bar{u} = \nabla \tilde{u} + \nabla  u$ by \eqref{eq: def u strich} and  \eqref{eq: uutilde}, as well as $\nabla u^\eps = \nabla \tilde{u}^\eps + \nabla  u$  by  \eqref{eq: rescaled disp},  \eqref{eq: def u strich}, and \eqref{eq: final def}.   Thus, $y^\eps \to (y,\bar{u},\mathcal{P})$ in the sense of Definition \ref{def:convergence1}. As  $u-\bar{u} \in \mathcal{T}(y,\mathcal{P})$, by Proposition \ref{prop:ex-coarsest-part}(iii) we then also find $y^\eps \to (y,u,\mathcal{P})$, as desired. This concludes this step of the proof.

\emph{Step 6: Convergence of   the  energies.} The goal of this step is to prove $\limsup_{\eps \to 0} \mathcal{E}_\eps(y^\eps) \le \mathcal{E}_0^{\mathcal{A}}(y,u,\mathcal{P})$.  To this end, fix $\delta, \theta>0$. Recalling the construction of the local recovery sequences in Step 3, it is not restrictive to suppose that
\begin{align}\label{eq: Dsmall}
\mathcal{L}^d\Big( \bigcup\nolimits_{i=1}^I D_i) \le \theta^2
\end{align}
by choosing the constant $h>0$ sufficiently small, see before equation \eqref{eq: shrinko}.  In view of   \eqref{eq: energy-tilde}, we see that we essentially need to estimate the difference of $\mathcal{E}_\eps(y^\eps)$ and $\mathcal{E}_\eps(\tilde{y}^\eps)$.  

First, we note that $\eps|\nabla u| \le \delta$ for $\eps$ small enough since $\nabla u \in C^\infty(\overline{\Omega};\M^{d\times d})$. Define $\Omega_\eps = \lbrace x\in \Omega \colon \, {\rm dist}(\nabla y^\eps,SO(d)\lbrace A,B \rbrace) <\delta  \rbrace$.  By  \eqref{eq: nonlinear energy}, Lemma \ref{lemma: limsup-W}, \eqref{eq: def u strich}, \eqref{eq: final def}, and a quadratic expansion we calculate  
\begin{align}\label{eq: first Taylor}
\mathcal{E}_\eps(y^\eps) \le  \mathcal{E}_\eps(\tilde{y}^\eps) + \frac{C_\delta}{\eps}\int_{\Omega} \sqrt{W(\nabla \tilde{y}^\eps)} |\nabla u| \, {\rm d}x + \int_{\Omega_\eps}  \frac{1}{2}D^2W(\nabla \tilde{y}^\eps) \nabla u:\nabla u\,{\rm d}x\EEE + \rho_\delta \int_{\Omega_\eps} |\nabla u|^2 \, {\rm d}x + \gamma_\eps, 
\end{align}
where $\rho_\delta$ and $C_\delta$ are the constants from Lemma \ref{lemma: limsup-W}, and $\gamma_\eps$ is defined by 
\begin{align*}
\gamma_\eps & := \eps^3 \int_\Omega 2\nabla^2 \tilde{y}^{\ep} : \nabla^2 u \,{\rm d}x + \eps^4 \int_\Omega |\nabla^2 u|^2\,{\rm d}x +   \bar{\eta}_{\eps,d}^2  \sum\nolimits_{1 \le \min\lbrace i,j\rbrace <d} \int_\Omega  \Big( 2\eps \partial^2_{ij} \tilde{y}^{\ep} \,  \partial^2_{ij} u  + \eps^2|\partial^2_{ij} u|^2 \Big) \,{\rm d}x.
\end{align*}
As $\mathcal{E}_\eps(\tilde{y}^\eps) \le C$ by \eqref{eq: energy-tilde} and $\nabla u \in C^\infty(\overline{\Omega};\mathbb{M}^{d\times d})$, the fact that $\lim_{\eps \to 0} \eps\bar{\eta}_{\eps,d} = 0$ (see \eqref{eq:alphad})  along with Young's inequality shows that $\lim_{\eps \to 0} \gamma_\eps= 0$. (More precisely, for the third term we use an estimate of the form    $\bar{\eta}_{\eps,d} ^2\eps\partial^2_{ij} \tilde{y}^{\ep} \, \partial^2_{ij} u \le \frac{\bar{\eta}_{\eps,d} ^2}{2\lambda_\eps^2} |\partial^2_{ij} \tilde{y}^{\ep}|^2 + \frac{1}{2} \eps^2\bar{\eta}_{\eps,d}^2 \lambda_\eps^2 |\partial^2_{ij} u|^2$ for a sequence $\lbrace\lambda_\eps\rbrace_\eps$ such that $\lambda_\eps \to \infty$ and $\lambda_\eps\eps\bar{\eta}_{\eps,d} \to 0$.) 
%$\eps\partial^2_{ij} \tilde{y}^{\ep} : \partial^2_{ij} u \le \frac{1}{2\lambda^2_\eps} |\partial^2_{ij} \tilde{y}^{\ep}|^2 + \frac{\lambda^2_\eps}{2} \eps^2 |\partial^2_{ij} u|^2$ for $\lambda_\eps \to \infty$ such that $\lim_{\eps \to 0} \eps\lambda_\eps\bar{\eta}_{\eps,d} = 0$.) 
Moreover,  for the second term in \eqref{eq: first Taylor} we compute by   H\"older's inequality, \eqref{eq: second gradient-same}, H3., and \eqref{eq: Dsmall}
\begin{align}\label{eq: second Taylor}
\frac{1}{\eps}\int_{\Omega} \sqrt{W(\nabla \tilde{y}^\eps)}|\nabla u| \, {\rm d}x & = \frac{1}{\eps}\int_{\bigcup_i D_i} \sqrt{W(\nabla \tilde{y}^\eps) }| \nabla u| \, {\rm d}x  \le \frac{1}{\eps} \Big( \int_\Omega W(\nabla \tilde{y}^\eps) \, {\rm d}x \Big)^{1/2}  \Vert \nabla u \Vert_{L^2(\bigcup_i D_i)}\notag \\ 
&\le C\Vert \nabla u \Vert_{L^\infty(\Omega)} \Big(\mathcal{L}^d\big(\bigcup\nolimits_i D_i\big) \Big)^{1/2} \le C\theta,
\end{align} 
where in the penultimate step we have also used  the fact that  $\int_\Omega W(\nabla \tilde{y}^\eps) \, {\rm d}x \le C\eps^2$ by \eqref{eq: energy-tilde}. Then, from \eqref{eq: nochmal strong}, \eqref{eq: first Taylor}, \eqref{eq: second Taylor}, $\gamma_\eps \to 0$, the regularity of $W$, and  the  dominated convergence theorem we obtain 
$$\limsup_{\eps \to 0}\mathcal{E}_\eps ({y}^\eps) \le \limsup_{\eps \to 0}\mathcal{E}_\eps(\tilde{y}^\eps) +   \int_{\Omega}\mathcal{Q}_{\rm lin}(\nabla y(x),\nabla u(x))\,{\rm d}x + C C_\delta\theta + \rho_\delta \Vert\nabla u \Vert^2_{L^2(\Omega)},$$
where $\mathcal{Q}_{\rm lin}$ is defined in \eqref{eq:def-Q}.  In view of \eqref{eq: energy-tilde}, this yields
$$\limsup_{\eps \to 0}\mathcal{E}_\eps ({y}^\eps) \le   \mathcal{E}_0^{\mathcal{A}} (y,u,\mathcal{P})  + CC_\delta\theta + \rho_\delta \Vert\nabla u \Vert^2_{L^2(\Omega)}.$$
The limsup inequality now  follows by first letting $\theta\to 0$, then $\delta \to 0$, and by recalling the comment in   \eqref{eq: delta-en}.

\emph{Step 7: Proof of \eqref{eq: part-prop}.} To conclude the proof, it remains to show the technical property \eqref{eq: part-prop}. We observe that it suffices to prove the estimates
\begin{align}\label{eq: three proofs} 
{\rm (i)} & \ \   J_u  \setminus J_{\nabla y}  \subset \big(J_{\tilde{u}}\cup\big( \bigcup\nolimits_j \partial \tilde{P}_j\cap \Omega\big)\big)\setminus J_{\nabla y},\notag\\
{\rm (ii)}  &\ \ \bigcup\nolimits_j (\partial {P}_j \cap \Omega) \setminus J_{\nabla y } \subset \bigcup\nolimits_j (\partial  \tilde{P}_j \cap \Omega) \setminus J_{\nabla y},\notag \\ 
{\rm (iii)} & \ \ \bigcup\nolimits_j (\partial {P}_j \cap \Omega) \setminus J_{\nabla y } \supset \bigcup\nolimits_j (\partial  \tilde{P}_j \cap \Omega) \setminus J_{\nabla y}.
\end{align}
In fact, \eqref{eq: three proofs}(ii),(iii) along with Definition \ref{def:convergence2} show  \eqref{eq: part-prop}(i). By  \eqref{eq: three proofs}(i) and Definition \ref{def:convergence2} we obtain one inclusion in  \eqref{eq: part-prop}(ii). The other one then follows from \eqref{eq: great inequ}.  

Let us now show \eqref{eq: three proofs} by contradiction. First,  if \eqref{eq: three proofs}(i) were wrong, we would find a cylindrical set  $\alpha_ie_d + D_{\omega_i,l}$ for $i \in \mathcal{I}_{\mathcal{P}} \cup \mathcal{I}_{u}$ (see \eqref{eq: diff-i}) and $l>0$ sufficiently small and some component $\tilde{P}_j$ of $\tilde{\mathcal{P}}$ such that $ (\alpha_ie_d +D_{\omega_i,l})  \cap \Omega \subset \tilde{P}_j$ and $(\alpha_ie_d +D_{\omega_i,l})  \cap J_{\tilde{u}} = \emptyset$. By Theorem \ref{thm:compactness} applied for $\lbrace \tilde{y}^\eps\rbrace_\eps$,  we then get  (see also  \eqref{eq: def u tilde-lim})  
\begin{align}\label{eq: contra}
\eps^{-1}\big(\tilde{y}^{\ep}- ( R^{\ep} M x+t^{\ep}_{j})\big) \to \tilde{u}  \quad \quad \text{in measure on $(\alpha_ie_d + D_{\omega_i,l}) \cap \tilde{P}_j$}, 
\end{align} 
where $R^\eps \to {\rm Id}$, $\lbrace t_j^\eps \rbrace_\eps \subset \R^d$, and $M$ such that $\nabla y \equiv M$ on $\tilde{P}_j$. In view of  the fact that $(\alpha_ie_d +D_{\omega_i,l})  \cap J_{\tilde{u}} = \emptyset$,  we obtain a contradiction to \eqref{eq: Pandu}(i),(ii).  On the other hand, if \eqref{eq: three proofs}(ii) were wrong, we would  find $i \in \mathcal{I}_{\mathcal{P}}$ such that  \eqref{eq: contra} holds. But then  \eqref{eq: contra} and the fact that $\tilde{u}$ is finite a.e.\  contradict \eqref{eq: Pandu}(ii).

Finally, suppose that \eqref{eq: three proofs}(iii) were wrong. Then, there would exist  a cylindrical set $D:= \alpha e_d + D_{\omega,l}$ which  intersects two components $\tilde{P}_{j_1}$ and $\tilde{P}_{j_2}$, but not $\bigcup_{i \in \mathcal{I}_\mathcal{P}}(\omega_i \times \lbrace \alpha_i\rbrace)$, i.e., there exists $P_j$ such that  $D \cap \Omega \subset P_j$.  Similarly to \eqref{eq: contra}, we find  sequences $\lbrace t^\eps_{j_1} \rbrace_\eps, \lbrace t^\eps_{j_2} \rbrace_\eps \subset \R^d$ from the sequence $\lbrace \mathcal{T}^\eps\rbrace_\eps$ given in Theorem \ref{thm:compactness} such that 
\begin{align}\label{eq: contra2}
\eps^{-1}\big(\tilde{y}^{\ep}- ( R^{\ep} M x+t^{\ep}_{j_k})\big) \to \tilde{u}  \quad \quad \text{in measure on $D \cap \tilde{P}_{j_k}$ for  $k=1,2$}, 
\end{align}
where $M$ is such that  $\nabla y \equiv M$ on ${P}_j$. On the other hand, we find a sequence of  isometries $\lbrace I^\eps\rbrace_\eps$ converging to the identity as $\eps \to 0$ such that $\eps^{-1} (\tilde{y}^\eps - I^\eps\circ y)$ converges to a finite value a.e.\ on $\Omega \cap D$ due to \eqref{eq: Pandu}--\eqref{eq: second gradient-same}, where  we exploit that   $D$ does not intersect $\bigcup_{i \in \mathcal{I}_\mathcal{P}}(\omega_i \times \lbrace \alpha_i\rbrace)$. This along with \eqref{eq: contra2} shows $\limsup_{\eps \to 0} |(t^\eps_{j_1} - t^\eps_{j_2})/\eps| <+\infty$. This, however, contradicts \eqref{eq: toinfty}. This argument concludes the proof of \eqref{eq: three proofs},  and thus we have completed the proof of \eqref{eq: part-prop}. 
\end{proof}

 We conclude this subsection by showing that $W$ satisfies the estimates in Lemma \ref{lemma: limsup-W}.

\begin{proof}[Proof of Lemma \ref{lemma: limsup-W}]
Fix $0< \delta \le \delta_W /2$. We start with (i). By a Taylor expansion,  by assumption  H5.,  and the fact that $D^2W$ is uniformly continuous on $\overline{\mathcal{V}_{\delta}}$ we find  that  for  any  $F \in \mathcal{V}_\delta$,  and  $G \in B_{\delta}(0)$  there holds
$$ W(F +  G)  \le W(F) + DW(F):G +  \frac{1}{2}D^2W(F)\,G : G    + \rho_\delta|G|^2,  $$
where $\rho_\delta \to 0$ as $\delta \to 0$. Letting $R_F \in SO(d)\lbrace A,B\rbrace$  be such that  $|R_F - F| = {\rm dist}(F,SO(d)\lbrace A,B\rbrace)$,  assumptions H3.\ and H4., together with  the fact that  $DW$ is Lipschitz on $\mathcal{V}_\delta$ and $DW(R_F)=0$, give
$$|DW(F)| \le |DW(R_F)| + C|F-R_F| = C{\rm dist}(F,SO(d)\lbrace A,B\rbrace) \le (C/ \sqrt{c_1}) \sqrt{W(F)}$$
for a constant $C$ only depending on $W$.  By the Cauchy-Schwarz inequality this concludes the proof of (i).  To prove (ii), we exploit H7.\ to find for $F \in \mathbb{M}^{d \times d}$ and $G \in B_\delta(0)$ that 
$$W(F + G) \le W(F) + c_3 (1+ 2 |F| + \delta) |G|. $$
For $F \in \mathbb{M}^{d \times d} \setminus \mathcal{V}_\delta$ one has $ \max  \{\delta,\,1+ 2 |F|\} \le C_\delta {\rm dist}(F,SO(d)\lbrace A,B\rbrace)$ for a sufficiently large constant depending on $\delta$. The desired estimate follows then again from H4. 
\end{proof}

\subsection{Properties of the double-profile energy}\label{sec: cell-formula}

In this subsection we analyze the double-profile energy functional introduced in \eqref{eq: k2-intro} and address  its relation to $K$ and $K_{\rm dp}^M$. In particular, we prove Proposition \ref{eq: KundKdo-new}. We start by stating the results of this subsection.

\begin{proposition}[Properties of the double-profile energy function]
\label{prop:cell-form-new}
The functions $\mathcal{F}^M_{\rm dp}$, $M \in \lbrace A,B\rbrace$,   satisfy for all $h>0$ and all  open, bounded   sets  $\omega \subset \R^{d-1}$ with $\mathcal{H}^{d-1}(\partial\omega)=0$: 
\begin{itemize}
\item[{\rm (i)}] $\mathcal{F}^M_{\rm dp}(\alpha \omega;\alpha h) \ge  \alpha^{d-1}\mathcal{F}^M_{\rm dp}(\omega;h) $ for all $0 < \alpha < 1$.
\item[{\rm (ii)}] $\mathcal{F}^M_{\rm dp}(\omega;h)= \mathcal{H}^{d-1}(\omega) \, \mathcal{F}^M_{\rm dp}(Q';h)$, where $Q' := (-\frac{1}{2},\frac{1}{2})^{d-1}$.
\item[{\rm (iii)}] $\mathcal{F}^M_{\rm dp}(\omega;h)= \mathcal{F}^M_{\rm dp}(\omega; 1)$. 
\end{itemize}
\end{proposition}

 We now address the relationship between the optimal-profile and double-profile energies.

\begin{proposition}[Relation between $K$ and $K^M_{\rm dp}$]\label{eq: KundKdo}
 There holds $K^M_{\rm dp} \ge \mathcal{F}^M_{\rm dp}(Q',1) \ge 2K$   for $M\in\lbrace A,B \rbrace$,  where $ Q'  = (-\frac{1}{2},\frac{1}{2})^{d-1}$, and  $K$, $K^M_{\rm dp }$ are defined   in \eqref{eq: our-k1} and  \eqref{eq: our-k2}, respectively.
\end{proposition}

 Finally, if $2K=K^M_{\rm dp}$ for $M\in \{A,B\}$,  in the definition \eqref{eq: our-k2} one can replace cubes by general Lipschitz domains, and the formula holds for every $h>0$ and general $\lbrace w_\eps \rbrace_\eps \in \mathcal{W}_d$.

\begin{proposition}[Characterization of  $K^M_{\rm dp}$]
\label{prop:kMdo}
Let $M \in \lbrace A,B\rbrace$, and   suppose that the constant $K^M_{\rm dp}$ defined in \eqref{eq: our-k2} satisfies $K^M_{\rm dp}  =  2K$. Then there holds
\begin{align}\label{eq: added-equ}
\inf\Big\{\limsup_{\ep\to 0} \mathcal{E}_{\ep}(y^{\ep},D_{\omega,h})\colon \  \frac{y^\eps -  Mx}{w_\eps} \to y_{\rm dp}^M  \text{ \rm  in measure  in $D_{\omega,h}$ as $\eps \to 0$}\Big\} = K^M_{\rm dp}  \, \mathcal{H}^{d-1}(\omega),
\end{align}
for every Lipschitz domain $\omega\subset \mathbb{R}^{d-1}$,  $h>0$, and  $ \lbrace w_\eps \rbrace_\eps \in \mathcal{W}_d$.
\end{proposition}

We point out that Propositions \ref{prop:cell-form-new} and \ref{eq: KundKdo}  directly imply Proposition \ref{eq: KundKdo-new}. Proposition \ref{prop:kMdo} will be instrumental in Subsection \ref{subs:local} below for the proof of Proposition \ref{lemma: local2}.  We prove it here as it completes the characterization of the relation between $K^M_{\rm dp}$, $M\in \{A,B\}$, and the double-profile energy  functions.  We now proceed with the proofs of Propositions \ref{prop:cell-form-new}, \ref{eq: KundKdo}, and  \ref{prop:kMdo}.  As a preparation, we start with a standard rescaling argument which we will use several times.

\begin{remark}\label{sub:2-prof}
{\normalfont For a configuration $y \in H^2(\alpha D_{\omega,h};\R^d)$ and $0 < \alpha <1$, we define $\bar{y} \in H^2(D_{\omega,h};\R^d)$ by $\bar{y}(x) =  y(\alpha x)/\alpha$. We observe that  $\nabla \bar{y}(x) = \nabla y (\alpha x)$ and  $\nabla^2 \bar{y}(x) = \alpha  \nabla^2y(\alpha x)$ for all $x \in D_{\omega,h}$.  Since $\lbrace  \bar{\eta}_{\ep,d} \rbrace_\eps$ is increasing as $\eps \to 0$ (see \eqref{eq:alphad}), we get $\bar{\eta}^2_{\sqrt{\alpha}\eps,d} \ge \alpha\bar{\eta}^2_{\ep,d}$. Thus, we obtain by \eqref{eq: nonlinear energy}--\eqref{eq: specific energy choice}  
\begin{align}\label{eq: transformation preparation}
\mathcal{E}_{\sqrt{\alpha}\ep}(y,\alpha D_{\omega,h})  &\ge  \frac{1}{\alpha\ep^2}\int_{\alpha D_{\omega,h}}W(\nabla y)\,{\rm d}x+ \alpha\ep^2\int_{\alpha D_{\omega,h}}|\nabla^2 y|^2\,{\rm d}x+ \alpha\bar{\eta}_{\ep,d}^2  \int_{\alpha D_{\omega,h}}(|\nabla^2 y|^2-|\partial^2_{dd} y|^2)\,{\rm d}x \notag \\
& =  \frac{\alpha^{d-1}}{\ep^2}\int_{D_{\omega,h}}W(\nabla \bar{y})\,{\rm d}x+\alpha^{d-1} \ep^2\int_{D_{\omega,h}}|\nabla^2 \bar{y}|^2\,{\rm d}x+ \alpha^{d-1} \bar{\eta}_{\ep,d}^2 \int_{D_{\omega,h}}(|\nabla^2 \bar{y}|^2-|\partial^2_{dd} \bar{y}|^2)\,{\rm d}x \notag\\ 
&=  \alpha^{d-1}  \mathcal{E}_{\ep}(\bar{y},D_{\omega,h}). 
\end{align}
} 
\end{remark}

\begin{proof}[Proof of Proposition \ref{prop:cell-form-new}]
We prove {\rm (i)}. Let $0< \alpha < 1$.  By \eqref{eq: k2-intro}, for a given $\delta>0$, we find sequences $\lbrace \eps_i\rbrace_{i}$ with $\eps_i \to 0$,  $\lbrace w_{i}\rbrace_i \in \mathcal{W}_d$, $u \in \mathcal{U}_{\rm dp}(\alpha D_{\omega,h})$, and   $\lbrace y^{i}\rbrace_i \subset H^2(\alpha D_{\omega,h};\R^d)$ with  $w_{i}^{-1}(y^{i} -  Mx) \to  u$    in measure in   $\alpha D_{\omega,h}$ such that 
\begin{align}\label{eq: quasi opt}
\liminf_{i \to \infty} \mathcal{E}_{\sqrt{\alpha}\eps_i}(y^{i}, \alpha D_{\omega,h}) \le \mathcal{F}^M_{\rm dp}(\alpha\omega; \alpha h) + \delta.
\end{align}
Let $\lbrace \bar{y}^{i}\rbrace_i\subset H^2(D_{\omega,h};\R^d)$ be the rescaled functions defined before \eqref{eq: transformation preparation}. Note that $\alpha w_{i}^{-1}(\bar{y}^{i} -  Mx) = w_{i}^{-1}(y^{i}(\alpha x) -  M (\alpha x)) \to \alpha\bar{u}$  in measure in  $D_{\omega,h}$, where $\bar{u}(x) = u(\alpha x)/\alpha$ for $x \in D_{\omega,h}$. Then the definition of $\mathcal{F}^M_{\rm dp}$,  \eqref{eq: transformation preparation}, and \eqref{eq: quasi opt} imply
$$\delta + \mathcal{F}^M_{\rm dp}(\alpha \omega; \alpha h) \ge  \liminf_{i \to \infty} \mathcal{E}_{\sqrt{\alpha}\eps_i}(y^{i}, \alpha D_{\omega,h}) \ge \alpha^{d-1}\liminf_{i \to \infty} \mathcal{E}_{\ep_i}(\bar{y}^{i}, D_{\omega,h}) \ge \alpha^{d-1}\mathcal{F}^M_{\rm dp}(\omega;h).   $$
Since $\delta >0$ was arbitrary, {\rm (i)} follows. 

The proof of (ii) and (iii) is exactly as in \cite[Propostion 4.6]{davoli.friedrich} and we refer therein for details. (See also \cite[Lemma 4.3]{conti.fonseca.leoni} for similar arguments.) 
\end{proof}

  We now move to the proof of Proposition \ref{eq: KundKdo}. We first state two technical lemmas.    
Recall the definition of $y^+_0$ and   $y^-_0$ below  \eqref{eq: conti-schweizer-k}.  

\begin{lemma}[Lower energy bound]\label{lemma: lower energy bound}
Let $\lbrace \eps_i\rbrace_i$ be an infinitesimal sequence, and  $\lbrace \tau_i \rbrace_i \subset \R$ be a bounded sequence with $\eps_i/\sqrt{\tau_i} \to 0$.  Let $\omega \subset \R^{d-1}$ be a bounded Lipschitz domain.  Suppose that there exists a sequence $\lbrace v^i\rbrace_i$ with $v^i \in H^2(D_{\omega, \tau_i};\R^d)$, and 
\begin{align}\label{eq: to confirm-old1}
\tau_i^{-1}\Vert \nabla v^i - \nabla y_0^+\Vert^2_{L^2(D_{\omega, \tau_i})} \to 0.
\end{align} 
Then 
\begin{align}\label{eq: to confirm-old2}
\liminf_{i \to \infty} \mathcal{E}_{\eps_i}(v^i,D_{\omega,\tau_i})\ge K\mathcal{H}^{d-1}(\omega),
\end{align}
where $K$ is the constant from \eqref{eq: our-k1}.  
\end{lemma}

\begin{lemma}[Zooming to two interfaces]\label{lemma: intefacefind2}
Let $\lbrace \eps_i\rbrace_i$ be an infinitesimal sequence. Let $\mathcal{Q}' \subset \R^{d-1}$ be a cube and  let $h>0$.  Let $M \in \lbrace A,B \rbrace$.   For every $i\in \mathbb{N}$, let $y^{i} \in H^2( D_{\mathcal{Q}',h};\R^d)$  with $\mathcal{E}_{\eps_i}(y^{i}, D_{\mathcal{Q}', h}) \le C_0 <+\infty$, let $\lbrace \tau_i \rbrace_i \in \mathcal{W}_d$, let $u \in \mathcal{U}_{\rm dp}(D_{\mathcal{Q}',h})$, and   assume that  
\begin{align}\label{eq: converg assu}
 \frac{y^i -  Mx}{\tau_i} \to u \text{ in measure in  $D_{\mathcal{Q}',h}$ as $i \to \infty$}.
\end{align} 
Then, there exist   $\mu>0$, sequences $\lbrace \alpha^1_i \rbrace_i, \lbrace \alpha^2_i \rbrace_i  \subset \R$ such that $D^j_i := \alpha^j_i e_d+D_{\mathcal{Q}',\mu\tau_i}$, $j=1,2$, satisfy $D^1_i,D^2_i \subset D_{\mathcal{Q}',h}$ and  $D^1_i \cap D^2_i = \emptyset$, and there exists a sequence  of isometries $\lbrace I_i\rbrace_i$ such that the  maps  $ v^i  \in H^2(D^1_i \cup D^2_i;\R^d)$, defined by 
\begin{align}\label{eq: vi}
v^i(x) = I_i \circ  y^{i}  (x) \ \ \ \text{ for every }\, x \in  D^1_i \cup D^2_i, 
\end{align}
satisfy, up to a subsequence,  for $j=1,2$ that
\begin{align}\label{eq: to confirm}
\min\Big\{ \tau_i^{-1}\Vert \nabla v^i(\cdot + \alpha^j_i e_d) - \nabla y^+_0\Vert^2_{L^2(D_{\mathcal{Q}',\mu \tau_i})}\, , \ \tau_i^{-1} \Vert \nabla v^i (\cdot + \alpha^j_i e_d) - \nabla y^-_0\Vert^2_{L^2(D_{\mathcal{Q}',\mu \tau_i})} \Big\} \to 0.
\end{align} 
\end{lemma}

The lemma states that one finds two cylindrical sets with height $\mu \tau_i$ such that each `contains an interface', i.e.,  asymptotically a big portion of $D^j_i \cap \lbrace x_d \ge \alpha^j_i \rbrace$ and  $D^j_i \cap \lbrace x_d \le \alpha^j_i \rbrace$, respectively, is contained in the $A$ and $B$-phase region, respectively,  cf.\ Figure \ref{fig:zoom-in}. 
%%%%%%%%%%%%%%% Figure: zoom-to-interfaces %%%%%%%%%%%%%%%%%%%%%%%%%%%%%%%%%%%%
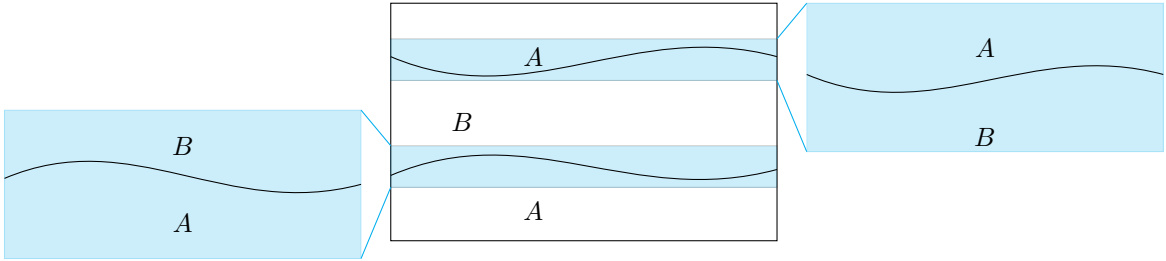
\begin{figure}[h]
\centering
\begin{tikzpicture}[scale=0.79]
% e_1 points right, e_2 points up, 
% bog rectangle
\coordinate (A) at (-3,3);
 \coordinate (B) at (3.5,3);
 \coordinate (E) at (3.5,2.1);
 \coordinate (F) at (3.5,1.4);
 \coordinate (G) at (3.5,0.6);
  \coordinate (O) at (-3,2.2);
 \coordinate (P) at (-3,1.4);
 \coordinate (Q) at (-3,0.6);
 \coordinate (S) at (-3,-1); 
 \coordinate (T) at (3.5,-1); 

  %lines and colors

 \draw (A)--(B)--(T)--(S)--cycle;
 \draw[fill=cyan, opacity=0.2] (-3,2.4)--(3.5,2.4)--(3.5,1.7)--(-3,1.7)--cycle;
 \draw [black] (-3,2.1)  to [out=-23,in=165] (3.5,2.1);
 
  \draw[fill=cyan, opacity=0.2] (-3,0.6)--(3.5,0.6)--(3.5,-0.1)--(-3,-0.1)--cycle;
 \draw [black] (-3,0.1)  to [out=23,in=-165] (3.5,0.2);
 
\node at (-1.8,1) {$B$};
\node at  (-0.6,2.1) {$A$};
\node at  (-0.6,-0.5) {$A$};

%first small rectangle
 \coordinate (C) at (4,3);
  \coordinate (D) at (10,3);
   \coordinate (H) at (4,0.5);
    \coordinate (I) at (10,0.5);
    \draw[cyan, fill=cyan, opacity=0.2] (C)--(D)--(I)--(H)--cycle;
    \draw[cyan] (3.5,2.4)--(C);
     \draw[cyan] (3.5,1.7)--(H);
     \draw[black] (4,1.8)  to [out=-23,in=165] (10,1.8);
     \node at (7,2.25) {$A$};
      \node at (7,0.75) {$B$};
      
%second small rectangle
 \coordinate (U) at (-3.5,1.2);
  \coordinate (V) at (-3.5,-1.3);
   \coordinate (W) at (-9.5,-1.3);
    \coordinate (X) at (-9.5,1.2);
    \draw[cyan, fill=cyan, opacity=0.2] (U)--(V)--(W)--(X)--cycle;
    \draw[cyan] (-3,0.6)--(U);
     \draw[cyan] (-3,-0.1)--(V);
    \draw[black] (-9.5,0.05)  to [out=23,in=-165] (-3.5,-0.05);
     \node at (-6.5,0.6) {$B$};
    \node at (-6.5,-0.7) {$A$};      

%% Following is for debugging purposes so you can see where the points are
%% These are last so that they show up on top
%\foreach \xy in {A, B, C, D, E, F, G, H, I, O, P, Q, R, S, T,U,V}{
% \node at (\xy) {\xy};
%}

\end{tikzpicture}
\caption{By `zooming in' one can identify two regions in which phase transitions occur: the interfaces between the $A$ and $B$-phase regions become asymptotically flat as $i\to \infty$.}
\label{fig:zoom-in}
\end{figure}

 Loosely speaking, the result shows that, under assumption \eqref{eq: converg assu}, there are at least two interfaces and  the interfaces between the $A$ and $B$-phase  regions  become asymptotically flat, where the nonflatness is asymptotically small compared to the sequence $\lbrace \tau_i \rbrace_i$. An analogous result for a single interface between the $A$ and $B$-phase region has been derived in \cite[Lemma 4.9]{davoli.friedrich}.

We postpone the proofs of these two lemmas  and proceed with the proof of Proposition \ref{eq: KundKdo}.

\begin{proof}[Proof of Proposition \ref{eq: KundKdo}]
Let $M \in \lbrace A,B\rbrace$. First, the inequality $K^M_{\rm dp} \ge \mathcal{F}^M_{\rm dp}(Q',1)$ follows immediately from the definitions in \eqref{eq: our-k2} and  \eqref{eq: k2-intro}. We now show  $\mathcal{F}^M_{\rm dp}(Q',1) \ge 2K$. We again let $Q = (-\frac{1}{2},\frac{1}{2})^d$. Given $\delta>0$, we choose sequences $\lbrace \eps_i\rbrace_i$, $\lbrace w_i\rbrace_i \in \mathcal{W}_d$, $u \in \mathcal{U}_{\rm dp}(Q)$, and  $\lbrace y^i \rbrace_i \subset H^2(Q;\R^d)$ such that
$w_i^{-1} (y^i - Mx) \to u$ in measure in $Q$, and
\begin{align}\label{eq: relation1}
\limsup_{i \to \infty} \mathcal{E}_{\eps_i} (y^i, Q) \le \mathcal{F}^M_{\rm dp}(Q',\tfrac{1}{2}) + \delta = \mathcal{F}^M_{\rm dp}(Q',1) + \delta,
\end{align}
where the last step follows from Proposition \ref{prop:cell-form-new}(iii). By Lemma \ref{lemma: intefacefind2} applied for $\mathcal{Q}' = Q'$, $h=1/2$, and $\tau_i = w_i$ we find $\mu>0$ and pairwise disjoint sets $D^j_i := \alpha^j_i e_d+D_{Q',\mu w_i}$, $j=1,2$, with $D^1_i,D^2_i \subset Q$, and isometries $\lbrace I_i\rbrace_i$ such that the  maps  $v^i \in H^2(D^1_i \cup D^2_i;\R^d)$, defined by 
$v^i(x) = I_i \circ  y^i  (x)$ for $x \in  D^1_i \cup D^2_i$ satisfy \eqref{eq: to confirm} (after extraction of a subsequence). Possibly after a transformation of the form $x \mapsto -v^i(-x)$, we may suppose that  $w_i^{-1}\Vert \nabla v^i(\cdot + \alpha^j_i e_d) - \nabla y^+_0\Vert^2_{L^2(D_{Q',\mu w_i})} \to 0 $ for $j=1,2$. Then H2.\ and Lemma \ref{lemma: lower energy bound} for $\tau_i = w_i$ (note that $\eps_i/\sqrt{\tau_i} \to 0$ by \eqref{ew: W sequence}) imply
\begin{align*}
\liminf_{i \to \infty} \mathcal{E}_{\eps_i} (y^i, Q) \ge \sum\nolimits_{j=1,2} \liminf_{i \to \infty} \mathcal{E}_{\eps_i}(v^i(\cdot + \alpha^j_i e_d), D_{Q', \mu w_i} ) \ge 2K.
\end{align*}  
This along with \eqref{eq: relation1} and the fact that $\delta>0$ was arbitrary concludes the proof. 
\end{proof}

We continue with the proofs of  Lemma \ref{lemma: lower energy bound} and Lemma \ref{lemma: intefacefind2}.

\begin{proof}[Proof of Lemma \ref{lemma: lower energy bound}]
First, suppose that $\tau_i \ge h>0$ for all $i \in \N$ for some $h>0$. Then, up to translations  we have $v^i \to y_0^+$ in $L^1(D_{\omega, {h }};\R^d)$, and we immediately get 
$$\liminf_{i \to \infty} \mathcal{E}_{\eps_i}(v^i,D_{\omega,\tau_i}) \ge \liminf_{i \to \infty} \mathcal{E}_{\eps_i}(v^i,D_{\omega,h})\ge \mathcal{F}(\omega;h)$$
 by \eqref{eq: k-intro}. The result now follows from Proposition \ref{prop:cell-form}.

We can therefore concentrate on the case $\tau_i \to 0$. We prove the statement first for $\omega = \mathcal{Q}'$, where $\mathcal{Q}' \subset \R^{d-1}$ is a cube. For notational convenience we set $\gamma_i := \tau^{-1}_i$.  We  define $y^i \in H^2(D_{\gamma_i\mathcal{Q}',1};\R^d)$ by  $y^i(x) =   v^i(\tau_i x) / \tau_i$.    By using \eqref{eq: transformation preparation} with $\alpha_i = \tau_i$, we get
\begin{align}\label{eq: l-e-b1}
\mathcal{E}_{\ep_i}(v^i, D_{\mathcal{Q}',\tau_i}) = \mathcal{E}_{\sqrt{\tau_i} \sqrt{\gamma_i}\ep_i}(v^i, D_{\mathcal{Q}',\tau_i})   &\ge  \tau_i^{d-1}  \mathcal{E}_{\sqrt{\gamma_i}\ep_i}(y^i,D_{\gamma_i\mathcal{Q}',1}). 
\end{align}
Let $\delta>0$. We can (almost) cover $D_{\gamma_i\mathcal{Q}',1}$ by $\lfloor \gamma_i \rfloor^{d-1}$  pairwise disjoint translated copies of $D_{\mathcal{Q}',1}$. This implies that   we can find   $z_i \in \R^{d-1} \times \lbrace 0 \rbrace$  such that, by a  classical De Giorgi argument  (see the explanation at the beginning of the proof of  \cite[Lemma 4.3]{conti.schweizer2} for details on this technique),  for $i \in \N$ sufficiently large   we get  by \eqref{eq: l-e-b1} and a change of variables  that
\begin{align}\label{eq: l-e-b2}
{\rm {\rm (i)}} & \ \ \mathcal{E}_{\sqrt{\gamma_i}\ep_i}(y^i, z_i  + D_{\mathcal{Q}',1}) \le  \frac{(1+\delta)}{ \lfloor \gamma_i \rfloor^{d-1} } \, \mathcal{E}_{\sqrt{\gamma_i}\ep_i}(y^i,D_{\gamma_i\mathcal{Q}',1}) \le  \frac{(1+\delta)}{(\lfloor \gamma_i \rfloor  \tau_i)^{ d-1}}  \,\mathcal{E}_{\ep_i}(v^i,D_{\mathcal{Q}',\tau_i}),\notag \\
{\rm {\rm (ii)}} & \ \  \Vert  \nabla y^i-  \nabla y_0^+ \Vert^{ 2 }_{L^{2}(z_i+ D_{\mathcal{Q}',1})} \le \frac{C}{\delta}\tau_i^{d-1}\Vert  \nabla y^i- \nabla  y_0^+ \Vert^{ 2 }_{L^{2}( D_{\gamma_i\mathcal{Q}',1})} = \frac{C}{\delta\tau_i}\Vert  \nabla v^i- \nabla  y_0^+ \Vert^{ 2 }_{L^{2}( D_{\mathcal{Q}',\tau_i})}.
\end{align}
Since $\tau_i \to 0$,  there holds  $\tau_i \lfloor \gamma_i \rfloor \to 1$. This along with \eqref{eq: l-e-b2}(i) yields
\begin{align}\label{eq: l-e-b3}
\liminf_{i \to \infty}\mathcal{E}_{\sqrt{\gamma_i}\ep_i}(y^i,  z_i  + D_{\mathcal{Q}',1}) \le  (1+\delta) \, \liminf_{i \to \infty}\mathcal{E}_{\ep_i}(v^i,D_{\mathcal{Q}',\tau_i}).
\end{align}
Moreover,  by \eqref{eq: to confirm-old1} (with $\omega = \mathcal{Q}'$) and \eqref{eq: l-e-b2}(ii) we obtain   $\Vert  \nabla y^i- \nabla y_0^+ \Vert^{ 2 }_{L^{2}(z_i  + D_{\mathcal{Q}',1})}  \to 0$. Since $\sqrt{\gamma_i}\ep_i \to 0$ by assumption on $\lbrace \tau_i\rbrace_i$, \eqref{eq: k-intro}, \eqref{eq: l-e-b3}, and the translational invariance of $\mathcal{E}_\eps$ imply
$$\mathcal{F}(\mathcal{Q}',1) \le \liminf_{i \to \infty}\mathcal{E}_{\sqrt{\gamma_i}\ep_i}(y^i,  z_i  + D_{\mathcal{Q}',1}) \le (1+\delta) \, \liminf_{i \to \infty}\mathcal{E}_{\ep_i}(v^i,D_{\mathcal{Q}',\tau_i}). $$
Since $\delta>0$ was arbitrary, in view of Proposition \ref{prop:cell-form}, the statement follows for $\omega = \mathcal{Q}'$.

Now we consider a general bounded Lipschitz domain $\omega \subset \R^{d-1}$. Given $\delta>0$, we can choose pairwise disjoint cubes $\mathcal{Q}_j' \subset \omega$, $j=1,\ldots,N$, contained in $\omega$ such that $\mathcal{H}^{d-1}(\omega \setminus \bigcup_{j=1}^N \mathcal{Q}_j') \le  \delta$. Then by applying \eqref{eq: to confirm-old2} on each cube $\mathcal{Q}'_j$ we get
\begin{align*}
\liminf_{i \to \infty}\mathcal{E}_{\eps_i}(v^i,D_{\omega,\tau_i}) \ge \sum\nolimits_{j=1}^N \liminf_{i \to \infty}\mathcal{E}_{\eps_i}(v^i,D_{\mathcal{Q}'_j,\tau_i}) \ge K  \sum\nolimits_{j=1}^N\mathcal{H}^{d-1}(\mathcal{Q}'_j) \ge K \big(\mathcal{H}^{d-1}(\omega) -\delta \big).
\end{align*}    
Since $\delta>0$ was arbitrary, \eqref{eq: to confirm-old2} holds. 
\end{proof}

\begin{proof}[Proof of Lemma \ref{lemma: intefacefind2}]
We prove the result only in the case $M= A$. The case $M=B$ is the same, up to a different notational realization. The proof is similar to the one of \cite[Lemma 4.9]{davoli.friedrich} where the problem with one interface only has been addressed.

\emph{Step 1:  Subdivision into  phases.} As $\lbrace \tau_i\rbrace_i \in \mathcal{W}_d$, see \eqref{ew: W sequence}, and $\alpha(d) = 1/(2d)$, we can choose   $ \lambda_i= \eps_i^{1+1/(4d)}\subset (0,1/4)$  such that 
\begin{align}\label{eq: cond on h2}
\tau_i^{-1}\lambda_i \to 0, \ \ \ \ \ \   \eps_i^{-2 + \alpha(d)}\tau_i  \, \lambda_i^{(d-1)/d} \to \infty.
\end{align}
We use  Proposition \ref{lprop: phases} for $ y^{i} \in H^2(D_{\mathcal{Q}',h};\R^d)$ to find a corresponding set $T_i$ with properties \eqref{eq: propertiesT}. Recall that $T_i$ corresponds to the $A$-phase regions and $D_{\mathcal{Q}',h} \setminus T_i$ to the $B$-phase regions of the function $y^{i} $. Let 
\begin{align}\label{eq: good layer5}
\mathcal{T}^i_A& = \big\{ t \in (-h,h)\colon \ \mathcal{H}^{d-1}( (\mathcal{Q}' \times \lbrace t\rbrace) \cap T_i) \ge (1- \lambda_i) \mathcal{H}^{d-1}(\mathcal{Q}')\big\}, \notag \\ 
\mathcal{T}^i_B& = \big\{ t \in (-h,h)\colon \ \mathcal{H}^{d-1}( (\mathcal{Q}' \times \lbrace t\rbrace) \setminus T_i) \ge  (1- \lambda_i) \mathcal{H}^{d-1}(\mathcal{Q}') \big\}.
\end{align}
 Define the indicator function $\psi_i\colon (-h,h) \to \lbrace A,B\rbrace$ by $\psi_i(t) = A$ if $\sup \lbrace t' \le t, t' \in \mathcal{T}^i_A \cup \mathcal{T}^i_B \rbrace \in \overline{\mathcal{T}^i_A}$ and $\psi_i(t) = B$ else.  We get that  
\begin{align}\label{eq: good layer3}
\mathcal{H}^{1}\big( (-h,h) \setminus (\mathcal{T}^i_A \cup \mathcal{T}^i_B)  \big) \le cC_0  \eps_i^{2-\alpha(d)} \,    \lambda_i^{{\frac{1-d}{d}}}  \big(\mathcal{H}^{d-1}(\mathcal{Q}')\big)^{\frac{2-d}{d-1}}, 
\end{align}
and that the function $\psi_i$ jumps at most
\begin{align}\label{eq: good layer1}
 N_i \le 2cC_0  \, (\mathcal{H}^{d-1}(\mathcal{Q}'))^{-1} + 1
\end{align}
times, where $c>0$ is the constant from Proposition \ref{lprop: phases}, and $C_0>0$ is such that $\mathcal{E}_{\eps_i}(y^{i}, D_{\mathcal{Q}', h}) \le C_0$ for all $i \in \N$. We point out that the above estimates are obtained by performing analogous arguments to the ones in the proof of \cite[Lemma 4.9; (4.39)-(4.43)]{davoli.friedrich}. The expert reader can thus skip the remaining part of this step and move directly to Step 2. To keep the presentation self-contained, we include here  a short proof   of \eqref{eq: good layer3} and \eqref{eq: good layer1}. 

For $i$ sufficiently large (i.e., $\lambda_i$ small), the relative isoperimetric inequality on $\mathcal{Q}' \times \lbrace t\rbrace$ in dimension $d{-}1$, cf.\ \cite[Theorem 2, Section 5.6.2]{EvansGariepy92}, shows that
\begin{align}\label{eq: if}
\mathcal{H}^{d-2}\big((\mathcal{Q}' \times \lbrace t \rbrace) \cap \partial^* T_i \big) \le \lambda_i^{\frac{d-1}{d}} (\mathcal{H}^{d-1}(\mathcal{Q}'))^{\frac{d-2}{d-1}}\quad \quad \Rightarrow \quad   \quad t \in \mathcal{T}^i_A \cup \mathcal{T}^i_B.
\end{align}
Indeed, by the relative isoperimetric inequality we get
$$\min\big\{ \mathcal{H}^{d-1} \big((\mathcal{Q}'\times \{t\})\cap   T_i   \big),\, \mathcal{H}^{d-1}\big((\mathcal{Q}'\times \{t\})\setminus   T_i   \big)\big\}  \le C\big(\lambda_i^{\frac{d-1}{d}} (\mathcal{H}^{d-1}(\mathcal{Q}'))^{\frac{d-2}{d-1}}\big)^{\frac{d-1}{d-2}} \le \lambda_i \mathcal{H}^{d-1}(\mathcal{Q}')$$
for $i$ large enough, where we used $(d-1)^2/(d(d-2))>1$. (For $d=2$, the term after the first inequality has to be interpreted as zero.)   This gives \eqref{eq: if}. Thus, by \eqref{eq: propertiesT}{\rm (iii)},  \eqref{eq: if},  and $\mathcal{E}_{\eps_i}(y^{i}, D_{\mathcal{Q}',h}) \le C_0$   we obtain \eqref{eq: good layer3}.

To prove \eqref{eq: good layer1}, we use the coarea formula to  get for $\mathcal{H}^1$-a.e.\ $t_A \in \mathcal{T}^i_A$, $t_B \in \mathcal{T}^i_B$
\begin{align*}
\mathcal{H}^{d-1}\big(\partial^*T_i \cap (\mathcal{Q}' \times (t_A,t_B))\big) &\ge  \int_{\partial^*T_i \cap (\mathcal{Q}' \times (t_A,t_B))} |\langle \nu_{T_i},   e_d   \rangle| \, {\rm d}\mathcal{H}^{d-1} \\
&= \int_{\Pi_d} \mathcal{H}^{0}\big( (  z + (t_A,t_B) e_d) \cap \partial^* T_i \cap (\mathcal{Q}'\times (t_A,t_B))\big)  \, {\rm d}\mathcal{H}^{d-1}(z),
\end{align*}
where $\Pi_d := \R^{d-1} \times \lbrace 0 \rbrace$, and $\nu_{T_i}$ denotes the outer unit normal to $T_i$. In view of \eqref{eq: good layer5}  and $\lambda_i \le \frac{1}{4}$, we get
\begin{align*}
%\label{eq: LLL}
\int_{\Pi_d} \mathcal{H}^{0}\big( ( z + (t_A,t_B) e_d) \cap \partial^* T_i \cap   (\mathcal{Q}'\times (t_A,t_B))\big)  \, {\rm d}\mathcal{H}^{d-1}( z )\geq \frac12 \mathcal{H}^{d-1}(\mathcal{Q}').
\end{align*} 
Property \eqref{eq: good layer1} follows then by \eqref{eq: propertiesT}{\rm (i)}.

\emph{Step 2: Rigidity estimates.} 
Theorem \ref{thm:rigiditythm}   and Proposition \ref{lprop: phases} yield rotations $R_i \in SO(d)$ such that 
\begin{align}\label{eq: good layer7}
\Vert\nabla  y^{i}  -R_iA \Vert_{L^{ 2 }(  D_{\mathcal{Q}',h}   \cap T_i)} + \Vert\nabla  y^{i}  -R_iB \Vert_{L^{ 2 }(  D_{\mathcal{Q}',h}  \setminus T_i)}\leq  C\eps_i,
\end{align}
where $C$ depends on the uniform energy bound $C_0$  and on $D_{\mathcal{Q}',h}$. (Note that the estimate holds in the entire set $D_{\mathcal{Q}',h}$ since it is a paraxial cuboid.) For later purposes, we estimate integrals on sets $D = \alpha e_d + D_{\mathcal{Q}',\sigma} \subset D_{\mathcal{Q}',h}$ for $\alpha \in \R$ and $\sigma>0$. Let $L\ge \sqrt{d}$  sufficiently large such that  $\dist(F,SO(d)\lbrace A,B\rbrace) \ge  |F-R M|/2 $ for all $F \in \M^{d \times d}$ with $|F| \ge L$,  $R \in SO(d)$, and $ M  \in \lbrace A,B\rbrace$. We now show that for  every $q \in \lbrace 1, 2 \rbrace$ there holds
\begin{align}\label{eq: abschluss-rigidty}
{\rm (i)} & \ \ \int_{D} |R_i^T \nabla  y^{i}  -A|^q \, {\rm d}x \le  C \big(\mathcal{L}^d(D)\big)^{1-q/2}   \eps_i^q  +  (2L)^q\mathcal{L}^d({D \setminus T_i}), \notag \\ 
{\rm (ii)} & \ \  \int_{D} |R_i^T \nabla  y^{i}  -B|^q \, {\rm d}x \le  C \big(\mathcal{L}^d(D)\big)^{1-q/2}   \eps_i^q  +  (2L)^q\mathcal{L}^d({D \cap T_i}).
\end{align}
To see this, define $E_i = D \cap \lbrace |\nabla  y^{i} | \le L \rbrace$. First, by using H4.\  we observe that  
\begin{align}\label{eq: bad set}
\Vert\nabla  y^{i}  -R_iA \Vert^2_{L^2(  D\setminus E_i  ) } + \Vert\nabla  y^{i}  -R_iB \Vert^2_{L^2(  D\setminus E_i  )} \le C\int_{  D  } W(\nabla  y^{i} )  \, {\rm d}x  \le C\eps^2_i,
\end{align}
where $C$ depends on $c_1$ and $C_0$. For the integral on  $E_i$, we calculate  
\begin{align*}
\int_{E_i} |R_i^T \nabla  y^{i}  -A|^q \, {\rm d}x & = \int_{E_i\cap T_i} | \nabla  y^{i}  -R_i|^q \, {\rm d}x +  \int_{E_i \setminus T_i} |\nabla  y^{i}  -R_i|^q \, {\rm d}x\\
& \le  \big(\mathcal{L}^d(D)\big)^{1-q/2} \Big(\int_{  D \cap T_i} | \nabla  y^{i}  -R_i|^{2} \, {\rm d}x\Big)^{q/2} +  (2L)^q\mathcal{L}^d({D \setminus T_i})
\end{align*}
for $q \in \lbrace 1, 2 \rbrace$, where in the second step we have used H\"older's inequality. This along with \eqref{eq: good layer7},  \eqref{eq: bad set}, and H\"older's inequality shows \eqref{eq: abschluss-rigidty}(i). In a similar fashion, one can show \eqref{eq: abschluss-rigidty}(ii).

\emph{Step 3: Asymptotic behavior of phases.} We now use \eqref{eq: abschluss-rigidty} to show the properties 
\begin{align}\label{eq: good layer2}
{\rm (i)} \  \liminf_{i \to\infty} \frac{1}{\tau_i}\mathcal{H}^1\Big(\mathcal{T}^i_B \cap (-\tfrac{h}{2},\tfrac{h}{2})  \Big)>0  \ \ \ \ \text{and}  \ \ \ {\rm (ii)}  \  \lim_{i \to \infty} \mathcal{H}^1\Big( \mathcal{T}^i_B\cap \big((-h,h) \setminus (-\tfrac{h}{2},\tfrac{h}{2}) \big)\Big) \to 0.
\end{align}
Suppose by contradiction that \eqref{eq: good layer2}(i) were false. Let $D^\sigma := D_{\mathcal{Q}',\sigma}$ for $0 < \sigma< \frac{h}{2}$. Then by \eqref{eq: cond on h2}--\eqref{eq: good layer3} we get (for a subsequence, not relabeled) 
\begin{align}\label{eq: needs to be repeated1}
\frac{1}{\tau_i}\mathcal{L}^d(D^\sigma\setminus T_i)\le \frac{1}{\tau_i} \mathcal{H}^{d-1}(\mathcal{Q}') \Big( \lambda_i \mathcal{H}^1\big((-\sigma,\sigma) \cap \mathcal{T}^i_A\big)  +  \mathcal{H}^1\big((-\sigma,\sigma) \cap \mathcal{T}^i_B\big)  + \mathcal{H}^1\big((-\sigma,\sigma) \setminus (\mathcal{T}^i_A \cup \mathcal{T}^i_B)\big)   \Big) \to 0.
\end{align}
By \eqref{eq: abschluss-rigidty}(i) for $q=1$ and the fact that $\limsup_{i\to \infty} \ (\eps_i/\tau_i) < + \infty$, see \eqref{ew: W sequence}, this implies
\begin{align*}
%\label{eq: needs to be repeated2}
\limsup_{i\to \infty} \frac{1}{\tau_i}\int_{D^\sigma} |R_i^T \nabla  y^{i}  -A| \, {\rm d}x \le C \big(2\sigma \mathcal{H}^{d-1}(\mathcal{Q}')\big)^{1/2} \, \limsup_{i\to \infty} \ (\eps_i/\tau_i)\le c_\sigma
\end{align*}
for a constant $c_\sigma$ with $c_\sigma \to 0 $ as $\sigma \to 0$. By Poincare's inequality and a $BV$ compactness result, we find $\lbrace b_i \rbrace_i \subset \R^d$ such that the sequence 
$$f^\sigma_i (x):= \tau_i^{-1} (y^i - (R_i\,x + b_i)) \ \text{ for $x \in D^\sigma$}$$
converges weakly* in $BV$ to some $f^\sigma \in BV(D^\sigma;\R^d)$ with $|Df^\sigma|(D^\sigma) \le c_\sigma$. In view of \eqref{eq: converg assu}, it is not hard to check that $f^\sigma(x) =  u(x)+ S  x+b$ for some $S \in \mathbb{M}^{d \times d}_{\rm skew}$ and $b \in \R^d$. On the other hand, by \eqref{eq: UUU}, for $\sigma$ sufficiently small we find $c_\sigma < |D^ju|(\mathcal{Q}' \times \lbrace 0 \rbrace)$, where $D^ju$ denotes the jump part of the distributional derivative. This contradicts the fact that $|D^ju|(D^\sigma) = |D^jf^\sigma|(D^\sigma) \le c_\sigma$.

Now suppose by contradiction that \eqref{eq: good layer2}(ii) were false. In view of \eqref{eq: good layer1}, by passing to a subsequence, we find $ h>\sigma>0$ and $\alpha \in (-h+\sigma,h-\sigma)$   such that  $\mathcal{H}^1((\alpha-\sigma,\alpha+\sigma) \cap \mathcal{T}^i_A)=0$  for all $i$ sufficiently large. Define  $ D := \alpha e_d + D_{\mathcal{Q}',\sigma}$. Repeating the argument in \eqref{eq: needs to be repeated1}, in particular using \eqref{eq: cond on h2}--\eqref{eq: good layer3}, we find $\tau_i^{-1}\mathcal{L}^d( D \cap T_i) \to 0$. Then, by \eqref{eq: abschluss-rigidty}(ii) and the fact that $\limsup_{i\to \infty} \ (\eps_i/\tau_i) < + \infty$ we get that
$$\limsup_{i\to \infty} \frac{1}{\tau_i}  \int_{D} |R_i^T \nabla  y^{i}  -B| \, {\rm d}x < + \infty. $$
By Poincare's inequality and a $BV$ compactness result, we find $\lbrace b_i \rbrace_i\subset  \R^d$ such that the sequence $f_i (x):= \tau_i^{-1} ( y^i - (R_i\, B\,x + b_i))$ for $x \in D$ converges pointwise a.e.\ to some $f \in BV(D;\R^d)$ (up to passing to a subsequence). By \eqref{eq: converg assu}, this implies that $\tau_i^{-1} (( R_i\, B\, - A)x + b_i)$ converges a.e.\ on $D$ to a finite limit. This, however, is impossible, and therefore \eqref{eq: good layer2}(ii) holds.

\emph{Step 4: Definition of cylindrical sets.}  In the following, we denote by $s^i_1 < s^i_2 \ldots < \ldots < s^i_{N_i}$ the jump points of the function $\psi_i$ defined below \eqref{eq: good layer5}. Let $\mathcal{J}_i =\lbrace 0\le j \le N_i\colon \,  (s^i_{j},s^i_{j+1}) \cap \mathcal{T}^i_A = \emptyset \rbrace$, where we set $s^i_0 = -h$ and $s^i_{N_i+1} = h$. Note that for $j \in \mathcal{J}_i \setminus \lbrace 0 \rbrace$ there holds $(s^i_{j-1},s^i_j) \cap \mathcal{T}^i_B = \emptyset$.  Recalling \eqref{eq: good layer1}, up to passing to a subsequence, we can assume that $\mathcal{J}_i$ and $N_i$ are independent of $i$, which we denote by $\mathcal{J}$ and $N$, respectively, for simplicity. Moreover, we can suppose that $\lbrace s^i_{j} \rbrace_i$ converges for all $1 \le j \le N$. In view of \eqref{eq: good layer2}(i), possibly by selecting a further subsequence, we find an index $k \in \mathcal{J}$ and a constant $\bar{c}>0$ independently of $i$ such that   $s_k^i,s_{k+1}^i \in (-\frac{h}{2},\frac{h}{2})$ and  
\begin{align}\label{eq: middle-length}
s_{k+1}^i - s_{k}^i \ge \bar{c}\,\tau_i.
\end{align}
 We now show that there exist $1 \le j_1 \le k$ and $k+1 \le j_2 \le N$, as well as $\mu_1,\mu_2>0$ such that
\begin{align}\label{eq: intervals}
{\rm (i)}& \ \ \lim_{i\to\infty}\tau_i^{-1} \mathcal{H}^1\big((s^i_{j_1} -\mu_1 \tau_i, s^i_{j_1})  \cap  \mathcal{T}^i_B\big) =0, \ \ \ \ \lim_{i\to\infty}\tau_i^{-1} \mathcal{H}^1\big((s^i_{j_1}, s^i_{j_1} + \mu_1 \tau_i)  \cap  \mathcal{T}^i_A\big) =0, \notag \\
{\rm (ii)}& \ \ \lim_{i\to\infty}\tau_i^{-1} \mathcal{H}^1\big((s^i_{j_2} -\mu_2 \tau_i, s^i_{j_2})  \cap  \mathcal{T}^i_A\big) =0, \ \ \ \ \lim_{i\to\infty}\tau_i^{-1} \mathcal{H}^1\big((s^i_{j_2}, s^i_{j_2} + \mu_2 \tau_i)  \cap  \mathcal{T}^i_B\big) =0.
\end{align} 
 Indeed, choose $j_1 \in \mathcal{J}$, $j_1 \le k$, as the largest index such that $\liminf_{i \to \infty} \tau_i^{-1}(s^i_{j_1} - s^i_{j_1-1}) >0$ and set 
$$\mu_1 := \min\big\{ \liminf_{i \to \infty} \tau_i^{-1}(s^i_{j_1} - s^i_{j_1-1}), \bar{c}/2\rbrace >0,$$ 
where $\bar{c}$ is the constant from  \eqref{eq: middle-length}.   Note that such an index exists by \eqref{eq: cond on h2}, \eqref{eq: good layer3}, \eqref{eq: good layer2}(ii), and the fact that $(s^i_{j-1}, s^i_{j})\cap \mathcal{T}^i_B= \emptyset$ for each $j \in \mathcal{J} \setminus \lbrace0 \rbrace$ by the definition of $\mathcal{J}$. 
 This immediately implies the first part of \eqref{eq: intervals}(i). The second part of \eqref{eq: intervals}(i) follows from the fact that  $\liminf_{i \to \infty} \tau_i^{-1}(s^i_{j} - s^i_{j-1}) =0$ for all $j \in \mathcal{J}$ with $j_1 < j \le k$, $(s^i_{j}, s^i_{j+1})\cap \mathcal{T}^i_A= \emptyset$ for $j \in \mathcal{J}$, \eqref{eq: middle-length}, and the fact that $\mu_1 \le \bar{c}/2$.   The index $j_2 \ge k+1$, $j_2 \notin \mathcal{J}$, and $\mu_2 \in (0,\bar{c}/2]$  in \eqref{eq: intervals}(ii) can be chosen in a similar fashion: let $j_2 \ge k+1$, $j_2 \notin \mathcal{J}$,  be the smallest index such that $\liminf_{i \to \infty} \tau_i^{-1}(s^i_{j_2+1} - s^i_{j_2}) >0$ and let $\mu_2 = \min\lbrace \liminf_{i \to \infty} \tau_i^{-1}(s^i_{j_2+1} - s^i_{j_2}), \bar{c}/2\rbrace$.     
 
 We define $\mu = \min\lbrace \mu_1,\mu_2 \rbrace$, \EEE  $\alpha^1_i = s^i_{j_1}$, and $\alpha^2_i = s_{j_2}^i$. Then, the sets $D^1_i := \alpha^1_i e_d + D_{\mathcal{Q}',\mu \tau_i}$ and $D^2_i := \alpha^2_i e_d + D_{\mathcal{Q}',\mu \tau_i}$ satisfy $D^1_i \cap D^2_i = \emptyset$ by \eqref{eq: middle-length} and the fact that $\mu \le \bar{c}/2$.   Moreover, there holds 
\begin{align}\label{eq: good layer6}
{\rm (i)} & \ \ \tau_i^{-1}\big(\mathcal{L}^d(D^1_i \cap \lbrace x_d \le \alpha^1_i  \rbrace \setminus T_i) + \mathcal{L}^d(D^1_i \cap \lbrace x_d \ge \alpha^1_i  \rbrace \cap T_i) \big) \to 0, \notag\\
{\rm (ii)} & \ \ \tau_i^{-1}\big(\mathcal{L}^d(D^2_i \cap \lbrace x_d \le \alpha^2_i  \rbrace \cap T_i) + \mathcal{L}^d(D^2_i \cap \lbrace x_d \ge \alpha^2_i  \rbrace \setminus T_i)  \big) \to 0
\end{align}
as $i \to \infty$. Indeed, e.g., for the first term in \eqref{eq: good layer6}(i), we compute by  \eqref{eq: cond on h2}--\eqref{eq: good layer3}  and \eqref{eq: intervals}(i)  that
\begin{align*} 
&\tau_i^{-1}\mathcal{L}^d( \lbrace x \in D^1_i\colon \  x_d \le \alpha^1_i  \rbrace\setminus T_i) \\
& \ \ \ \ \le  \tau_i^{-1} \mathcal{H}^{d-1}(\mathcal{Q}') \Big( \mathcal{H}^{1}\big( (-h,h) \setminus (\mathcal{T}^i_A \cup \mathcal{T}^i_B)  \big)  + \mathcal{H}^1\big((s^i_{j_1} -\mu_1 \tau_i, s^i_{j_1})  \cap  \mathcal{T}^i_B\big) + \mu \tau_i  \, \lambda_i\Big) \to 0
\end{align*}
as $i \to \infty$. The other three terms can be treated in a similar fashion.

\emph{Step 5: Proof of \eqref{eq: to confirm}.} We define $v^i$ as in \eqref{eq: vi}    for  isometries $I_i$ whose derivative is given by $R^T_i$.  To see \eqref{eq: to confirm},  we apply \eqref{eq: abschluss-rigidty}(i) for $q=2$ on $D=D_i^1 \cap \lbrace x_d \le\EEE \alpha^1_i\rbrace$ and $D=D_i^2 \cap \lbrace x_d \ge\EEE \alpha^2_i\rbrace$, as well as \eqref{eq: abschluss-rigidty}(ii) for $q=2$ on $D=D_i^1 \cap \lbrace x_d \ge \EEE \alpha^1_i\rbrace$ and $D=D_i^2 \cap \lbrace x_d  \le \EEE \alpha^2_i\rbrace$. This along with \eqref{eq: good layer6} and $\tau_i^{-1}\eps_i^2 \to 0$ (see \eqref{ew: W sequence}) shows the desired estimate.  This concludes the proof.  
\end{proof}

  We conclude this subsection with the proof of Proposition \ref{prop:kMdo}.  

\begin{proof}[Proof of Proposition \ref{prop:kMdo}]
Let $M \in \lbrace A,B\rbrace$.  First, it is clear that the left hand side in \eqref{eq: added-equ} is not smaller than $\mathcal{F}^M_{\rm dp}(\omega,h)$, see \eqref{eq: k2-intro}. We also note by Proposition \ref{eq: KundKdo-new} that 
\begin{align}\label{eq: onemore}
\mathcal{F}^M_{\rm dp}(\omega,h) \ge 2K \, \mathcal{H}^{d-1}(\omega) = K^M_{\rm dp}  \, \mathcal{H}^{d-1}(\omega), 
\end{align} 
 where in the last step we used the assumption  $K^M_{\rm dp}  =  2K$. To prove the reverse inequality, we argue by contradiction: if the statement  were   false, there would exist $\delta>0$, a Lipschitz domain $\omega \subset \R^{d-1}$, $h>0$, and    a sequence $\{w_\ep\}_\ep\in \mathcal{W}_d$ such that
\begin{align}
\label{eq:contr}
\inf\left\{\limsup_{\ep\to 0}\mathcal{E}_\ep(y^\ep;D_{\omega,h}):\  \frac{y^\eps -  Mx}{w_\eps} \to y_{\rm dp}^M  \text{  in measure  in $D_{\omega,h}$ as $\eps \to 0$}\right\} \ge (K^M_{\rm dp} +2\delta)\mathcal{H}^{d-1}(\omega).
\end{align}
Up to translations of $\omega$, we can select a cube $\mathcal{Q}' \subset \R^{d-1}$ containing both  $\omega$ and $Q'=(-\frac{1}{2},\frac{1}{2})^{d-1}$ such that $\alpha \mathcal{Q}' = Q'$ for some $0 < \alpha < 1$.  In view of \eqref{eq: our-k2},  we can find a sequence of functions $\lbrace y^{\ep} \rbrace_\ep \subset H^2(D_{Q',\alpha h};\R^d)$ such that $ (w_{\eps}\alpha)^{-1}(y^{\ep} -  M x)  \to y_{\rm dp}^M$  in measure  in $D_{Q',\alpha h}$ and 
\begin{align}\label{eq: for cubes only3}
\limsup_{\ep \to 0} \mathcal{E}_{\sqrt{\alpha}\ep}(y^{\ep},D_{Q',\alpha h}) \le  K^M_{\rm dp}+ \delta\alpha^{d-1}\mathcal{H}^{d-1}(\omega).
\end{align}
Then the functions $\lbrace \bar{y}^{\ep}\rbrace_\ep\subset H^2(D_{\mathcal{Q}',h};\R^d)$ defined by $\bar{y}^{\ep}(x) = y^{\ep}(\alpha x)/\alpha$ satisfy $ w_{\ep}^{-1}(\bar{y}^{\ep} -  Mx) = (w_{\ep}\alpha)^{-1}(y^{\ep}(\alpha x) -  M(\alpha x)) \to y_{\rm dp}^M$ in measure in  $D_{\mathcal{Q}',h}$. In particular, as $D_{\omega,h} \subset D_{\mathcal{Q}',h}$, by \eqref{eq:contr} we find an infinitesimal sequence $\{\ep_i\}_i$ such that 
\begin{align}\label{eq: for cubes only2}
\liminf_{i \to \infty} \mathcal{E}_{\ep_i}( \bar{y}^{\ep_i} ,D_{\omega,h}) \ge (K^M_{\rm dp} + 2\delta) \, \mathcal{H}^{d-1}(\omega).
\end{align}
Then, using  \eqref{eq: k2-intro},  \eqref{eq: transformation preparation}, \eqref{eq: onemore}, and   \eqref{eq: for cubes only2},  we derive
\begin{align*}
\liminf_{i \to \infty}  \alpha^{1-d} \mathcal{E}_{\sqrt{\alpha}\ep_i}( {y}^{\ep_i} ,D_{Q',\alpha h}) &\ge \liminf_{i \to \infty}  \mathcal{E}_{\ep_i}( \bar{y}^{\ep_i} ,D_{\mathcal{Q}',h}) \\
& \ge \liminf_{i \to \infty}\mathcal{E}_{\ep_i}( \bar{y}^{\ep_i} ,D_{\mathcal{Q}',h} \setminus D_{\omega,h}) + \liminf_{i \to \infty} \mathcal{E}_{\ep_i}(\bar{y}^{\ep_i} ,D_{\omega,h}) \\
& \ge  \mathcal{F}^M_{\rm dp}(\mathcal{Q}'\setminus \omega;h) +   (K_{\rm dp}^M + 2\delta) \mathcal{H}^{d-1}(\omega) \\
& \ge   K^M_{\rm dp} \,  \mathcal{H}^{d-1}(\mathcal{Q}'\setminus \omega) + ( K^M_{\rm dp} + 2\delta) \mathcal{H}^{d-1}(\omega) = \alpha^{1-d}  K^M_{\rm dp}  +2\delta \mathcal{H}^{d-1}(\omega).
\end{align*}
 In the last step, we used $\alpha \mathcal{Q}' = Q'$. This estimate, however, contradicts \eqref{eq: for cubes only3}.
\end{proof}

\subsection{Construction of local recovery sequences}
\label{subs:local}
This subsection is devoted to the proofs of Propositions  \ref{lemma: local1}  and  \ref{lemma: local2}, namely to the construction of local recovery sequences performing single and double phase transitions, respectively, in an energetically optimal way. The crucial point is that the sequences coincide with isometries far from the interfaces as this allows to `glue together' different sequences, as done in the proof of Theorem \ref{thm:limsup-new}. We begin with the proof of Proposition \ref{lemma: local1}. 

\begin{proof}[Proof of Proposition \ref{lemma: local1}]
The result has been proved in \cite[Proposition 4.7]{davoli.friedrich} in the special case in which $\Omega = D_{\omega',h}$. We briefly explain how to obtain the result for strictly star-shaped sets $\Omega$  and cylindrical sets $D_{\omega',h}$ such that  $(\partial \omega' \times (-h,h)) \cap \Omega = \emptyset$. Choose $\omega\subset \R^{d-1}$ such that $\omega \times \lbrace 0\rbrace = (\omega' \times \lbrace 0 \rbrace) \cap \Omega$.  As $\Omega$ is strictly star-shaped, we can find sequences  $\lbrace h_i \rbrace_i, \lbrace \alpha_i\rbrace_i \subset \R$, with $h_i\to 0 $ and $\alpha_i \to 0$ as $i \to \infty$, and a sequence of decreasing Lipschitz sets $\lbrace \omega_i \rbrace_i$ with $\omega \subset \subset \omega_i \subset \subset \omega'     $   for all $i \in \N$  and 
\begin{align}\label{eq: always better}
\mathcal{H}^{d-1}(\omega_i) \le \mathcal{H}^{d-1}(\omega) + 1/i
\end{align}
such that $\alpha_i e_d + D_{\omega_i,h_i} \subset D_{\omega',h}$, and $(\partial \omega_i \times (-h_i + \alpha_i, \alpha_i + h_i)) \cap \Omega = \emptyset$.

We apply \cite[Proposition 4.7]{davoli.friedrich} on $D_i := \alpha_i e_d + D_{\omega_i,h_i}$, to obtain a recovery sequence $v^{\pm,i}_\eps \subset H^2(D_i;\R^d)$ and isometries $\{I^{\pm,i}_{1,\eps}\}_{ \eps}$, $\{I^{\pm,i}_{2,\eps}\}_{ \eps}$ such that \eqref{eq:local1-1} holds  for $D_i$ in place of $D_{\omega',h} \cap \Omega$ and for  $y_0^\pm(\cdot - \alpha_i e_d )$  in place of $y_0^\pm$, and  \eqref{eq:local1-3} holds for   $h_i$ in place of $h$, up to a translation by $\alpha_i e_d$. Moreover, instead of \eqref{eq:local1-3a} we get 
\begin{align}\label{eq: always better2}
 \lim_{\eps \to 0} \mathcal{E}_{\ep} (v^{\pm,i}_\eps, D_i)  =  K \, \mathcal{H}^{d-1}( \omega_i ).
 \end{align}
 In view of \eqref{eq:local1-3} for $v^{\pm,i}_\eps$ and the fact that $(\partial \omega_i \times (-h_i + \alpha_i, \alpha_i + h_i)) \cap \Omega = \emptyset$, we can extend $v^{\pm,i}_\eps$ to an $H^2$-function on $D_{\omega',h} \cap \Omega$ by setting $v^{\pm,i}_\eps = I^{\pm,i}_{1,\eps} \circ y_0^\pm$ on $\lbrace    \alpha_i  + \EEE  3h_i/4 \le x_d < h\rbrace$ and $v^{\pm,i}_\eps = I^{\pm, { i }}_{2,\eps} \circ y_0^\pm$ on $\lbrace   - h < x_d \le  \alpha_i \EEE   -3h_i/4 \rbrace$, respectively. Note that the extensions (not relabeled) still satisfy \eqref{eq:local1-1}  (for  $y_0^\pm(\cdot - \alpha_i e_d )$  in place of $y_0^\pm$). \EEE Now we obtain a sequence satisfying \eqref{eq:local1-1}--\eqref{eq:local1-3} by choosing a suitable diagonal sequence in $\lbrace v^{\pm,i}_{\eps}\rbrace_{\eps,i}$ as $\eps \to 0 $ and $i \to \infty$  via Attouch's diagonalization lemma \cite[Lemma 1.15 and Corollary 1.16]{attouch},  and by taking  \eqref{eq: always better}--\eqref{eq: always better2} into account. 
\end{proof}

The remaining part of this subsection is devoted to the proof of Proposition \ref{lemma: local2}. The argument hinges upon applying some careful transformations to maps locally attaining the double-profile energy in Proposition \ref{prop:kMdo}, so that the modified maps satisfy \eqref{eq:local2-2}. As a first step, we show that the energy of optimal  sequences   concentrates near the interface. We recall  the definitions of $\mathcal{W}_d$ and $y^M_{\rm dp}$ in \eqref{ew: W sequence} and \eqref{eq: step functions}, respectively.

\begin{lemma}[Concentration of the energy near the interface]\label{cor: layer energy}
Let  $h>\tau>0$,  and let  $\omega \subset \R^{d-1}$ be a bounded Lipschitz domain.  Let $M \in \lbrace A,B\rbrace$ and suppose that $K_{\rm dp}^M   =  2K$. Let $\lbrace \eps_i\rbrace_i$ be an infinitesimal sequence and  let $\lbrace w_{\eps_i}\rbrace_i \in \mathcal{W}_d$.   Then,   there exists   $\lbrace y^{\eps_i}\rbrace_i \subset H^2(D_{\omega,h};\mathbb{R}^d)$ satisfying  $\lim_{i \to \infty} \Vert  y^{\eps_i}-  Mx \Vert_{H^1(D_{{\omega},h})} = 0$, and, as $i \to \infty$, we have 
$${\mathcal{E}_{\eps_i}(y^{\eps_i}, D_{\omega,h})\to 2K\mathcal{H}^{d-1}(\omega), \ \ \   \mathcal{E}_{\eps_i}(y^{\eps_i}, D_{\omega,h} \setminus D_{\omega,  \tau}) \to 0, \ \   \ \ \frac{y^{\ep_i}-Mx}{w_{\ep_i}}\to y^M_{\rm dp} \ \ \text{ in measure in }D_{{\omega},h}.}$$
\end{lemma}

\begin{proof}
First, by Proposition \ref{prop:kMdo},  $K_{\rm dp}^M =  2K$, and a standard diagonal argument we find a sequence $\lbrace y^{\eps_i}\rbrace_i \subset  H^2(D_{\omega,h};\R^d)$ with 
$$
{ \limsup_{i\to \infty} \mathcal{E}_{\eps_i}(y^{\eps_i}, D_{\omega,h}) \le 2K\mathcal{H}^{d-1}(\omega), \ \ \ \ \ \ \ \ \ \ \frac{y^{\ep_i}-Mx}{w_{\ep_i}}\to y^M_{\rm dp} \ \ \text{ in measure in }D_{{\omega},h}.}
$$
By \eqref{eq: k2-intro} and  Proposition \ref{eq: KundKdo-new}, we also  get $\liminf_{i\to \infty} \mathcal{E}_{\eps_i}(y^{\eps_i}, D_{\omega, \tau}) \ge 2K\mathcal{H}^{d-1}(\omega)$. This in turn implies $\mathcal{E}_{\eps_i}(y^{\eps_i}, D_{\omega,h} \setminus D_{\omega, \tau}) \to 0$ and $\mathcal{E}_{\eps_i}(y^{\eps_i}, D_{\omega,h})\to 2K\mathcal{H}^{d-1}(\omega)$. The convergence in measure to $y^M_{\rm dp}$ along with $w_{\ep_i} \to 0$ implies that $y^{\ep_i} \to Mx$ in measure on $D_{{\omega},h}$. Then, by Lemma \ref{lemma:comp-def} we deduce $\lim_{i \to \infty} \Vert  y^{\eps_i}-  Mx \Vert_{H^1(D_{{\omega},h})} = 0$.  
\end{proof}

Motivated by  Lemma \ref{cor: layer energy}, for $0<\tau\le h/4$ we introduce the notion of \emph{$\ep$-closeness  of $y$ to $Mx$}, defined as 
\begin{equation}
\label{eq:eta2}
\delta_{\ep}^M(y;\omega,h,\tau): = \mathcal{E}_{\eps}(y,  D_{\omega,h}\setminus D_{\omega,\tau})  +     (\mathcal{L}^d(D_{\omega,4\tau}))^{-1}  \Vert \nabla y -  M  \Vert^2_{L^2(D_{\omega,4\tau})},
\end{equation}
for $M\in \{A,B\}$. In the following, we will use that, for given $\omega \subset \R^{d-1}$, $0<\tau\le h/4$, and $\lbrace \eps_i \rbrace_i$ converging to zero, there exists a  sequence   $\lbrace y^{\eps_i}\rbrace_i \subset H^2(D_{\omega,h};\mathbb{R}^d)$ of deformations attaining asymptotically the double-profile energy $K_{\rm dp}^M =2K$ such that  
\begin{align*}
\delta_{\ep_i}^M(y^{\eps_i};\omega,h, \tau) \to 0 \ \ \ \ \text{ as } i \to \infty.
\end{align*}
Owing to the quantitative rigidity estimate in Theorem \ref{thm:rigiditythm}, it is possible to find $(d-1)$-dimensional  slices on which the energy of $y$  and the $L^{2}$-distance  of $\nabla y$  from suitable rotations of $M\in \{A,B\}$ can be quantified  in terms of $\delta_{\ep}^M(y;\omega,h, \tau)$. Recall $\kappa = |A-B|$, and $c_1$ in H4.  In addition, define 
$$p_d:=\begin{cases} 2&\text{if }d=2\\2(d-1)/d&\text{if }d>2. \end{cases}$$

\begin{proposition}[Properties  of $(d-1)$-dimensional slices]\label{lemma: optimal profile}
Let $d\in \mathbb{N}$, $d\geq 2$, and let $M\in \{A,B\}$. Let $h>0$,  $0 < \tau \le h/4$,   and let  $\omega, \hat{\omega} \subset \mathbb{R}^{d-1}$ be bounded Lipschitz domains such that $\omega \subset \subset \hat{\omega}$. Then there  exist $\ep_0=\eps_0(\omega,\hat{\omega},h,\kappa,c_1,\tau) \in (0,1)$ and $C=C(\omega,\hat{\omega},h,\kappa,c_1)>0$  with the following properties: \\  
For all $0<\eps \le \eps_0$ and for each  $y\in H^2(D_{\hat{\omega},h};\mathbb{R}^d)$ with $\delta_{\ep}^M(y;\hat{\omega},h,\tau) \le   (\kappa/64)^2  $ we can find two rotations $R^+,R^- \in SO(d)$  and two constants   $s^+ \in (\tau,2\tau)$, $s^- \in (-2\tau, - \tau)$ such that  
\begin{align*}
 {\rm (i)}&\ \    \int_{\Gamma^+} | \nabla y -  R^+ M|^{p} \, {\rm d}\mathcal{H}^{d-1} +  \int_{\Gamma^-} | \nabla y -  R^- M|^{p} \, {\rm d}\mathcal{H}^{d-1} \leq  \frac{C}{\tau}(\delta_{\ep}^M(y;\hat{\omega},h, \tau))^{ {p}/  2} \,\eps^{p} \  \, \text{for all }\, 1\le p \le p_d, \notag\\
 {\rm (ii)} & \ \ \Vert \nabla y - M \Vert^2_{L^2(  s^+ e_d  + D_{\omega,\eps^2} )}   +   \Vert \nabla y -M \Vert^2_{L^2(  s^- e_d  + D_{\omega,\eps^2} )} \le C\eps^2\delta^{M}_{\ep}(y;\hat{\omega},h, \tau),\notag\\
{\rm (iii)}& \ \   \eps^2 \int_{\Gamma^+  \cup \Gamma^-} | \nabla^2 y|^2 \, {\rm d}\mathcal{H}^{d-1} + \bar{\eta}_{\eps,d}^2   \int_{\Gamma^+  \cup \Gamma^-} (|\nabla^2 y|^2-|\partial^2_{dd}y|^2) \, {\rm d}\mathcal{H}^{d-1} \le  \frac{C}{\tau}\delta_{\ep}^M(y;\hat{\omega},h, \tau)  \notag,\\
{\rm (iv)}&  \ \   \mathcal{E}_{\ep}\big(y,  s^+ e_d  + D_{\omega,\eps^2} \big) + \mathcal{E}_{\ep}\big(y,  s^- e_d  + D_{\omega,\eps^2}  \big) \le  \frac{C\eps^2}{\tau}   \delta_{\ep}^M(y;\hat{\omega},h, \tau), \\
{\rm (v)} & \ \  |R^+ -  {\rm Id}|^2  + |R^- - {\rm Id}|^2 \le C\delta_{\ep}^M(y;\hat{\omega},h, \tau),\\ 
%{\rm (vi)} & \ \ |R^+-R^-|^2\leq C\ep^2
\end{align*}
where we set  $\Gamma^\pm=\omega \times \lbrace s^\pm \rbrace$  for brevity.  

\end{proposition}
\begin{proof}
The statement has been proven in \cite[Proposition 4.12]{davoli.friedrich} in the case in which the bound on $\delta_\ep^M(y;\omega,h,\tau)$ is replaced by a smallness assumption on 
\begin{equation}
\label{eq:eta}
\delta_{\ep}(y;\omega,h,\tau): = \mathcal{E}_{\eps}(y,  D_{\omega,h}\setminus D_{\omega,\tau})  +     (\mathcal{L}^d(D_{\omega,4\tau}))^{-1}  \Vert \nabla y -  \nabla y^+_0  \Vert^2_{L^2(D_{\omega,4\tau})},
\end{equation}
where $y_0^+$ is the map defined right after \eqref{eq: conti-schweizer-k} (see also \cite[Subsection 4.5]{davoli.friedrich}). Since the identifications of $R^\pm$ and $s^\pm$ are completely independent from each other (see also \cite[Remark 4.21]{davoli.friedrich}), the proof of Proposition \ref{lemma: optimal profile} follows by analogous arguments. 
\end{proof}

\begin{remark}[Integrability exponent]
Note that the results in \cite{davoli.friedrich} are proved using the most general formulation of the quantitative rigidity estimate in \cite[Theorem 3.1]{davoli.friedrich}, thus allowing for different integrability exponents $p$, as well as for a smaller penalization $\eta_{\ep,d}<\bar{\eta}_{\ep,d}$ (see \eqref{eq:alphad}). The proposition is stated in its generality in order to ease the reference to \cite{davoli.friedrich}. Under suitable simplifications (see \cite[Remark 4.17]{davoli.friedrich}), analogous estimates hold for $p=2$.  
\end{remark}

The following lemma deals with the transition between a  $(d-1)$-dimensional slice  and a rigid movement.  Recall the  definition  of $c_2$ in  H6.

\begin{lemma}[Transition to a rigid movement]\label{lemma: transition1}
 Let $d\in \mathbb{N}$, $d\geq 2$, and let $M\in \{A,B\}$.  Let $h, \tau,\ep>0$ and $\omega \subset \subset \hat{\omega} \subset \R^{d-1}$ satisfy the assumptions of Proposition \ref{lemma: optimal profile}.  Assume that the elastic energy density $W$ satisfies assumptions H1.--H4.\ and H6.  Let $y\in H^2(D_{\hat{\omega},h};\mathbb{R}^d)$ with $\delta_{\ep}^M(y;\hat{\omega},h,\tau) \le  (\kappa/64)^{2} $  and let $R^+,R^- \in SO(d)$, $s^+ \in (\tau,2\tau)$, $s^- \in (-2\tau, - \tau)$ be the associated rotations and  constants  provided by Proposition \ref{lemma: optimal profile}.  
 Then there exist a map $y^M_+ \in H^2(\omega \times (0,\infty);\R^d )$ and a constant $b^M_+ \in \R^d$ such that 
\begin{align}\label{eq: trans-equ}
{\rm (i)} & \ \  y^M_+ =  y \ \ \text{on} \ \ \omega \times (0, s^+),   \ \ \ \ \
y^M_+(x) = R^+ Mx   + b^M_+ \ \  \text{ for all $x \in \omega \times (s^+ + \tau,\infty)$}, \notag\\
{\rm (ii)} & \ \  \Vert \nabla y^M_+ - R^+M \Vert^2_{L^2(\omega \times (s^+,\infty) )}  \le 
C \eps^2  \delta_{\ep}^M(y;\hat{\omega},h, \tau),   \notag\\
{\rm (iii)} & \ \   \mathcal{E}_{\eps}(y^M_+, \omega \times (s^+,\infty))  \le 
C \delta_{\ep}^M(y;\hat{\omega},h, \tau)
\end{align}
 where $C=C(\omega,\hat{\omega},h,\tau,\kappa,c_1, c_2)>0$.  Analogously, there exist a map ${y}^M_- \in H^2(\omega \times (-\infty,0);  \R^d)$ and a constant $b^M_- \in \R^d$ for which  \eqref{eq: trans-equ} holds  with $s^-$, and $R^-$ in place of $s^+$, and $R^+$, respectively. 
\end{lemma}
\begin{proof}
The result follows directly by \cite[Lemma 4.20]{davoli.friedrich}. Indeed, in \cite[Lemma 4.20]{davoli.friedrich} an analogous result is proven in the case in which the $\eps$-closeness $\delta_\ep^M$ is replaced by the quantity defined in \eqref{eq:eta}. The thesis follows by observing that the constructions around the slices $s^+$ and $s^-$ are independent (see also \cite[Remark 4.21]{davoli.friedrich}).
\end{proof}

After these preparations, we are now in a position to exhibit local recovery sequences performing a double phase transition in an energetically optimal way. 
 
 \begin{proof}[Proof of Proposition \ref{lemma: local2}]
We will prove the result only in the special case that  $\Omega = D_{\omega',h}$. In fact, to treat the general case of  strictly star-shaped sets $\Omega$  and cylindrical sets $D_{\omega',h}$  with $(\partial \omega' \times (-h,h)) \cap \Omega = \emptyset$ one can apply the diagonal argument explained in the proof of Proposition \ref{lemma: local1} in a similar fashion and therefore we omit the details. For simplicity, we will write $\omega$ in place of $\omega'$  in the following.

  Let $M \in \lbrace A,B\rbrace$, let $h>0$, let $\omega \subset \R^{d-1}$ be a bounded Lipschitz domain, and  let $\lbrace w_\eps \rbrace_\eps \in \mathcal{W}_d$. Fix $\rho>0$ and choose a Lipschitz domain $\tilde{\omega}$ such that $\omega\subset\subset \tilde{\omega}$, with $\mathcal{H}^{d-1}(\tilde{\omega}\setminus \omega)\leq \rho$. We first observe that by  Lemma \ref{cor: layer energy}  there exists a sequence $\lbrace y^\eps \rbrace_\eps \subset  H^2(D_{ \tilde{\omega}  ,h}; \R^d)$ such that 
   \begin{equation}
   \label{eq:concentrate}
{\lim_{\eps \to 0} \Vert  y^\eps-  Mx \Vert_{H^1(D_{\tilde{\omega},h})} = 0 , \ \   \ \ \frac{y^{\ep}-Mx}{w_{\ep}}\to y^M_{\rm dp}\quad \text{ in measure on } D_{\tilde{\omega},h}},
 \end{equation}
where $y^M_{\rm dp}$ is the function defined in \eqref{eq: step functions}, as well as
\begin{equation}
   \label{eq:concentrate2}
\lim_{\eps \to 0}{\mathcal{E}_{\eps}(y^{\eps}, D_{\tilde{\omega},h}) =  2K\mathcal{H}^{d-1}(\tilde{\omega}), \ \ \  \quad   \lim_{\eps \to 0}\mathcal{E}_{\eps}(y^{\eps}, D_{\tilde{\omega},h} \setminus D_{\tilde{\omega}, h/16}) = 0.}
 \end{equation}
In view of  Lemma \ref{cor: layer energy}, the existence of a sequence $\{y^{\ep_i}\}_i$ satisfying \eqref{eq:concentrate}--\eqref{eq:concentrate2} is guaranteed for every $\{\ep_i\}_i$ with $\ep_i\to 0$. Hence, in what follows, for notational simplicity we directly work with the continuous parameter $\ep$.

 Fix  $\tau=h/8$.  By \eqref{eq:eta2} and  \eqref{eq:concentrate}--\eqref{eq:concentrate2} we find that $\delta_\ep^M(y^\ep;\tilde{\omega},h,\tau)\to 0$ as $\ep\to 0$. Without loss of generality we can assume that $\ep<\ep_0$ (see Proposition \ref{lemma: optimal profile}) and $\delta_\ep^M(y^\ep;\tilde{\omega},h,\tau)\leq (\kappa/64)^2$.  Applying Proposition \ref{lemma: optimal profile}  to $\{y^\ep\}_\ep$ for $\hat{\omega} = \tilde{\omega}$, we find sequences of rotations $\{R^+_\ep\}_\ep,\,\{R^-_\ep\}_\ep\subset SO(d)$, and of slices $\{s^+_\ep\}_\ep\subset ( \tau  ,2\tau)$, and $\{s^-_\ep\}_\ep\subset (-2\tau, - \tau)$. 
 %\RRR{\tt see comment above. Maybe the proposition is not needed anymore. Note that a slightly changed the set where the slices are taken. Should not make any difference up to changing some constants... Alternative: Use $\tau = h/8$ in \eqref{eq:eta2}, but also this needs to be explained to the reader.} \EEE 
 Let now $\{y_{\ep,\pm}^{ M}\}_\ep$ be the maps provided by Lemma \ref{lemma: transition1}. We define $v_\ep^{M} \in H^2(D_{\omega,h};\R^d)$ by
 \begin{equation}
 \label{eq:vaep}
 v^M_\ep(x):=\begin{cases}y^{M}_{\ep,+}&\text{if }x_d\geq s^+_\ep,\\y^\ep&\text{if }s^-_\ep\leq x_d\leq s^+_\ep,\\ y^{M}_{\ep,-}&\text{if }x_d\leq s^-_\ep,\end{cases}
 \end{equation}
 for every $x\in D_{{\omega},h}$.  We proceed by checking that $\{v^M_{\ep}\}_\ep$ satisfies \eqref{eq:local2-1}--\eqref{eq:local2-2}. First, since $|s^{\pm}_\ep|\le 2\tau$ and $\tau=  h/8 $, by Lemma \ref{lemma: transition1} we have that $v^M_\ep= I^M_{1,\ep} \circ Mx$ and $v^M_\ep= I^M_{2,\ep}\circ Mx$ for $x_d\geq 3h/  8$ and $x_d\leq -3h/  8$, respectively, for two suitable sequences of isometries $\{I^M_{1,\ep}\}_\ep, \{I^M_{2,\ep}\}_\ep$. This yields the second part of \eqref{eq:local2-2}. For brevity, we define the sets $F_{\omega,h}^+ = \omega \times (h/16,h)$ and  $F_{\omega,h}^- = \omega \times (-h,-h/16)$. A key step will be to show that for $\eps \to 0$
\begin{align}\label{eq: last thing to show}
w_\eps^{-1} (v^M_\eps - Mx)   \to y^M_{\rm dp} \quad \text{in measure on  $F_{\omega,h}^- \cup F_{\omega,h}^+$.}
\end{align} 
 This along with \eqref{eq:concentrate} and the fact that $v^M_\eps = y^\eps$ on $D_{\omega,h/8}$ then shows \eqref{eq:local2-1}. Moreover, note that \eqref{eq: last thing to show} also implies that the isometries $\{I^M_{1,\ep}\}_\ep$ and $\{I^M_{2,\ep}\}_\ep$ converge to the identity as $\eps \to 0$.
 
Let us now show \eqref{eq: last thing to show}. We only show the result on $F_{\omega,h}^+$ as the argument on $F_{\omega,h}^-$ is analogous. Moreover, it clearly suffices to prove the property for any subsequence as then convergence holds for the whole sequence by Urysohn's property.  First, we note that  $\mathcal{E}_\eps(v^M_\eps,F_{\omega,h}^+) \to 0$ as $\eps \to 0$ by  Lemma  \ref{lemma: transition1}(iii), \eqref{eq:concentrate2},  \eqref{eq:vaep}, and the fact that $\delta_\ep^M(y^\ep;\tilde{\omega},h, \tau )\to 0$. Then, applying the compactness result and the lower bound for $\Omega = F_{\omega,h}^+$ (see Theorem \ref{thm:compactness} and Theorem \ref{thm:liminf}) we find a subsequence (not relabeled) and $(y,u,\mathcal{P})\in \mathcal{A}$ such that $v^M_\eps \to (y,u,\mathcal{P})$  and  $\mathcal{E}_0^{\mathcal{A}}(y,u,\mathcal{P}) = 0$, where here the limiting energy $\mathcal{E}_0^{\mathcal{A}}$ defined in \eqref{eq: limiting energy}  has to be understood with respect to the set $F_{\omega,h}^+$. 

In view of \eqref{eq: limiting energy} and $\mathcal{E}_0^{\mathcal{A}}(y,u,\mathcal{P}) = 0$, we find that $\mathcal{P}$ is trivial, consisting just of the component $F_{\omega,h}^+$. Moreover,  $\nabla y$ is constant, and then $\nabla y = M$ by  \eqref{eq:comp-y},  \eqref{eq:concentrate}, and the fact that $v^M_\eps = y^\eps$ on $G_{\omega,h}^+  := \omega \times (h/16,h/8)$. (Recall that $s_\eps^+ \ge \tau = h/8$.) As $\mathcal{E}_0^{\mathcal{A}}(y,u,\mathcal{P}) = 0 $ and $F \mapsto \mathcal{Q}_{\rm lin}(M,FM)$ is positive definite on $\mathbb{M}^{d \times d}_{\rm sym}$ (see \eqref{eq: only symmetric}) we also get that $u$ is affine on $F_{\omega,h}^+$ and has the form $u(x) = SMx + s $ for each $x \in F_{\omega,h}^+$, where $S \in \mathbb{M}^{d \times d}_{\rm skew}$ and $s \in \R^d$. Moreover, in view of \eqref{eq: rescaled disp}--\eqref{eq:comp-u},  we find $\lbrace t^\eps\rbrace_\eps \subset \R^d$ and $\lbrace\bar{R}^\eps\rbrace_\eps \subset SO(d)$ such that
\begin{align}\label{eq: measure1}
\eps^{-1} \big(v^M_\eps - (\bar{R}^\eps M x + t^\eps)\big) \to u \quad \text{ in measure in $F_{\omega,h}^+$.}
\end{align}
On the other hand,  by \eqref{eq:concentrate} and the fact that $v^M_\eps = y^\eps$ on $G_{\omega,h}^+   = \omega \times (h/16,h/8)$  we have 
\begin{align}\label{eq: measure2}
w_\eps^{-1} (v^M_\eps -M x) \to y^M_{\rm dp}  \quad \text{ in measure in $G_{\omega,h}^+$}.
\end{align}
Passing to another subsequence (not relabeled) we can assume that $\lambda: = \lim_{\eps \to 0} \eps/w_\eps$ exists, cf.\ \eqref{ew: W sequence}. By multiplying \eqref{eq: measure1}  with $\eps/w_\eps$  and by subtracting \eqref{eq: measure2}     we get
\begin{align*}
w_\eps^{-1} \big(Mx - (\bar{R}^\eps Mx + t^\eps)    \big) \to \lambda u  - y^M_{\rm dp} \quad \text{ in measure in $G_{\omega,h}^+$.}
\end{align*} 
As the mappings  in the left-hand side, as well as  $u$ and $y^M_{\rm dp}$  are affine, this convergence holds also on the larger set $F_{\omega,h}^+$. This along with \eqref{eq: measure1} yields  
$$w_\eps^{-1} (v^M_\eps -M x) \to \lambda u - (\lambda u -  y^M_{\rm dp}) =  y^M_{\rm dp} \quad   \text{ in measure on $F_{\omega,h}^+$.} $$ 
This concludes the proof of \eqref{eq: last thing to show}. To conclude, it remains to show the asymptotic behavior of the energies in \eqref{eq:local2-2}. Using \eqref{eq: k2-intro}, \eqref{eq:local2-1}, and   Proposition \ref{eq: KundKdo-new}, it follows that $\liminf_{\ep\to 0}\mathcal{E}_\ep(v_\ep^{M},D_{\omega,h})\geq 2K\mathcal{H}^{d-1}(\omega)$. To prove the opposite inequality, we observe that by \eqref{eq:vaep} and Lemma \ref{lemma: transition1}(iii) there holds
 \begin{align*}
 \mathcal{E}_\ep\big(v^M_\ep,D_{\omega,h}\big)&\leq \mathcal{E}_\ep\big(y^M_{\ep,+},\omega\times (s^+_\ep,h)\big)+\mathcal{E}_\ep\big(y^M_{\ep,-},\omega\times (-h, s^-_\ep)\big)+\mathcal{E}_\ep\big(y^\ep,\omega\times (s^-_\ep, s^+_\ep)\big)\\
&\leq C\delta_\ep^M(y^\ep,\tilde{\omega}, h, \tau)+\mathcal{E}_\ep(y^\ep, D_{\tilde{\omega},h}).
\end{align*}
Thus, by  \eqref{eq:concentrate2}, the fact that $\delta_\ep^M(y^\ep;\tilde{\omega},h,\tau)\to 0$, and $\mathcal{H}^{d-1}(\tilde{\omega}\setminus \omega)\leq \rho$,  we have
$$\limsup_{\ep\to 0}\mathcal{E}_\ep(v_\ep^{M},D_{\omega,h})\leq 2K\mathcal{H}^{d-1}(\tilde{\omega})\leq 2K \mathcal{H}^{d-1}({\omega})+2K\rho.$$
The convergence in \eqref{eq:local2-2} follows then by the arbitrariness of $\rho$ and by a diagonal argument. 
 \end{proof}

 \subsection{One-dimensional profiles and compatibility condition}
 \label{subs:1d}
 
In this subsection we assume that the density $W$ satisfies  \eqref{eq: isotropy}. We will show that in this case optimal profiles for single transitions are one-dimensional in a sense to make precise below. Moreover, we show that the {compatibility condition}  $K^A_{\rm dp} = K^B_{\rm dp} = 2K$  holds. Let us start by discussing a model case for \eqref{eq: isotropy}, see \eqref{eq: modelli}. Suppose that $W$ is of the form 
$$
W(F) =  \phi\big(  {\rm dist}(F, SO(d)A),   {\rm dist}(F, SO(d)B) \big) \quad \quad \quad \text{for all } F \in \mathbb{M}^{d \times d},
$$ 
where $\phi\colon ([0,\infty))^2 \to [0,\infty)$ is a smooth function with  $c_1 (\min \lbrace t_1,t_2 \rbrace)^2 \le \phi(t_1,t_2)  \le c_2 (\min \lbrace t_1,t_2 \rbrace)^2$ for all $t_1,t_2 \in [0,\infty)$, and  is increasing in both entries. Then, we can check that H1.--H6.\ hold. Moreover, also H7.\ is satisfied if $\phi$ fulfills a corresponding local Lipschitz condition. We can also confirm  \eqref{eq: isotropy}. Indeed, for  each $F \in \mathbb{M}^{d \times d}$,  by H3., the monotonicity assumptions on $\phi$, and the triangle inequality we compute
\begin{align}\label{eq: model case for details}
W(F)  & = \phi\big(  {\rm dist}(F, SO(d)A),   {\rm dist}(F, SO(d)B) \big) = \phi\Big(  \min_{R \in SO(d)}|F - RA|,   \min_{R \in SO(d)}|F - RB|\Big)\notag \\ 
& \ge \phi\Big(  \min_{R \in SO(d)}|Fe_d - RA e_d|,   \min_{R \in SO(d)}|F e_d - RB e_d|\Big)  \ge \phi\Big(  \big||Fe_d| - |A e_d|\big|,   \big||F e_d| - |B e_d|\big|\Big)\notag \\ 
& = \phi\Big(  \big||Fe_d| - 1\big|,   \big||F e_d| - (1+\kappa)\big|\Big) =  \phi\Big(  \big|{\rm Id} + (|F e_d| - 1)e_{dd} - A\big|,   \big|{\rm Id} + (|F e_d| - 1)e_{dd} - B\big|\Big)\notag\\
& \ge \phi\Big(  {\rm dist}\big({\rm Id} + (|F e_d| - 1)e_{dd}, SO(d)A\big),   {\rm dist}\big({\rm Id} + (|F e_d| - 1)e_{dd}, SO(d)B\big) \Big)\notag \\ 
&= W\big({\rm Id} + (|F e_d| - 1)e_{dd}\big).
\end{align}
We now check that under condition \eqref{eq: isotropy} optimal profiles for single transitions are one-dimensional.

\begin{lemma}[One-dimensional profiles]\label{lemma: 1d}
Under condition  \eqref{eq: isotropy}, there holds
\begin{align}\label{new-K}
K =\inf\Big\{&\liminf_{\ep\to 0} \mathcal{E}_{\ep}(y^{\ep},Q)\colon \ y^\eps(x) = (x',\psi^\eps(x_d)) \text{ {\rm for} $x =(x',x_d)  \in Q$},  \ \ \   \lim_{\eps \to 0}  \Vert  y^\eps -  y_0^+ \Vert_{L^1(Q)}  = 0\Big\},
\end{align}
where $K$ is defined in \eqref{eq: our-k1}.
\end{lemma}

\begin{proof}
We denote the right-hand side of \eqref{new-K} by $K_{1d}$. Clearly, we get $K_{1d} \ge K$. To see the reverse inequality, by a standard diagonal argument,  we choose a sequence $\lbrace y^\eps \rbrace_\eps \subset H^2(Q;\R^d)$ with $\lim_{\eps \to 0}  \Vert  y^\eps -  y_0^+ \Vert_{L^1(Q)}  = 0$ and   
$$\liminf_{\ep\to 0} \mathcal{E}_{\ep}(y^{\ep},Q)  = K. $$
Then, by Fatou's lemma, and by Lemma \ref{lemma:comp-def}, we can find $x' \in (-\frac{1}{2},\frac{1}{2})^{d-1}$ such that 
\begin{align}\label{eq: first-Fatou}   
\liminf_{\eps \to 0}\int_{-\frac{1}{2}}^{\frac{1}{2}}  \Big( \frac{1}{\ep^2}W(\nabla y^\eps(x',t))+\ep^2|\nabla^2 y^\eps(x',t)|^2+\bar{\eta}^2_{\ep,d} (|\nabla^2 y^\eps(x',t)|^2-|\partial^2_{dd} y^\eps(x',t)|^2) \Big) \,{\rm d}t \le K  
\end{align}
as well as 
\begin{align}\label{eq: second-Fatou}  
\lim_{\eps \to 0 } \Big(  \int_{-\frac{1}{2}}^{0}  |\nabla y^\eps(x',t)  - B |^2 \, {\rm d}t +\int_{0}^{\frac{1}{2}}  |\nabla y^\eps(x',t)  - A |^2 \, {\rm d}t  \Big) = 0.  
\end{align}
We let $\tau^\eps:= \partial_d y^\eps(x',\cdot) \in H^1((-\frac{1}{2},\frac{1}{2});\R^d)$ and we choose the unique function $\psi^\eps\colon (-\frac{1}{2},\frac{1}{2}) \to \R$ with  $\psi^\eps(0) = 0$ and $(\psi^\eps)' = |\tau^\eps|$. Then, we define the sequence $\lbrace v^\eps\rbrace_\eps \subset H^2(Q;\R^d)$ by $v^\eps(x',x_d) = (x',\psi^\eps(x_d))$ for $(x',x_d) \in Q$. We observe that 
\begin{align}\label{eq: v-grad}
\nabla v^\eps (x) = \sum\nolimits_{i=1}^{d-1} e_{ii}  +  |\tau^\eps(x_d)| e_{dd}.
\end{align}
We note that $\lbrace v^\eps\rbrace_\eps$ is an admissible sequence in the definition of $K_{1d}$. Indeed, by  H3., \eqref{eq: second-Fatou}, \eqref{eq: v-grad},   and the triangle inequality we find  
\begin{align*}
\int_Q |\nabla v^\eps- \nabla y_0^+|^2\, {\rm d}x & = \int_{Q \cap \lbrace x_d \le 0\rbrace} |\partial_d v^\eps   -  B  e_d   |^2 \, {\rm d}t  + \int_{Q \cap \lbrace x_d \ge 0 \rbrace} |\partial_d v^\eps  -  A e_d   |^2 \, {\rm d}t  \\
&= \int_{-\frac{1}{2}}^{0} \big| |\partial_d y^\eps(x',t) |  -  |B e_{d}|  \big|^2 \, {\rm d}t  + \int_{0}^{\frac{1}{2}} \big| |\partial_d y^\eps(x',t)| -  |A e_{d}| \big  |^2 \, {\rm d}t \\
& \le   \int_{-\frac{1}{2}}^{0}  \big| (\nabla y^\eps(x',t)  - B ) e_d \big|^2 \, {\rm d}t +\int_{0}^{\frac{1}{2}}  \big|(\nabla y^\eps(x',t)  - A)e_d \big|^2 \, {\rm d}t    \to 0,   
\end{align*}
and therefore also $v^\eps \to y_0^+$ in $L^1(Q;\R^d)$ since $v^\eps(0) = 0$ for all $\eps$. Consequently, in view of \eqref{eq: isotropy}, \eqref{new-K}, \eqref{eq: first-Fatou},  \eqref{eq: v-grad}, and the fact that $\frac{\rm d}{{\rm d}t}|\tau^\eps|(t) \le |\partial_{dd} y^\eps(x',t)| $ for $t \in (-\frac{1}{2},\frac{1}{2})$  we get
\begin{align*}
K_{1d} &\le \liminf_{\eps \to 0} \mathcal{E}_\eps(v^\eps,Q)  = \liminf_{\eps \to 0} \int_{Q}  \Big( \frac{1}{\ep^2}W(\nabla v^\eps)+\ep^2|\partial^2_{dd} v^\eps|^2 \Big) \,{\rm d}x  \\
 &  \le \liminf_{\eps \to 0} \int_{-\frac{1}{2}}^{\frac{1}{2}}  \Big( \frac{1}{\ep^2}W\Big( {\rm Id} +  (|\nabla y^\eps(x',t)  e_d|-1) e_{dd}  \Big)+\ep^2\Big|\frac{\rm d}{{\rm d}t}|\tau^\eps|(t)  \Big|^2 \Big) \,{\rm d}t \\ 
 & \le  \liminf_{\eps \to 0} \int_{-\frac{1}{2}}^{\frac{1}{2}}  \Big( \frac{1}{\ep^2}W\big(\nabla y^\eps(x',t) \big)+\ep^2|\partial^2_{dd}y^\eps(x',t)|^2 \Big) \,{\rm d}t  \le K.  
\end{align*}
This concludes the proof. 
\end{proof}

We point out that without an additional assumption, such as \eqref{eq: isotropy},  optimal profiles for single transitions are in general not one-dimensional, see \cite[Remark 6.2]{conti.schweizer2} for an example in a linearized setting. We are now in a position to prove Proposition \ref{prop: compat}. 

\begin{proof}[Proof of Proposition \ref{prop: compat}]
We start with a consequence of Lemma  \ref{lemma: 1d}. Define $\tilde{W}\colon \R \to \R$ by $\tilde{W}(t) = W({\rm Id} + (t-1)    e_{dd}  )$ for $t \in \R$. Note that $\tilde{W}$ is a two-well potential with $\tilde{W}(t) = 0$ if and only if $t \in \lbrace 1, 1+\kappa\rbrace$, see H3.    In view of \eqref{eq: nonlinear energy} and \eqref{new-K}, we find
\begin{align*}
 K =\inf\Big\{&\liminf_{\ep\to 0} \int_{-\frac{1}{2}}^{\frac{1}{2}} \Big(\frac{1}{\eps^2}  \tilde{W}\big(\psi'_\eps\big) + \eps^2|\psi_\eps''|^2\Big) \, {\rm d}t    \colon \  \psi^\eps \in H^2((-\tfrac{1}{2},\tfrac{1}{2});\R),  \ \ \   \lim_{\eps \to 0}  \Vert   \psi^\eps  -  \tilde{y}_0^+ \Vert_{L^2(-\tfrac{1}{2},\tfrac{1}{2})}  = 0\Big\},
\end{align*}
where $\tilde{y}_0^+(t) := t \chi_{\lbrace t > 0 \rbrace} + (1+\kappa)t \chi_{\lbrace t < 0\rbrace}$ for $t \in (-\frac{1}{2},\frac{1}{2})$. By a cut-off argument one can further show that (see e.g. \cite[Proof of Proposition 5.3]{conti.fonseca.leoni}  for details)  
\begin{align}\label{eq: strangeKdef}
 K =\inf\Big\{\liminf_{\ep\to 0} \int_{-\frac{1}{2}}^{\frac{1}{2}} &\Big(\frac{1}{\eps^2} \tilde{W}\big(\psi'_\eps\big) + \eps^2|\psi_\eps''|^2\Big)\,  {\rm d}t    \colon \  \psi^\eps \in H^2((-\tfrac{1}{2},\tfrac{1}{2});\R),  \ \   \lim_{\eps \to 0}  \Vert \psi^\eps  -  \tilde{y}_0^+ \Vert_{L^2(-\tfrac{1}{2},\tfrac{1}{2})}  = 0, \notag\\  & \quad \quad  \quad \quad  \quad \quad  \quad \quad   \psi_\eps'(t) = 1 + \kappa \text{ near $t = -\tfrac{1}{2}$}, \ \psi_\eps'(t) = 1 \text{ near $t = \tfrac{1}{2}$}                \Big\}.
\end{align}
We now start with the proof. We prove the result only for $M=A$. The arguments for $M=B$ are similar up to a different notational realization. Let $Q' = (-\frac{1}{2},\frac{1}{2})^{d-1}$.  Fix $\delta >0$. In view of \eqref{eq: our-k2}, we choose  $h>0$ and $\lbrace w_\eps\rbrace_\eps \in \mathcal{W}_d$ such that 
\begin{align}\label{eq: our-k2-NNN}
 K^A_{\rm dp} -\delta \le \inf\Big\{&\limsup_{\ep\to 0}  \mathcal{E}_{\ep}\big(y^{\ep}, D_{Q',h}\big)\colon \,             w_\eps^{-1}(y^\eps -  x) \to   y_{\rm dp}^A  \text{ in measure in $D_{Q',h}$}\Big\},
\end{align} 
where we recall the notations in \eqref{eq: step functions} and  \eqref{eq:def-dlh}. \EEE We start by observing that it suffices to show that there exists a sequence $\lbrace z_\eps \rbrace_\eps \subset H^2((-h,h);\R)$ such that 
\begin{align}\label{eq: oned-goal}
{\rm (i)} & \ \ w_\eps^{-1}\big( z_\eps  -  {\rm id}) \to \chi_{\lbrace t>0\rbrace} \quad \quad \text{ in measure in  $(-h, h)$}, \notag \\
{\rm (ii)}& \ \    \limsup_{\eps \to 0} \int_{-h}^{h} \Big(\frac{1}{\eps^2}  \tilde{W}\big(z_\eps'\big) + \eps^2|z_\eps''|^2 \Big) \,  {\rm d}t  \le 2K + \delta.       
\end{align}
In fact, then the sequence $y^\eps \in H^2(D_{Q',h};\R^d)$ defined by $y^\eps(x',x_d) = (x',z_\eps(x_d))$ clearly satisfies 
$$w_\eps^{-1} ( y^\eps - x)   \to  y_{\rm dp}^A \quad \quad \text{in measure in $D_{Q',h}$}$$
 by \eqref{eq: oned-goal}(i). Therefore, it is  an admissible sequence in \eqref{eq: our-k2-NNN}, and thus $ \limsup_{\eps \to 0} \mathcal{E}_\eps(y^\eps,D_{Q',h}) \ge  K^A_{\rm dp} -\delta$. By \eqref{eq: nonlinear energy},    \eqref{eq: oned-goal}(ii), and the definition of $\tilde{W}$, we also have  $\limsup_{\eps \to 0} \mathcal{E}_\eps(y^\eps,D_{Q',h})\le 2K+\delta$. Thus,  $K_{\rm dp}^{A}-\delta \le 2K+\delta$ and therefore  $K_{\rm dp}^{A} \le 2K$ holds by passing to $\delta \to 0$. The other inequality $K_{\rm dp}^{A}  \ge 2K$ follows from Proposition \ref{prop:cell-form-chaper2}.

We now construct a sequence $\lbrace z_\eps \rbrace_\eps \subset H^2((-h,h);\R)$ satisfying \eqref{eq: oned-goal}. 
Given $\delta>0$,  we use \eqref{eq: strangeKdef} to find $\eps_0>0$ and  a function $\psi \in H^2((-\frac{1}{2},\frac{1}{2});\R)$ such that
\begin{align}\label{eq: 1den}
\int_{-\frac{1}{2}}^{\frac{1}{2}} \Big(\frac{1}{\eps_0^{2}} \tilde{W}\big(\psi'\big) + \eps^2_0|\psi''|^2\Big) \, {\rm d}t \le K + \delta/2, 
\end{align}
as well as $\psi'(t) = 1+ \kappa $ near $t =-\frac{1}{2}$ and $\psi'(t) = 1 $ near $t =\frac{1}{2}$.  Let $\eps>0$ sufficiently small and let  $\epsilon := \eps /\eps_0$ for brevity. We define $z_\eps \in H^2((-h,h);\R)$ as the continuous function with $z_\eps(0)=0$ and the derivative
\begin{align*}
z_\eps' (t) := \begin{cases}
1 & \text{ if }  t  \in   (-h,  - \epsilon^2) \\      
\psi'\big( \tfrac{1}{\epsilon^2} (-\tfrac{1}{2}\epsilon^2-t)  \big)    & \text{ if }  t  \in   ( - \epsilon^2, 0) \\ 
1 + \kappa  & \text{ if }  t  \in   (0, w_\eps^\kappa) \\ 
\psi'\big( \tfrac{1}{\epsilon^2} (t- w_\eps^\kappa -  \tfrac{1}{2}\epsilon^2)  \big)    & \text{ if }  t  \in   (w_\eps^\kappa, w_\eps^\kappa + \epsilon^2)\\ 
1 & \text{ if }  t  \in   (w_\eps^\kappa + \epsilon^2, h)
\end{cases}
\end{align*}
for $t \in (-h,h)$, where for shorthand we write $w_\eps^\kappa = w_\eps/\kappa$. Indeed, we note that $z_\eps'$  is continuous since $\psi'$ is constant near $t =-\frac{1}{2}$ and $t =\frac{1}{2}$.
By using $\tilde{W}(t) = 0$ for $t \in \lbrace 1, 1+\kappa \rbrace$ and \eqref{eq: 1den}, it is not hard to check that
\begin{align*}
\int_{-h}^{h} \Big(\frac{1}{\eps^2}  \tilde{W}\big(z_\eps'\big) + \eps^2|z_\eps''|^2\Big) \,  {\rm d}t 
 &= 2 \int_{-\epsilon^2/2}^{\epsilon^2/2}   \Big( \frac{1}{\eps^{2}} \tilde{W}\big(  \psi' (t / \epsilon^2  )     \big) + \frac{\eps^2}{\epsilon^{4}} \big|\psi''(t/\epsilon^2)\big|^2\Big) \,  {\rm d}t  \\
 &  = 2 \int_{-\frac{1}{2}}^{\frac{1}{2}}   \Big( \frac{1}{\eps_0^{2}} \tilde{W}\big(  \psi' (s)    \big) + \eps^2_0 |\psi''(s)|^2\Big) \,  {\rm d}s  \le 2K+\delta,
\end{align*} 
where in the second step we used a change of variables and $\epsilon = \eps /\eps_0$. This shows \eqref{eq: oned-goal}(ii). We now prove \eqref{eq: oned-goal}(i).  As by a scaling argument we have
$$\Vert z_\eps'\Vert_{L^1((-\epsilon^2,0))} +   \Vert z_\eps'\Vert_{L^1((w_\eps^\kappa, w_\eps^\kappa + \epsilon^2))}  \le 2\epsilon^2 \int_{-\frac{1}{2}}^{\frac{1}{2}}| \psi'| \, {\rm d}t \le C\eps^2,        $$
we get that
$$ \Vert z_\eps' - \tilde{z}_\eps' \Vert_{L^1((-h,h))} \le C\eps^2, $$
where $\tilde{z}_\eps$ denotes the continuous piecewise affine function with  $\tilde{z}_\eps(0) = 0$,  $\tilde{z}_\eps' = 1$ on $(-h,0) \cup (w_\eps^\kappa, h)$, and    $\tilde{z}_\eps' = 1+\kappa$ on $(0,w_\eps^\kappa)$. By Poincar\'e's inequality  and  ${z}_\eps(0) = \tilde{z}_\eps(0) = 0$ this also yields 
\begin{align}\label{eq: after Poinc}
 \Vert z_\eps - \tilde{z}_\eps \Vert_{L^1((-h,h))} \le C\eps^2.
 \end{align} 
Since $w_\eps \to 0$ as $\eps\to 0$ and $w_\eps^\kappa = w_\eps/\kappa$, it is easy to check that $w_\eps^{-1} (\tilde{z}_\eps - {\rm id}) \to \chi_{\lbrace t>0\rbrace}$ in measure in $(-h,h)$. This along with \eqref{eq: after Poinc} and the fact that $\eps^2 /w_\eps \to 0$ as $\eps \to 0$ (see \eqref{ew: W sequence}) implies \eqref{eq: oned-goal}(i). This concludes the proof.       
\end{proof}

 \appendix

\section{$SBV$ functions and Caccioppoli partitions}
\label{sec:appendix}
 Let $d\in \N$, and let $\Omega\subset \R^d$ be a bounded open set. In the whole paper we use standard notations for the space $BV(\Omega)$. We refer the reader to \cite{Ambrosio-Fusco-Pallara:2000}  for definitions and main properties. We discuss here only some basic properties of special functions of bounded variation $(SBV)$ and Caccioppoli partitions. 

%-----------------------------------------------------------------------------------------------------------------
\subsection*{Special functions of bounded variation}\label{sec: SBV}
Let $m \in \N$. We say that a function $u\in BV(\Omega; \R^m)$ is a \textit{special function of bounded variation}, i.e., $u\in SBV(\Omega;  \R^m  )$, if the Cantor part of its gradient (see \cite[Definition 3.91]{Ambrosio-Fusco-Pallara:2000}) satisfies
$$D^c u=0.$$
In particular, for every $u\in SBV(\Omega;  \R^m  )$ there holds
$$Du=\nabla u\, \mathcal{L}^d+(u^+-u^-)\otimes \nu_u\mathcal{H}^{d-1}\lfloor J_u,$$
where $\nabla u$ is the approximate differential, $u^+$ and $u^-$ are the approximate one-sided limits, $J_u$ is the jump set of $u$, and $\nu_u$ is the normal to $J_u$ (see \cite[Chapter 3]{Ambrosio-Fusco-Pallara:2000}).

A function $u\in L^1_{\rm loc}(\Omega;\R^m)$ (namely $u\in L^1(K;\R^m)$ for every compact set $K\subset \Omega$) is a \textit{special function of locally bounded variation}, i.e., $u\in SBV_{\rm loc}(\Omega;\R^m)$, if $u\in SBV(O;\R^m)$ for every open set $O \subset \subset \Omega$. 

We stress that  $SBV(\Omega;\R^m)$ is a proper subset of $BV(\Omega;\R^m)$, see \cite[Corollary 4.3]{Ambrosio-Fusco-Pallara:2000}.   The set $SBV^2(\Omega;\R^m)$ is defined as the collection of maps $u\in SBV(\Omega;\R^m)$ such that $\nabla u\in L^2(\Omega;\R^{m \times d})$ and $\mathcal{H}^{d-1}(J_u)<+\infty$. 

 %------------------------------------------------------------------------------------------------------------------
\subsection*{Sets of finite perimeter and Caccioppoli partitions}\label{sec: caccio}
%------------------------------------------------------------------------------------------------------------------
%------------------------------------------------------------------------------------------------------------------

 For every $\mathcal{L}^d$-measurable set $E\subset \R^d$ and every $t\in [0,1]$, we denote by $E^t$ the set of points of $E$ having density $t$, namely
$$E^t:=\big\{x\in E\colon\,\lim_{\rho\to 0} \mathcal{L}^d (E\cap B_\rho(x))/\mathcal{L}^d(B_\rho(x))=t\big\}.$$
The \emph{essential boundary of} $E$, denoted by $\partial^\ast E$, is defined as $\partial^\ast E:=\R^d\setminus (E^0\cup E^1)$. We say that $E$ has finite perimeter if $\mathcal{H}^{d-1}(\partial^* E) < + \infty$. We refer the reader to \cite[Subsections 3.3 and 3.5]{Ambrosio-Fusco-Pallara:2000} for basic properties. Moreover, a partition $\mathcal{P} = \{P_j\}_j$ of $\Omega$ is called a \textit{Caccioppoli partition} if 
$$
\sum\nolimits_j \mathcal{H}^{d-1}(\partial^* P_j) < + \infty.
$$
%where $\partial^* P_j$ denotes  the \emph{essential boundary} of $P_j$ (see \cite[Definition 3.60]{Ambrosio-Fusco-Pallara:2000}). 
We say that a partition is \textit{ordered} if $\mathcal L^d\left(P_i\right) \ge \mathcal L^d\left(P_j\right)$ for $i \le j$, and recall that every Caccioppoli partition of a bounded domain induces an ordered one just by a permutation of the indices.  \

We say that a set of finite perimeter $E$ is \emph{indecomposable} if it cannot be written as $E^1 \cup E^2$ with $E^1 \cap E^2 = \emptyset$, $\mathcal{L}^d(E^1), \mathcal{L}^d(E^2) >0$ and $\mathcal{H}^{d-1}(\partial^* E) = \mathcal{H}^{d-1}(\partial^* E^1) + \mathcal{H}^{d-1}(\partial^* E^2)$. Note  that this notion generalizes the concept of  connectedness to sets of finite perimeter. By \cite[Theorem 1]{Ambrosio-Morel}  for each set of finite perimeter $E$  there exists a unique finite or countable family of pairwise disjoint indecomposable sets $\lbrace E_i\rbrace_i$ such that $\mathcal{H}^{d-1}(\partial^* E) = \sum_i \mathcal{H}^{d-1}(\partial^* E_i)$. The set   $\lbrace E_i\rbrace_i$ are called the \emph{connected components} of $E$.

We conclude this section by stating a compactness result for ordered Caccioppoli partitions (see \cite[Theorem 4.19, Remark 4.20]{Ambrosio-Fusco-Pallara:2000}).

\begin{theorem}[Compactness for Caccioppoli partitions]\label{th: comp cacciop}
Let $\Omega \subset \R^d$ be a bounded open set with Lipschitz boundary. Let $\mathcal{P}_i = \{P_{j,i}\}_j$, $i \in \N$, be a sequence of ordered Caccioppoli partitions of $\Omega$ with 
$$\sup\nolimits_{ i \in \N} \Big\{\sum\nolimits_{j}\mathcal{H}^{d-1}(\partial^* P_{j,i})\Big\} < + \infty.$$
Then, there exists a Caccioppoli partition $\mathcal{P} = \{P_j\}_j$ and a not relabeled subsequence such that $P_{j,i} \to  P_j $ in measure for all $j \in \N$   as $i \to \infty$.
\end{theorem}

In the proofs, we also sometimes use the fact that $P_{i,j} \to P_j$ in measure for  all $j \in \N$ is equivalent to $\sum_j \mathcal{L}^d(P_{i,j}\triangle P_j) \to 0$.

\section*{Acknowledgements}
We would like to thank {\sc Roberto Alicandro}, {\sc Giuliano Lazzaroni}, and {\sc Ben Schweizer} for interesting discussions on the topics of this paper. The support by the Alexander von Humboldt Foundation is gratefully acknowledged. This work was supported by the DFG projects FR 4083/1-1,  FR 4083/3-1,  and  by the Deutsche Forschungsgemeinschaft (DFG, German Research Foundation) under Germany's Excellence Strategy EXC 2044 -390685587, Mathematics M\"unster: Dynamics--Geometry--Structure. This work has been funded by  the Vienna Science and Technology Fund (WWTF) through Project MA14-009,  as well as by the Austrian Science Fund (FWF) projects F65, V 662 N32, and I 4052 N32, and by BMBWF through the OeAD-WTZ project CZ04/2019.

\end{document}